\documentclass[12pt]{amsart}
\usepackage{lmodern}
\usepackage{ifthen}
\usepackage{amsfonts}
\usepackage{amsxtra}
\usepackage{amssymb}
\usepackage{array}
\usepackage{eucal}
\usepackage[margin=1.0in]{geometry}
\usepackage{xcolor}
\definecolor{cite}{rgb}{0.30,0.60,1.00}
\definecolor{url}{rgb}{0.00,0.00,0.80}
\definecolor{link}{rgb}{0.40,0.10,0.20}
\usepackage[colorlinks,linkcolor=link,urlcolor=url,citecolor=cite,breaklinks]{hyperref}
\usepackage{bbm}
\usepackage{mathtools}
\usepackage{mathrsfs}
\usepackage{appendix}
\usepackage[all]{xy}
\usepackage[lite,abbrev,msc-links,alphabetic]{amsrefs}
\usepackage{enumitem}

\usepackage[OT2,T1]{fontenc}
\DeclareSymbolFont{cyrletters}{OT2}{wncyr}{m}{n}
\DeclareMathSymbol{\Sha}{\mathalpha}{cyrletters}{"58}

\numberwithin{equation}{section}

\theoremstyle{plain}
\newtheorem{proposition}{Proposition}[section]

\newtheorem{corollary}[proposition]{Corollary}
\newtheorem{lem}[proposition]{Lemma}
\newtheorem{theorem}[proposition]{Theorem}

\theoremstyle{definition}
\newtheorem{definition}[proposition]{Definition}
\newtheorem{construction}[proposition]{Construction}
\newtheorem{notation}[proposition]{Notation}
\newtheorem{assumption}[proposition]{Assumption}

\theoremstyle{remark}
\newtheorem{remark}[proposition]{Remark}
\newtheorem{example}[proposition]{Example}

\renewcommand{\b}[1]{\mathbf{#1}}
\renewcommand{\c}[1]{\mathcal{#1}}
\renewcommand{\d}[1]{\mathbb{#1}}
\newcommand{\f}[1]{\mathfrak{#1}}
\renewcommand{\r}[1]{\mathrm{#1}}
\newcommand{\s}[1]{\mathscr{#1}}
\renewcommand{\sf}[1]{\mathsf{#1}}
\renewcommand{\(}{\left(}
\renewcommand{\)}{\right)}
\newcommand{\res}{\mathbin{|}}

\newcommand{\Sec}{\S}

\newcommand{\cA}{\c A}
\newcommand{\cB}{\c B}
\newcommand{\cC}{\c C}

\newcommand{\cE}{\c E}
\newcommand{\cF}{\c F}

\newcommand{\cJ}{\c J}

\newcommand{\cM}{\c M}

\newcommand{\cO}{\c O}
\newcommand{\cP}{\c P}

\newcommand{\cR}{\c R}
\newcommand{\cS}{\c S}

\newcommand{\cV}{\c V}

\newcommand{\cX}{\c X}
\newcommand{\cY}{\c Y}
\newcommand{\cZ}{\c Z}

\newcommand{\dA}{\d A}

\newcommand{\dC}{\d C}

\newcommand{\dE}{\d E}
\newcommand{\dF}{\d F}

\newcommand{\dL}{\d L}

\newcommand{\dP}{\d P}
\newcommand{\dQ}{\d Q}
\newcommand{\dR}{\d R}

\newcommand{\dT}{\d T}

\newcommand{\dZ}{\d Z}

\newcommand{\fD}{\f D}

\newcommand{\fH}{\f H}

\newcommand{\fM}{\f M}
\newcommand{\fN}{\f N}

\newcommand{\fa}{\f a}

\newcommand{\fd}{\f d}

\newcommand{\fl}{\f l}
\newcommand{\fm}{\f m}

\newcommand{\fp}{\f p}
\newcommand{\fq}{\f q}
\newcommand{\fr}{\f r}
\newcommand{\fs}{\f s}
\newcommand{\ft}{\f t}

\newcommand{\rA}{\r A}
\newcommand{\rB}{\r B}

\newcommand{\rD}{\r D}
\newcommand{\rE}{\r E}
\newcommand{\rF}{\r F}
\newcommand{\rG}{\r G}
\newcommand{\rH}{\r H}
\newcommand{\rI}{\r I}

\newcommand{\rM}{\r M}
\newcommand{\rN}{\r N}

\newcommand{\rP}{\r P}

\newcommand{\rR}{\r R}
\newcommand{\rS}{\r S}
\newcommand{\rT}{\r T}

\newcommand{\rZ}{\r Z}

\newcommand{\rd}{\r d}

\newcommand{\rt}{\r t}

\newcommand{\sF}{\s F}

\newcommand{\sO}{\s O}

\newcommand{\sT}{\s T}

\newcommand{\sfC}{\sf C}
\newcommand{\sfD}{\sf D}
\newcommand{\sfE}{\sf E}
\newcommand{\sfF}{\sf F}

\newcommand{\sfH}{\sf H}

\newcommand{\sfM}{\sf M}

\newcommand{\sfh}{\sf h}

\newcommand{\tp}[1]{\prescript{\rt}{}{#1}}
\newcommand{\pres}[2]{\prescript{#1}{}{#2}}

\newcommand{\ab}{\r{ab}}
\newcommand{\abs}{\r{abs}}

\newcommand{\cl}{\r{cl}}
\newcommand{\Cl}{\r{Cl}}

\newcommand{\cris}{\r{cris}}

\newcommand{\dr}{\r{dR}}
\newcommand{\EC}{\acute{\r{E}}\r{tCor}}

\newcommand{\loc}{\r{loc}}

\newcommand{\sing}{\r{sing}}

\newcommand{\TF}{\widetilde{F}}

\newcommand{\unr}{\r{unr}}
\newcommand{\ur}{\r{ur}}

\DeclareMathOperator{\AJ}{AJ}

\DeclareMathOperator{\CH}{CH}

\DeclareMathOperator{\coker}{coker}

\DeclareMathOperator{\disc}{disc}

\DeclareMathOperator{\End}{End}

\DeclareMathOperator{\Frob}{Frob}

\DeclareMathOperator{\Gal}{Gal}

\DeclareMathOperator{\GL}{GL}

\DeclareMathOperator{\Hom}{Hom}
\DeclareMathOperator{\IM}{im}

\DeclareMathOperator{\Ind}{Ind}

\DeclareMathOperator{\Ker}{ker}
\DeclareMathOperator{\Lie}{Lie}

\DeclareMathOperator{\Mat}{Mat}

\DeclareMathOperator{\Nm}{N}

\DeclareMathOperator{\ord}{ord}

\DeclareMathOperator{\Res}{Res}

\DeclareMathOperator{\SL}{SL}

\DeclareMathOperator{\Sp}{Sp}

\DeclareMathOperator{\Spec}{Spec}

\DeclareMathOperator{\Sym}{Sym}

\DeclareMathOperator{\Tr}{Tr}
\DeclareMathOperator{\tr}{tr}

\begin{document}

\title{Bounding cubic-triple product Selmer groups of elliptic curves}

\author{Yifeng Liu}
\address{Department of Mathematics, Northwestern University, Evanston, IL 60208}
\email{liuyf@math.northwestern.edu}

\date{\today}

\begin{abstract}
Let $E$ be a modular elliptic curve over a totally real cubic field. We have a cubic-triple product motive over $\mathbb{Q}$ constructed from $E$ through multiplicative induction; it is of rank $8$. We show that, under certain assumptions on $E$, the non-vanishing of the central critical value of the $L$-function attached to the motive implies that the dimension of the associated Bloch--Kato Selmer group is $0$.
\end{abstract}

\maketitle

\tableofcontents

\section{Introduction}
\label{ss:1}

\subsection{Main result}
\label{ss:main}

Let $F$ be a cubic \'{e}tale algebra over $\dQ$. Consider the reductive group $G\coloneqq\Res_{F/\dQ}\GL_2$ over $\dQ$. Denote by $\tau\colon\pres{L}{G}\to\GL_8(\dC)$ the triple product $L$-homomorphism \cite{PSR87}*{\Sec 0}. Thus, for every irreducible cuspidal automorphic
representation $\Pi$ of $G(\dA)$ where $\dA$ denotes the ring of ad\`{e}les of $\dQ$ as usual, we have the attached triple product $L$-function
$L(s,\Pi,\tau)$. By the work of Garrett \cite{Gar87} and Piatetski-Shapiro--Rallis \cite{PSR87}, it is a meromorphic function on the complex plane. Suppose that the central character $\omega_\Pi$ of $\Pi$ is trivial after restriction to $\dA^\times$. Then we have a functional
equation
\[L(s,\Pi,\tau)=\epsilon(\Pi,\tau)C(\Pi,\tau)^{1/2-s}L(1-s,\Pi,\tau)\]
for some $\epsilon(\Pi,\tau)\in\{\pm 1\}$ and positive integer $C(\Pi,\tau)$. The global root number $\epsilon(\Pi,\tau)$ is the product of local ones:
\[\epsilon(\Pi,\tau)=\prod_v\epsilon(\Pi_v,\tau),\]
where $v$ runs over all places of $\dQ$. Here, we have $\epsilon(\Pi_v,\tau)\in\{\pm 1\}$ and that it equals $1$ for all but finitely
many $v$. Put
\[\Sigma(\Pi,\tau)=\{v\res \epsilon(\Pi_v,\tau)=-1\}.\]

In our previous work \cite{Liu}, we consider the case where $F$ is the product of $\dQ$ and a real quadratic field. In this article, we consider the case where $F$ is a totally real cubic field. Thus, from now on, we fix such a field $F$ (except in the two appendices). Denote by $O_F$, $\dA_F$, $\Phi_F$, and $\TF\subset\dC$ the ring of integers, the ring of ad\`{e}les, the set of archimedean places, and the normal closure in $\dC$, of $F$ respectively. In particular, elements of $\Phi_F$ are different embeddings from $F$ to $\TF$. Let $F_0\subset\dC$ be the unique subfield contained in $\TF$ of degree $[\TF:F]$.

Let $E$ be an elliptic curve over $F$. We have the $\dQ$-motive $\otimes\Ind^F_\dQ\sfh^1(E)$
(with coefficient $\dQ$) of rank $8$, which is the multiplicative induction of the $F$-motive $\sfh^1(E)$ to $\dQ$. The \emph{cubic-triple product
motive} of $E$ is defined to be
\begin{align}\label{eq:cubic_motive}
\sfM(E)\coloneqq(\otimes\Ind^F_\dQ\sfh^1(E))(2).
\end{align}
It is canonically polarized. For every prime $p$, the $p$-adic realization of $\sfM(E)$, denoted by $\sfM(E)_p$, is a Galois representation of $\dQ$ of dimension $8$ with $\dQ_p$-coefficients. In fact, up to a twist, it is the multiplicative induction from $F$ to $\dQ$ of the rational $p$-adic Tate module of $E$.

Now we assume that $E$ is modular. Then it gives rise to an irreducible cuspidal automorphic representation $\Pi_E$ of $G(\dA)$ with trivial central character. In particular, the set $\Sigma(\Pi_E,\tau)$ contains $\infty$. We have $L(s,\sfM(E))=L(s+1/2,\Pi_E,\tau)$.

\begin{assumption}\label{as:elliptic_curve}
We propose the following three assumptions.
\begin{description}
  \item[(E1)] For every finite place $w$ of $F$ over some prime in $\Sigma(\Pi_E,\tau)$, the elliptic curve $E$ has either good or multiplicative reduction at $w$.

  \item[(E2)] For distinct elements $\tau_1,\tau_2\in\Phi_F$, the $\TF$-elliptic curve $E\otimes_{F,\tau_1}\TF$ is not isogenous to
      any (possibly trivial) quadratic twist of $E\otimes_{F,\tau_2}\TF$.

  \item[(E3)] The geometric fiber $E\otimes_F\overline{F}$ does not admit complex multiplication.
\end{description}
\end{assumption}

The following theorem is a new case of the Bloch--Kato conjecture, for the motive $\sfM(E)$. We refer the reader to the Introduction of our previous article \cite{Liu} for the background of this conjecture. In particular, one may compare the following theorem with \cite{Liu}*{Theorems 1.3, 1.4}.

\begin{theorem}\label{th:main}
Let $E$ be a modular elliptic curve over $F$ satisfying Assumption \ref{as:elliptic_curve}. If the central critical value $L(0,\sfM(E))$ is nonzero, then there exists a finite set $\cP_E$ of primes such that for every prime $p\not\in\cP_E$, we have
\[\dim_{\dQ_p}\rH^1_f(\dQ,\sfM(E)_p)=0.\]
\end{theorem}

\begin{remark}
Once the elliptic curve $E$ is given, (an upper bound for) the finite set $\cP_E$ can be computed effectively. See Lemma \ref{le:prime_bad}.
\end{remark}

\begin{remark}
The technical heart of the proof of Theorem \ref{th:main} consists of a new level raising result (Theorem \ref{th:cubic_level_raising}) for Hilbert threefolds, and a new reciprocity law (Theorem \ref{th:reciprocity}) for their special cycles of codimension two. The version for CM points on Shimura curves is known as the \emph{first explicit reciprocity law} in \cite{BD05}. In \cite{Liu}, the corresponding results are Theorems 4.10 (3) and 4.11, which are called \emph{congruence formulae} there. However, unlike the situation in \cite{Liu}, our threefold here is not a product of a curve and a surface. Thus, the pullback technique will not work anymore; and in particular, we cannot make use of the reciprocity law in \cite{BD05}. In other words, we cannot reduce the computation to divisors on curves. The main innovation of the current article is to develop a new tool, for higher (co)dimensional cycles, to compute the local Galois cohomology; and our new reciprocity law will be the first practice.

In \cite{Liu}, we also have a result when $\ord_{s=0}L(s,\sfM(E))=1$ (\cite{Liu}*{Theorem 1.5}). To obtain a similar result here, we need another reciprocity law, which plays the role as the \emph{second explicit reciprocity law} in \cite{BD05} for CM points on Shimura curves generalizing \cite{Rib90}. This reciprocity (and its application) are currently work in progress.
\end{remark}

\begin{remark}
When $F$ is normal, Assumption \ref{as:elliptic_curve} (E2) is equivalent to that $E$ does not admit a nontrivial homomorphism to any twist of
$A\otimes_\dQ F$ for any $\GL(2)$-type abelian variety $A$ over $\dQ$. In general, (E2) always excludes the case where $E=\cE\otimes_\dQ F$ for some rational elliptic curve $\cE$ (hence $E$ is modular by \cite{MR00} and the well-known modularity of $\cE$ \cites{Wil95,TW95,BCDT01}). However, this case has potential application to the symmetric cube Selmer group of $\cE$. Precisely, if we want to prove a version of Theorem \ref{th:main} for the motive $(\Sym^3\sfh^1(\cE))(2)$ (of rank $4$) by embedding it into some triple product motive (of rank $8$) through base change to a totally real cubic \'{e}tale $\dQ$-algebra $F$, then we should necessarily take $F$ to be a \emph{field}; otherwise, the motive $\sfh^1(\cE)(1)$ will appear
simultaneously. This is not good since we may have $L(0,(\Sym^3\sfh^1(\cE))(2))\neq 0$ but $L(0,\sfh^1(\cE)(1))(=L(1,\cE))=0$. In fact, (E2) is not a serious assumption to our method. We think we are able to remove it by some extra combinatorial arguments which are quite involved. However, we choose not to pursue this generality here because otherwise the article would be rather technical and the logical uniformity would be broken.
\end{remark}

\subsection{Outline of the article}

The main part of the article has three chapters.

In \Sec\ref{ss:2}, we develop a general tool for computing the Galois cohomology of \'{e}tale cohomology of varieties over local fields having (strictly) semistable reduction, using the weight spectral sequence. The core is the construction of the so-called \emph{potential map}. We also study its relation with the Abel--Jacobi map. This chapter does not involve Shimura varieties, automorphic representations, or global Galois representations.

In \Sec\ref{ss:3}, we prove a level raising result (Theorem \ref{th:cubic_level_raising}) for certain Galois representations of $F$ of rank $2$ with torsion coefficients. We realize the level raising in the middle degree cohomology of some Hilbert modular variety of dimension $3$. The proof uses the general tool developed in \Sec\ref{ss:2}, by looking at the local Galois cohomology at a bad prime. We start from a natural semistable model of the Hilbert threefold at such a bad prime. The most technical part is to compute the potential map. For this, we have to calculate a bunch of intersection numbers for well-selected cycles, and to show that all other cycles are irrelevant.

In \Sec\ref{ss:4}, we first introduce the notation of Hirzebruch--Zagier cycle, which is of dimension $1$, on the Hilbert threefold, and prove an explicit reciprocity law for its image in the local Galois cohomology in terms of the level raising proved in \Sec\ref{ss:3}. Then we use these cycles to produce sufficiently many annihilators for Selmer elements, which confirms the triviality of the Selmer group in Theorem \ref{th:main}.

There are two appendices. In \Sec\ref{ss:a}, we study oriented Eichler orders over totally real number fields, and use them to canonically parameterize certain abelian varieties in positive characteristic. This generalizes the work of Ribet \cite{Rib89}. Such parameterization is crucial for the level raising in \Sec\ref{ss:3}. In \Sec\ref{ss:b}, we improve the work of Zink \cite{Zin82} on the global description of some pluri-nodal integral model of Hilbert modular varieties at certain bad primes. We also prove some extra results which are used in \Sec\ref{ss:3}.

\subsection{Notation and conventions}
\label{ss:notation}

\begin{itemize}
  \item We simply write $\otimes$ for $\otimes_\dZ$.

  \item Let $\Lambda$ be an abelian group. We put $\widehat\Lambda\coloneqq\Lambda\otimes_\dZ\varprojlim_N\dZ/N$ and $\Lambda_\dQ\coloneqq\Lambda\otimes_\dZ\dQ$. For a finite set $\cS$, we denote by $\Gamma(\cS,\Lambda)$ the abelian group of $\Lambda$-valued functions on $\cS$.

  \item For every rational prime $\ell$, we fix an algebraic closure $\dF_{\ell^\infty}$ of $\dF_\ell$ and let $\dF_{\ell^f}$ be the unique subfield of $\dF_{\ell^\infty}$ of cardinality $\ell^f$. For $1\leq f\leq \infty$, let $\dZ_{\ell^f}$ be the ring of Witt vectors in $\dF_{\ell^f}$ and put $\dQ_{\ell^f}=\dZ_{\ell^f}\otimes\dQ$.

  \item For a number field $K$, we denote by $O_K$ its ring of integers. For an ideal $\fa$ of $O_K$, we denote by $\Nm\fa=[O_K:\fa]$ the numerical norm of $\fa$.

  \item For a field $K$, we denote by $\rG_K$ the absolutely Galois group of $K$. Denote by $\overline\dQ$ the algebraic closure of $\dQ$ in $\dC$. When $K$ is a subfield of $\overline\dQ$, we take $\rG_K$ to be $\Gal(\overline\dQ/K)$ hence a subgroup of $\rG_\dQ$.

  \item If $K$ is a local field, then we denote by $O_K$ its ring of integers, $\rI_K\subset\rG_K$ the inertia subgroup. If $v$ is a rational prime, then we simply write $\rG_v$ for $\rG_{\dQ_v}$ and $\rI_v$ for $\rI_{\dQ_v}$.

  \item Let $K$ be a local field, $\Lambda$ a ring, and $N$ a $\Lambda[\rG_K]$-module. We have an exact sequence of $\Lambda$-modules
      \begin{align*}
      0\to\rH^1_\unr(K,N)\to\rH^1(K,N)\xrightarrow{\partial}\rH^1_\sing(K,N)\to0,
      \end{align*}
      where $\rH^1_\unr(K,N)$ is the submodule of unramified classes.

  \item Let $\Lambda$ be a ring and $N$ a $\Lambda[\rG_\dQ]$-module. For each prime power $v$, we have the localization map $\loc_v\colon\rH^1(\dQ,N)\to\rH^1(\dQ_v,N)$ of $\Lambda$-modules.

  \item All cohomology for schemes are computed in the \'{e}tale site, unless otherwise specified.
\end{itemize}

\subsubsection*{Acknowledgements}

The author would like to thank Benedict~Gross, Victor~Rotger, Sug~Woo~Shin, Yunqing~Tang, Yichao~Tian, Liang~Xiao, Shouwu~Zhang, Wei~Zhang, and
Xinwen~Zhu for their interest, helpful comments and discussions. He especially thanks Yichao~Tian for informing him the result of Zink. He also thanks the anonymous referee for the extremely careful reading and many valuable suggestions and corrections to improve the quality of the article. The author is partially supported by NSF grant DMS--1302000.

\section{Weight spectral sequence and Galois cohomology}
\label{ss:2}

Let $K$ be a henselian discrete valuation field with residue field $\kappa$ and a separable closure $\overline{K}$. We fix a prime $p$ that is different from the characteristic of $\kappa$. Throughout this section, the coefficient ring $\Lambda$ will be either $\dQ_p$ or $\dZ/p^\nu$ for some $\nu\geq 1$, and we put $\Lambda_0=\dZ_p$. In what follows, the letter $p$ will serve as an index (mainly) for spectral sequences; this will not cause confusion as we will always use $\Lambda$ or $\Lambda_0$ for coefficient rings.

\subsection{Semistable schemes}
\label{ss:semistable_schemes}

We first recall the following definition.

\begin{definition}[Strictly semistable scheme]\label{de:semistable_schemes}
Let $X$ be a scheme locally of finite presentation over $\Spec O_K$. We say that $X$ is \emph{strictly semistable} if it is Zariski locally \'{e}tale over
\[\Spec O_K[t_1,\dots,t_n]/(t_1\cdots t_s-\varpi)\]
for some integers $0\leq s\leq n$ (which may vary) and a uniformizer $\varpi$ of $K$.
\end{definition}

Let $X$ be a proper strictly semistable scheme over $O_K$. The special fiber $X_\kappa\coloneqq X\otimes_{O_K}\kappa$ is a normal crossing divisor of $X$. Suppose that $\{X_1,\dots,X_m\}$ is the set of irreducible components of $X_\kappa$. For $p\geq 0$, put
\[X^{(p)}_\kappa=\coprod_{I\subset\{1,\dots,m\},|I|=p+1}\bigcap_{i\in I}X_i.\] Then $X^{(p)}_\kappa$ is a finite disjoint union of smooth proper $\kappa$-schemes of codimension $p$. From \cite{Sai03}*{page 610}, we have the pullback map
\[\delta_p^*\colon\rH^q(X^{(p)}_{\overline{\kappa}},\Lambda(j))\to\rH^q(X^{(p+1)}_{\overline{\kappa}},\Lambda(j))\]
and the pushforward (Gysin) map
\[\delta_{p*}\colon\rH^q(X^{(p)}_{\overline{\kappa}},\Lambda(j))\to\rH^{q+2}(X^{(p-1)}_{\overline{\kappa}},\Lambda(j+1))\]
for every integer $j$. These maps satisfy the formula
\begin{align}\label{eq:delta}
\delta_{p-1}^*\circ\delta_{p*}+\delta_{p+1*}\circ\delta_p^*=0
\end{align}
For reader's convenience, we recall the definition here. For subsets $J\subset I\subset\{1,\dots,m\}$ such that $|I|=|J|+1$, let $i_{JI}\colon\bigcap_{i\in I}X_i\to\bigcap_{i\in J}X_i$ denote the closed immersion. If $I=\{i_0<\cdots<i_p\}$ and $J=I\setminus\{i_j\}$, then we put $\epsilon(J,I)=(-1)^j$. We define $\delta_p^*$ to be the alternating sum $\sum_{I\subset J,|I|=|J|-1=p+1}\epsilon(I,J)i_{IJ}^*$ of the pullback maps, and $\delta_{p*}$ to be the alternating sum $\sum_{I\supset J,|I|=|J|+1=p+1}\epsilon(J,I)i_{JI*}$ of the Gysin maps.

\begin{remark}\label{re:sign}
In general, the maps $\delta_p^*$ and $\delta_{p*}$ depend on the ordering of the irreducible components of $X_\kappa$. However, it is easy to see that the composite map $\delta_{1*}\circ\delta^*_0$ does not depend on such ordering.
\end{remark}

Let us recall the weight spectral sequence attached to $X$. Denote by $K^\ur\subset\overline{K}$ the maximal unramified extension, with the residue field $\overline{\kappa}$ which is a separable closure of $\kappa$. Then we have $\rG_K/\rI_K\simeq\rG_\kappa$. Denote by $t_0\colon\rI_K\to\Lambda_0(1)$ the ($p$-adic) tame quotient homomorphism, that is, the one sending $\sigma\in\rI_K$ to $(\sigma(\varpi^{1/p^n})/\varpi^{1/p^n})_n$ for a uniformizer $\varpi$ of $K$. We fix an element $T\in\rI_K$ such that $t_0(T)$ is a topological generator of $\Lambda_0(1)$.

We have the weight spectral sequence $\rE_X$ attached to the (proper strictly semistable) scheme $X$, where
\[(\rE_X)^{p,q}_1=\bigoplus_{i\geq\max(0,-p)}\rH^{q-2i}(X_{\overline{\kappa}}^{(p+2i)},\Lambda(-i))
\Rightarrow\rH^{p+q}(X_{\overline{K}},\Lambda).\]
This is also known as the Rapoport--Zink spectral sequence, first studied in \cite{RZ82}; here we will follow the convention and discussion in \cite{Sai03}. For $r\in\dZ$, put $\pres{r}\rE_X=\rE_X(r)$ and we will suppress the subscript $X$ in the notation of the spectral sequence if it causes no confusion.

By \cite{Sai03}*{Corollary 2.8 (2)}, we have a map $\mu\colon\rE^{\bullet-1,\bullet+1}_\bullet\to\rE^{\bullet+1,\bullet-1}_\bullet$ of spectral sequences (depending on $T$) and its version for $\pres{r}\rE$. For later use, we write down part of the map in the first page of $\pres{r}\rE$ \footnote{There is a typo in \cite{Sai03}*{Corollary 2.8 (2)} where the limit in the second direct summand should be $i-1\geq\max(0,-p-2)$.}.
\begin{align}\label{eq:spectral_sequence}
\resizebox{16cm}{!}{\xymatrix{
\pres{r}\rE^{-3,2r}_1=\bigoplus\limits_{i=3}^r\rH^{2r-2i}(X_{\overline{\kappa}}^{(2i-3)},\Lambda(r-i))
\ar[d]_-{\rd^{-3,2r}_1}\ar[rr]^-{\mu^{-2,2r-1}}&&
\bigoplus\limits_{i=2}^r\rH^{2r-2i}(X_{\overline{\kappa}}^{(2i-3)},\Lambda(r-i+1))=\pres{r}\rE^{-1,2r-2}_1 \ar[d]^-{\rd^{-1,2r-2}_1}\\
\pres{r}\rE^{-2,2r}_1=\bigoplus\limits_{i=2}^r\rH^{2r-2i}(X_{\overline{\kappa}}^{(2i-2)},\Lambda(r-i))
\ar[d]_-{\rd^{-2,2r}_1}\ar[rr]^-{\mu^{-1,2r-1}}&&
\bigoplus\limits_{i=1}^r\rH^{2r-2i}(X_{\overline{\kappa}}^{(2i-2)},\Lambda(r-i+1))=\pres{r}\rE^{0,2r-2}_1 \ar[d]^-{\rd^{0,2r-2}_1}\\
\pres{r}\rE^{-1,2r}_1=\bigoplus\limits_{i=1}^r\rH^{2r-2i}(X_{\overline{\kappa}}^{(2i-1)},\Lambda(r-i))
\ar[d]_-{\rd^{-1,2r}_1}\ar[rr]^-{\mu^{0,2r-1}}&&
\bigoplus\limits_{i=1}^r\rH^{2r-2i}(X_{\overline{\kappa}}^{(2i-1)},\Lambda(r-i+1))=\pres{r}\rE^{1,2r-2}_1 \ar[d]^-{\rd^{1,2r-2}_1}\\
\pres{r}\rE^{0,2r}_1=\bigoplus\limits_{i=0}^r\rH^{2r-2i}(X_{\overline{\kappa}}^{(2i)},\Lambda(r-i))
\ar[d]_-{\rd^{0,2r}_1}\ar[rr]^-{\mu^{1,2r-1}}&&
\bigoplus\limits_{i=1}^r\rH^{2r-2i}(X_{\overline{\kappa}}^{(2i)},\Lambda(r-i+1))
=\pres{r}\rE^{2,2r-2}_1 \ar[d]^-{\rd^{1,2r-2}_1}\\
\pres{r}\rE^{1,2r}_1=\bigoplus\limits_{i=0}^r\rH^{2r-2i}(X_{\overline{\kappa}}^{(2i+1)},\Lambda(r-i))
\ar[rr]^-{\mu^{2,2r-1}}&&
\bigoplus\limits_{i=1}^r\rH^{2r-2i}(X_{\overline{\kappa}}^{(2i+1)},\Lambda(r-i+1))=\pres{r}\rE^{3,2r-2}_1,
}}
\end{align}
in which all differentials $\rd^{p,q}_1$ are appropriate sums of pullback and pushforward maps. The map $\mu^{p,q}\coloneqq\mu^{p,q}_1\colon\pres{r}\rE^{p-1,q+1}_1\to\pres{r}\rE^{p+1,q-1}_1$ is the sum of its restrictions to each direct summand $\rH^{q+1-2i}(X_{\overline{\kappa}}^{(2i+1)},\Lambda(r-i))$, and such restriction is the tensor product by $t_0(T)$ (resp.\ the zero map) if $\rH^{q+1-2i}(X_{\overline{\kappa}}^{(2i+1)},\Lambda(r-i+1))$ does (resp.\ does not) appear in the target. In particular, the map $\mu^{0,2r-1}$ is an isomorphism.

Denote by
\begin{align}\label{eq:filtration}
\cdots\subset\sF^{p+1}\rH^n(X_{\overline{K}},\Lambda(r))\subset
\sF^p\rH^n(X_{\overline{K}},\Lambda(r))\subset\cdots\subset\rH^n(X_{\overline{K}},\Lambda(r))
\end{align}
the filtration induced by the spectral sequence, such that
\[\frac{\sF^p\rH^n(X_{\overline{K}},\Lambda(r))}{\sF^{p+1}\rH^n(X_{\overline{K}},\Lambda(r))}\simeq\pres{r}\rE^{p,n-p}_\infty.\]

\subsection{Potential map and Galois cohomology}
\label{ss:potential_map}

Let $X$ be a proper strictly semistable scheme over $O_K$ purely of relative dimension $n$.

For every integer $j$, put
\[B^j(X_\kappa,\Lambda)=\Ker[\delta_0^*\colon\rH^{2j}(X_{\overline{\kappa}}^{(0)},\Lambda(j))\to\rH^{2j}(X_{\overline{\kappa}}^{(1)},\Lambda(j))],\]
and dually
\[B_j(X_\kappa,\Lambda)=
\coker[\delta_{1*}\colon\rH^{2(n-1-j)}(X_{\overline{\kappa}}^{(1)},\Lambda(n-1-j))\to\rH^{2(n-j)}(X_{\overline{\kappa}}^{(0)},\Lambda(n-j))].\]
Note that if we regard $B^j(X_\kappa,\Lambda)$ as a submodule of $\rE^{0,2j}_1$, then it is contained in the kernel of $\rd^{0,2j}_1$ since the restriction of $\rd^{0,2j}_1$ to $\rH^{2j}(X_{\overline{\kappa}}^{(0)},\Lambda(j))$ coincides with $\delta_0^*$ by \cite{Sai03}*{Proposition 2.10}. Thus we may put
\[B^j(X_\kappa,\Lambda)^0=\Ker[B^j(X_\kappa,\Lambda)\to\pres{j}\rE^{0,2j}_2],\]
which is a submodule of $B^j(X_\kappa,\Lambda)$. Set $B_j(X_\kappa,\Lambda)\to B_j(X_\kappa,\Lambda)_0$ to be the quotient map dual to the inclusion $B^j(X_\kappa,\Lambda)^0\subset B^j(X_\kappa,\Lambda)$.

\begin{lem}\label{le:galois_cohomology}
In \eqref{eq:spectral_sequence}, we have the equality
\[\IM\rd^{-1,2r}_1\cap\rH^{2r}(X_{\overline{\kappa}}^{(0)},\Lambda(r))=B^r(X_\kappa,\Lambda)^0\]
of submodules of $\pres{r}\rE^{0,2r}_1$.
\end{lem}

\begin{proof}
Note that $\Ker\rd^{0,2r}_1\cap\rH^{2r}(X_{\overline{\kappa}}^{(0)},\Lambda(r))=B^r(X_\kappa,\Lambda)$. Thus
\[\IM\rd^{-1,2r}_1\cap\rH^{2r}(X_{\overline{\kappa}}^{(0)},\Lambda(r))=\IM\rd^{-1,2r}_1\cap B^r(X_\kappa,\Lambda),\]
which coincides with $B^r(X_\kappa,\Lambda)^0$ by definition.
\end{proof}

\begin{lem}\label{le:galois_cohomology_1}
Consider the composite map
\[\rH^{2r-2}(X_{\overline{\kappa}}^{(0)},\Lambda(r-1))\xrightarrow{\delta_0^*}\rH^{2r-2}(X_{\overline{\kappa}}^{(1)},\Lambda(r-1))
\xrightarrow{\delta_{1*}}\rH^{2r}(X_{\overline{\kappa}}^{(0)},\Lambda(r)).\]
Then its image is contained in $B^r(X_\kappa,\Lambda)^0$ and it factorizes through $B_{n+1-r}(X_\kappa,\Lambda)_0$.
\end{lem}

\begin{proof}
It suffices to show that the image of $\delta_{1*}\circ\delta_0^*$ is contained in $B^r(X_\kappa,\Lambda)^0$, since the other half follows from duality.

By \eqref{eq:delta}, we have $\delta_0^*\circ\delta_{1*}=-\delta_{2*}\circ\delta_1^*$. Thus the image of $\delta_{1*}\circ\delta_0^*$ is contained in $B^r(X_\kappa,\Lambda)$ as $\delta_1^*\circ\delta_0^*=0$. It remains to show that the image of $\delta_{1*}\circ\delta_0^*$ is contained in the image of $\rd^{-1,2r}_1$. This follows from \cite{Sai03}*{Proposition 2.10} that the restriction of $\rd^{-1,2r}_1$ to $\rH^{2r-2}(X_{\overline{\kappa}}^{(1)},\Lambda(r-1))$ is the sum of $\delta_{1*}$ and $\delta_1^*$.
\end{proof}

By the above lemma, we obtain the following two maps
\begin{align}\label{eq:potential_prepremap}
B_{n+1-r}(X_\kappa,\Lambda)\to B^r(X_\kappa,\Lambda);
\end{align}
\begin{align}\label{eq:potential_premap}
B_{n+1-r}(X_\kappa,\Lambda)_0\to B^r(X_\kappa,\Lambda)^0.
\end{align}
Put $A^j(X_\kappa,\Lambda)^?=(B^j(X_\kappa,\Lambda)^?)^{\rG_\kappa}$ and $A_j(X_\kappa,\Lambda)_?=B_j(X_\kappa,\Lambda)_?^{\rG_\kappa}$ for $?=\emptyset,0$.

\begin{definition}[Potential map]\label{de:potential_map}
We call the following map
\begin{align*}
\Delta^r\colon A_{n+1-r}(X_\kappa,\Lambda)_0\to A^r(X_\kappa,\Lambda)^0.
\end{align*}
induced by \eqref{eq:potential_premap} the \emph{potential map}.
\end{definition}

Now we use the potential map to study the group $\rH^1_\sing(K,\rH^{2r-1}(X_{\overline{K}},\Lambda(r)))$. We start from the following lemma.

\begin{lem}\label{le:singular}
There is a canonical isomorphism
\begin{align*}
\rH^1_\sing(K,\rH^{2r-1}(X_{\overline{K}},\Lambda(r)))\simeq\rH^1(\rI_K,\rH^{2r-1}(X_{\overline{K}},\Lambda(r)))^{\rG_\kappa}.
\end{align*}
\end{lem}

\begin{proof}
By the inflation-restriction exact sequence, it suffices to show the cohomology group $\rH^2(\rG_\kappa,\rH^{2r-1}(X_{\overline{K}},\Lambda(r))^{\rI_K})$ vanishes. Suppose that $\Lambda$ is not $\dQ_p$. Then
\[\rH^2(\rG_\kappa,\rH^{2r-1}(X_{\overline{K}},\Lambda(r))^{\rI_K})\simeq
\varprojlim_{n}\rH^2(\rG_\kappa,\rH^{2r-1}(X_{\overline{K}},\Lambda(r))^{\rI_K}/p^n)\]
by \cite{Jan88}*{(2.1)}. Thus $\rH^2(\rG_\kappa,\rH^{2r-1}(X_{\overline{K}},\Lambda(r))^{\rI_K})=0$ as $\rH^2(\rG_\kappa,\rH^{2r-1}(X_{\overline{K}},\Lambda(r))^{\rI_K}/p^n)=0$. When $\Lambda=\dQ_p$, $\rH^2(\rG_\kappa,\rH^{2r-1}(X_{\overline{K}},\Lambda(r))^{\rI_K})\simeq
\rH^2(\rG_\kappa,\rH^{2r-1}(X_{\overline{K}},\Lambda_0(r))^{\rI_K})\otimes_{\Lambda_0}\Lambda=0$ as well. Thus the lemma follows.
\end{proof}

\begin{definition}[Nice coefficient]\label{de:nice_coefficient}
We say that $\Lambda$ is a \emph{nice coefficient} for the spectral sequence $\pres{r}\rE=\pres{r}\rE_X$ if the following are satisfied:
\begin{enumerate}
  \item $\pres{r}\rE$, with the coefficient $\Lambda$, degenerates at the second page;

  \item if $\pres{r}\rE^{p,2r-1-p}_2(-1)$ has a non-trivial subquotient on which $\rG_\kappa$ acts trivially, then $p=1$.
\end{enumerate}
We say that $\Lambda$ is a \emph{very nice coefficient} if moreover
\begin{enumerate}\setcounter{enumi}{2}
  \item for every subquotient $\Lambda[\rG_\kappa]$-module $M$ of $\ker\mu^{1,2r-1}\oplus\coker\mu^{-1,2r-1}(-1)$, the natural map
      $M^{\rG_\kappa}\to M_{\rG_\kappa}$ is an isomorphism.
\end{enumerate}
\end{definition}

\begin{remark}
The coefficient ring $\Lambda=\dQ_p$ is always very nice for $\pres{r}\rE_X$.
\end{remark}

\begin{theorem}\label{th:galois_cohomology}
Suppose that $\Lambda$ is a nice coefficient for $\pres{r}\rE_X$.
\begin{enumerate}
  \item We have a canonical map
        \[\eta^r\colon A^r(X_\kappa,\Lambda)^0\to\rH^1_\sing(K,\rH^{2r-1}(X_{\overline{K}},\Lambda(r))).\]

  \item If $\Lambda$ is moreover very nice, then we have a canonical exact sequence
        \begin{multline*}
        0 \to \rH^1_\unr(K,\rH^{2r-1}(X_{\overline{K}},\Lambda(r)))\to
        A_{n+1-r}(X_\kappa,\Lambda)_0 \\
        \xrightarrow{\Delta^r} A^r(X_\kappa,\Lambda)^0
        \xrightarrow{\eta^r}
        \rH^1_\sing(K,\rH^{2r-1}(X_{\overline{K}},\Lambda(r))) \to
        0.
        \end{multline*}
\end{enumerate}
\end{theorem}

\begin{proof}
Denote by $\rP_K\subset\rI_K$ the kernel of $t_0\colon\rI_K\to\Lambda_0(1)$. Again by the inflation-restriction exact sequence, we have the following exact sequence
\[0 \to \rH^1(\rI_K/\rP_K,\rH^{2r-1}(X_{\overline{K}},\Lambda(r))^{\rP_K})\to \rH^1(\rI_K,\rH^{2r-1}(X_{\overline{K}},\Lambda(r)))
\to\rH^1(\rP_K,\rH^{2r-1}(X_{\overline{K}},\Lambda(r))).\]
Since $\rH^{2r-1}(X_{\overline{K}},\Lambda(r))$ is a finitely generated $\Lambda$-module, we have
\[\rH^{2r-1}(X_{\overline{K}},\Lambda(r))^{\rP_K}=\rH^{2r-1}(X_{\overline{K}},\Lambda(r));\quad
\rH^1(\rP_K,\rH^{2r-1}(X_{\overline{K}},\Lambda(r)))=0.\]
Thus, we have a canonical isomorphism
\begin{align}\label{eq:galois1}
\rH^1(\rI_K,\rH^{2r-1}(X_{\overline{K}},\Lambda(r)))^{\rG_\kappa}
\simeq\rH^1(\rI_K/\rP_K,\rH^{2r-1}(X_{\overline{K}},\Lambda(r)))^{\rG_\kappa}.
\end{align}

On the other hand, the map $t_0\colon\rI_K\to\Lambda_0(1)$ induces a $\rG_\kappa$-equivariant isomorphism $\rI_K/\rP_K\simeq\Lambda_0(1)$. In
particular, the quotient group $\rI_K/\rP_K$ is topologically generated by $T$ and we have a canonical isomorphism of $\Lambda[\rG_\kappa]$-modules
\begin{align}\label{eq:galois2}
\xymatrix{\rH^1(\rI_K/\rP_K,\rH^{2r-1}(X_{\overline{K}},\Lambda(r)))
\simeq\(\frac{\rH^{2r-1}(X_{\overline{K}},\Lambda(r))}{(T-1)\rH^{2r-1}(X_{\overline{K}},\Lambda(r))}\)(-1).}
\end{align}
Denote the right-hand side of \eqref{eq:galois2} by $\rH^{2r-1}(X_{\overline{K}},\Lambda(r))_T$. The map $\mu^{p,q}$ in \eqref{eq:spectral_sequence} induces a map, known as the monodromy operator,
\[\tilde\mu^{p,q}\colon\pres{r}\rE^{p-1,q+1}_2\to\pres{r}\rE^{p+1,q-1}_2(-1)\]
of $\Lambda[\rG_\kappa]$-modules.

Now we consider (1). The filtration \eqref{eq:filtration} induces a quotient filtration
\[\cdots\subset\sF^{p+1}\rH^{2r-1}(X_{\overline{K}},\Lambda(r))_T
\subset\sF^p\rH^{2r-1}(X_{\overline{K}},\Lambda(r))_T\subset\cdots\subset\rH^{2r-1}(X_{\overline{K}},\Lambda(r))_T.\]
By the computation of the action of $T$ on $\rH^{2r-1}(X_{\overline{K}},\Lambda(r))$ via the weight spectral sequence (see \cite{Sai03}*{Corollary 2.8 (2)}) and Definition \ref{de:nice_coefficient} (N1), we have canonical isomorphisms
\[\xymatrix{\frac{\sF^p\rH^{2r-1}(X_{\overline{K}},\Lambda(r))_T}{\sF^{p+1}\rH^{2r-1}(X_{\overline{K}},\Lambda(r))_T}
\simeq\frac{\pres{r}\rE^{p,2r-1-p}_2(-1)}{\tilde\mu^{p-1,2r-p}\pres{r}\rE^{p-2,2r+1-p}_2}}\]
of $\Lambda[\rG_\kappa]$-modules. By Definition \ref{de:nice_coefficient} (N2), we have that
\[\xymatrix{\rH^0\(\rG_\kappa,\frac{\sF^p\rH^{2r-1}(X_{\overline{K}},\Lambda(r))_T}{\sF^{p+1}\rH^{2r-1}(X_{\overline{K}},\Lambda(r))_T}\)
=\rH^1\(\rG_\kappa,\frac{\sF^p\rH^{2r-1}(X_{\overline{K}},\Lambda(r))_T}{\sF^{p+1}\rH^{2r-1}(X_{\overline{K}},\Lambda(r))_T}\)=0}\]
unless $p=1$. By induction, we obtain the following canonical isomorphisms
\begin{align}\label{eq:galois3}
\xymatrix{\(\rH^{2r-1}(X_{\overline{K}},\Lambda(r))_T\)^{\rG_\kappa}
\simeq\(\frac{\pres{r}\rE^{1,2r-2}_2(-1)}{\tilde\mu^{0,2r-1}\pres{r}\rE^{-1,2r}_2}\)^{\rG_\kappa};}
\end{align}
\begin{align}\label{eq:galois7}
\xymatrix{\(\rH^{2r-1}(X_{\overline{K}},\Lambda(r))_T\)_{\rG_\kappa}
\simeq\(\frac{\pres{r}\rE^{1,2r-2}_2(-1)}{\tilde\mu^{0,2r-1}\pres{r}\rE^{-1,2r}_2}\)_{\rG_\kappa}.}
\end{align}
By definition, we have the equality
\begin{align}\label{eq:galois6}
\xymatrix{\frac{\pres{r}\rE^{1,2r-2}_2(-1)}{\tilde\mu^{0,2r-1}\pres{r}\rE^{-1,2r}_2}=
\frac{\Ker(\rd^{1,2r-2}_1(-1))}{\mu^{0,2r-1}(\Ker\rd^{-1,2r}_1)+\IM(\rd^{0,2r-2}_1(-1))}.}
\end{align}
In \eqref{eq:spectral_sequence}, we have an induced isomorphism
\[\xymatrix{\rd^{-1,2r}_1\circ(\mu^{0,2r-1})^{-1}\colon
\frac{\Ker(\rd^{1,2r-2}_1(-1))}{\mu^{0,2r-1}(\Ker\rd^{-1,2r}_1)}\xrightarrow{\sim}
\IM\rd^{-1,2r}_1\cap\rH^{2r}(X_{\overline{\kappa}}^{(0)},\Lambda(r))}\]
which equals $B^r(X_\kappa,\Lambda)^0$ by Lemma \ref{le:galois_cohomology}. The kernel of the induced quotient map
\[\xymatrix{B^r(X_\kappa,\Lambda)^0\to\frac{\Ker(\rd^{1,2r-2}_1(-1))}{\mu^{0,2r-1}(\Ker\rd^{-1,2r}_1)+\IM(\rd^{0,2r-2}_1(-1))}}\]
is equal to $\IM(\rd_1^{-1,2r}\circ(\mu^{0,2r-1})^{-1}\circ\rd_1^{0,2r-2}(-1))\cap B^r(X_\kappa,\Lambda)^0$. In other words, we have
the exact sequence
\[\xymatrix{\rH^{2r-2}(X_\kappa^{(0)},\Lambda(r))\to B^r(X_\kappa,\Lambda)^0\to\frac{\Ker(\rd^{1,2r-2}_1(-1))}{\mu^{0,2r-1}(\Ker\rd^{-1,2r}_1)+\IM(\rd^{0,2r-2}_1(-1))}\to0,}\]
where the first map is $\rd_1^{-1,2r}\circ(\mu^{0,2r-1})^{-1}\circ\rd_1^{0,2r-2}(-1)$. We may rewrite the above sequence as
\begin{align}\label{eq:galois4}
\xymatrix{B_{n+1-r}(X_\kappa,\Lambda)_0\to B^r(X_\kappa,\Lambda)^0\to\frac{\Ker(\rd^{1,2r-2}_1(-1))}{\mu^{0,2r-1}(\Ker\rd^{-1,2r}_1)+\IM(\rd^{0,2r-2}_1(-1))}\to0.}
\end{align}
Taking $\rG_\kappa$-invariants and combining with Lemma \ref{le:singular}, \eqref{eq:galois1}, \eqref{eq:galois2}, \eqref{eq:galois3} and \eqref{eq:galois6}, we obtain two canonical maps
\begin{align}\label{eq:galois5}
A_{n+1-r}(X_\kappa,\Lambda)_0\xrightarrow{\Delta^r}A^r(X_\kappa,\Lambda)^0
\xrightarrow{\eta^r}\rH^1_\sing(K,\rH^{2r-1}(X_{\overline{K}},\Lambda(r))),
\end{align}
whose composition is zero. In particular, (1) follows.

Now we consider (2). We have the following commutative diagram
\begin{align}\label{eq:coinvariant}
\xymatrix{
\rH^1(\rI_K,\rH^{2r-1}(X_{\overline{K}},\Lambda(r)))^{\rG_\kappa} \ar[r]^-{\sim}\ar[d]
& \(\frac{\pres{r}\rE^{1,2r-2}_2(-1)}{\tilde\mu^{0,2r-1}\pres{r}\rE^{-1,2r}_2}\)^{\rG_\kappa} \ar[d] \\
\rH^1(\rI_K,\rH^{2r-1}(X_{\overline{K}},\Lambda(r)))_{\rG_\kappa}
\ar[r]^-{\sim} &
\(\frac{\pres{r}\rE^{1,2r-2}_2(-1)}{\tilde\mu^{0,2r-1}\pres{r}\rE^{-1,2r}_2}\)_{\rG_\kappa}
}\end{align}
from \eqref{eq:galois3} and \eqref{eq:galois7}.

Recall that $B^r(X_\kappa,\Lambda)^0$ is a $\Lambda[\rG_\kappa]$-submodule of $\ker\mu^{1,2r-1}$ and $B_{n+1-r}(X_\kappa^{(0)},\Lambda)_0$ is a quotient $\Lambda[\rG_\kappa]$-module of $\coker\mu^{-1,2r-1}(-1)$. By Definition \ref{de:nice_coefficient} (N3) and \eqref{eq:galois4}, the following maps
\[A^r(X_\kappa,\Lambda)^0\to(B^r(X_\kappa,\Lambda)^0)_{\rG_\kappa},\quad
A_{n+1-r}(X_\kappa,\Lambda)_0\to(B_{n+1-r}(X_\kappa,\Lambda)_0)_{\rG_\kappa},\]
and the right vertical map in \eqref{eq:coinvariant} are all isomorphisms. In particular, the left vertical map in \eqref{eq:coinvariant} is an isomorphism as well. Therefore, \eqref{eq:galois5} is exact and $\eta^r$ is surjective. Thus, (2) follows from Lemma \ref{le:singular} and the Tate and Poincar\'{e} dualities.
\end{proof}

\begin{remark}
When the coefficient ring $\Lambda$ is $\dQ_p$, the monodromy-weight conjecture implies that the map $\Delta^r\colon A_{n+1-r}(X_\kappa,\Lambda)_0\to A^r(X_\kappa,\Lambda)^0$ is an isomorphism, which is equivalent to the vanishing of $\rH^1(K,\rH^{2r-1}(X_{\overline{K}},\Lambda(r)))$ by Theorem \ref{th:galois_cohomology}.
\end{remark}

\subsection{Correspondence, functoriality, and localization}
\label{ss:correspondence}

\begin{definition}[\'{E}tale correspondence]\label{de:correspondence}
Let $X$ be a presheaf (valued in groupoids) on the overcategory of a scheme. An \emph{\'{e}tale correspondence} $T=(X',f,g)$ of $X$ is a diagram
\[X\xleftarrow{f}X'\xrightarrow{g}X,\]
where both $f$ and $g$ are finite \'{e}tale morphisms. The composition $T_2\circ T_1=(X_3,f_3,g_3)$ of two \'{e}tale correspondences $T_1=(X_1,f_1,g_1)$ and $T_2=(X_2,f_2,g_2)$ is defined as
\[\xymatrix{
&& X_3 \ar[ld]\ar[rd] \ar@/_2.3pc/[ddll]_-{f_3} \ar@/^2.3pc/[ddrr]^-{g_3} \\
& X_1 \ar[ld]_-{f_1}\ar[rd]^-{g_1} && X_2 \ar[ld]_-{f_2}\ar[rd]^-{g_2} \\
X && X && X, }\]
in which the middle square Cartesian. The collection of all \'{e}tale correspondences of $X$ forms a monoidal category $\EC(X)$ where the
multiplication is given by compositions and units are those $T=(X',f,g)$ such that $f$ and $g$ are isomorphisms.
\end{definition}

Now let $X$ be a proper strictly semistable scheme over $O_K$ purely of relative dimension $n$. Let $f\colon X'\to X$ be a finite \'{e}tale morphism. In particular, $X'$ is a proper strictly semistable scheme over $O_K$ as well.

\begin{construction}
By \cite{Sai03}*{Corollary 2.14}, we have the following commutative diagram
\[\xymatrix{
\rH^{2j}(X^{\prime(0)}_{\overline\kappa},\Lambda(j)) \ar[r]^-{\delta_0^*}\ar[d]_-{f^{(0)}_{\overline\kappa*}}
& \rH^{2j}(X^{\prime(1)}_{\overline\kappa},\Lambda(j)) \ar[d]^-{f^{(1)}_{\overline\kappa*}} \\
\rH^{2j}(X^{(0)}_{\overline\kappa},\Lambda(j)) \ar[r]^-{\delta_0^*} &
\rH^{2j}(X^{(1)}_{\overline\kappa},\Lambda(j)),
}\]
which induces, together with duality, maps
\[f_{\kappa!}\colon B^j(X'_\kappa,\Lambda)\to B^j(X_\kappa,\Lambda),\quad f_\kappa^!\colon B_j(X_\kappa,\Lambda)\to B_j(X'_\kappa,\Lambda)\]
for every integer $j$. Trivially, we have maps
\[f_\kappa^*\colon B^j(X_\kappa,\Lambda)\to B^j(X'_\kappa,\Lambda),\quad f_{\kappa*}\colon B_j(X'_\kappa,\Lambda)\to B_j(X_\kappa,\Lambda)\]
for every integer $j$.
\end{construction}

\begin{construction}\label{co:correspondence}
An \'{e}tale correspondence $T=(X',f,g)$ of $X$ induces an endomorphism on $B_j(X_\kappa,\Lambda)$ as $f_{\kappa*}\circ g_\kappa^!$, and an endomorphism on $B^j(X_\kappa,\Lambda)$ as $f_{\kappa!}\circ g_\kappa^*$, for every integer $j$. In both cases, we denote this endomorphism by $T^\star$. Clearly, we have $(T_2\circ T_1)^\star=T_1^\star\circ T_2^\star$.

By the functoriality of weight spectral sequences \cite{Sai03}*{\Sec 2.3}, the following commutative diagram commutes
\[\xymatrix{
B_{n+1-r}(X_\kappa,\Lambda) \ar[r]^-{\eqref{eq:potential_premap}}
& B^r(X_\kappa,\Lambda)  \\
B_{n+1-r}(X_\kappa,\Lambda) \ar[u]^-{T^\star}\ar[r]^-{\eqref{eq:potential_premap}}
& B^r(X_\kappa,\Lambda). \ar[u]_-{T^\star}}
\]
Moreover, the map $T^\star$ factorizes through (resp.\ preserves) $B_j(X_\kappa,\Lambda)_0$ (resp.\ $B^j(X_\kappa,\Lambda)^0$).

Now if $\Lambda$ is a nice coefficient for $\pres{r}\rE_X$, then by (the proof of) Theorem \ref{th:galois_cohomology}, we have the following commutative diagram
\[\xymatrix{
A^r(X_\kappa,\Lambda)^0 \ar[r]^-{\eta^r}
& \rH^1_\sing(K,\rH^{2r-1}(X_{\overline{K}},\Lambda(r))) \\
A^r(X_\kappa,\Lambda)^0 \ar[r]^-{\eta^r}\ar[u]_-{T^\star} &
\rH^1_\sing(K,\rH^{2r-1}(X_{\overline{K}},\Lambda(r))), \ar[u]_-{T^\star} }\]
where the second endomorphism $T^\star$ is defined as $f_*\circ g^*$.
\end{construction}

Suppose that we have a monoid $\dT$, regarded as a discrete monoidal category, and a monoidal functor $\dT\to\EC(X)$. Then $\pres{r}\rE=\pres{r}\rE_X$ becomes a spectral sequence of $\Lambda[\dT]$-complexes through the monoidal functor. In particular, every term $\pres{r}\rE^{p,q}_s$ in the spectral sequence and cohomology groups $\rH^q(X_{\overline\kappa}^{(p)},\Lambda(r))$, $\rH^{p+q}(X_{\overline{K}},\Lambda(r))$ become $\Lambda[\dT]$-modules. In what follow, we will always assume that $\dT$ is commutative. Let $\fm$ be a prime ideal of $\Lambda[\dT]$. Since localization is exact, we have the localized weight spectral sequences $\pres{r}\rE_\fm=\pres{r}\rE_{X,\fm}$ such that $\pres{r}\rE_{\fm,s}^{p,q}=(\pres{r}\rE^{p,q}_s)_\fm$ and
\[\pres{r}\rE_{\fm,1}^{p,q}=\bigoplus_{i\geq\max(0,-p)}\rH^{q-2i}(X_{\overline{\kappa}}^{(p+2i)},\Lambda(r-i))_\fm
\Rightarrow\rH^{p+q}(X_{\overline{K}},\Lambda(r))_\fm.\]

By Construction \ref{co:correspondence}, the map \eqref{eq:potential_premap} is $\Lambda[\dT]$-linear. Thus, we have the following localized map
\begin{align}\label{eq:potential_premap_localized}
(B_{n+1-r}(X_\kappa,\Lambda)_0)_\fm\to B^r(X_\kappa,\Lambda)^0_\fm.
\end{align}
For the sake of notational symmetry, we write $B_{n+1-r}(X_\kappa,\Lambda)_0^\fm$ for $(B_{n+1-r}(X_\kappa,\Lambda)_0)_\fm$. Since taking $\rG_\kappa$-invariants and $\rG_\kappa$-coinvariants commute with localization, we have \[A_{n+1-r}(X_\kappa,\Lambda)_0^\fm\coloneqq(A_{n+1-r}(X_\kappa,\Lambda)_0)_\fm=(B_{n+1-r}(X_\kappa,\Lambda)_0^\fm)^{\rG_\kappa},\] and $A^r(X_\kappa,\Lambda)^0_\fm=(B^r(X_\kappa,\Lambda)^0_\fm)^{\rG_\kappa}$.

\begin{definition}[Localized potential map]
We call the following map
\begin{align*}
\Delta^r_\fm\colon A_{n+1-r}(X_\kappa,\Lambda)_0^\fm\to A^r(X_\kappa,\Lambda)^0_\fm,
\end{align*}
induced by \eqref{eq:potential_premap_localized} the \emph{localized potential map}.
\end{definition}

\begin{definition}[Nice coefficient, case of localization]\label{de:nice_coefficient_localized}
We say that $\Lambda$ is a \emph{nice coefficient} for the spectral sequence $\pres{r}\rE_\fm=\pres{r}\rE_{X,\fm}$ if the following are satisfied:
\begin{description}
  \item[(N1)] $\pres{r}\rE_\fm$, with the coefficient $\Lambda$, degenerates at the second page;

  \item[(N2)] if $\pres{r}\rE_{\fm,2}^{p,2r-1-p}(-1)$ has a non-trivial subquotient on which $\rG_\kappa$ acts trivially, then $p=1$.
\end{description}
We say that $\Lambda$ is a \emph{very nice coefficient} if moreover
\begin{description}
  \item[(N3)] for every subquotient $\Lambda[\rG_\kappa]$-module $M$ of $\ker\mu^{1,2r-1}_\fm\oplus\coker\mu^{-1,2r-1}_\fm(-1)$, the natural map $M^{\rG_\kappa}\to M_{\rG_\kappa}$ is an isomorphism.
\end{description}
\end{definition}

The following is the localized version of Theorem \ref{th:galois_cohomology}, with the same proof \footnote{However, it is not a corollary of Theorem \ref{th:galois_cohomology} since $\Lambda$ might not be a nice coefficient for the original spectral sequence $\pres{r}\rE_X$.}.

\begin{theorem}\label{th:galois_cohomology_localized}
Suppose that $\Lambda$ is a nice coefficient for $\pres{r}\rE_{X,\fm}$.
\begin{enumerate}
  \item We have a canonical map
        \[\eta^r_\fm\colon A^r(X_\kappa,\Lambda)^0_\fm\to\rH^1_\sing(K,\rH^{2r-1}(X_{\overline{K}},\Lambda(r))_\fm).\]

  \item If $\Lambda$ is moreover very nice, then we have a canonical exact sequence
        \begin{multline*}
        0 \to \rH^1_\unr(K,\rH^{2r-1}(X_{\overline{K}},\Lambda(r))_\fm)\to
        A_{n+1-r}(X_\kappa,\Lambda)_0^\fm \\
        \xrightarrow{\Delta^r_\fm} A^r(X_\kappa,\Lambda)^0_\fm
        \xrightarrow{\eta^r_\fm}
        \rH^1_\sing(K,\rH^{2r-1}(X_{\overline{K}},\Lambda(r))_\fm) \to
        0.
        \end{multline*}
\end{enumerate}
\end{theorem}

\subsection{Relation with Abel--Jacobi maps}
\label{ss:relation_abel}

Let $X$ be a proper strictly semistable scheme over $O_K$ purely of relative dimension $n$. For every integer $r\geq0$, we have the (absolute) cycle class map
\begin{align}\label{eq:cycle_class}
\cl\colon\CH^r(X_K)\otimes\Lambda\to\rH^{2r}(X_K,\Lambda(r)).
\end{align}
Denote by $\CH^r(X_K,\Lambda)^0$ the kernel of $\cl$ composed with $\rH^{2r}(X_K,\Lambda(r))\to\rH^{2r}(X_{\overline{K}},\Lambda(r))$. Then we further have the Abel--Jacobi map
\begin{align*}
\alpha\colon\CH^r(X_K,\Lambda)^0\to\rH^1(K,\rH^{2r-1}(X_{\overline{K}},\Lambda(r))).
\end{align*}

Suppose that we have a commutative monoid $\dT$ and a monoidal functor $\dT\to\EC(X)$. Let $\fm$ be a prime ideal of $\Lambda[\dT]$. Then \eqref{eq:cycle_class} is a map of $\Lambda[\dT]$-modules. Thus we obtain the localized Abel--Jacobi map
\begin{align}\label{eq:abel_jacobi}
\alpha_\fm\colon\CH^r(X_K,\Lambda)^0_\fm\to\rH^1(K,\rH^{2r-1}(X_{\overline{K}},\Lambda(r))_\fm)
\end{align}
as the canonical map $\rH^1(K,\rH^{2r-1}(X_{\overline{K}},\Lambda(r)))_\fm\to\rH^1(K,\rH^{2r-1}(X_{\overline{K}},\Lambda(r))_\fm)$ is an isomorphism.

Let $z$ be an algebraic cycle on $X_K$ of codimension $r$. Denote by $Z$ the Zariski closure of the support of $z$ in $X$. We have a closed immersion $\iota\colon Z\to X$. Let $z^\sharp$ be the unique cycle on $X$ of codimension $r$ supported on $Z$ whose restriction to $X_K$ is $z$. Then $z^\sharp$ determines a class $[z^\sharp]\in\rH^{2r}_Z(X,\Lambda(r))$ by the cycle class map. Denote by $\widetilde{z}$ the image of $[z^\sharp]$ under the composite map
\[\rH^{2r}_Z(X,\Lambda(r))=\rH^{2r}(X,\iota_!\iota^!\Lambda(r))\to\rH^{2r}(X,\Lambda(r))
\to\rH^{2r}(X_\kappa,\Lambda(r))\to\rH^{2r}(X_{\overline\kappa},\Lambda(r))^{\rG_\kappa}.\]

For further discussion, let us review some notation for the nearby cycles functor. Put $S=\Spec O_{K^\ur}$. Denote by $\overline{j}\colon X_{\overline{K}}\to X_S$ a geometric generic fiber and $\overline{i}\colon X_{\overline\kappa}\to X_S$ the special fiber. Then the nearby cycles functor $\rR\Psi\colon\rD^+(X_S,\Lambda)\to\rD^+(X_{\overline\kappa},\Lambda)$ is defined to be the composite functor $\overline{i}^*\circ\overline{j}_*\circ\overline{j}^*$. Here, the $*$-pushforward is understood in the derived sense. By adjunction, we have a natural transformation $\overline{i}^*\to\rR\Psi$ whose cone is usually called the vanishing cycles functor. In particular, we have a map \[\rH^{2r}(X_{\overline\kappa},\Lambda(r))\to\rH^{2r}(X_{\overline\kappa},\rR\Psi\Lambda(r))\simeq\rH^{2r}(X_{\overline{K}},\Lambda(r))\]
of $\Lambda[\dT]$-modules, where the isomorphism is given by the proper base change. Denote by $\rH^{2r}(X_{\overline\kappa},\Lambda(r))^0$ the kernel of the above map.

\begin{lem}\label{le:cohomology}
If (the induced localized Chow class of) $z$ belongs to $\CH^r(X_K,\Lambda)^0_\fm$, then $\widetilde{z}$ belongs to $\rH^{2r}(X_{\overline\kappa},\Lambda(r))^0_\fm$.
\end{lem}

\begin{proof}
We may take the base change to $S$. Thus we assume that $K=K^\ur$ and $\kappa=\overline\kappa$. Now take an element $t\in\Lambda[\dT]$ that does not belong to $\fm$ such that the Chow class of $tz$ belongs to $\CH^r(X_K,\Lambda)^0$. If we replace $z$ by $tz$ (and $Z$ by $tZ$), then $t\widetilde{z}=\widetilde{tz}$ and we only need to show that $\widetilde{tz}\in\rH^{2r}(X_{\overline\kappa},\Lambda(r))^0$. Thus we may assume at the beginning that $z$ belongs to $\CH^r(X_K,\Lambda)^0$.

Note that the following diagram
\[\xymatrix{
\rH^{2r}(X,\iota_!\iota^!\Lambda(r)) \ar[r]\ar[d] & \rH^{2r}(X_{\overline{K}},\iota_!\iota^!\Lambda(r)) \ar[d] \\
\rH^{2r}(X,\Lambda(r)) \ar[r] & \rH^{2r}(X_{\overline{K}},\Lambda(r))
}\]
commutes. Since $z$ belongs to $\CH^r(X_K,\Lambda)^0$, we have that the image of $[z^\sharp]$ under the composite map
\[\rH^{2r}(X,\iota_!\iota^!\Lambda(r))\to\rH^{2r}(X,\Lambda(r))\to\rH^{2r}(X_{\overline{K}},\Lambda(r))\]
is zero. Since the second map $\rH^{2r}(X,\Lambda(r))\to\rH^{2r}(X_{\overline{K}},\Lambda(r))$ coincides with the composition
\[\rH^{2r}(X,\Lambda(r))\to\rH^{2r}(X_\kappa,\Lambda(r))\to\rH^{2r}(X_\kappa,\rR\Psi\Lambda(r))
\xrightarrow{\sim}\rH^{2r}(X_{\overline{K}},\Lambda(r)),\]
we have $\widetilde{z}\in\rH^{2r}(X_{\overline\kappa},\Lambda(r))^0$. The lemma follows.
\end{proof}

Lemma \ref{le:cohomology} implies that if $z$ belongs to $\CH^r(X_K,\Lambda)^0_\fm$ and $\pres{r}\rE^{0,2r}_{\fm,2}=\pres{r}\rE^{0,2r}_{\fm,\infty}$, then the image of $\widetilde{z}$ under the composite map \[\rH^{2r}(X_{\overline\kappa},\Lambda(r))^{\rG_\kappa}\to\rH^{2r}(X^{(0)}_{\overline\kappa},\Lambda(r))
\to\rH^{2r}(X^{(0)}_{\overline\kappa},\Lambda(r))_\fm\]
belongs to the subspace $A^r(X_\kappa,\Lambda)^0_\fm$. Recall from \Sec\ref{ss:notation} that we have a map
\[\partial\colon\rH^1(K,\rH^{2r-1}(X_{\overline{K}},\Lambda(r)))\to\rH^1_\sing(K,\rH^{2r-1}(X_{\overline{K}},\Lambda(r)))\]
of $\Lambda[\dT]$-modules. The following theorem establishes the compatibility between the (localized) Abel--Jacobi map and the map $\eta^r_\fm$ constructed in Theorem \ref{th:galois_cohomology_localized}.

\begin{theorem}\label{th:relation_abel}
Suppose that $\Lambda$ is a nice coefficient for $\pres{r}\rE_\fm=\pres{r}\rE_{X,\fm}$ (Definition \ref{de:nice_coefficient_localized}). Then for every algebraic cycle $z$ on $X_K$ of codimension $r$ belonging to $\CH^r(X_K,\Lambda)^0_\fm$, we have
\[\eta^r_\fm(\widetilde{z})=\partial(\alpha_\fm(z))\]
as an equality in $\rH^1_\sing(K,\rH^{2r-1}(X_{\overline{K}},\Lambda(r))_\fm)$. Here, we view $\widetilde{z}$ as an element in $A^r(X_\kappa,\Lambda)^0_\fm$ by the above discussion as $\pres{r}\rE^{0,2r}_{\fm,2}=\pres{r}\rE^{0,2r}_{\fm,\infty}$.
\end{theorem}

\begin{proof}
The proof is divided into several steps.

\emph{Step 1.} As in the proof of Lemma \ref{le:cohomology}, we take an element $t\in\Lambda[\dT]$ that does not belong to $\fm$ such that the Chow class of $tz$ belongs to $\CH^r(X_K,\Lambda)^0$. If we can show that $\eta^r_\fm(\widetilde{tz})=\partial(\alpha_\fm(tz))$, then the theorem follows since $t$ is an invertible operator on $\rH^1_\sing(K,\rH^{2r-1}(X_{\overline{K}},\Lambda(r))_\fm)$. Thus we may assume at the beginning that $z$ belongs to $\CH^r(X_K,\Lambda)^0$.

\emph{Step 2.} Let us first recall the definition of $\partial(\alpha(z))$. By the semi-purity theorem \cite{Fuj02}*{\Sec 8}, we know that $\iota_!\iota^!\Lambda\in\rD^{\geq 2r}(X_{\overline\kappa},\Lambda)$. Therefore, we have a short exact sequence
\begin{align}\label{eq:vanishing1}
\rH^{2r-1}(X_{\overline{K}},\Lambda(r))\hookrightarrow\rH^{2r-1}((X\setminus Z)_{\overline{K}},\Lambda(r))
\to\rH^{2r}(X_{\overline{K}},\iota_!\iota^!\Lambda(r))\to\rH^{2r}(X_{\overline{K}},\Lambda(r)).
\end{align}
Denote by $[z]$ the cycle class of $z$ in $\rH^{2r}(X_{\overline{K}},\iota_!\iota^!\Lambda(r))$, which maps to the geometric cycle class of $z$ in $\rH^{2r}(X_{\overline{K}},\Lambda(r))$. Then $[z]$ is in the kernel of the last map in \eqref{eq:vanishing1} since $z$ belongs to $\CH^r(X_K,\Lambda)^0$. As $[z]$ is fixed by $\rG_K$, it induces an element
\[\delta[z]\in\rH^1(\rI_K,\rH^{2r-1}(X_{\overline{K}},\Lambda(r)))^{\rG_\kappa}\]
by the coboundary map for the functor $\rH^0(\rI_K,-)$. Then $\partial(\alpha(z))=\delta[z]$ under the isomorphism in Lemma \ref{le:singular}.

\emph{Step 3.} To compute $\delta[z]$, we pick up an arbitrary $[z]'\in\rH^{2r-1}((X\setminus Z)_{\overline{K}},\Lambda(r))$ in the preimage of $[z]$. Then $(T-1)[z]'$ belongs to $\rH^{2r-1}(X_{\overline{K}},\Lambda(r))$, and its image in the quotient
\begin{align}\label{eq:vanishing3}
\xymatrix{\frac{\rH^{2r-1}(X_{\overline{K}},\Lambda(r))_\fm}{(T-1)\rH^{2r-1}(X_{\overline{K}},\Lambda(r))_\fm},}
\end{align}
which is isomorphic to $\rH^1(\rI_K,\rH^{2r-1}(X_{\overline{K}},\Lambda(r)))$ as $\Lambda$-modules, coincides with $\delta[z]$. Therefore, our goal is to compute the image of $(T-1)[z]'$ in \eqref{eq:vanishing3}. It does not depend on the choice of $[z]'$.

\emph{Step 4.} Now we want to make a convenient choice of $[z]'$. Denote by $F$ the cone of the composite map $(\iota_{S!}\iota_S^!\Lambda)\res_{X_{\overline\kappa}}\to\Lambda\to\rR\Psi\Lambda$ (in the triangulated category $\rD^+(X_{\overline\kappa},\Lambda)$), and by $G$ the cone of the map $\rR\Psi\iota_{S!}\iota_S^!\Lambda\to\rR\Psi\Lambda$. Then we have the following map
\[\xymatrix{
(\iota_{S!}\iota_S^!\Lambda)\res_{X_{\overline\kappa}}  \ar[r]\ar[d] & \rR\Psi\Lambda \ar[r]\ar[d]^-{=} & F \ar[r]^-{+1}\ar[d] & \\
\rR\Psi\iota_{S!}\iota_S^!\Lambda  \ar[r] & \rR\Psi\Lambda \ar[r] & G \ar[r]^-{+1} &
}\]
of exact triangles. It induces the following commutative diagram
\begin{align}\label{eq:vanishing2}
\xymatrix{
\rH^{2r-1}(X_{\overline{K}},\Lambda(r)) \ar@{^(->}[r]\ar[d]^-{=} & \rH^{2r-1}(X_{\overline\kappa},F(r)) \ar[r]\ar[d] &
\rH^{2r}(X_{\overline\kappa},\iota_{S!}\iota_S^!\Lambda) \ar[r]\ar[d] & \rH^{2r}(X_{\overline{K}},\Lambda(r)) \ar[d]^-{=}  \\
\rH^{2r-1}(X_{\overline{K}},\Lambda(r)) \ar@{^(->}[r] & \rH^{2r-1}(X_{\overline\kappa},G(r)) \ar[r] &
\rH^{2r}(X_{\overline\kappa},\rR\Psi\iota_{S!}\iota_S^!\Lambda) \ar[r] & \rH^{2r}(X_{\overline{K}},\Lambda(r)) }
\end{align}
in which the bottom line is canonically isomorphic to \eqref{eq:vanishing1}. Note that $z$, or rather $z^\sharp$, induces a class $[z^\sharp]\in\rH^{2r}(X,\iota_!\iota^!\Lambda(r))$. Denote by $[z^\sharp]_0$ the image of $[z^\sharp]$ under the restriction map $\rH^{2r}(X,\iota_!\iota^!\Lambda(r))\to\rH^{2r}(X_{\overline\kappa},\iota_{S!}\iota_S^!\Lambda)$. Then $[z^\sharp]_0$ maps to $[z]$ under the third vertical map in \eqref{eq:vanishing2}. Pick up an element $[z^\sharp]'_0\in\rH^{2r-1}(X_{\overline\kappa},F(r))$ that maps to $[z^\sharp]_0$, and then we may take $[z]'$ to be the image of $[z^\sharp]'_0$ under the second vertical map in \eqref{eq:vanishing2}.

Now since $T$ acts trivially on $(\iota_{S!}\iota_S^!\Lambda)\res_{X_{\overline\kappa}}$, the map $T-1\colon F\to F$ lifts to a map $T-1\colon F\to\rR\Psi\Lambda$ in $\rD^+(X_{\overline\kappa},\Lambda)$. It induces a canonical map $T-1\colon\rH^{2r-1}(X_{\overline\kappa},F(r))\to\rH^{2r-1}(X_{\overline\kappa},\rR\Psi\Lambda(r))$ on the level of cohomology. So the question is reduced to the computation of the image of $(T-1)[z^\sharp]'_0$, after localization and quotient, in \eqref{eq:vanishing3}.

\emph{Step 5.} To proceed, we need to use the monodromy filtration studied in \cite{Sai03}*{\Sec 2.2}. It is an increasing filtration $M_\bullet$ of $\rR\Psi\Lambda$ such that the monodromy operator $T-1$ sends $M_\bullet\rR\Psi\Lambda$ into $M_{\bullet-2}\rR\Psi\Lambda$. In particular, we have a map $T-1\colon \rR\Psi\Lambda/M_0\rR\Psi\Lambda\to \rR\Psi\Lambda/M_{-2}\rR\Psi\Lambda$. Therefore we have submodules
$(T-1)\rH^{2r-1}(X_{\overline\kappa},\rR\Psi\Lambda(r)/M_{-2}\rR\Psi\Lambda(r))\subset
(T-1)\rH^{2r-1}(X_{\overline\kappa},\rR\Psi\Lambda(r)/M_0\rR\Psi\Lambda(r))\subset
\rH^{2r-1}(X_{\overline\kappa},\rR\Psi\Lambda(r)/M_{-2}\rR\Psi\Lambda(r))$. Since $\Lambda$ is a nice coefficient for $\pres{r}\rE_{X,\fm}$, a similar argument in the proof of Theorem \ref{th:galois_cohomology} (1) will show that
the canonical inclusion
\[\xymatrix{\(\frac{(T-1)\rH^{2r-1}(X_{\overline\kappa},\rR\Psi\Lambda(r)/M_0\rR\Psi\Lambda(r))_\fm}
{(T-1)\rH^{2r-1}(X_{\overline\kappa},\rR\Psi\Lambda(r)/M_{-2}\rR\Psi\Lambda(r))_\fm}\)^{\rG_\kappa}
\subset\(\frac{\rH^{2r-1}(X_{\overline\kappa},\rR\Psi\Lambda(r)/M_{-2}\rR\Psi\Lambda(r))_\fm}
{(T-1)\rH^{2r-1}(X_{\overline\kappa},\rR\Psi\Lambda(r)/M_{-2}\rR\Psi\Lambda(r))_\fm}\)^{\rG_\kappa}}\]
is an isomorphism, both being canonically isomorphic to
\[\xymatrix{
\(\frac{\pres{r}\rE^{1,2r-2}_{\fm,1}(-1)}{\mu^{0,2r-1}_\fm(\Ker\rd^{-1,2r}_{\fm,1})+\IM(\rd^{0,2r-2}_{\fm,1}(-1))}\)^{\rG_\kappa}.}\]

\emph{Step 6.} By the construction in \cite{Sai03}*{\Sec 2.2}, the adjunction map $\Lambda\to\rR\Psi\Lambda$ factorizes through $M_0\rR\Psi\Lambda$. Thus we have the composite map $(\iota_{S!}\iota_S^!\Lambda)\res_{X_{\overline\kappa}}\to\Lambda\to M_0\rR\Psi\Lambda/M_{-2}\rR\Psi\Lambda$, hence a map
\[\xymatrix{
(\iota_{S!}\iota_S^!\Lambda)\res_{X_{\overline\kappa}}  \ar[r]\ar[d] & \rR\Psi\Lambda \ar[r]\ar[d] & F \ar[r]^-{+1}\ar[d] & \\
M_0\rR\Psi\Lambda/M_{-2}\rR\Psi\Lambda  \ar[r] & \rR\Psi\Lambda/M_{-2}\rR\Psi\Lambda \ar[r] & \rR\Psi\Lambda/M_0\rR\Psi\Lambda \ar[r]^-{+1} &
}\]
of exact triangles. Applying the monodromy operator $T-1$, we obtain the following commutative diagram
\begin{align}\label{eq:vanishing4}
\xymatrix{
\frac{(T-1)\rH^{2r-1}(X_{\overline\kappa},F(r))}{(T-1)\rH^{2r-1}(X_{\overline\kappa},\rR\Psi\Lambda(r))}
\ar[r]\ar@{^(->}[d] & \frac{(T-1)\rH^{2r-1}(X_{\overline\kappa},\rR\Psi\Lambda(r)/M_0\rR\Psi\Lambda(r))}
{(T-1)\rH^{2r-1}(X_{\overline\kappa},\rR\Psi\Lambda(r)/M_{-2}\rR\Psi\Lambda(r))}  \ar@{^(->}[d] \\
\frac{\rH^{2r-1}(X_{\overline\kappa},\rR\Psi\Lambda(r))}{(T-1)\rH^{2r-1}(X_{\overline\kappa},\rR\Psi\Lambda(r))}
\ar[r] & \frac{\rH^{2r-1}(X_{\overline\kappa},\rR\Psi\Lambda(r)/M_{-2}\rR\Psi\Lambda(r))}
{(T-1)\rH^{2r-1}(X_{\overline\kappa},\rR\Psi\Lambda(r)/M_{-2}\rR\Psi\Lambda(r))}
}
\end{align}
By Step 5 and a similar argument in the proof of Theorem \ref{th:galois_cohomology} (1), the map
\[\xymatrix{\(\frac{\rH^{2r-1}(X_{\overline\kappa},\rR\Psi\Lambda(r))_\fm}{(T-1)\rH^{2r-1}(X_{\overline\kappa},\rR\Psi\Lambda(r))_\fm}\)^{\rG_\kappa}
\to\(\frac{\rH^{2r-1}(X_{\overline\kappa},\rR\Psi\Lambda(r)/M_{-2}\rR\Psi\Lambda(r))_\fm}
{(T-1)\rH^{2r-1}(X_{\overline\kappa},\rR\Psi\Lambda(r)/M_{-2}\rR\Psi\Lambda(r))_\fm}\)^{\rG_\kappa}}\]
induced from the bottom arrow in \eqref{eq:vanishing4} is canonically isomorphic to the map
\begin{align}\label{eq:vanishing6}
\xymatrix{\(\frac{\Ker(\rd^{1,2r-2}_{\fm,1}(-1))}{\mu^{0,2r-1}_\fm(\Ker\rd^{-1,2r}_{\fm,1})+\IM(\rd^{0,2r-2}_{\fm,1}(-1))}\)^{\rG_\kappa}
\to\(\frac{\pres{r}\rE^{1,2r-2}_{\fm,1}(-1)}{\mu^{0,2r-1}_\fm(\Ker\rd^{-1,2r}_{\fm,1})+\IM(\rd^{0,2r-2}_{\fm,1}(-1))}\)^{\rG_\kappa}.}
\end{align}
In particular, \eqref{eq:vanishing6} is injective. Since $(T-1)[z^\sharp]'_0$ belongs to the $\rG_\kappa$-invariant submodule of \eqref{eq:vanishing3}, it suffices to compute its image in
\begin{align}\label{eq:vanishing7}
\xymatrix{\(\frac{(T-1)\rH^{2r-1}(X_{\overline\kappa},\rR\Psi\Lambda(r)/M_0\rR\Psi\Lambda(r))_\fm}
{(T-1)\rH^{2r-1}(X_{\overline\kappa},\rR\Psi\Lambda(r)/M_{-2}\rR\Psi\Lambda(r))_\fm}\)^{\rG_\kappa}.}
\end{align}

\emph{Step 7.} Now we use the following map
\[\xymatrix{
M_0\rR\Psi\Lambda/M_{-1}\rR\Psi\Lambda  \ar[r]\ar[d]^-{T-1} & \rR\Psi\Lambda/M_{-1}\rR\Psi\Lambda \ar[r]\ar[d]^-{T-1}
& \rR\Psi\Lambda/M_0\rR\Psi\Lambda \ar[r]^-{+1}\ar[d]^-{T-1} & \\
M_{-2}\rR\Psi\Lambda/M_{-3}\rR\Psi\Lambda  \ar[r] & \rR\Psi\Lambda/M_{-3}\rR\Psi\Lambda \ar[r]
& \rR\Psi\Lambda/M_{-2}\rR\Psi\Lambda \ar[r]^-{+1} &
}\]
of exact triangles. Note that the image of the map $T-1\colon\rH^{2r}(X_{\overline\kappa},M_0\rR\Psi\Lambda/M_{-1}\rR\Psi\Lambda)
\to\rH^{2r}(X_{\overline\kappa},M_{-2}\rR\Psi\Lambda/M_{-3}\rR\Psi\Lambda)$ is canonically isomorphic to
\[\xymatrix{\frac{\pres{r}\rE^{0,2r}_1}{\IM(\rd_1^{-1,2r}\circ(\mu^{0,2r-1})^{-1}\circ\rd_1^{0,2r-2}(-1))}.}\]
Together with the discussion in Step 5, we have the induced map
\[\xymatrix{\(\frac{\pres{r}\rE^{1,2r-2}_{\fm,1}(-1)}{\mu^{0,2r-1}_\fm(\Ker\rd^{-1,2r}_{\fm,1})+\IM(\rd^{0,2r-2}_{\fm,1}(-1))}\)^{\rG_\kappa}
\simeq\eqref{eq:vanishing7}\to
\(\frac{\pres{r}\rE^{0,2r}_{\fm,1}}{\IM(\rd_{\fm,1}^{-1,2r}\circ(\mu^{0,2r-1}_\fm)^{-1}\circ\rd_{\fm,1}^{0,2r-2}(-1))}\)^{\rG_\kappa}}.\]
It is simply $\rd_{\fm,1}^{-1,2r}\circ(\mu^{0,2r-1}_\fm)^{-1}$, which is in particular injective. Therefore, to prove the theorem, it suffices to show that the image of $[z^\sharp]'_0$ under the composite map
\[\rH^{2r-1}(X_{\overline\kappa},F(r))\to\rH^{2r-1}(X_{\overline\kappa},\rR\Psi\Lambda(r)/M_0\rR\Psi\Lambda(r))
\to\rH^{2r}(X_{\overline\kappa},M_0\rR\Psi\Lambda/M_{-1}\rR\Psi\Lambda)\simeq\pres{r}\rE^{0,2r}_1\]
coincides with the image of $\widetilde{z}$ under the restriction map
\[\rH^{2r}(X_{\overline\kappa},\Lambda(r))\to\rH^{2r}(X^{(0)}_{\overline\kappa},\Lambda(r))\to\pres{r}\rE^{0,2r}_1.\]
This follows from the following diagram
\[\xymatrix{
F \ar[r]\ar[d] & (\iota_{S!}\iota_S^!\Lambda[1])\res_{X_{\overline\kappa}} \ar[r]\ar[d] & \Lambda[1] \ar[d] \\
\rR\Psi\Lambda/M_0\rR\Psi\Lambda \ar[r] & (M_0\rR\Psi\Lambda/M_{-1}\rR\Psi\Lambda)[1] \ar[r]^-{=} & (M_0\rR\Psi\Lambda/M_{-1}\rR\Psi\Lambda)[1]
}\]
which is tautologically commutative by definition. The theorem follows.
\end{proof}

The following proposition provides a way to compute the image of $\widetilde{z}$ in $\rH^{2r}(X^{(0)}_{\overline\kappa},\Lambda(r))$, which will be used later. We first introduce more notation. For an algebraic cycle $z$ on $X_K$ of codimension $r$, let $Z$ and $z^\sharp$ be as in the previous discussion. Suppose that $Z$ has irreducible components $Z_1,\dots,Z_t$ where each is of codimension $r$ in $X$. If we have $z=\sum_{i=1}^ta_iZ_{iK}$, then $z^\sharp=\sum_{i=1}^ta_iZ_i$. For each $i$, the fiber product $Z_i\times_XX^{(0)}_\kappa$ is a closed subscheme of $X^{(0)}_\kappa$ of pure codimension $r$. Thus $z^\sharp\times_XX^{(0)}_\kappa\coloneqq\sum_{i=1}^ta_i(Z_i\times_XX^{(0)}_\kappa)$ is an algebraic cycle on $X^{(0)}_\kappa$ of codimension $r$.

\begin{proposition}\label{pr:relation_abel}
Let $z$ be an algebraic cycle on $X_K$ of codimension $r$. Suppose that for every irreducible component $Z_i$ of $Z$, the codimension of the singular locus $Z_i^\sing$ of $Z_i$ in each irreducible component $X_j$ of $X_\kappa$ is at least $r+1$. Then the image of $\widetilde{z}$ under the restriction map $\rH^{2r}(X_{\overline\kappa},\Lambda(r))\to\rH^{2r}(X^{(0)}_{\overline\kappa},\Lambda(r))$ coincides with the geometric cycle class of $z^\sharp\times_XX^{(0)}_\kappa$.
\end{proposition}

\begin{proof}
Without lost of generality, we assume that $Z$ is a prime divisor. Since the codimension of $Z^\sing$ in $X$ is at least $r+1$, the restriction map $\rH^{2r}(X,\Lambda(r))\to\rH^{2r}(X\setminus Z^\sing,\Lambda(r))$ is an isomorphism by semi-purity theorem. Under such isomorphism, we have $\cl(Z)=\cl(Z\setminus Z^\sing)$. Now we choose an arbitrary irreducible component $X_j$ of $X_\kappa$, and consider the morphism $f\colon X_j\setminus Z^\sing\to X\setminus Z^\sing$. By the functoriality of the cycle class map \cite{Fuj02}*{Proposition 1.1.3}, we know that the image of $\widetilde{z}$ under the restriction map $\rH^{2r}(X_{\overline\kappa},\Lambda(r))\to\rH^{2r}((X_j\setminus Z^\sing)_{\overline\kappa},\Lambda(r))$ coincides with the geometric cycle class of $f^*(Z\setminus Z^\sing)$. By the assumption that the codimension of $Z^\sing$ in $X_j$ is at least $r+1$, we have the canonical isomorphism $\rH^{2r}(X_{j,\overline\kappa},\Lambda(r))\to\rH^{2r}((X_j\setminus Z^\sing)_{\overline\kappa},\Lambda(r))$ under which the geometric cycle class of $f^*(Z)$ coincides with the geometric cycle class of $f^*(Z\setminus Z^\sing)$. The proposition follows.
\end{proof}

\section{A reciprocity law for cubic Hirzebruch--Zagier cycles}
\label{ss:3}

Recall from \Sec\ref{ss:main} that we have fixed a totally real cubic field $F$, with various notation $O_F$, $\dA_F$, $\Phi_F$, and
$\TF\subset\dC$, $F_0\subset\dC$.

In this chapter, we fix a finite set of even cardinality $\nabla$ of places of $F$ containing $\Phi_F$.

\subsection{Level raising on Hilbert threefolds}
\label{ss:level_raising}

The initial data for the level raising is a quadruple $(\rho,\fr_\rho,\fr_0,\fr_1)$ where
\begin{itemize}
 \item $\rho\colon\rG_F\to\GL(\rN_\rho)$ is a homomorphism with a free $\dZ/p^\nu$-module $\rN_\rho$ of rank $2$, such that $\det\rho\simeq\dZ/p^\nu(-1)$;

  \item $\fr_\rho$ is an ideal of $O_F$ coprime to $\nabla$ such that $\rho$ is unramified outside $\nabla$ and $\fr_\rho$;

  \item $\fr_0$ is an ideal of $O_F$ coprime to $\nabla$;

  \item $\fr_1$ is an absolutely neat ideal (Definition \ref{bde:absolutely_neat}) of $O_F$ that is coprime to $\nabla$ and $\fr_0$, and such that $\fr_0\fr_1\subset\fr_\rho$.
\end{itemize}
For a given quadruple, we will denote by $\fr$ the ideal $\fr_0\fr_1$, and $\Lambda$ for $\dZ/p^\nu$.

\begin{notation}\label{no:nabla}
Let $\fr$ be an ideal of $O_F$.
\begin{enumerate}
  \item Put $\cS_\fr=\cS(\nabla)_\fr$ (Definition \ref{ade:eichler}) if $\fr$ is coprime to $\nabla$.
  \item Let $\dT^\fr$ be the coproduct of $\dT_\fq$ (Definition \ref{ade:hecke}) for all primes $\fq$ of $F$ that are coprime to $\nabla$ and $\fr$ in the category of commutative monoids.
\end{enumerate}
\end{notation}

For every ideal $\fs$ of $O_F$ contained in $\fr_\rho$, we have an induced homomorphism
\[\phi^\fs_\rho\colon\dZ[\dT^\fs]\to\Lambda(=\dZ/p^\nu)\]
such that $\phi^\fs_\rho(\rT_\fq)=\tr\rho(\Frob_\fq)$ and $\phi^\fs_\rho(\rS_\fq)=1$, for every prime $\fq$ of $F$ coprime to $\nabla$ and $\fs$. Here, $\Frob_\fq$ denotes a geometric Frobenius at $\fq$. Put
\[\fm_\rho^\fs=\Ker[\dZ[\dT^\fs]\xrightarrow{\phi^\fs_\rho}\Lambda\to\dF_p]\]
which is a maximal ideal of $\dZ[\dT^\fs]$. If we put $\bar\rho\coloneqq\rho\mod p$, then $\fm_\rho^\fs=\fm_{\bar\rho}^\fs$.

\begin{definition}[Perfect quadruple]\label{de:perfect_prime}
We say that
\begin{enumerate}
  \item $\bar\rho$ is \emph{generic} if $(\Ind^\dQ_F\bar\rho)\res_{\rG_{\TF}}$ has the largest possible image, which is isomorphic to $\rG(\SL_2(\dF_p)\oplus\SL_2(\dF_p)\oplus\SL_2(\dF_p))$;

  \item a quadruple $(\rho,\fr_\rho,\fr_0,\fr_1)$ is \emph{$\fs$-isolated}, for an ideal $\fs$ of $O_F$ contained in $\fr$, if $\Gamma(\cS_{\fr_\rho},\Lambda)_{\fm^\fs_\rho}$ (see \Sec\ref{ss:notation} for $\Gamma$) is a free $\Lambda$-module of rank $1$ and the maps in the following commutative diagram
      \[\xymatrix{
      \Gamma(\cS_\fr,\Lambda)_{\fm^\fs_\rho}
      \ar[rrr]^-{\bigoplus_{\fd\in\fD(\fr,\fr_\rho)}\delta^\fd_*}\ar[d] &&&
      \bigoplus_{\fd\in\fD(\fr,\fr_\rho)}\Gamma(\cS_{\fr_\rho},\Lambda)_{\fm^\fs_\rho} \ar[d] \\
      \Gamma(\cS_\fr,\Lambda)/\Ker\phi^\fs_\rho
      \ar[rrr]^-{\bigoplus_{\fd\in\fD(\fr,\fr_\rho)}\delta^\fd_*}  &&&
      \bigoplus_{\fd\in\fD(\fr,\fr_\rho)}\Gamma(\cS_{\fr_\rho},\Lambda)/\Ker\phi^\fs_\rho
      }\]
      are all isomorphisms, where $\fD(\fr,\fr_\rho)$ is introduced in Notation \ref{ano:divisor} and $\delta^\fd_*$ is the (normalized) pushforward map in Definition \ref{ade:pushforward};

  \item a quadruple $(\rho,\fr_\rho,\fr_0,\fr_1)$ is \emph{perfect} if
      \begin{enumerate}
        \item $p\geq 11$ and $p\neq 13,19$;
        \item $p$ is coprime to $\nabla$ and $\fr\cdot\mu(\fr,\fr_\rho)\cdot|\Cl(F)_{\fr_1}|\cdot\disc F$ (Notation \ref{ano:divisor} for $\mu(\fr,\fr_\rho)$);
        \item $\bar\rho$ is generic; and
        \item it is $\fr$-isolated.
      \end{enumerate}
\end{enumerate}
\end{definition}

From now on, we will fix a perfect quadruple $(\rho,\fr_\rho,\fr_0,\fr_1)$.

\begin{definition}[Cubic-level raising prime]\label{de:cubic_level_raising}
We say that a rational prime $\ell$ is a \emph{cubic-level raising prime} for the quadruple $(\rho,\fr_\rho,\fr_0,\fr_1)$ if
\begin{description}
  \item[(C1)] $\ell$ is inert in $F$, unramified in $\TF$, and coprime to $\nabla$ and $2\fr$;

  \item[(C2)] $(\rho,\fr_\rho,\fr_0,\fr_1)$ is $\fr\fl$-isolated, where $\fl$ is the unique prime of $F$ above $\ell$;

  \item[(C3)] $p$ does not divide $\ell(\ell^{18}-1)(\ell^6+1)$;

  \item[(C4)] $\phi_\Pi(\rT_\fl)\equiv\ell^3+1\mod p^\nu$.
\end{description}
\end{definition}

\begin{notation}\label{no:reduction}
We denote by $\rho^\sharp\colon\rG_\dQ\to\GL(\rN_\rho^\sharp)$ the multiplicative induction of $\rho$ from $\rG_F$ to $\rG_\dQ$, where $\rN_\rho^\sharp=\rN_\rho^{\otimes 3}$. Put $\bar\rN_\rho\coloneqq\rN_\rho\otimes\dF_p$ and $\bar\rN_\rho^\sharp\coloneqq\rN_\rho^\sharp\otimes\dF_p$, which are $\dF_p[\rG_\dQ]$-modules.
\end{notation}

Let $\ell$ be a prime satisfying (C1) in Definition \ref{de:cubic_level_raising} and $\fl$ the unique prime of $F$ above $\ell$. From Definition \ref{bde:hilbert_shimura} (with $\Delta=\{\fl\}\cup\nabla\setminus\Phi_F$), we have the scheme \[\cX(\ell)_{\fr_0,\fr_1}\coloneqq\cX(\{\fl\}\cup\nabla\setminus\Phi_F)_{\fr_0,\fr_1}\]
which is projective and of relative dimension $3$ over $\Spec\dZ[(\fr\disc F)^{-1}]$. Moreover, we have a monoidal functor
\begin{align}\label{eq:monoidal}
\dT^{\fr\fl}\to\EC(\cX(\ell)_{\fr_0,\fr_1})
\end{align}
(\Sec\ref{ss:correspondence}) given by Hecke correspondences (Definition \ref{bde:hecke}, Remark \ref{bre:hecke}).

\begin{theorem}\label{th:cubic_level_raising}
Let $(\rho,\fr_\rho,\fr_0,\fr_1)$ be a perfect quadruple. Let $\ell$ be a cubic-level raising prime for $(\rho,\fr_\rho,\fr_0,\fr_1)$. Then we have
\begin{enumerate}
  \item $\rH^j(\cX(\ell)_{\fr_0,\fr_1}\otimes\overline\dQ,\dZ_p)_{\fm_\rho^{\fr\fl}}=0$ for $j\neq 3$;

  \item a canonical decomposition of the $\Lambda[\rG_\dQ]$-module
      \[\rH^3(\cX(\ell)_{\fr_0,\fr_1}\otimes\overline\dQ,\dZ_p(2))/\Ker\phi^{\fr\fl}_\rho=\bigoplus_{\fd\in\fD(\fr,\fr_\rho)}\rM_0\]
      where $\rM_0$ is isomorphic to $\rN_\rho^\sharp(2)$ as a $\Lambda[\rG_{F_0}]$-module;

  \item a canonical isomorphism
      \[\rH^1_\sing(\dQ_\ell,\rH^3(\cX(\ell)_{\fr_0,\fr_1}\otimes\overline\dQ,\dZ_p(2))/\Ker\phi^{\fr\fl}_\rho)
      \simeq\Gamma(\cS_\fr,\dZ)/\Ker\phi^\fr_\rho.\]
\end{enumerate}
\end{theorem}

\begin{proof}[Proof of Theorem \ref{th:cubic_level_raising} (1)]
The proof is similar to \cite{Dim05}*{Theorem 6.6 (1)}. By the Poincar\'{e} duality and the Nakayama lemma, it suffices to show that $\rH^j(\cX(\ell)_{\fr_0,\fr_1}\otimes\overline\dQ,\dF_p)/{\fm_\rho^{\fr\fl}}=0$ for $j=0,1,2$. As $p$ is coprime to $\nabla$ and $\fr\cdot\disc F$, the scheme $\cX(\ell)_{\fr_0,\fr_1}\otimes\dZ_p$ is smooth and projective over $\Spec\dZ_p$. By Faltings' Comparison Theorem \cite{Fal88}, we know that $3$ is not a Fontaine--Laffaille weight for $\rH^j(\cX(\ell)_{\fr_0,\fr_1}\otimes\overline\dQ,\dF_p)$ for $j=0,1,2$. The rest of the proof is same as for \cite{Dim05}*{Theorem 6.6 (1)}, as $\bar\rho$ is generic (Definition \ref{de:perfect_prime}).
\end{proof}

The remaining sections of this chapter are dedicated to the proof of Theorem \ref{th:cubic_level_raising} (2,3). Before the end of this section, we record the following lemma.

\begin{lem}\label{le:sharp}
The $\Lambda[\rG_\ell]$-module $\rN_\rho^\sharp(2)$ is unramified and isomorphic to $\Lambda(-1)\oplus\Lambda\oplus\Lambda(1)\oplus\Lambda(2)\oplus\rR\oplus\rR(1)$, where $\rR=\Lambda^2$ as a $\Lambda$-module on which $\Frob_\ell$ acts via the matrix $\big(\begin{array}{cc} 0 & 1 \\ -1 & -1 \\ \end{array} \big)$. In particular, $\rH^1_\sing(\dQ_\ell,\rN_\rho^\sharp(2))$ is a free $\Lambda$-module of rank $1$.
\end{lem}

\begin{proof}
The first part is a consequence of Definition \ref{de:cubic_level_raising} (C1, C3, C4) and the definition of multiplicative induction. The second part is a consequence of the first part and Lemma \ref{le:singular}.
\end{proof}

\subsection{Semistable model and featuring cycles}
\label{ss:featuring}

As in Example \ref{bex:cube}, we identify $\Phi$ with the set $\dZ/6\dZ=\{0,1,2,3,4,5\}$ such that $\sigma i=i+1$, and $\Phi_F$ with the set $\dZ/3\dZ=\{0,1,2\}$ such that $\Phi\to\Phi_F$ is the natural map of modulo $3$. Put $K=\dQ_{\ell^6}$ hence $O_K=\dZ_{\ell^6}$ and $\kappa=\dF_{\ell^6}$ in the setup of \Sec\ref{ss:2}. Recall that $\rG_\kappa=\Gal(\overline\kappa/\kappa)$ is the absolute Galois group of $\kappa$.

From now to the end of \Sec\ref{ss:3}, we fix a rational prime $\ell$ that is a \emph{cubic-level raising prime} for the perfect quadruple $(\rho,\fr_\rho,\fr_0,\fr_1)$.

We have the following geometric objects from \Sec\ref{bss:semistable}.
\begin{itemize}
  \item $\cX_{\fr_0,\fr_1}\coloneqq\cX(\ell)_{\fr_0,\fr_1}=\cX(\{\fl\}\cup\nabla\setminus\Phi_F)_{\fr_0,\fr_1}$, with the special fiber $X_{\fr_0,\fr_1}\coloneqq\cX_{\fr_0,\fr_1}\otimes\kappa$,

  \item the canonical semistable resolution $\pi\colon\cY_{\fr_0,\fr_1}\to\cX_{\fr_0,\fr_1}\otimes O_K$ (Definition \ref{bde:blowup}),

  \item subschemes $X_{\fr_0,\fr_1}^{ijk\cdots}$ of $X_{\fr_0,\fr_1}$,

  \item the morphism $\wp_{ijk\cdots}\colon X_{\fr_0,\fr_1}^{ijk\cdots}\to Z_{\fr_0,\fr_1}^{ijk}$ if $\{i,j,k,\cdots\}\subset\{0,1,2,3,4,5\}$ is a proper subset satisfying $\{i,j,k,\cdots\}^\dag=\{i,j,k\}$.
\end{itemize}

We collect some facts from Appendix \ref{ss:b} in this particular case.

\begin{lem}\label{le:strata}
Let notation be as above. Then for $i\in\{0,1,2,3,4,5\}$,
\begin{enumerate}
  \item $\wp_{i(i+1)(i+2)}\colon X_{\fr_0,\fr_1}^{i(i+1)(i+2)}\to Z_{\fr_0,\fr_1}^{i(i+1)(i+2)}$ is a $\dP^1$-bundle;

  \item $\wp_{i(i+2)(i+4)}\colon X_{\fr_0,\fr_1}^{i(i+2)(i+4)}\to Z_{\fr_0,\fr_1}^{i(i+2)(i+4)}$ is a $(\dP^1)^3$-bundle;

  \item $\wp_{i(i+1)(i+2)(i+3)}\colon X_{\fr_0,\fr_1}^{i(i+1)(i+2)(i+3)}\to Z_{\fr_0,\fr_1}^{(i+1)(i+2)(i+3)}$ is an isomorphism;

  \item the composite morphism
       \[X_{\fr_0,\fr_1}^{i(i+1)(i+2)(i+3)}\to X_{\fr_0,\fr_1}^{i(i+1)(i+2)}\xrightarrow{\wp_{i(i+1)(i+2)}}Z_{\fr_0,\fr_1}^{i(i+1)(i+2)}\]
       is a Frobenius factor, which is finite flat of degree $\ell$;
%
%  \item $X_{\fr_0,\fr_1}^{i(i+1)(i+2)(i+4)}=\{(a,b,\b{f}a)\}\subset\dP_i\times\dP_{i+2}\times\dP_{i+4}=X_{\fr_0,\fr_1}^{i(i+2)(i+4)}$;
%
%  \item $X_{\fr_0,\fr_1}^{i(i+1)(i+2)(i+4)(i+5)}=\{(a,\b{f}^2a,\b{f}a)\}\subset\dP_i\times\dP_{i+2}\times\dP_{i+4}=X_{\fr_0,\fr_1}^{i(i+2)(i+4)}$;
%
%  \item $X_{\fr_0,\fr_1}^{012345}=\{(a,\b{f}^2a,\b{f}a)\res\b{f}^3a=a\}\subset\dP_i\times\dP_{i+2}\times\dP_{i+4}=X_{\fr_0,\fr_1}^{i(i+2)(i+4)}$.
\end{enumerate}
%Here, the products in (4--6) are fiber products over $Z_{\fr_0,\fr_1}^{i(i+2)(i+4)}$, and $\b{f}=\b{f}_i\colon\dP_{i+2}\to\dP_i$ is the morphism %\eqref{beq:bundle_frob}.
\end{lem}

\begin{proof}
Parts (1--3) follow from Theorem \ref{bth:hilbert_subbundle}. Part (4) follows from Proposition \ref{bpr:translation}.
%Parts (4--6) follow from Corollary \ref{bco:hilbert_subbundle}.
\end{proof}

\begin{notation}\label{no:tilde}
We label the two sparse types as $S^+=\{0,2,4\}$ and $S^-=\{1,3,5\}$. Sometimes we write
$Y_{\fr_0,\fr_1}^+$ (resp.\ $X_{\fr_0,\fr_1}^+$, $Z_{\fr_0,\fr_1}^+$) for $Y_{\fr_0,\fr_1}^{024}$ (resp.\ $X_{\fr_0,\fr_1}^{024}$, $Z_{\fr_0,\fr_1}^{024}$), and $Y_{\fr_0,\fr_1}^-$ (resp.\ $X_{\fr_0,\fr_1}^-$, $Z_{\fr_0,\fr_1}^-$) for $Y_{\fr_0,\fr_1}^{135}$ (resp.\ $X_{\fr_0,\fr_1}^{135}$, $Z_{\fr_0,\fr_1}^{135}$). We have canonical $\rG_\kappa$-equivariant maps
\[\delta^\pm\colon X_{\fr_0,\fr_1}^{012345}(\overline\kappa)\to Z_{\fr_0,\fr_1}^\pm(\overline\kappa)\]
induced from the composite morphisms $X_{\fr_0,\fr_1}^{012345}\to X_{\fr_0,\fr_1}^\pm\xrightarrow{\wp_\pm}Z_{\fr_0,\fr_1}^\pm$, respectively. For an element $g\in Z_{\fr_0,\fr_1}^\pm(\overline\kappa)$, we put $\rT(g)=\delta^\mp((\delta^\pm)^{-1}(g))\subset Z_{\fr_0,\fr_1}^\mp(\overline\kappa)$.
\end{notation}

The following lemma identifies certain $1$-cycles on the surfaces $Z_{\fr_0,\fr_1,\overline\kappa}^{i(i+1)(i+2)}$.

\begin{lem}\label{le:cycle_base}
Suppose $i\in S^\pm$. Under the isomorphism (Lemma \ref{le:strata})
\[\wp=\wp_{i(i+1)(i+2)(i+5)}\colon X_{\fr_0,\fr_1,\overline\kappa}^{i(i+1)(i+2)(i+5)}\xrightarrow{\sim}Z_{\fr_0,\fr_1,\overline\kappa}^{i(i+1)(i+2)},\]
we have
\begin{enumerate}
  \item $\wp X_{\fr_0,\fr_1,\overline\kappa}^{i(i+1)(i+2)(i+4)(i+5)}=\coprod_{g\in Z_{\fr_0,\fr_1}^\pm(\overline\kappa)}C^i_g$ where each $C^i_g$ is isomorphic to $\dP^1_{\overline\kappa}$;

  \item $\wp X_{\fr_0,\fr_1,\overline\kappa}^{i(i+1)(i+2)(i+3)(i+5)}=\coprod_{g'\in Z_{\fr_0,\fr_1}^\mp(\overline\kappa)}C^i_{g'}$ where each $C^i_{g'}$ is isomorphic to $\dP^1_{\overline\kappa}$;

  \item the self intersection number of $C^i_g$ (resp.\ $C^i_{g'}$) equals $-2\ell$ (resp.\ $-2\ell^2$) for $g\in Z_{\fr_0,\fr_1}^\pm(\overline\kappa)$ (resp.\ $g'\in Z_{\fr_0,\fr_1}^\mp(\overline\kappa)$).
\end{enumerate}
\end{lem}

\begin{proof}
Without lost of generality, we assume $i=0\in S^+$.

For (1), we have $X_{\fr_0,\fr_1,\overline\kappa}^{01245}=X_{\fr_0,\fr_1,\overline\kappa}^{0125}\cap X_{\fr_0,\fr_1,\overline\kappa}^{024}$, which is a $\dP^1$-bundle over $Z_{\fr_0,\fr_1,\overline\kappa}^{024}$ by Theorem \ref{bth:hilbert_subbundle}. Therefore, we have a canonical bijection
\[\pi_0(X_{\fr_0,\fr_1,\overline\kappa}^{01245})\simeq Z_{\fr_0,\fr_1}^{024}(\overline\kappa)=Z_{\fr_0,\fr_1}^+(\overline\kappa).\]
Thus (1) follows. Part (2) is similar.

For (3), by Proposition \ref{bpr:divisor}, we know that $C^0_g$ (resp.\ $C^0_{g'}$) is an irreducible component of the Goren--Oort divisor $Z_{\fr_0,\fr_1,\overline\kappa}^{012,4}$ (resp.\ $Z_{\fr_0,\fr_1,\overline\kappa}^{012,3}$) of $Z_{\fr_0,\fr_1,\overline\kappa}^{012}$. By \cite{TX14}*{Proposition 2.31 (2)}, we know that the normal bundle of the inclusion $Z_{\fr_0,\fr_1,\overline\kappa}^{012,4}\hookrightarrow Z_{\fr_0,\fr_1,\overline\kappa}^{012}$ (resp.\ $Z_{\fr_0,\fr_1,\overline\kappa}^{012,3}\hookrightarrow Z_{\fr_0,\fr_1,\overline\kappa}^{012}$) is $\sO(-2\ell^{n_4})$ (resp.\ $\sO(-2\ell^{n_3})$). In particular, the self intersection number of $C^i_g$ (resp.\ $C^i_{g'}$) is equal to $-2\ell^{n_4}$ (resp.\ $-2\ell^{n_3}$). Part (3) follows as
$n_4=1$ and $n_3=2$.
\end{proof}

\begin{lem}\label{le:cycle_strata}
Suppose $i\in S^\pm$. We have
\begin{enumerate}
  \item $Y_{\fr_0,\fr_1,\overline\kappa}^{i(i+1)(i+2)(i+4)}=\coprod_{g\in Z_{\fr_0,\fr_1}^\pm(\overline\kappa)}F^i_g$ where each $F^i_g$ is a connected smooth projective surface over $\Spec\overline\kappa$, which coincides with the inverse image of $C^i_g$ under $\wp_{i(i+1)(i+2)}\circ\pi$;

  \item $Y_{\fr_0,\fr_1,\overline\kappa}^{i(i+1)(i+2)(i+3)(i+5)}=\coprod_{g'\in Z_{\fr_0,\fr_1}^\mp(\overline\kappa)}E^i_{g'}$ where each $E^i_{g'}$ is a connected smooth projective surface over $\Spec\overline\kappa$, which is the exceptional divisor whose image under $\wp_{i(i+1)(i+2)}\circ\pi$ coincides with $C^i_{g'}$.

  \item $Y_{\fr_0,\fr_1,\overline\kappa}^{012345}=\coprod_{h\in X_{\fr_0,\fr_1}^{012345}(\overline\kappa)}H_h$ where each $H_h$ is a connected smooth projective surface over $\Spec\overline\kappa$.
\end{enumerate}
\end{lem}

\begin{proof}
Parts (1) and (2) are consequences of Proposition \ref{bpr:blowup} (5) and (6), respectively, together with the similar discussion as for Lemma \ref{le:cycle_base}.

For (3), by definition, we have $Y_{\fr_0,\fr_1}^{012345}=Y_{\fr_0,\fr_1}^{024}\cap Y_{\fr_0,\fr_1}^{135}$. By Proposition \ref{bpr:blowup} (2,3), we have the canonical isomorphism $\pi_0(Y_{\fr_0,\fr_1,\overline\kappa}^{012345})\simeq X_{\fr_0,\fr_1}^{012345}(\overline\kappa)$, and that every geometric fiber of $Y_{\fr_0,\fr_1,\overline\kappa}^{012345}\to X_{\fr_0,\fr_1,\overline\kappa}^{012345}$ is a smooth projective surface.
\end{proof}

To proceed, we review some intersection theory from \cite{Ful}. Let $k$ be a field. For every smooth proper scheme $X$ over $\Spec k$, let $\rZ^r(X)$ be the group of cycles on $X$ of codimension $r$. We have a (Chow) cycle class map $\rZ^r(X)\to\CH^r(X)$ sending $D$ to $[D]$. If $X$ is of pure dimension $d$, then we have the degree map $\rZ^d(X)\to\CH^d(X)\xrightarrow{\deg}\dZ$. Let $f\colon Y\to X$ be a morphism of smooth proper schemes over $\Spec k$. Then we have a pullback map $f^*\colon\CH^r(X)\to\CH^r(Y)$ and a pushforward map $f_*\colon\CH^r(Y)\to\CH^{r-d}(X)$ where $d=\dim(Y)-\dim(X)$.

\begin{lem}\label{le:fulton}
Let notation be as above.
\begin{enumerate}
  \item If $f$ is flat, then $f^*[D]=[f^{-1}D]$ for every $D\in\rZ^\bullet(X)$.

  \item We have $f^*[D_1]\cup f^*[D_2]=f^*[D_1\cdot D_2]$ for every $D_1,D_2\in\rZ^\bullet(X)$, where $\cup$ (resp.\ $\cdot$) is the cup product (resp.\ intersection product) on $\CH^\bullet$ (resp.\ $\rZ^\bullet$).

  \item We have $f_*(E\cup f^*D)=f_*E\cup D$ for every $D\in\CH^\bullet(X)$ and $E\in\CH^\bullet(Y)$.
\end{enumerate}
\end{lem}

\begin{proof}
These can be found in \cite{Ful}*{\Sec 17}. In particular, we use \cite{Ful}*{Corollary 17.4}.
\end{proof}

Now we can introduce the notion of featuring cycles.

\begin{definition}\label{de:featuring}
Let $\cB(Y_{\fr_0,\fr_1})$ be the free abelian group generated by the following symbols, which we call \emph{(abstract) featuring cycles},
\begin{itemize}
  \item $\sfF^i_g$ for $i\in S^\pm$ and $g\in Z_{\fr_0,\fr_1}^\pm(\overline\kappa)$;

  \item $\sfF^i_{g'}$ for $i\in S^\pm$ and $g'\in Z_{\fr_0,\fr_1}^\mp(\overline\kappa)$;

  \item $\sfE^i_{g'}$ for $i\in S^\pm$ and $g\in Z_{\fr_0,\fr_1}^\mp(\overline\kappa)$;

  \item $\sfH^i_h$ for $h\in X_{\fr_0,\fr_1}^{012345}(\overline\kappa)$.
\end{itemize}
Note that by Proposition \ref{bpr:blowup} (7), we have $\CH^1(Y^{(0)}_{\overline\kappa})=\bigoplus_{|S|=3}\CH^1(Y_{\overline\kappa}^S)$. We define a map $\beta_0\colon\cB(Y_{\fr_0,\fr_1})\to\CH^1(Y^{(0)}_{\overline\kappa})$ that sends
\begin{itemize}
  \item $\sfF^i_g$ to $[F^i_g]\in\CH^1(Y_{\fr_0,\fr_1,\overline\kappa}^{i(i+1)(i+2)})$ for $i\in S^\pm$ and $g\in Z_{\fr_0,\fr_1}^\pm(\overline\kappa)$;

  \item $\sfF^i_{g'}$ to $[F^i_{g'}]\in\CH^1(Y_{\fr_0,\fr_1,\overline\kappa}^{i(i+1)(i+2)})$, where $F^i_{g'}=(\wp_{i(i+1)(i+2)}\circ\pi)^{-1}C^i_{g'}$, for $i\in S^\pm$ and $g'\in Z_{\fr_0,\fr_1}^\mp(\overline\kappa)$;

  \item $\sfE^i_{g'}$ to $[E^i_{g'}]\in\CH^1(Y_{\fr_0,\fr_1,\overline\kappa}^{i(i+1)(i+2)})$ for $i\in S^\pm$ and $g\in Z_{\fr_0,\fr_1}^\mp(\overline\kappa)$;

  \item $\sfH^i_h$ to $[H_h]\in\CH^1(Y_{\fr_0,\fr_1,\overline\kappa}^+)$ for $h\in X_{\fr_0,\fr_1}^{012345}(\overline\kappa)$.
\end{itemize}
Note that for a featuring cycle $\sfC$, we always have $\beta_0\sfC=[C]$ for some explicit cycle $C$ on $Y_{\fr_0,\fr_1,\overline\kappa}^S$ for some type $S$. Then we say that $\sfC$ has the \emph{realization} $C$ and the \emph{location} $Y_{\fr_0,\fr_1,\overline\kappa}^S$.
\end{definition}

The abelian group $\cB(Y_{\fr_0,\fr_1})$ is a $\dZ[\dT^{\fr\fl}][\rG_\kappa]$-module, where the actions of $\dT^{\fr\fl}$ and $\rG_\kappa$ are induced from those on the indexing sets $Z_{\fr_0,\fr_1}^\pm(\overline\kappa)$ and $X_{\fr_0,\fr_1}^{012345}(\overline\kappa)$ accordingly. Put $\cB(Y_{\fr_0,\fr_1},\Lambda)=\cB(Y_{\fr_0,\fr_1})\otimes\Lambda$. We have the following composite map
\begin{align}\label{eq:beta0}
\beta\colon\cB(Y_{\fr_0,\fr_1},\Lambda)\xrightarrow{\beta_0}\CH^1(Y^{(0)}_{\fr_0,\fr_1,\overline\kappa})\otimes\Lambda
\xrightarrow{\r{cl}}\rH^2(Y^{(0)}_{\fr_0,\fr_1,\overline\kappa},\Lambda(1))\to B_2(Y_{\fr_0,\fr_1},\Lambda)
\end{align}
where $\r{cl}$ denotes the class map as before, and $B_2(Y_{\fr_0,\fr_1},\Lambda)$ is defined in \Sec\ref{ss:potential_map} with the canonical quotient map $\rH^2(Y^{(0)}_{\fr_0,\fr_1,\overline\kappa},\Lambda(1))\to B_2(Y_{\fr_0,\fr_1},\Lambda)$. Taking $\Lambda$-dual, we obtain another map
\begin{align}\label{eq:beta}
\beta'\colon B^2(Y_{\fr_0,\fr_1},\Lambda)\simeq\Hom(B_2(Y_{\fr_0,\fr_1},\Lambda),\Lambda)
\xrightarrow{\beta^\vee}\Hom(\cB(Y_{\fr_0,\fr_1}),\Lambda).
\end{align}
In summary, we have the following commutative diagram
\begin{align}\label{eq:beta1}
\xymatrix{
\cB(Y_{\fr_0,\fr_1},\Lambda)  \ar[r]^-{\beta}\ar[d] & B_2(Y_{\fr_0,\fr_1},\Lambda) \ar[r]^-{\eqref{eq:potential_prepremap}}  & B^2(Y_{\fr_0,\fr_1},\Lambda) \ar[r]^-{\beta'}\ar[d] & \Hom(\cB(Y_{\fr_0,\fr_1}),\Lambda)  \\
\CH^1(Y^{(0)}_{\fr_0,\fr_1,\overline\kappa})\otimes\Lambda \ar[r]^-{\r{cl}} &  \rH^2(Y^{(0)}_{\fr_0,\fr_1,\overline\kappa},\Lambda(1)) \ar[r]^-{\delta_{1*}\circ\delta^*_0}\ar[u]
&  \rH^4(Y^{(0)}_{\fr_0,\fr_1,\overline\kappa},\Lambda(2)) \ar[r]^-{\r{cl}^\vee} & \Hom(\CH^1(Y^{(0)}_{\fr_0,\fr_1,\overline\kappa}),\Lambda). \ar[u]
}
\end{align}
Denote the composition of the top line by
\begin{align}\label{eq:potential_featuring}
\tilde\Delta\colon\cB(Y_{\fr_0,\fr_1},\Lambda)\to\Hom(\cB(Y_{\fr_0,\fr_1}),\Lambda).
\end{align}
The motivation behind Definition \ref{de:featuring} is explained in Proposition \ref{pr:featuring}.

\subsection{Computation of potential map}
\label{ss:computation}

We start to compute the map $\tilde\Delta$ \eqref{eq:potential_featuring}. To make formulae easier to read, throughout this section, we will suppress all subscripts $\fr_0,\fr_1$ indicating levels, as they will not be changed in the discussion here.

Let us first explain what it means for computing $\tilde\Delta$. Note that by definition, $\tilde\Delta$ is determined by the value of the function $\tilde\Delta\sfC$ on $\sfD$ for every pair of featuring cycles $\sfC$ and $\sfD$. Suppose that $\sfC$ (resp.\ $\sfD$) has the realization $C$ (resp.\ $D$) and the location $Y_{\overline\kappa}^S$ (resp.\ $Y_{\overline\kappa}^{S'}$). To compute $\tilde\Delta\sfC(\sfD)$, there are three cases:
\begin{enumerate}
  \item We have $Y_{\overline\kappa}^S\cap Y_{\overline\kappa}^{S'}=\emptyset$. Then $\tilde\Delta\sfC(\sfD)=0$.

  \item We have $Y_{\overline\kappa}^S=Y_{\overline\kappa}^{S'}$. Let $Y_{\overline\kappa}^{S''},Y_{\overline\kappa}^{S'''},\dots$ be all other strata of dimension $3$ that have nonempty intersection with $Y_{\overline\kappa}^S$. Then
      \[\tilde\Delta\sfC(\sfD)=\deg(C\cdot(Y_{\overline\kappa}^S\cap Y_{\overline\kappa}^{S''}+Y_{\overline\kappa}^S\cap Y_{\overline\kappa}^{S'''}+\cdots)\cdot D)\]
      where the intersection number is computed on $Y_{\overline\kappa}^S$.

  \item We have $Y_{\overline\kappa}^S\neq Y_{\overline\kappa}^{S'}$ but $Y_{\overline\kappa}^S\cap Y_{\overline\kappa}^{S'}\neq\emptyset$. Let $j\colon Y_{\overline\kappa}^S\cap Y_{\overline\kappa}^{S'}\to Y_{\overline\kappa}^S$ and $j'\colon Y_{\overline\kappa}^S\cap Y_{\overline\kappa}^{S'}\to Y_{\overline\kappa}^{S'}$ be the canonical embeddings, respectively. Then
      \[\tilde\Delta\sfC(\sfD)=\deg(-j'_*j^*(C\cdot(Y_{\overline\kappa}^S\cap Y_{\overline\kappa}^{S'}))\cdot D)
      =-\deg(j^*C\cdot j'^*D).\]
\end{enumerate}
The sign difference in (2) and (3) is due to the fact that $\delta_0^*$ and $\delta_{1*}$ are defined as alternating sums (see \Sec\ref{ss:semistable_schemes}, in particular Remark \ref{re:sign}). In particular, we have
\begin{align}\label{eq:symmetric}
\tilde\Delta\sfC(\sfD)=\tilde\Delta\sfD(\sfC).
\end{align}

In what follows, we will suppress the notation ``$\deg$'' if we are expecting a number (rather than a $0$-cycle). Moreover, in case (3), we will simply write the intersection number as $\tilde\Delta\sfC(\sfD)=-C\cdot(Y_{\overline\kappa}^S\cap Y_{\overline\kappa}^{S'})\cdot D$, with the interpretation implicitly.

\begin{remark}\label{re:symmetric}
By Definition \ref{de:featuring}, the location of $\sfH_h$ is always $Y_{\overline\kappa}^+$. However, since $\tilde\Delta$ factorizes through $B^2(Y,\Lambda)$, we may compute $\tilde\Delta\sfC(\sfH_h)$ by regarding $H_h$ as a $2$-cycle on $Y_{\overline\kappa}^-$.
\end{remark}

We need the following notation and lemma before the computation of $\tilde\Delta\sfC(\sfD)$.

\begin{notation}\label{no:fiber}
For $i\in S^\pm$ and $g'\in Z^\mp(\overline\kappa)$, we have the exceptional divisor $E^i_{g'}$. Let $D^i_{g'}$ be its image under $\pi$.
Then $E^i_{g'}$ is a ruled surface with the base curve $D^i_{g'}$. Moreover, we have
\[X_{\overline\kappa}^{i(i+1)(i+2)(i+3)(i+5)}=\coprod_{g'\in Z^\mp(\overline\kappa)}D^i_{g'}.\]
Finally, denote by $f^i_{g'}$ the the general fiber of $E^i_{g'}\to D^i_{g'}$.
\end{notation}

\begin{lem}\label{le:pullback}
For every $i\in S^\pm$, consider the morphisms
\[Y_{\overline\kappa}^{i(i+1)(i+2)}\xrightarrow{\pi}X_{\overline\kappa}^{i(i+1)(i+2)}
\xrightarrow{\wp_{i(i+1)(i+2)}}Z_{\overline\kappa}^{i(i+1)(i+2)}.\]
Then we have
\begin{enumerate}
  \item $[F^i_g]=\pi^*\wp_{i(i+1)(i+2)}^*[C^i_g]$ for $g\in Z^\pm(\overline\kappa)$;
  \item $[F^i_{g'}]=\pi^*\wp_{i(i+1)(i+2)}^*[C^i_{g'}]$ for $g\in Z^\mp(\overline\kappa)$.
\end{enumerate}
\end{lem}

\begin{proof}
Note that the morphism $\pi\colon Y_{\overline\kappa}^{i(i+1)(i+2)}\to X_{\overline\kappa}^{i(i+1)(i+2)}$ is the blow-up along a smooth proper curve $X_{\overline\kappa}^{i(i+1)(i+2)(i+3)(i+5)}$ by Proposition \ref{bpr:blowup}. We put $\wp=\wp_{i(i+1)(i+2)}$ for simplicity.

For (1), note that the smooth surface $\wp^{-1}C^i_g$ intersects with $X_{\overline\kappa}^{i(i+1)(i+2)(i+3)(i+5)}$ properly by Lemma \ref{le:cycle_base}. By \cite{Ful}*{Proposition 6.7}, we know that $\pi^*[\wp^{-1}C^i_g]=[\pi^{-1}\wp^{-1}C^i_g]+[D]$ where $D$ is supported on $\pi^{-1}(\wp^{-1}C^i_g\cap X_{\overline\kappa}^{i(i+1)(i+2)(i+3)(i+5)})$. Thus $[D]=0$, and $[\pi^{-1}\wp^{-1}C^i_g]=\pi^*[\wp^{-1}C^i_g]=\pi^*\wp^*[C^i_g]$ by Lemma \ref{le:fulton} (1).

For (2), by the same argument, we know that $\pi^*[\wp^{-1}C^i_{g'}]=[\pi^{-1}\wp^{-1}C^i_{g'}]+a[E^i_{g'}]$ for some integer $a$. Let $G^i_{g'}\in\rZ^1(Y_{\overline\kappa}^{i(i+1)(i+2)})$ be the strict transform of $\wp^{-1}C^i_{g'}$ under $\pi$. Then we have $\pi^{-1}\wp^{-1}C^i_{g'}=G^i_{g'}+E^i_{g'}$. On one hand, by Lemma \ref{le:fulton} (3), we have (Notation \ref{no:fiber})
\[\pi^*[\wp^{-1}C^i_{g'}]\cup[f^i_{g'}]=[\wp^{-1}C^i_{g'}]\cup\pi_*[f^i_{g'}]=0.\]
On the other hand, we have
\[\deg(([\pi^{-1}\wp^{-1}C^i_{g'}]+a[E^i_{g'}])\cup[f^i_{g'}])=(G^i_{g'}+(a+1)E^i_{g'})\cdot f^i_{g'}=1+(-1)(a+1)\]
by \cite{AG5}*{Lemma 2.2.14 (ii b)}. Therefore, $a=0$ and $[F^i_{g'}]=\pi^*\wp_{i(i+1)(i+2)}^*[C^i_{g'}]$.
\end{proof}

In the proof of Lemmas \ref{le:intersection1}, \ref{le:intersection2}, \ref{le:intersection3} and \ref{le:intersection4}, We will abuse notation by simply writing $D$ instead of $[D]$ for a cycle $D$ on a $\overline\kappa$-scheme.

\begin{lem}\label{le:intersection1}
For $i\in S^\pm$ and $g\in Z^\pm(\overline\kappa)$, we have
\[
\tilde\Delta\sfF^i_g
\begin{cases}
(\sfF^i_g)= -2\ell(\ell+1), \\
(\sfF^i_{g'})= \ell+1, & \forall g'\in\rT(g), \\
(\sfE^i_{g'})= 1, & \forall g'\in\rT(g),\\
(\sfF^{i-1}_{g'})= -1, & \forall g'\in\rT(g),\\
(\sfF^{i-1}_g)=2\ell^2, \\
(\sfE^{i-1}_g)=2\ell, \\
(\sfF^{i+1}_{g'})=-1, & \forall g'\in\rT(g), \\
(\sfF^{i+1}_g)=2\ell^2, \\
(\sfE^{i+1}_g)=2\ell^2,
\end{cases}
\]
where $\rT(g)$ is introduced in Notation \ref{no:tilde}. On other featuring cycles, the function $\tilde\Delta\sfF^i_g$ takes value zero.
\end{lem}

\begin{proof}
We may assume $i=0$ without lost of generality by Remark \ref{re:symmetric}. Note that $\sfF^0_g$ has the realization $F^0_g$ and location $Y_{\overline\kappa}^{012}$. All other $3$-dimensional strata that have nonempty intersection with $Y_{\overline\kappa}^{012}$ are: $Y_{\overline\kappa}^{015}$, $Y_{\overline\kappa}^{123}$, $Y_{\overline\kappa}^{024}$ and $Y_{\overline\kappa}^{135}$. We have
\begin{align*}
Y_{\overline\kappa}^{012}\cap Y_{\overline\kappa}^{015}=Y_{\overline\kappa}^{0125},&\qquad
Y_{\overline\kappa}^{012}\cap Y_{\overline\kappa}^{123}=Y_{\overline\kappa}^{0123},\\
Y_{\overline\kappa}^{012}\cap Y_{\overline\kappa}^{024}=Y_{\overline\kappa}^{0124}=\coprod_{\tilde{g}\in Z^+(\overline\kappa)}F^0_{\tilde{g}},&\qquad
Y_{\overline\kappa}^{012}\cap Y_{\overline\kappa}^{135}=Y_{\overline\kappa}^{01235}=\coprod_{g'\in Z^-(\overline\kappa)}E^0_{g'},
\end{align*}
by Lemma \ref{le:cycle_strata}. Moreover, $F^0_g\cdot F^0_{\tilde{g}}=0$ if $\tilde{g}\neq g$, and $F^0_g\cdot E^0_{g'}=0$ if $g'\not\in\rT(g)$. Now we show the 9 equalities and the vanishing part one by one.

\begin{enumerate}
  \item We have $\tilde\Delta\sfF^0_g(\sfF^0_g)=F^0_g\cdot(\sum_{g'\in\rT(g)}E^0_{g'}+F^0_g+Y_{\overline\kappa}^{0125}+
      Y_{\overline\kappa}^{0123})\cdot F^0_g$. By Lemmas \ref{le:fulton} and \ref{le:pullback}, we have
      \begin{multline}\label{eq:111}
      F^0_g\cdot E^0_{g'}\cdot F^0_g=(F^0_g\cdot F^0_g)\cdot E^0_{g'}=((\wp_{123}\circ\pi)^*C^0_g\cdot(\wp_{123}\circ\pi)^*C^0_g)\cdot E^0_{g'}\\
      =(\wp_{123}\circ\pi)^*(C^0_g\cdot C^0_g)\cdot E^0_{g'}=(C^0_g\cdot C^0_g)\cdot(\wp_{123}\circ\pi)_*E^0_{g'}=
      (C^0_g\cdot C^0_g)\cdot0=0.
      \end{multline}
      Similarly, we have
      \begin{align}\label{eq:112}
      F^0_g\cdot F^0_g\cdot F^0_g=(C^0_g\cdot C^0_g)\cdot(\wp_{012}\circ\pi)_*F^0_g=(C^0_g\cdot C^0_g)\cdot0=0.
      \end{align}
      Similarly, we have for $k=3,5$ that
      \begin{align*}
      F^0_g\cdot Y_{\overline\kappa}^{012k}\cdot F^0_g=(F^0_g\cdot F^0_g)\cdot Y_{\overline\kappa}^{012k}=(C^0_g\cdot C^0_g)\cdot\wp_{012*}\pi_*Y_{\overline\kappa}^{012k}
      =(C^0_g\cdot C^0_g)\cdot\wp_{012*}X_{\overline\kappa}^{012k}.
      \end{align*}
      By Lemma \ref{le:cycle_base} (3), we have $C^0_g\cdot C^0_g=-2\ell$. By Lemma \ref{le:strata} (3) (resp.\ Lemma \ref{le:strata} (4)), we know that $\wp_{012*}X_{\overline\kappa}^{0125}=Z_{\overline\kappa}^{012}$ (resp.\ $\wp_{012*}X_{\overline\kappa}^{0123}=\ell Z_{\overline\kappa}^{012}$). Thus we have
      \begin{align}\label{eq:113}
      F^0_g\cdot (Y_{\overline\kappa}^{0123}+Y_{\overline\kappa}^{0125})\cdot F^0_g=-2\ell(\ell+1).
      \end{align}
      The first equality follows from \eqref{eq:111}, \eqref{eq:112} and \eqref{eq:113}.

  \item We have $\tilde\Delta\sfF^0_g(\sfF^0_{g'})=F^0_g\cdot
      (\sum_{\tilde{g}'\in\rT(g)}E^0_{\tilde{g}'}+F^0_g+Y_{\overline\kappa}^{0125}+Y_{\overline\kappa}^{0123})\cdot F^0_{g'}$. Similar to \eqref{eq:111}, we have
      \begin{align}\label{eq:121}
      F^0_g\cdot E^0_{\tilde{g}'}\cdot F^0_{g'}=(F^0_g\cdot F^0_{g'})\cdot E^0_{\tilde{g}'}=(C^0_g\cdot C^0_{g'})\cdot(\wp_{123}\circ\pi)_*E^0_{\tilde{g}'}=0.
      \end{align}
      Similar to \eqref{eq:112}, we have
      \begin{align}\label{eq:122}
      F^0_g\cdot F^0_g\cdot F^0_{g'}=0.
      \end{align}
      Similar to \eqref{eq:113}, we have for $k=3,5$ that
      \begin{align*}
      F^0_g\cdot Y_{\overline\kappa}^{012k}\cdot F^0_{g'}=(F^0_g\cdot F^0_{g'})\cdot Y_{\overline\kappa}^{012k}=(C^0_g\cdot C^0_{g'})\cdot\wp_{012*}\pi_*Y_{\overline\kappa}^{012k}.
      \end{align*}
      Note that $C^0_g\cdot C^0_{g'}$ equals $1$ (resp.\ $0$) if $g'\in\rT(g)$ (resp.\ $g'\not\in\rT(g)$). Thus we have for $g'\in\rT(g)$,
      \begin{align}\label{eq:123}
      F^0_g\cdot (Y_{\overline\kappa}^{0123}+Y_{\overline\kappa}^{0125})\cdot F^0_{g'}=\ell+1.
      \end{align}
      The second equality follows from \eqref{eq:121}, \eqref{eq:122} and \eqref{eq:123}.

  \item We have $\tilde\Delta\sfF^0_g(\sfE^0_{g'})=F^0_g\cdot
      (\sum_{\tilde{g}'\in\rT(g)}E^0_{\tilde{g}'}+F^0_g+Y_{\overline\kappa}^{0125}+Y_{\overline\kappa}^{0123})\cdot E^0_{g'}$. Note that $E^0_{\tilde{g}'}\cdot E^0_{g'}=0$ unless $\tilde{g}'=g'$. Thus $\tilde\Delta\sfF^0_g(\sfE^0_{g'})=F^0_g\cdot(E^0_{g'}+F^0_g+Y_{\overline\kappa}^{0125}+Y_{\overline\kappa}^{0123})\cdot E^0_{g'}$. As $F^0_g$ intersects $E^0_{g'}$ transversally at a fiber $f^0_{g'}$ of the ruled surface $E^0_{g'}$ (Notation \ref{no:fiber}), we have
      \begin{align}\label{eq:131}
      F^0_g\cdot E^0_{g'}\cdot E^0_{g'}=f^0_{g'}\cdot E^0_{g'}=-1
      \end{align}
      by \cite{AG5}*{Lemma 2.2.14 (ii b)}. For $k=3,5$, the restricted morphism $\pi\colon Y_{\overline\kappa}^{012k}\to X_{\overline\kappa}^{012k}$ is an isomorphism by Proposition \ref{bpr:blowup} (4). Thus the morphism $\wp_{012}\circ\pi$ induces an isomorphism from $Y_{\overline\kappa}^{012k}\cap E^0_{g'}$ to $C^0_{g'}$. As the intersection $Y_{\overline\kappa}^{012k}\cap E^0_{g'}$ is proper, we have by Lemma \ref{le:fulton}
      \begin{align}\label{eq:132}
      F^0_g\cdot Y_{\overline\kappa}^{012k}\cdot E^0_{g'}=C^0_g\cdot C^0_{g'}=1.
      \end{align}
      The third equality follows from \eqref{eq:111}, \eqref{eq:131} and \eqref{eq:132}.

  \item We have $\tilde\Delta\sfF^0_g(\sfF^5_{g'})=-F^0_g\cdot Y_{\overline\kappa}^{0125}\cdot F^5_{g'}$. Note that $F^0_g\subset Y_{\overline\kappa}^{0124}$ intersects transversally with $Y_{\overline\kappa}^{0125}$ (in $Y_{\overline\kappa}^{012}$), and the restricted morphism $\pi\colon Y_{\overline\kappa}^{0124}\cap Y_{\overline\kappa}^{0125}\to X_{\overline\kappa}^{01245}$ is an isomorphism by Proposition \ref{bpr:blowup}. The following composite morphism
      \[F^0_g\cap Y_{\overline\kappa}^{0125}\to X_{\overline\kappa}^{01245}
      \to X_{\overline\kappa}^{0145}\to Z_{\overline\kappa}^{015},\]
      which is same as $\wp_{015}\circ\pi$, induces an isomorphism onto the image $C^5_g$ by Lemma \ref{le:cycle_base} (2). Therefore,
      \[F^0_g\cdot Y_{\overline\kappa}^{0125}\cdot F^5_{g'}=(F^0_g\cap Y_{\overline\kappa}^{0125})\cdot\pi^*\wp_{015}^*C^5_{g'}=
      \wp_{015*}\pi_*(F^0_g\cap Y_{\overline\kappa}^{0125})\cdot C^5_{g'}=C^5_g\cdot C^5_{g'}\]
      which equals $1$ (resp.\ $0$) if $g'\in\rT(g)$ (resp\ $g'\not\in\rT(g)$). The fourth equality follows.

  \item Similar to (4), we have $F^0_g\cdot Y_{\overline\kappa}^{0125}\cdot F^5_g=C^5_g\cdot C^5_g=-2\ell^2$ by Lemma \ref{le:cycle_base} (3). The fifth equality follows. However, on the other hand,
      \[F^0_g\cdot Y_{\overline\kappa}^{0125}\cdot F^5_g=\wp_{012*}\pi_*(F^5_g\cdot Y_{\overline\kappa}^{0125})\cdot C^0_g
      =aC^0_g\cdot C^0_g.\]
      As $C^0_g\cdot C^0_g=-2\ell$, we have $a=\ell$. In other words,
      \begin{align}\label{eq:151}
      \wp_{012*}\pi_*(F^5_g\cdot Y_{\overline\kappa}^{0125})=\ell C^0_g.
      \end{align}
      This will be used later.

  \item We have $\tilde\Delta\sfF^0_g(\sfE^5_g)=-F^0_g\cdot Y_{\overline\kappa}^{0125}\cdot E^5_g$. Note that $F^0_g\cap Y_{\overline\kappa}^{0125}\subset Y_{\overline\kappa}^{0124}\cap Y_{\overline\kappa}^{0125}$. Therefore, $F^0_g\cap Y_{\overline\kappa}^{0125}$, as a $1$-cycle on $Y_{\overline\kappa}^{0125}$, is contained in $Y_{\overline\kappa}^{01245}$ and $Y_{\overline\kappa}^{0125}$. By Lemma \ref{le:cycle_strata}, $Y_{\overline\kappa}^{01245}=\coprod_{\tilde{g}\in Z^+(\overline\kappa)}E^5_g$, thus $F^0_g\cap Y_{\overline\kappa}^{0125}$ is simply $E^5_g\cap Y_{\overline\kappa}^{0125}$ as a $1$-cycle of $Y_{\overline\kappa}^{0125}$. Therefore, $F^0_g\cdot Y_{\overline\kappa}^{0125}\cdot E^5_g=Y_{\overline\kappa}^{0125}\cdot (E^5_g\cdot E^5_g)$. Applying the projection formula (Lemma \ref{le:fulton} (3)) to the closed embedding $Y_{\overline\kappa}^{0125}\to Y_{\overline\kappa}^{015}$, we obtain
      \[Y_{\overline\kappa}^{0125}\cdot (E^5_g\cdot E^5_g)=(E^5_g\cap Y_{\overline\kappa}^{0125})\cdot(E^5_g\cap Y_{\overline\kappa}^{0125}).\]
      Under the isomorphism $Y_{\overline\kappa}^{0125}\xrightarrow{\pi}X_{\overline\kappa}^{0125}\xrightarrow{\wp_{0125}}Z_{\overline\kappa}^{012}$ (Proposition \ref{bpr:blowup} and Lemma \ref{le:strata}), the image of $E^5_g\cap Y_{\overline\kappa}^{0125}$ is simply $C^0_g$ ($g\in Z^+(\overline\kappa)$). Thus the sixth equality follows from Lemma \ref{le:cycle_base} (3).

  \item The seventh equality follows from \eqref{eq:symmetric} and the fourth with $i=1$.

  \item Similar to (5), we have $\tilde\Delta\sfF^0_g(\sfF^1_g)=-F^0_g\cdot Y_{\overline\kappa}^{0123}\cdot F^5_g=-C^1_g\cdot C^1_g=2\ell^2$
      by Lemma \ref{le:cycle_base} (3). The eighth equality follows.

  \item We have $\tilde\Delta\sfF^0_g(\sfE^1_g)=-F^0_g\cdot Y_{\overline\kappa}^{0123}\cdot E^1_g$. Similar to (6), we have
      \[F^0_g\cdot Y_{\overline\kappa}^{0123}\cdot E^1_g=Y_{\overline\kappa}^{0123}\cdot (E^5_g\cdot E^5_g)=(E^1_g\cap Y_{\overline\kappa}^{0123})\cdot(E^1_g\cap Y_{\overline\kappa}^{0123}).\]
      Under the isomorphism $Y_{\overline\kappa}^{0123}\xrightarrow{\pi}X_{\overline\kappa}^{0123}\xrightarrow{\wp_{0123}}Z_{\overline\kappa}^{123}$ (Proposition \ref{bpr:blowup} and Lemma \ref{le:strata}), the image of $E^1_g\cap Y_{\overline\kappa}^{0123}$ is simply $C^1_g$ ($g\in Z^+(\overline\kappa)$). Thus the ninth equality follows from Lemma \ref{le:cycle_base} (3).

  \item For the remaining featuring cycles, we trivially have $\tilde\Delta\sfF^0_g(\sfC)=0$ unless $\sfC=\sfH_h$, since for else, there is obvious empty intersection in the computation of intersection numbers. Now we show that $\tilde\Delta\sfF^0_g(\sfH_h)=0$ for every $h\in X^{012345}(\overline\kappa)$. By Remark \ref{re:symmetric}, we may regard $H_h$ as a cycle on $Y_{\overline\kappa}^{135}$. Thus
      \[\tilde\Delta\sfF^0_g(\sfH_h)=-F^0_g\cdot Y_{\overline\kappa}^{01235}\cdot H_h=
      -F^0_g\cdot(\sum_{g'\in Z^-(\overline\kappa)}E^0_{g'})\cdot H_h.\]
      Note that $F^0_g\cdot E^0_{g'}=f^0_{g'}$ for $g'\in\rT(g)$. Now we can move the fiber $f^0_{g'}$ such that their projection in $X_{\overline\kappa}^{01235}$ does not belong to $X_{\overline\kappa}^{024}$. In other words, we can move $f^0_{g'}$ such that it has empty intersection with $Y_{\overline\kappa}^{024}$ hence with $Y_{\overline\kappa}^{012345}=\coprod_{h\in X^{012345}(\overline\kappa)}H_h$. Therefore, $F^0_g\cdot Y_{\overline\kappa}^{01235}\cdot H_h=0$ hence $\tilde\Delta\sfF^0_g(\sfH_h)=0$.
\end{enumerate}
\end{proof}

\begin{lem}\label{le:intersection2}
For $i\in S^\pm$ and $g'\in Z^\mp(\overline\kappa)$, we have
\[
\tilde\Delta\sfF^i_{g'}
\begin{cases}
(\sfF^i_{g'})= -2\ell^2(\ell+1), \\
(\sfF^i_g)= \ell+1, & \forall g\in\rT(g'), \\
(\sfE^i_{g'})=-2\ell^2, \\
(\sfF^{i-1}_g)= -\ell, & \forall g\in\rT(g'),\\
(\sfF^{i-1}_{g'})=2\ell^2, \\
(\sfE^{i-1}_g)=-1, &\forall g\in\rT(g'),\\
(\sfF^{i+1}_g)=-\ell, & \forall g\in\rT(g'), \\
(\sfF^{i+1}_{g'})=2\ell^2, \\
(\sfE^{i+1}_g)=-\ell, & \forall g\in\rT(g').
\end{cases}
\]
On other featuring cycles, the function $\tilde\Delta\sfF^i_{g'}$ takes value zero.
\end{lem}

\begin{proof}
Again we assume that $i=0$.

\begin{enumerate}
  \item We have $\tilde\Delta\sfF^0_{g'}(\sfF^0_{g'})=F^0_{g'}\cdot
      (E^0_{g'}+\sum_{g\in\rT(g')}F^0_g+Y_{\overline\kappa}^{0125}+Y_{\overline\kappa}^{0123})\cdot F^0_{g'}$. By the same argument in Lemma \ref{le:intersection1} (1), we have $F^0_{g'}\cdot E^0_{g'}\cdot F^0_{g'}=F^0_{g'}\cdot F^0_g\cdot F^0_{g'}=0$. For $k=3,5$, we have
      $F^0_{g'}\cdot Y_{\overline\kappa}^{012k}\cdot F^0_{g'}=(C^0_{g'}\cdot C^0_{g'})\cdot\wp_{012*}\pi_*Y_{\overline\kappa}^{012k}
      =(C^0_g\cdot C^0_g)\cdot\wp_{012*}X_{\overline\kappa}^{012k}$. By Lemma \ref{le:cycle_base} (3), we have $C^0_g\cdot C^0_g=-2\ell^2$. Thus we have
      \begin{align*}
      F^0_{g'}\cdot (Y_{\overline\kappa}^{0123}+Y_{\overline\kappa}^{0125})\cdot F^0_{g'}=-2\ell^2(\ell+1).
      \end{align*}
      The first equality is proved.

  \item The second equality follows from \eqref{eq:symmetric} and the second equality of Lemma \ref{le:intersection1}.

  \item We have $\tilde\Delta\sfF^0_{g'}(\sfE^0_{g'})=F^0_{g'}\cdot
      (E^0_{g'}+\sum_{g\in\rT(g')}F^0_g+Y_{\overline\kappa}^{0125}+Y_{\overline\kappa}^{0123})\cdot E^0_{g'}$.  By the same argument in Lemma \ref{le:intersection1} (3), we have
      \begin{align}\label{eq:231}
      F^0_{g'}\cdot Y_{\overline\kappa}^{012k}\cdot E^0_{g'}=C^0_{g'}\cdot C^0_{g'}=-2\ell^2
      \end{align}
      for $k=3,5$. By \cite{AG5}*{Lemma 2.2.14 (ii b)}, we have $E^0_{g'}\cdot E^0_{g'}=-\pi^* D^0_{g'}+a f^0_{g'}$ (Notation \ref{no:fiber}) for some integer $a$. We have $\pi_*\pi^* D^0_{g'}=D^0_{g'}$ by \cite{AG5}*{Lemma 2.2.14 (i)}; $\pi_*f^0_{g'}=0$; and that $\wp_{012}\colon D^0_{g'}\to C^0_{g'}$ is an isomorphism. Thus
      \begin{align}\label{eq:232}
      F^0_{g'}\cdot E^0_{g'}\cdot E^0_{g'}=C^0_{g'}\cdot\wp_{012*}\pi_*(-\pi^* D^0_{g'}+a f^0_{g'})=C^0_{g'}\cdot \wp_{012*}(-D^0_{g'})
      =-C^0_{g'}\cdot C^0_{g'}=2\ell^2.
      \end{align}
      The third equality follows from \eqref{eq:231}, \eqref{eq:232} and \eqref{eq:121}.

  \item We have $\tilde\Delta\sfF^0_{g'}(\sfF^5_g)=-F^0_{g'}\cdot Y_{\overline\kappa}^{0125}\cdot F^5_g$. We have
      \[F^0_{g'}\cdot Y_{\overline\kappa}^{0125}\cdot F^5_g=\pi^*\wp_{012}^*C^0_{g'}\cdot(Y_{\overline\kappa}^{0125}\cdot F^5_g)=C^0_{g'}\cdot\wp_{012*}\pi_*(Y_{\overline\kappa}^{0125}\cdot F^5_g)\]
      which by \eqref{eq:151} equals to $C^0_{g'}\cdot(\ell C^0_g)=\ell$ as $g\in\rT(g')$. The fourth equality follows.

  \item The fifth equality follows from \eqref{eq:symmetric} and the eighth equality of Lemma \ref{le:intersection1}.

  \item We have $\tilde\Delta\sfF^0_{g'}(\sfE^5_g)=-F^0_{g'}\cdot Y_{\overline\kappa}^{0125}\cdot E^5_g$. Again $\pi$ induces an isomorphism from $Y_{\overline\kappa}^{0125}\cdot E^5_g=Y_{\overline\kappa}^{0125}\cap E^5_g$ to $D^5_g\subset X_{\overline\kappa}^{01245}$. By Lemma \ref{le:strata}, we have $\wp_{012*}\pi_*(Y_{\overline\kappa}^{0125}\cdot E^5_g)=C^0_g$.
      Thus, $F^0_{g'}\cdot Y_{\overline\kappa}^{0125}\cdot E^5_g=\pi^*\wp_{012}^*C^0_{g'}\cdot(Y_{\overline\kappa}^{0125}\cdot E^5_g)
      =C^0_{g'}\cdot C^0_g=1$ as $g\in\rT(g')$. The sixth equality follows.

  \item The seventh equality follows from \eqref{eq:symmetric} and the fourth equality with $i=1$.

  \item The eighth equality follows from \eqref{eq:symmetric} and the fifth equality of Lemma \ref{le:intersection1}.

  \item We have $\tilde\Delta\sfF^0_{g'}(\sfE^1_g)=-F^0_{g'}\cdot Y_{\overline\kappa}^{0123}\cdot E^1_g$. To compute $F^0_{g'}\cdot Y_{\overline\kappa}^{0123}$ as an element in $\CH^2(Y_{\overline\kappa}^{123})$, we may write $F^0_{g'}\cdot Y_{\overline\kappa}^{0123}=a(Y_{\overline\kappa}^{0123}\cap F^1_{g'})=a(Y_{\overline\kappa}^{0123}\cdot F^1_{g'})$ by looking at the support. By \eqref{eq:151}, we have $a=\ell$. Thus $F^0_{g'}\cdot Y_{\overline\kappa}^{0123}\cdot E^1_g=\ell Y_{\overline\kappa}^{0123}\cdot F^1_{g'}\cdot E^1_g=\ell Y_{\overline\kappa}^{0123}\cdot f^1_g=\ell$. The ninth equality follows.

  \item Similar to Lemma \ref{le:intersection1}, the only nontrivial equality is $\tilde\Delta\sfF^0_{g'}(\sfH_h)=0$ for all $h\in X^{012345}(\overline\kappa)$. We have
      $tilde\Delta\sfF^0_{g'}(\sfH_h)=-F^0_{g'}\cdot Y_{\overline\kappa}^{0124}\cdot H_h=
      -F^0_{g'}\cdot(\sum_{g\in Z^+(\overline\kappa)}F^0_g)\cdot H_h$.
      Note that $F^0_{g'}\cdot(\sum_{g\in Z^+(\overline\kappa)}F^0_g)=\sum_{\tilde{h}\in(\delta^-)^{-1}(g')}\pi^*\wp_{012}^*\wp_{012}(\tilde{h})$, where $\wp_{012}(\tilde{h})\in Z_{\overline\kappa}^{012}(\overline\kappa)$. Since on $Y_{\overline\kappa}^{024}$, we have $\pi^*\wp_{012}^*\wp_{012}(\tilde{h})\cdot H_h=\wp_{012}(\tilde{h})\cdot\wp_{012*}\pi_*H_h=0$, we conclude $\tilde\Delta\sfF^0_{g'}(\sfH_h)=0$.
\end{enumerate}
\end{proof}

\begin{lem}\label{le:intersection3}
For $i\in S^\pm$ and $g'\in Z^\mp(\overline\kappa)$, we have
\[
\tilde\Delta\sfE^i_{g'}
\begin{cases}
(\sfF^i_{g'})= -2\ell^2, \\
(\sfF^i_g)=1, & \forall g\in\rT(g'), \\
(\sfE^i_{g'})=-(\ell^3+1), \\
(\sfF^{i-1}_g)= -\ell, & \forall g\in\rT(g'),\\
(\sfF^{i-1}_{g'})=2\ell^2, \\
(\sfE^{i-1}_g)=-1, &\forall g\in\rT(g'),\\
(\sfF^{i+1}_g)=-1, & \forall g\in\rT(g'), \\
(\sfF^{i+1}_{g'})=2\ell, \\
(\sfE^{i+1}_g)=-1, & \forall g\in\rT(g'), \\
(\sfH_h)=1, & \forall h\in(\delta^\mp)^{-1}(g').
\end{cases}
\]
On other featuring cycles, the function $\tilde\Delta\sfE^i_{g'}$ takes value zero.
\end{lem}

\begin{proof}
Again we assume that $i=0$ hence $g'\in Z^-(\overline\kappa)$.

\begin{enumerate}
  \item The first equality follows from \eqref{eq:symmetric} and the third equality of Lemma \ref{le:intersection2}.

  \item The second equality follows from \eqref{eq:symmetric} and the third equality of Lemma \ref{le:intersection1}.

  \item We have $\tilde\Delta\sfE^0_{g'}(\sfE^0_{g'})=E^0_{g'}\cdot(\sum_{g\in\rT(g')}F^0_g+E^0_{g'}+Y_{\overline\kappa}^{0123}
  +Y_{\overline\kappa}^{0125})\cdot E^0_{g'}$. By \eqref{eq:131}, we know that $E^0_{g'}\cdot F^0_g\cdot E^0_{g'}=-1$ for $g\in\rT(g')$. Note that $|\rT(g')|=\ell^3+1$. Thus we only need to show that
  \begin{align}\label{eq:331}
  E^0_{g'}\cdot(E^0_{g'}+Y_{\overline\kappa}^{0123}+Y_{\overline\kappa}^{0125})\cdot E^0_{g'}=0.
  \end{align}

  By a similar argument for Lemma \ref{le:pullback} (2), we know that $\pi^*Y_{\overline\kappa}^{012k}=Y_{\overline\kappa}^{012k}+\sum_{g'\in Z^\mp(\overline\kappa)}E^0_{g'}$ for $k=3,5$. Therefore,
  \begin{align*}
  &E^0_{g'}\cdot(E^0_{g'}+Y_{\overline\kappa}^{0123}+Y_{\overline\kappa}^{0125})\cdot E^0_{g'}
  =(E^0_{g'}+Y_{\overline\kappa}^{0123}-Y_{\overline\kappa}^{0123})\cdot
  (E^0_{g'}+Y_{\overline\kappa}^{0123}+Y_{\overline\kappa}^{0125})\cdot E^0_{g'} \\
  &=(\pi^*X_{\overline\kappa}^{0123}-Y_{\overline\kappa}^{0123})\cdot
  (\pi^*X_{\overline\kappa}^{0123}+Y_{\overline\kappa}^{0125})\cdot E^0_{g'}\\
  &=\pi^*X_{\overline\kappa}^{0123}\cdot\pi^*X_{\overline\kappa}^{0123}\cdot E^0_{g'}+\pi^*X_{\overline\kappa}^{0123}\cdot
  Y_{\overline\kappa}^{0125}\cdot E^0_{g'}-\pi^*X_{\overline\kappa}^{0123}\cdot Y_{\overline\kappa}^{0123}\cdot E^0_{g'}
  -Y_{\overline\kappa}^{0123}\cdot Y_{\overline\kappa}^{0125}\cdot E^0_{g'}.
  \end{align*}
  By \cite{AG5}*{Lemma 2.2.14 (ii b)}, we have
  \begin{align}\label{eq:332}
  \pi^*X_{\overline\kappa}^{0123}\cdot\pi^*X_{\overline\kappa}^{0123}\cdot E^0_{g'}=
  \pi^*(X_{\overline\kappa}^{0123}\cdot X_{\overline\kappa}^{0125})\cdot E^0_{g'}=0.
  \end{align}
  We also have (Notation \ref{no:fiber})
  \begin{multline}\label{eq:333}
  \pi^*X_{\overline\kappa}^{0123}\cdot
  Y_{\overline\kappa}^{0125}\cdot E^0_{g'}-\pi^*X_{\overline\kappa}^{0123}\cdot Y_{\overline\kappa}^{0123}\cdot E^0_{g'}\\
  =X_{\overline\kappa}^{0123}\cdot(\pi_*Y_{\overline\kappa}^{0125}\cdot E^0_{g'}-\pi_*Y_{\overline\kappa}^{0123}\cdot E^0_{g'})
  =X_{\overline\kappa}^{0123}\cdot(D^0_{g'}-D^0_{g'})=0.
  \end{multline}
  Finally, as $Y_{\overline\kappa}^{0123}\cap Y_{\overline\kappa}^{0125}\cap E^0_{g'}\subset Y_{\overline\kappa}^{0123}\cap Y_{\overline\kappa}^{0125}\cap Y_{\overline\kappa}^{01235}\subset Y_{\overline\kappa}^{012}\cap Y_{\overline\kappa}^{015}\cap Y_{\overline\kappa}^{123}\cap Y_{\overline\kappa}^{135}$, which is empty from the dual reduction building \eqref{beq:reduction}, we have
  \begin{align}\label{eq:334}
  Y_{\overline\kappa}^{0123}\cdot Y_{\overline\kappa}^{0125}\cdot E^0_{g'}=0.
  \end{align}
  Thus \eqref{eq:331} hence the third equality follow from \eqref{eq:332}, \eqref{eq:333} and \eqref{eq:334}.

  \item The fourth equality follows from \eqref{eq:symmetric} and the ninth equality of Lemma \ref{le:intersection2}.

  \item The fifth equality follows from \eqref{eq:symmetric} and the ninth equality of Lemma \ref{le:intersection1}.

  \item We have $\tilde\Delta\sfE^0_{g'}(\sfE^5_g)=-E^0_{g'}\cdot Y_{\overline\kappa}^{0125}\cdot E^5_g$. Let $G^0_{g'}$ be the strict transform of $\wp_{012}^{-1}C^0_{g'}$ as in the proof of Lemma \ref{le:pullback} (2). Then
      \[-E^0_{g'}\cdot Y_{\overline\kappa}^{0125}\cdot E^5_g=(G^0_{g'}-F^0_{g'})\cdot Y_{\overline\kappa}^{0125}\cdot E^5_g
      =G^0_{g'}\cdot Y_{\overline\kappa}^{0125}\cdot E^5_g+\tilde\Delta\sfF^0_{g'}(\sfE^5_g)=G^0_{g'}\cdot Y_{\overline\kappa}^{0125}\cdot E^5_g-1\]
      by the sixth equality of Lemma \ref{le:intersection2}.
      Note that $\pi$ induces an isomorphism from $Y_{\overline\kappa}^{0125}\cap E^5_g$ to $D^5_g\subset X_{\overline\kappa}^{01245}$. In other words, $Y_{\overline\kappa}^{0125}\cap E^5_g$ is the strict transform of $D^5_g$ under the blow-up $\pi\colon Y_{\overline\kappa}^{012}\to X_{\overline\kappa}^{012}$. Now as $X_{\overline\kappa}^{01245}$ intersects transversally with $\wp_{012}^*C^0_{g'}$ at $X_{\overline\kappa}^{012345}$, we know that $Y_{\overline\kappa}^{0125}\cap E^5_g\cap G^0_{g'}=\emptyset$. Thus $Y_{\overline\kappa}^{0125}\cdot E^5_g=0$ and the sixth equality follows.

  \item The seventh equality follows from \eqref{eq:symmetric} and the sixth equality of Lemma \ref{le:intersection2}.

  \item The eighth equality follows from \eqref{eq:symmetric} and the sixth equality of Lemma \ref{le:intersection1}.

  \item The ninth equality follows from \eqref{eq:symmetric} and the sixth equality with $i=1$.

  \item We have $\tilde\Delta\sfE^0_{g'}(\sfH_h)=-E^0_{g'}\cdot Y_{\overline\kappa}^{0124}\cdot H_h$. However,
     \[-E^0_{g'}\cdot Y_{\overline\kappa}^{0124}\cdot H_h=G^0_{g'}\cdot Y_{\overline\kappa}^{0124}\cdot H_h-F^0_{g'}\cdot Y_{\overline\kappa}^{0124}\cdot H_h=
     G^0_{g'}\cdot Y_{\overline\kappa}^{0124}\cdot H_h+\tilde\Delta\sfF^0_{g'}(\sfH_h)=G^0_{g'}\cdot Y_{\overline\kappa}^{0124}\cdot H_h\]
     by Lemma \ref{le:intersection2}. By Proposition \ref{bpr:blowup} (5), $\pi\colon Y_{\overline\kappa}^{0124}\to X_{\overline\kappa}^{0124}$ is the blow-up along $X_{\overline\kappa}^{012345}$. Moreover, $H_h$ intersects with $Y_{\overline\kappa}^{0124}$ transversally at the exceptional divisor (inside $Y_{\overline\kappa}^{0124}$) above $h\in X_{\overline\kappa}^{012345}(\overline\kappa)$; and $G^0_{g'}$ intersects with $Y_{\overline\kappa}^{0124}$ transversally at the strict transform (inside $Y_{\overline\kappa}^{0124}$) of the curve $\wp_{0124}^{-1}\rT(g')\subset X_{\overline\kappa}^{0124}$, where we regard $\rT(g')$ as a (reduced) $0$-cycle of $Z_{\overline\kappa}^{024}$. Therefore, $H_h\cap Y_{\overline\kappa}^{0124}$ and $G^0_{g'}\cap Y_{\overline\kappa}^{0124}$ intersect transversally in $Y_{\overline\kappa}^{0124}$. The intersection consists of one point (resp.\ is empty) if $g'=\delta^-(h)$ (resp.\ if $g'\neq\delta^-(h)$). In summary, we have
     \[G^0_{g'}\cdot Y_{\overline\kappa}^{0124}\cdot H_h=(G^0_{g'}\cap Y_{\overline\kappa}^{0124})\cdot(H_h\cap Y_{\overline\kappa}^{0124})=1.\]
     The tenth equality follows.

  \item The vanishing of the remaining values is obvious.
\end{enumerate}
\end{proof}

\begin{lem}\label{le:intersection4}
For $h\in X^{012345}(\overline\kappa)$, we have
\[
\tilde\Delta\sfH_h
\begin{cases}
(\sfE^0_{\delta^-(h)})= 1, \\
(\sfE^1_{\delta^+(h)})= 1, \\
(\sfE^2_{\delta^-(h)})= 1, \\
(\sfE^3_{\delta^+(h)})= 1, \\
(\sfE^4_{\delta^-(h)})= 1, \\
(\sfE^5_{\delta^+(h)})= 1, \\
(\sfH_h)=-2.
\end{cases}
\]
On other featuring cycles, the function $\tilde\Delta\sfH_h$ takes value zero.
\end{lem}

\begin{proof}
The first six equalities follow from \eqref{eq:symmetric} and the tenth equality of Lemma \ref{le:intersection3}. The vanishing of the the remaining values follows from \eqref{eq:symmetric} and (the part of vanishing of) Lemmas \ref{le:intersection1}, \ref{le:intersection2} and \ref{le:intersection3}.

It remains to prove the seventh equality. Put $g=\delta^+(h)$. We have
\begin{align*}
\tilde\Delta\sfH_h(\sfH_h)&=H_h\cdot(Y_{\overline\kappa}^{012345}+Y_{\overline\kappa}^{0124}+Y_{\overline\kappa}^{0234}
+Y_{\overline\kappa}^{0245}+Y_{\overline\kappa}^{01234}+Y_{\overline\kappa}^{02345}+Y_{\overline\kappa}^{01245})\cdot H_h \\
&=H_h\cdot(H_h+F^0_g+F^2_g+F^4_g+E^1_g+E^3_g+E^5_g)\cdot H_h.
\end{align*}
Here, we regard all terms in the middle summation as $2$-cycles on $Y_{\overline\kappa}^{024}$. By Proposition \ref{bpr:blowup} (2), we can decompose $\pi$ as
\[Y_{\overline\kappa}^{024}\xrightarrow{\pi_2} X_{\overline\kappa}^{\prime024}\xrightarrow{\pi_1}X_{\overline\kappa}^{024}\]
where $\pi_1$ is the blow-up along $X_{\overline\kappa}^{012345}$ and $\pi_2$ is the blow-up along the strict transform of $X_{\overline\kappa}^{01234}\cup X_{\overline\kappa}^{01245}\cup X_{\overline\kappa}^{02345}$. Let $H'_g\subset  X_{\overline\kappa}^{\prime024}$ be the exceptional divisor above $h\in X_{\overline\kappa}^{012345}(\overline\kappa)$. By a similar argument for Lemma \ref{le:pullback} (1), we have $H_h=\pi_2^*H'_h$. Thus
\begin{align}\label{eq:471}
H_h \cdot H_h \cdot H_h=\pi_2^*(H'_h\cdot H'_h)\cdot \pi_2^*H'_h=(H'_h\cdot H'_h)\cdot\pi_{2*}\pi_2^*H'_h=H'_h\cdot H'_h\cdot H'_h=-1
\end{align}
by \cite{AG5}*{Lemma 2.2.14 (ii a)}. Note that for $k=1,3,5$, $E^k_g$ is an exceptional divisor for the blow-up $\pi_2$. Thus
\begin{align}\label{eq:472}
H_h \cdot E^k_g \cdot H_h=\pi_2^*H'_h\cdot \pi_2^*H'_h\cdot E^k_g=\pi_2^*(H'_h\cdot H'_h)\cdot E^k_g=0
\end{align}
by \cite{AG5}*{Lemma 2.2.14 (ii b)}. For $k=0,2,4$, the Chow cycle $\pi_*F^k_g$ is equal to the image of $F^k_g$ under $\pi$. Since $\pi_*F^k_g\cap X_{\overline\kappa}^{012345}=(\delta^+)^{-1}(g)$, we have $\pi_1^*\pi_*F^k_g=\pi_{2*}F^k_g+\sum_{h\in(\delta^+)^{-1}(g)}H'_h$. In particular,
\begin{multline*}
H_h \cdot F^k_g \cdot H_h=\pi_2^*(H'_h\cdot H'_h)\cdot F^k_g=(H'_h\cdot H'_h)\cdot\pi_{2*}F^k_g\\
=(H'_h\cdot H'_h)\cdot(\pi_1^*\pi_*F^k_g-\sum_{h\in(\delta^+)^{-1}(g)}H'_h)
=H'_h\cdot H'_h\cdot \pi_1^*\pi_*F^k_g-H'_h\cdot H'_h\cdot H'_h.
\end{multline*}
By \cite{AG5}*{Lemma 2.2.14 (ii a)}, $H'_h\cdot H'_h\cdot \pi_1^*\pi_*F^k_g=0$, and by \eqref{eq:471}, $H'_h\cdot H'_h\cdot H'_h=-1$. Together with \eqref{eq:471} and \eqref{eq:472}, we obtain the seventh equality.
\end{proof}

We will now write down the map $\tilde\Delta$ in terms of a block matrix. Thus, we need to fix bases of the free $\Lambda$-modules $\cB(Y,\Lambda)$ and $\Hom(\cB(Y,\Lambda),\Lambda)$ and make them into $19$ groups. For $\cB(Y,\Lambda)$, we
\begin{itemize}
  \item for $0\leq i\leq 5$, let the $(2i+1)$-th group of the basis be $\{\sfF^i_g\res g\in Z^\pm(\overline\kappa)\}$ if $i\in S^\pm$;
  \item for $0\leq i\leq 5$, let the $(2i+2)$-th group of the basis be $\{\sfF^i_{g'}\res g\in Z^\mp(\overline\kappa)\}$ if $i\in S^\pm$;
  \item for $0\leq i\leq 5$, let the $(13+i)$-th group of the basis be $\{\sfE^i_{g'}\res g\in Z^\mp(\overline\kappa)\}$ if $i\in S^\pm$;
  \item let the 19-th group of the basis be $\{\sfH_h\res h\in X^{012345}(\overline\kappa)\}$.
\end{itemize}
For $\Hom(\cB(Y,\Lambda),\Lambda)$, we take the dual basis and group it accordingly.

\begin{proposition}\label{pr:potential_matrix}
Under the above chosen bases, the map $\tilde\Delta$ is given by the following 19-by-19 block matrix:
\[\left(
    \begin{array}{ccc}
      A & \tp{B} & 0 \\
      B & C & D^\star \\
      0 & D & -2 \\
    \end{array}
  \right),
\] where
\[A=
\left(
  \begin{array}{cccccc}
    A_1 & A_2 & 0 & 0 & 0 & A_2 \\
    A_2 & A_1 & A_2 & 0 &0  & 0 \\
    0 & A_2 & A_1 & A_2 & 0 & 0 \\
    0 & 0 & A_2 & A_1 & A_2 & 0 \\
    0 & 0 & 0 & A_2 & A_1 & A_2 \\
    A_2 &0  & 0 & 0 & A_2 & A_1 \\
  \end{array}
\right)\] with
\[A_1=\left(
      \begin{array}{cc}
        -2\ell(\ell+1) & (\ell+1)\rT \\
        (\ell+1)\rT & -2\ell^2(\ell+1) \\
      \end{array}
    \right),\quad
A_2=\left(
      \begin{array}{cc}
        -\rT & 2\ell^2 \\
        2\ell^2 & -\ell\rT \\
      \end{array}
    \right);\]
and
\begin{align*}
B=\left(
    \begin{array}{cccccccccccc}
       \rT & -2\ell^2  &  2\ell & -\rT  & 0 & 0 & 0 & 0 &0  & 0 &  2\ell^2 & -\ell\rT  \\
       2\ell^2 & -\ell\rT  &  \rT & -2\ell^2  &  2\ell & -\rT  & 0 & 0 & 0 & 0 & 0 & 0 \\
       0& 0 &  2\ell^2 & -\ell\rT  &  \rT & -2\ell^2  &  2\ell & -\rT  & 0 &0  & 0 &  0\\
       0& 0 & 0 & 0 &  2\ell^2 & -\ell\rT  &  \rT & -2\ell^2  &  2\ell & -\rT  & 0 &0  \\
       0& 0 & 0 & 0 & 0 & 0 &  2\ell^2 & -\ell\rT  &  \rT & -2\ell^2  &  2\ell & -\rT  \\
      2\ell & -\rT  & 0 & 0 &0  & 0 & 0 & 0 & 2\ell^2 & -\ell\rT &  \rT & -2\ell^2  \\
    \end{array}
  \right),
\end{align*}
\begin{align*}
C=\left(
    \begin{array}{cccccc}
     -(\ell^3+1) & -\rT & 0 & 0 & 0 & -\rT \\
      -\rT & -(\ell^3+1) & -\rT & 0 &0  & 0 \\
      0 & -\rT & -(\ell^3+1) & -\rT & 0 & 0 \\
      0 & 0 & -\rT & -(\ell^3+1) & -\rT & 0 \\
      0 & 0 & 0 & -\rT & -(\ell^3+1) & -\rT \\
      -\rT & 0 & 0 & 0 & -\rT & -(\ell^3+1) \\
    \end{array}
  \right),
\end{align*}
\begin{align*}
D=\left(
     \begin{array}{cccccc}
      \delta^{+*} & \delta^{-*} & \delta^{+*} & \delta^{-*} & \delta^{+*} & \delta^{-*} \\
     \end{array}
   \right),\quad
D^\star=\tp{\left(
     \begin{array}{cccccc}
      \delta^+_* & \delta^-_* & \delta^+_* & \delta^-_* & \delta^+_* & \delta^-_* \\
     \end{array}
   \right)}.
\end{align*}
Here, $\rT=\delta_{\pm*}\circ\delta_{\mp}^*$ by abuse of notation.
\end{proposition}

\begin{proof}
This is a direct consequence of Lemmas \ref{le:intersection1}, \ref{le:intersection2}, \ref{le:intersection3} and \ref{le:intersection4}.
\end{proof}

For later use, we put
\begin{align}\label{eq:potential_matrix_change}
C'=C+\frac{1}{2}D^\star D=\left(
    \begin{array}{cccccc}
     -\frac{\ell^3+1}{2} & -\frac{\rT}{2} & \frac{\ell^3+1}{2} & \frac{\rT}{2}  & \frac{\ell^3+1}{2} & -\frac{\rT}{2} \\
     -\frac{\rT}{2} & -\frac{\ell^3+1}{2} & -\frac{\rT}{2} & \frac{\ell^3+1}{2} & \frac{\rT}{2}  & \frac{\ell^3+1}{2} \\
      \frac{\ell^3+1}{2} & -\frac{\rT}{2} & -\frac{\ell^3+1}{2} & -\frac{\rT}{2} & \frac{\ell^3+1}{2} & \frac{\rT}{2}  \\
      \frac{\rT}{2} & \frac{\ell^3+1}{2} & -\frac{\rT}{2} & -\frac{\ell^3+1}{2} & -\frac{\rT}{2} & \frac{\ell^3+1}{2} \\
      \frac{\ell^3+1}{2} & \frac{\rT}{2} & \frac{\ell^3+1}{2} & -\frac{\rT}{2} & -\frac{\ell^3+1}{2} & -\frac{\rT}{2} \\
      -\frac{\rT}{2} & \frac{\ell^3+1}{2} & \frac{\rT}{2}  & \frac{\ell^3+1}{2} & -\frac{\rT}{2} & -\frac{\ell^3+1}{2} \\
    \end{array}
  \right).
\end{align}

\subsection{Cohomology of Shimura surfaces}
\label{ss:cohomology}

By Remark \ref{bre:hecke_cube}, we have a monoidal functor $\dT^{\fr\fl}\to\EC(\cY_{\fr_0,\fr_1})$ canonically lifting \eqref{eq:monoidal}. We also have natural monoidal functors $\dT^{\fr\fl}\to\EC(Y_{\fr_0,\fr_1}^S)$ for $S$ ample, and $\dT^{\fr\fl}\to\EC(Z_{\fr_0,\fr_1}^S)$ for $S$ a type. They are compatible under embeddings and $\wp_S$ in the obvious sense. Moreover, \eqref{eq:beta1} is a diagram of $\Lambda[\dT^{\fr\fl}][\rG_\kappa]$-modules.

Put $\fm=\fm_\rho^{\fr\fl}$. In what follows, we will focus on the spectral sequence $\pres{2}\rE_{\cY_{\fr_0,\fr_1},\fm}$. To simplify notation, we denote it by $\dE$. We start from the following vanishing result.

\begin{lem}\label{le:vanishing}
We have for $i\in\{0,1,2,3,4,5\}$,
\begin{enumerate}
  \item $\rH^j(Z_{\fr_0,\fr_1,\overline\kappa}^{i(i+1)(i+2)},\dZ_p)_\fm=0$ for $j\neq 2$;
  \item $\rH^2(Z_{\fr_0,\fr_1,\overline\kappa}^{i(i+1)(i+2)},\dZ_p)_\fm$ is a free $\dZ_p$-module.
\end{enumerate}

\end{lem}

\begin{proof}
Without lost of generality, we assume $i=0$. Let $\tau_0,\tau_1,\tau_2\colon F\hookrightarrow\TF$ be the three distinct embeddings corresponding to $0,1,2$ in $\Phi_F$, respectively.

For (1), we may assume $j=0,1$ by the Poincar\'{e} duality. By the Nakayama lemma, it suffices to show that $\rH^j(Z_{\fr_0,\fr_1,\overline\kappa}^{012},\dF_p)/\fm=0$. Let $\cZ$ be the non-PEL type Shimura surface $\cX(\nabla\setminus\{\tau_0,\tau_1\})_{O_F,\fr}$ (over $\Spec\dZ_{\ell^6}$) as in Lemma \ref{ble:zink}. Then by the same lemma, it suffices to show that $\rH^j(\cZ\otimes_{\dZ_{\ell^6}}\overline\dQ,\dF_p)/\fm=0$ for $j=0,1$. We use the same ideal in the proof of \cite{Dim05}*{Theorem 6.6 (i)}.

Let $Z$ be the canonical model of $\cZ\otimes_{\dZ_{\ell^6}}\overline\dQ$ over $\Spec\TF$. For $j=0,1$, let $\rho'_j$ be the representation of $\rG_{\TF}$ on $\rH^j(\cZ\otimes_{\dZ_{\ell^6}}\overline\dQ,\dF_p)/\fm\simeq\rH^j(Z\otimes_{\TF}\overline\dQ,\dF_p)/\fm$. For $i=0,1,2$, put $\bar\rho_i\coloneqq\bar\rho\res_{\rG_{\TF},\tau_i}$, which is a two dimensional $\dF_p$-representation of $\rG_{\TF}$. By the Eichler--Shimura relation proved in \cite{Nek}*{\Sec A6} and the Chebotarev density theorem, we know that for every $g\in\rG_{\TF}$, the characteristic polynomial of $\bar\rho_0\otimes\bar\rho_1$ annihilates $\rho'_j(g)$. Since $\bar\rho$ is generic (Definition \ref{de:perfect_prime}), by a similar argument for \cite{Dim05}*{Lemma 6.5}, we know that each $\rG_{\TF}$-irreducible subquotient of $\rho'_j$ is isomorphic to $\bar\rho_0\otimes\bar\rho_1$. Since $p>4$, the Fontaine--Laffaille weights of $\bar\rho_0\otimes\bar\rho_1$ are $\{0,1,1,2\}$. As $p$ is coprime to $\nabla$ and $\fr\cdot\disc F$, the surface $Z$ has a smooth projective model over $\Spec O_{\TF}\otimes\dZ_p$. Together with Faltings' Comparison Theorem \cite{Fal88}, we know that $2$ is not a Fontaine--Laffaille weight for both $\rH^0(Z\otimes_{\TF}\overline\dQ,\dF_p)$ and $\rH^1(Z\otimes_{\TF}\overline\dQ,\dF_p)$. Therefore, $\rH^j(Z\otimes_{\TF}\overline\dQ,\dF_p)/\fm=0$ for $j=0,1$, and (1) follows.

Part (2) follows from (1) by the similar argument for \cite{Dim05}*{Theorem 6.6 (ii)}.
\end{proof}

\begin{lem}\label{le:shimura_set}
There are canonical $\rG_\kappa$-invariant isomorphisms (Notation \ref{no:nabla})
\[Z_{\fr_0,\fr_1}^\pm(\overline\kappa)/\Cl(F)_{\fr_1}\simeq\cS_\fr,\qquad X_{\fr_0,\fr_1}^{012345}(\overline\kappa)/\Cl(F)_{\fr_1}\simeq\cS_{\fr\fl}.\]
\end{lem}

\begin{proof}
These are special cases of Proposition \ref{bpr:hilbert_sparse}.
\end{proof}

\begin{definition}\label{de:quotient}
Denote by
\[\psi^\pm\colon Z_{\fr_0,\fr_1}^\pm(\overline\kappa)\to\cS_\fr\qquad \psi\colon X_{\fr_0,\fr_1}^{012345}(\overline\kappa)\to\cS_{\fr\fl}\]
the canonical maps obtained from Lemma \ref{le:shimura_set}. For an abelian group $M$, we denote by
\begin{enumerate}
  \item $\psi^{\pm*}\colon\Gamma(\cS_\fr,M)\to\Gamma(Z_{\fr_0,\fr_1}^\pm(\overline\kappa),M)$ and $\psi^*\colon\Gamma(\cS_{\fr\fl},M)\to\Gamma(X_{\fr_0,\fr_1}^{012345}(\overline\kappa),M)$ the usual pullback maps;

  \item  $\psi^\pm_*\colon\Gamma(Z_{\fr_0,\fr_1}^\pm(\overline\kappa),M)\to\Gamma(\cS_\fr,M)$ and $\psi_*\colon\Gamma(X_{\fr_0,\fr_1}^{012345}(\overline\kappa),M)\to\Gamma(\cS_{\fr\fl},M)$ the (normalized) pushforward maps as in Definition \ref{ade:pushforward} where we replace $\mu(\fN,\fM)$ by $|\Cl(F)_{\fr_1}|$.
\end{enumerate}
\end{definition}

The following lemma shows that the ray class group $\Cl(F)_{\fr_1}$ appearing in Lemma \ref{le:shimura_set} is negligible after localization.

\begin{lem}\label{le:quotient}
The maps in the following two commutative diagrams
\[\xymatrix{
\Gamma(Z_{\fr_0,\fr_1}^\pm(\overline\kappa),\Lambda)_\fm \ar[r]^-{\psi^\pm_*}\ar[d]
& \Gamma(\cS_\fr,\Lambda)_\fm \ar[d] \\
\Gamma(Z_{\fr_0,\fr_1}^\pm(\overline\kappa),\Lambda)/\Ker\phi^{\fr\fl}_\rho \ar[r]^-{\psi^\pm_*}
& \Gamma(\cS_\fr,\Lambda)/\Ker\phi^{\fr\fl}_\rho
}\quad
\xymatrix{
\Gamma(X_{\fr_0,\fr_1}^{012345}(\overline\kappa),\Lambda)_\fm
\ar[r]^-{\psi_*}\ar[d] & \Gamma(\cS_{\fr\fl},\Lambda)_\fm \ar[d] \\
\Gamma(X_{\fr_0,\fr_1}^{012345}(\overline\kappa),\Lambda)/\Ker\phi^{\fr\fl}_\rho
\ar[r]^-{\psi_*} & \Gamma(\cS_{\fr\fl},\Lambda)/\Ker\phi^{\fr\fl}_\rho
}\]
are all isomorphisms, where we recall $\fm=\fm^{\fr\fl}_\rho$.
\end{lem}

\begin{proof}
We only give the proof for the first diagram for $+$, as it is similar for the others.

Since $(\rho,\fr_\rho,\fr_0,\fr_1)$ is $\fr\fl$-isolated, the right vertical map is an isomorphism. It suffices to show that the two horizontal maps are isomorphisms. We fix a set of representatives $\{\fq_1,\dots,\fq_{|\Cl(F)_{\fr_1}|}\}$ of the ray class group $\Cl(F)_{\fr_1}$ using prime ideals of $O_F$ that are coprime to $\nabla$ and $\fr\fl$. Put $\rS\coloneqq\sum_{i=1}^{|\Cl(F)_{\fr_1}|}\rS_{\fq_i}\in\dZ[\dT^{\fr\fl}]$. As $\phi_\rho^{\fr\fl}(s)=|\Cl(F)_{\fr_1}|$ which is not divisible by $p$ by Definition \ref{de:perfect_prime} (3b), we know that $\rS\in\dZ[\dT^{\fr\fl}]\setminus\fm$.

We have maps $\psi^+_*\colon\Gamma(Z_{\fr_0,\fr_1}^+(\overline\kappa),\Lambda)\to\Gamma(\cS_\fr,\Lambda)$ and $\psi^{+*}\colon\Gamma(\cS_\fr,\Lambda)\to\Gamma(Z_{\fr_0,\fr_1}^+(\overline\kappa),\Lambda)$ such that $\psi^+_*\circ\psi^{+*}=|\Cl(F)_{\fr_1}|\in\Lambda^\times$. It is clear that the composite map $\psi^{+*}\circ\psi^+_*\colon\Gamma(Z_{\fr_0,\fr_1}^+(\overline\kappa),\Lambda)_\fm\to\Gamma(Z_{\fr_0,\fr_1}^+(\overline\kappa),\Lambda)_\fm$ is equal to the multiplication by $\rS$, which is an automorphism. Thus the upper horizontal map is an isomorphism. The composite map $\psi^{+*}\circ\psi^+_*\colon\Gamma(Z_{\fr_0,\fr_1}^+(\overline\kappa),\Lambda)/\Ker\phi^{\fr\fl}_\rho
\to\Gamma(Z_{\fr_0,\fr_1}^+(\overline\kappa),\Lambda)/\Ker\phi^{\fr\fl}_\rho$ is also equal to the multiplication by $|\Cl(F)_{\fr_1}|$, which is an automorphism. Thus the lower horizontal map is an isomorphism as well. The lemma is proved.
\end{proof}

The following proposition confirms the Tate conjecture for the localized cohomology $\rH^2(Z_{\fr_0,\fr_1,\overline\kappa}^{i(i+1)(i+2)},\Lambda(1))_\fm$. Since we are working with torsion coefficients and localization, the results of \cite{TX14} do not seem to be applicable directly. We will provide a proof following the method of \cite{TX14}.

\begin{proposition}\label{pr:tate}
For $i\in S^\pm$, we have
\begin{enumerate}
  \item $\rH^2(Z_{\fr_0,\fr_1,\overline\kappa}^{i(i+1)(i+2)},\Lambda(1))_\fm\simeq
      (\Lambda(1)\oplus\Lambda^{\oplus2}\oplus\Lambda(-1))^{\oplus|\fD(\fr,\fr_\rho)|}$ as $\Lambda[\rG_\kappa]$-modules;

  \item that the map $\rH^2(Z_{\fr_0,\fr_1,\overline\kappa}^{i(i+1)(i+2)},\Lambda(1))\to
      \Gamma(Z_{\fr_0,\fr_1}^\pm(\overline\kappa),\Lambda)\oplus\Gamma(Z_{\fr_0,\fr_1}^\mp(\overline\kappa),\Lambda)$
      sending $c$ in the domain to the function assigning $(g,g')\in Z_{\fr_0,\fr_1}^\pm(\overline\kappa)\times Z_{\fr_0,\fr_1}^\mp(\overline\kappa)$ to $(c\cdot\cl(C^i_g),c\cdot\cl(C^i_{g'}))$ induces an isomorphism
      \[(\rH^2(Z_{\fr_0,\fr_1,\overline\kappa}^{i(i+1)(i+2)},\Lambda(1))_\fm)_{\rG_\kappa}\xrightarrow{\sim} \Gamma(Z_{\fr_0,\fr_1}^\pm(\overline\kappa),\Lambda)_\fm\oplus\Gamma(Z_{\fr_0,\fr_1}^\mp(\overline\kappa),\Lambda)_\fm.\]
\end{enumerate}
\end{proposition}

\begin{proof}
Note that there is a canonical bijection between $Z_{\fr_0,\fr_1}^+(\overline\kappa)$ and $Z_{\fr_0,\fr_1}^-(\overline\kappa)$ induced by the arithmetic Frobenius isomorphism (Remark \ref{bre:frobenius}). It is compatible with Hecke actions.

Let $L$ be an algebraic closure of $\dQ_p$, $O_L\subset L$ the ring of integers, and $P_L\subset O_L$ the maximal ideal. By the theory of automorphic forms, there is a finite set $R$ of homomorphisms $\phi\colon\dZ[\dT^{\fr\fl}]\to O_L$ such that $\Gamma(\cS_{\fr_\rho},L)/\Ker\phi\neq 0$ and $\Gamma(\cS_{\fr_\rho},L)=\bigoplus_{\phi\in R}\Gamma(\cS_{\fr_\rho},L)/\Ker\phi$. In particular, we obtain an injective homomorphism $\Gamma(\cS_{\fr_\rho},O_L)\hookrightarrow\bigoplus_{\phi\in R}\Gamma(\cS_{\fr_\rho},O_L)/\Ker\phi$ with cokernel of finite $O_L$-length. Let $R_\rho\subset R$ be the subset of $\phi$ such that $\Ker[\dZ[\dT^{\fr\fl}]\xrightarrow{\phi} O_L\to O_L/P_L]=\fm$. Then $(\Gamma(\cS_{\fr_\rho},O_L)/\Ker\phi)_\fm=\Gamma(\cS_{\fr_\rho},O_L)/\Ker\phi$ (resp.\ $(\Gamma(\cS_{\fr_\rho},O_L)/\Ker\phi)_\fm=0$) if $\phi\in R_\rho$ (resp.\ $\phi\in R\setminus R_\rho$). In other words, we obtain an injective map
\begin{align}\label{eq:tate1}
\Gamma(\cS_{\fr_\rho},O_L)_\fm\hookrightarrow\bigoplus_{\phi\in R_\rho}\Gamma(\cS_{\fr_\rho},O_L)/\Ker\phi
\end{align}
Note that $\Gamma(\cS_{\fr_\rho},O_L)_\fm\otimes_{O_L}O_L/P_L=\Gamma(\cS_{\fr_\rho},\dF_p)_\fm\otimes_{\dF_p}O_L/P_L$, and $(\Gamma(\cS_{\fr_\rho},O_L)/\Ker\phi)\otimes_{O_L}O_L/P_L=(\Gamma(\cS_{\fr_\rho},\dF_p)/\fm)\otimes_{\dF_p}O_L/P_L$ for $\phi\in R_\rho$.
By Definition \ref{de:cubic_level_raising} (C2) and the Nakayama lemma, we know that \eqref{eq:tate1} is an isomorphism of free $O_L$-modules of rank $1$, and $R_\rho=\{\phi\}$ is a singleton.

By the similar argument, we obtain an injective map
\begin{align}\label{eq:tate2}
\Gamma(Z_{\fr_0,\fr_1}^\pm(\overline\kappa),O_L)_\fm\hookrightarrow\bigoplus_{\phi^\pm\in R^\pm_\rho}\Gamma(Z_{\fr_0,\fr_1}^\pm(\overline\kappa),O_L)/\Ker\phi^\pm
\end{align}
for similarly defined sets $R^+_\rho$ and $R^-_\rho$, both containing $\phi$. By Lemma \ref{le:quotient}, we have induced isomorphisms $\psi^\pm_*\colon\Gamma(Z_{\fr_0,\fr_1}^\pm(\overline\kappa),O_L)_\fm\otimes_{O_L}O_L/P_L
\xrightarrow{\sim}\Gamma(\cS_\fr,\dF_p)_\fm\otimes_{\dF_p}O_L/P_L$, and $\psi^\pm_*\colon((Z_{\fr_0,\fr_1}^\pm(\overline\kappa),O_L)/\Ker\phi^\pm)\otimes_{O_L}O_L/P_L\to(\Gamma(\cS_\fr,\dF_p)/\fm)\otimes_{\dF_p}O_L/P_L$. Again by Definition \ref{de:cubic_level_raising} (C2) and the Nakayama lemma, we know that \eqref{eq:tate2} is an isomorphism of free $O_L$-modules of rank $|\fD(\fr,\fr_\rho)|$, and $R^\pm_\rho=\{\phi\}$ are both singletons.

By Definition \ref{de:cubic_level_raising} (C2) one more time, we have isomorphisms \[\Gamma(\cS_{\fr_\rho},\Lambda)_\fm\xrightarrow{\sim}\Gamma(\cS_{\fr_\rho},\Lambda)/\Ker\phi_\rho^{\fr\fl}
\xrightarrow{\sim}\Gamma(\cS_{\fr_\rho},\Lambda)/\Ker\phi_\rho^\fr.\]
Together, we know that the composite map $\dZ[\dT^{\fr\fl}]\xrightarrow{\phi}O_L\to O_L/p^\nu$ coincides with $\phi_\rho^{\fr\fl}$; and $\rT_\fl$ acts on $\Gamma(\cS_{\fr_\rho},O_L/p^\nu)/\Ker\phi$ by $\ell^3+1$. In other words, the Galois representation attached to $\phi$ has eigenvalues $\ell^3$ and $1$ modulo $p^\nu$ for a geometric Frobenius at $\fl$.

Now we move to the cohomology $\rH^2(Z_{\fr_0,\fr_1,\overline\kappa}^{i(i+1)(i+2)},\Lambda(1))$. Without lost of generality, we assume that $i=0$. By Lemma \ref{le:vanishing} (2), the Jacquet--Langlands correspondence, and a similar argument as above, we obtain an injective map
\begin{align*}
\rH^2(Z_{\fr_0,\fr_1,\overline\kappa}^{012},O_L(1))_\fm\hookrightarrow(\rH^2(Z_{\fr_0,\fr_1,\overline\kappa}^{012},O_L(1))/\Ker\phi)_\fm
\simeq\rH^2(Z_{\fr_0,\fr_1,\overline\kappa}^{012},O_L(1))/\Ker\phi
\end{align*}
with cokernel of finite $O_L$-length. By \cite{TX14}*{Proposition 2.26}, we know that as a $\Lambda[\rG_\kappa]$-module, $\rH^2(Z_{\fr_0,\fr_1,\overline\kappa}^{012},\Lambda(1))_\fm$ is isomorphic to $|\fD(\fr,\fr_\rho)|$ copies of $Q$ such that $Q^{\r{ss}}\simeq\Lambda(1)\oplus\Lambda^{\oplus2}\oplus\Lambda(-1)$. As $\Gamma(Z_{\fr_0,\fr_1}^+(\overline\kappa),\Lambda)_\fm\oplus\Gamma(Z_{\fr_0,\fr_1}^-(\overline\kappa),\Lambda)_\fm$ is a free $\Lambda$-module of rank $2|\fD(\fr,\fr_\rho)|$ by the previous discussion, both (1) and (2) will follow if we can show that $\rH^2(Z_{\fr_0,\fr_1,\overline\kappa}^{012},\Lambda(1))_\fm\to
\Gamma(Z_{\fr_0,\fr_1}^+(\overline\kappa),\Lambda)_\fm\oplus\Gamma(Z_{\fr_0,\fr_1}^-(\overline\kappa),\Lambda)_\fm$ is surjective. For this, consider the composite map
\begin{align}\label{eq:tate3}
\Lambda[Z_{\fr_0,\fr_1}^+(\overline\kappa)]\oplus\Lambda[Z_{\fr_0,\fr_1}^-(\overline\kappa)]\to
\rH^2(Z_{\fr_0,\fr_1,\overline\kappa}^{012},\Lambda(1))\to
\Gamma(Z_{\fr_0,\fr_1}^+(\overline\kappa),\Lambda)\oplus\Gamma(Z_{\fr_0,\fr_1}^-(\overline\kappa),\Lambda)
\end{align}
where the first map sends $g\in Z_{\fr_0,\fr_1}^+(\overline\kappa)$ (resp.\ $g'\in Z_{\fr_0,\fr_1}^-(\overline\kappa)$) to $\cl(C^0_g)$ (resp.\ $\cl(C^0_{g'})$). Under natural bases, \eqref{eq:tate3} is given by the matrix
\[\left(
    \begin{array}{cc}
      -2\ell & \rT \\
      \rT & -2\ell^2 \\
    \end{array}
  \right)
\]
by Lemma \ref{le:cycle_base} (3). Now after localization at $\fm$, we know that $\rT$ acts by the constant $\ell^3+1$. Therefore, the localization of \eqref{eq:tate3} at $\fm$ is an isomorphism as $p\nmid\ell^3-1$ by Definition \ref{de:cubic_level_raising} (C3). In particular, $\rH^2(Z_{\fr_0,\fr_1,\overline\kappa}^{012},\Lambda(1))_\fm\to
\Gamma(Z_{\fr_0,\fr_1}^+(\overline\kappa),\Lambda)_\fm\oplus\Gamma(Z_{\fr_0,\fr_1}^-(\overline\kappa),\Lambda)_\fm$ is surjective. The proposition is proved.
\end{proof}

\begin{lem}\label{le:modification}
Let $Z$ be a smooth proper $\kappa$-scheme, and $\varphi\colon Y\to Z$ the composition of finitely many morphisms that are either blow-ups along smooth centers or projective bundles. Then
\begin{enumerate}
  \item $\varphi^*\colon\rH^j(Z_{\overline\kappa},\Lambda)\to\rH^j(Y_{\overline\kappa},\Lambda)$ is an isomorphism for $j$ odd;
  \item $\varphi^*\colon\rH^2(Z_{\overline\kappa},\Lambda(1))\to\rH^2(Y_{\overline\kappa},\Lambda(1))$ is injective whose image is a direct summand as a $\Lambda[\rG_\kappa]$-module. Moreover, the action of $\rG_\kappa$ on $\coker\varphi^*$ is trivial.
\end{enumerate}
\end{lem}

\begin{proof}
Part (1) is well-known. Part (2) follows from \cite{SGA5}*{VII, 2.2.6 \& 8.5}.
\end{proof}

\begin{lem}\label{le:spectral_sequence_cubic}
The following diagram contains all nonzero terms in the first page $\dE_1$.
\begin{align*}
\resizebox{16cm}{!}{\xymatrix{
\dE^{p,q}_1  \ar@{}[r] & q=0 \ar@{}[r] & q=2 \ar@{}[r] & q=4 \ar@{}[r] & q=6 \\
p=-3 \ar@{}[r] & & & & \rH^0(Y_{\fr_0,\fr_1,\overline\kappa}^{(3)},\Lambda(-1))_\fm \ar@{-->}[ddl]\ar[d] \\
p=-2 \ar@{}[r] & & & \rH^0(Y_{\fr_0,\fr_1,\overline\kappa}^{(2)},\Lambda)_\fm \ar[d]
& \rH^2(Y_{\fr_0,\fr_1,\overline\kappa}^{(2)},\Lambda)_\fm \ar[d] \\
p=-1 \ar@{}[r] & & \rH^0(Y_{\fr_0,\fr_1,\overline\kappa}^{(1)},\Lambda(1))_\fm \ar[d]
& \rH^2(Y_{\fr_0,\fr_1,\overline\kappa}^{(1)},\Lambda(1))_\fm\oplus\rH^0(Y_{\fr_0,\fr_1,\overline\kappa}^{(3)},\Lambda)_\fm
\ar@{-->}[ddl]_-{\mu^{0,3}}\ar[d]
& \rH^4(Y_{\fr_0,\fr_1,\overline\kappa}^{(1)},\Lambda(1))_\fm \ar[d] \\
p=0 \ar@{}[r] & \rH^0(Y_{\fr_0,\fr_1,\overline\kappa}^{(0)},\Lambda(2))_\fm \ar[d]
& \rH^2(Y_{\fr_0,\fr_1,\overline\kappa}^{(0)},\Lambda(2))_\fm\oplus\rH^0(Y_{\fr_0,\fr_1,\overline\kappa}^{(2)},\Lambda(1))_\fm \ar[d]
& \rH^4(Y_{\fr_0,\fr_1,\overline\kappa}^{(0)},\Lambda(2))_\fm\oplus\rH^2(Y_{\fr_0,\fr_1,\overline\kappa}^{(2)},\Lambda(1))_\fm \ar[d]
& \rH^6(Y_{\fr_0,\fr_1,\overline\kappa}^{(0)},\Lambda(2))_\fm \\
p=1 \ar@{}[r] & \rH^0(Y_{\fr_0,\fr_1,\overline\kappa}^{(1)},\Lambda(2))_\fm \ar[d]
& \rH^2(Y_{\fr_0,\fr_1,\overline\kappa}^{(1)},\Lambda(2))_\fm\oplus\rH^0(Y_{\fr_0,\fr_1,\overline\kappa}^{(3)},\Lambda(1))_\fm \ar@{-->}[ddl]\ar[d]
& \rH^4(Y_{\fr_0,\fr_1,\overline\kappa}^{(1)},\Lambda(2))_\fm   \\
p=2 \ar@{}[r] & \rH^0(Y_{\fr_0,\fr_1,\overline\kappa}^{(2)},\Lambda(2))_\fm \ar[d]
& \rH^2(Y_{\fr_0,\fr_1,\overline\kappa}^{(2)},\Lambda(2))_\fm   \\
p=3 \ar@{}[r] & \rH^0(Y_{\fr_0,\fr_1,\overline\kappa}^{(3)},\Lambda(2))_\fm}}
\end{align*}
\end{lem}
in which the dashed maps are those in \eqref{eq:spectral_sequence}.

\begin{proof}
This is a consequence of Lemmas \ref{le:strata}, \ref{le:vanishing} and \ref{le:modification} (1).
\end{proof}

\begin{proposition}\label{pr:degenerate}
The Galois group $\rG_\kappa$ acts trivially on $\dE^{p,q}_2(\frac{q-4}{2})$.
\end{proposition}

\begin{proof}
By the Poincar\'{e} duality and the fact that $\rG_\kappa$ acts trivially on $\rH^0(Y_{\overline\kappa},\Lambda)$ for every $\kappa$-scheme $Y$, we only need to consider two terms: $\dE^{0,2}_2(-1)$ and $\dE^{1,2}_2(-1)$. By Proposition \ref{pr:tate} and Lemma \ref{le:modification}, we have a canonical decomposition $\dE^{0,2}_1(-1)=M_1\oplus M_0\oplus M_{-1}$ of $\Lambda[\rG_\kappa]$-modules in which $M_i$ is isomorphic to the direct sum of finitely many copies of $\Lambda(i)$ for $i=-1,0,1$. Similarly, we have $\dE^{1,2}_1(-1)=N_1\oplus N_0\oplus N_{-1}$.

We claim that the restricted differential map $\rd^{0,2}_1\colon M_1\oplus M_{-1}\to N_1\oplus N_{-1}$ is an isomorphism. Then the action of $\rG_\kappa$ on both $\dE^{0,2}_2(-1)$ and $\dE^{1,2}_2(-1)$ will be trivial.

By Lemma \ref{le:modification}, we have an injective map
\[\bigoplus_{i=0}^5\rH^2(Z_{\fr_0,\fr_1,\overline\kappa}^{i(i+1)(i+2)},\Lambda(1))_\fm
\hookrightarrow\bigoplus_{i=0}^5\rH^2(Y_{\fr_0,\fr_1,\overline\kappa}^{i(i+1)(i+2)},\Lambda(1))_\fm\subset
\rH^2(Y_{\fr_0,\fr_1,\overline\kappa}^{(0)},\Lambda(1))_\fm\subset\dE^{0,2}_1(-1)\]
whose image contains $M_1\oplus M_{-1}$. By Proposition \ref{bpr:blowup} (4) and Lemma \ref{le:strata} (3), we have a surjective map
\[\dE^{1,2}_1(-1)\to\rH^2(Y_{\fr_0,\fr_1,\overline\kappa}^{(1)},\Lambda(1))_\fm\to
\bigoplus_{i=0}^5\rH^2(Y_{\fr_0,\fr_1,\overline\kappa}^{i(i+1)(i+2)(i+5)},\Lambda(1))_\fm
\simeq\bigoplus_{i=0}^5\rH^2(Z_{\fr_0,\fr_1,\overline\kappa}^{i(i+1)(i+2)},\Lambda(1))_\fm\]
under which $N_1\oplus N_{-1}$ maps isomorphically onto its image. Since $M_1\oplus M_{-1}$ and $N_1\oplus N_{-1}$ have the same length, to prove the claim, it suffices to show the map
\[\rd\colon\bigoplus_{i=0}^5\rH^2(Z_{\fr_0,\fr_1,\overline\kappa}^{i(i+1)(i+2)},\Lambda(1))_\fm
\to\bigoplus_{i=0}^5\rH^2(Z_{\fr_0,\fr_1,\overline\kappa}^{i(i+1)(i+2)},\Lambda(1))_\fm\]
induced from $\rd^{0,2}_1$ has injective restriction to $M_1\oplus M_{-1}$. Recall from \Sec\ref{ss:semistable_schemes}, to define $\rd^{0,2}_1$, we have to fix an order of (relevant) irreducible components. Without lost of generality, we suppose that $Y_{\fr_0,\fr_1}^{012},Y_{\fr_0,\fr_1}^{123},Y_{\fr_0,\fr_1}^{234},Y_{\fr_0,\fr_1}^{345},Y_{\fr_0,\fr_1}^{450},Y_{\fr_0,\fr_1}^{501}$ are listed in the order. Then under the natural basis from $i=0$ to $5$, the map $\rd$ is given by the matrix (applied from left)
\[\left(
    \begin{array}{cccccc}
      -\r{id} &  &  &  &  & \b{f}^* \\
      -\b{f}^* & \r{id} &  &  &  &  \\
       &-\b{f}^*  & \r{id} &  &  &  \\
       &  & -\b{f}^* & \r{id} &  &  \\
       &  &  & -\b{f}^* & \r{id} &  \\
        &  &  &  & -\b{f}^* & \r{id} \\
    \end{array}
  \right)
\]
where $\b{f}$ is the morphism \eqref{beq:translation2}. In particular,
\[\Ker\rd=\{(c,\b{f}^*c,\b{f}^{2*}c,\b{f}^{3*}c,\b{f}^{4*}c,\b{f}^{5*}c)\res
c\in\rH^2(Z_{\fr_0,\fr_1,\overline\kappa}^{012},\Lambda(1))_\fm,c=\b{f}^{6*}c\}\]
By Proposition \ref{bpr:translation2}, we know that $c=\b{f}^{6*}c$ implies $c\in\rH^2(Z_{\fr_0,\fr_1,\overline\kappa}^{012},\Lambda(1))_\fm^{\rG_\kappa}$. Thus, $\Ker\rd\cap(M_1\oplus M_{-1})=\{0\}$. The proposition is proved.
\end{proof}

\begin{proposition}\label{pr:featuring}
Put $\cA(Y_{\fr_0,\fr_1},\Lambda)=\cB(Y_{\fr_0,\fr_1},\Lambda)^{\rG_\kappa}$.
\begin{enumerate}
  \item The canonical map $\Hom(\cB(Y_{\fr_0,\fr_1}),\Lambda)^{\rG_\kappa}\to\Hom(\cA(Y_{\fr_0,\fr_1},\Lambda),\Lambda)$ is an isomorphism.

  \item $\Lambda$ is a very nice coefficient for $\dE$ (Definition \ref{de:nice_coefficient_localized}).

  \item The maps $\beta$ \eqref{eq:beta0} and $\beta'$ \eqref{eq:beta} induce isomorphisms
     \[\cA(Y_{\fr_0,\fr_1},\Lambda)_\fm\xrightarrow{\sim}A_2(Y_{\fr_0,\fr_1},\Lambda)_0^\fm,\qquad
     A^2(Y_{\fr_0,\fr_1},\Lambda)^0_\fm\xrightarrow{\sim}\Hom(\cA(Y_{\fr_0,\fr_1},\Lambda),\Lambda)_\fm.\]
     Here, $A_2(Y_{\fr_0,\fr_1},\Lambda)_0^\fm$ and $A^2(Y_{\fr_0,\fr_1},\Lambda)^0_\fm$ are defined in \Sec\ref{ss:correspondence}.
\end{enumerate}
\end{proposition}

\begin{proof}
For (1), by Lemma \ref{le:shimura_set}, the action of $\Gamma_\kappa$ on $\cB(Y_{\fr_0,\fr_1},\Lambda)$ factorizes through $\Cl(F)_{\fr_1}$. Since $p\nmid|\Cl(F)_{\fr_1}|$, the canonical map $\cB(Y_{\fr_0,\fr_1},\Lambda)^{\rG_\kappa}\to\cB(Y_{\fr_0,\fr_1},\Lambda)_{\rG_\kappa}$ is an isomorphism. Thus the composite map $\Hom(\cB(Y_{\fr_0,\fr_1})_{\rG_\kappa},\Lambda)
\xrightarrow{\sim}\Hom(\cB(Y_{\fr_0,\fr_1}),\Lambda)^{\rG_\kappa}\to\Hom(\cA(Y_{\fr_0,\fr_1},\Lambda),\Lambda)$ is an isomorphism; (1) follows.

For (2), conditions (N1) and (N2) of Definition \ref{de:nice_coefficient_localized} are satisfied due to Proposition \ref{pr:degenerate} and the fact that $p\nmid(\ell^6-1)(\ell^{12}-1)(\ell^{18}-1)$ by Definition \ref{de:cubic_level_raising} (C3). For (N3), we have $\Ker\mu^{1,3}_\fm=\rH^4(Y_{\fr_0,\fr_1,\overline\kappa}^{(0)},\Lambda(2))_\fm$ and $\coker\mu^{-1,3}_\fm(-1)=\rH^2(Y_{\fr_0,\fr_1,\overline\kappa}^{(0)},\Lambda(1))_\fm$. By the Poincar\'{e} duality, it suffices to check the subquotient for $\rH^2(Y_{\fr_0,\fr_1,\overline\kappa}^{(0)},\Lambda(1))_\fm$. By Proposition \ref{bpr:blowup}, Lemmas \ref{le:strata} and \ref{le:modification}, $\rH^2(Y_{\fr_0,\fr_1,\overline\kappa}^{(0)},\Lambda(1))_\fm$ is isomorphic to the direct sum of $\bigoplus_{i=1}^5\rH^2(Z_{\fr_0,\fr_1,\overline\kappa}^{i(i+1)(i+2)},\Lambda(1))_\fm$ and a trivial $\Lambda[\rG_\kappa]$-module. Thus (N3) follows from Proposition \ref{pr:tate} (1).

For (3), by definition we have $B^2(Y_{\fr_0,\fr_1},\Lambda)^0_\fm=\Ker[B^2(Y_{\fr_0,\fr_1},\Lambda)_\fm\to\dE^{0,4}_2]$. By (2), Theorem \ref{th:cubic_level_raising} (1), and the fact that $\pi\colon\cY_{\fr_0,\fr_1}\to\cX_{\fr_0,\fr_1}$ induces an isomorphism on the generic fiber, we have $\dE^{0,4}_2=\dE^{0,4}_\infty=0$. Thus, $B^2(Y_{\fr_0,\fr_1},\Lambda)^0_\fm=B^2(Y_{\fr_0,\fr_1},\Lambda)_\fm$, and the first isomorphism amounts to the following isomorphism
\[\cA(Y_{\fr_0,\fr_1},\Lambda)_\fm\xrightarrow{\sim}\coker[\rH^0(Y_{\fr_0,\fr_1,\overline\kappa}^{(1)},\Lambda)_\fm\xrightarrow{\delta_{1*}}
\rH^2(Y_{\fr_0,\fr_1,\overline\kappa}^{(0)},\Lambda(1))_\fm^{\rG_\kappa}].\]
However, this is a straightforward consequence of (the dual of) Proposition \ref{pr:tate} (2), Lemma \ref{le:modification}, and the fact that $\rH^0(Y_{\fr_0,\fr_1,\overline\kappa}^{i(i+1)(i+2)(i+3)},\Lambda)_\fm=0$ for every $i$ as $\fm$ is not an Eisenstein ideal. The second isomorphism follows from the first one and the Poincar\'{e} duality.
\end{proof}

\subsection{Proof of Theorem \ref{th:cubic_level_raising}}
\label{ss:cubic}

By Lemma \ref{le:quotient}, we may identify $\Hom(\cB(Y_{\fr_0,\fr_1},\Lambda),\Lambda)_\fm$ with
\begin{align}\label{eq:19}
\Gamma(\cS_\fr,\Lambda)_\fm^{\oplus 18}\bigoplus\Gamma(\cS_{\fr\fl},\Lambda)_\fm
\end{align}
via the (ordered) basis chosen before Proposition \ref{pr:potential_matrix}. We write an element in \eqref{eq:19} in the form $(\boldsymbol{f},\boldsymbol{e},\boldsymbol{h})$, where $\boldsymbol{f}$ encodes the first 12 components; $\boldsymbol{e}$ encodes the next six components; and $\boldsymbol{h}$ encodes the last component.

\begin{proposition}\label{pr:corank_cubic}
The assignment
\begin{align}\label{eq:corank_cubic}
(\boldsymbol{f},\boldsymbol{e},\boldsymbol{h})\mapsto(e_1+e_4,e_2+e_5,e_1+e_6)
\end{align}
where $e_i$ is the $i$-th component of $\boldsymbol{e}-BA^{-1}\boldsymbol{f}+\frac{1}{2}D^\star\boldsymbol{h}$, induces an isomorphism \[\coker\tilde\Delta_\fm\simeq\Gamma(\cS_\fr,\Lambda)_\fm^{\oplus 3}.\]
In particular, $\coker\tilde\Delta_\fm$ is a free $\Lambda$-module of rank $3|\fD(\fr,\fr_\rho)|$.
\end{proposition}

\begin{proof}
The statement implicitly asserts that $A$ is invertible after localization at $\fm$. Since $(\rho,\fr_\rho,\fr_0,\fr_1)$ is both $\fr$-isolated and $\fr\fl$-isolated by Definition \ref{de:cubic_level_raising} (C2), the canonical maps \[\Gamma(\cS_\fr,\Lambda)_\fm\to\Gamma(\cS_\fr,\Lambda)/\Ker\phi^{\fr\fl}_\rho\to\Gamma(\cS_\fr,\Lambda)/\Ker\phi^\fr_\rho\]
are isomorphisms by Lemma \ref{le:quotient}. In particular, $\rT=\phi^\fr_\rho(\rT_\fl)=\ell^3+1$ on $\Hom(\cB(Y_{\fr_0,\fr_1},\Lambda),\Lambda)_\fm$ by Definition \ref{de:cubic_level_raising} (C4). Using \emph{Mathematica},
we find that
\begin{enumerate}
  \item the determinant of the matrix $A$ is equal to $Q_0(\ell)\coloneqq(\ell-1)^{16}(\ell^2+\ell+1)^{16}$;

  \item if we let $A^*$ be the adjoint matrix of $A$, then
        \[BA^*\tp{B}=
        \left(
        \begin{array}{cccccc}
        Q_1(\ell) & Q_2(\ell) & -Q_2(\ell) & -Q_1(\ell) & -Q_2(\ell) &Q_2(\ell) \\
        Q_2(\ell) & Q_1(\ell) & Q_2(\ell) & -Q_2(\ell) & -Q_1(\ell) & -Q_2(\ell)  \\
        -Q_2(\ell) & Q_2(\ell) & Q_1(\ell) & Q_2(\ell) & -Q_2(\ell) & -Q_1(\ell) \\
        -Q_1(\ell) & -Q_2(\ell) & Q_2(\ell) & Q_1(\ell) & Q_2(\ell) & -Q_2(\ell) \\
        -Q_2(\ell) & -Q_1(\ell) & -Q_2(\ell) & Q_2(\ell) & Q_1(\ell) & Q_2(\ell) \\
        Q_2(\ell) & -Q_2(\ell) & -Q_1(\ell) & -Q_2(\ell) & Q_2(\ell) & Q_1(\ell)  \\
        \end{array}
        \right),\]
        where $Q_1(\ell)=-2\ell(\ell+1)(\ell^2-\ell+1)^2(\ell-1)^{14}(\ell^2+\ell+1)^{15}$ and $Q_2(\ell)=2\ell^2(\ell^3+1)(\ell-1)^{14}(\ell^2+\ell+1)^{15}$.
\end{enumerate}
In particular, (1) implies that $A$ is invertible after localization at $\fm$ by Definition \ref{de:cubic_level_raising} (C3).

Note that the following diagram
\[\xymatrix{
\Gamma(\cS_\fr,\Lambda)_\fm^{\oplus 18}\bigoplus\Gamma(\cS_{\fr\fl},\Lambda)_\fm
\ar[r]\ar[d]^-{\tilde\Delta_\fm} & \Gamma(\cS_\fr,\Lambda)_\fm^{\oplus 6} \ar[d] \\
\Gamma(\cS_\fr,\Lambda)_\fm^{\oplus 18}\bigoplus\Gamma(\cS_{\fr\fl},\Lambda)_\fm
\ar[r] &  \Gamma(\cS_\fr,\Lambda)_\fm^{\oplus 6}
}\]
commutes, in which horizontal vertical maps are the assignment
\[(\boldsymbol{f},\boldsymbol{e},\boldsymbol{h})\mapsto\boldsymbol{e}-BA^{-1}\boldsymbol{f}+\frac{1}{2}D^\star\boldsymbol{h},\]
and the right vertical map is given by the matrix $C'-BA^{-1}\tp{B}$ where $C'$ is in \eqref{eq:potential_matrix_change}. Moreover, the induced map
\begin{align}\label{eq:corank}
\coker\tilde\Delta_\fm\to\coker[C'-BA^{-1}\tp{B}\colon\Gamma(\cS_\fr,\Lambda)_\fm^{\oplus 6}\xrightarrow{\sim}
\Gamma(\cS_\fr,\Lambda)_\fm^{\oplus 6}]
\end{align}
is an isomorphism. As $p$ does not divide
\begin{align*}
Q_1(\ell)+Q_2(\ell)+(\ell^3+1)Q_0(\ell)&=(\ell-1)^{16}(\ell^2+\ell+1)^{15}(\ell^2-\ell+1)(\ell^3+1),\\
Q_1(t)-2Q_2(t)-\frac{t^3+1}{2}Q_0(t)&=-\frac{1}{2}(\ell+1)^2(\ell-1)^{14}(\ell^2+\ell+1)^{16}(\ell^3+1)
\end{align*}
by Definition \ref{de:cubic_level_raising} (C3), we know from linear algebra and \eqref{eq:corank} that the assignment \eqref{eq:corank_cubic} induces an isomorphism $\coker\tilde\Delta_\fm\simeq\Gamma(\cS_\fr,\Lambda)_\fm^{\oplus 3}$.
\end{proof}

%\begin{lem}
%Let $R$ be a ring. Let $\phi\colon R^{\oplus 6}\to R^{\oplus 6}$ be a homomorphism defined by the matrix
%\[\left(
%  \begin{array}{cccccc}
%    a & b & -b & -a & -b & b \\
%    b & a & b & -b & -a & -b \\
%    -b & b & a & b & -b & -a \\
%    -a & -b & b & a & b & -b \\
%    -b & -a & -b & b & a & b \\
%    b & -b & -a & -b & b & a \\
%  \end{array}
%\right).\]
%Then the image of $\phi$ is contained in the submodule $\{(e_1,\dots,e_6)\in R^{\oplus 6}\res e_1+e_4=e_2+e_5=e_3+e_6=0\}$; and they are equal %if and only if $(a+b)(a-2b)$ in invertible in $R$.
%\end{lem}

%\begin{proof}
%The first statement is obvious. For the second, we know that the image is exactly $\{(e_1,\dots,e_6)\in R^{\oplus 6}\res %e_1+e_4=e_2+e_5=e_3+e_6=0\}$ if and only if the matrix
%\[\left(
%  \begin{array}{ccc}
%    a & b & -b \\
%    b & a & b \\
%    -b & b & a \\
%  \end{array}
%\right)\]
%is invertible. Since its determinant equals $(a+b)^2(a-2b)$, the lemma follows.
%\end{proof}

\begin{corollary}\label{co:corank_cubic}
We have a canonical isomorphism
\[(\Gamma(\cS_\fr,\Lambda)/\Ker\phi^{\fr\fl}_\rho)^{\oplus 3}
\xrightarrow{\sim}\rH^1_\sing(K,\rH^3(\cX_{\fr_0,\fr_1}\otimes\overline\dQ,\Lambda(2))/\Ker\phi^{\fr\fl}_\rho).\]
\end{corollary}

\begin{proof}
By Proposition \ref{pr:featuring} and Theorem \ref{th:galois_cohomology_localized}, we have a canonical isomorphism
\[\coker\tilde\Delta_\fm\xrightarrow{\sim}\rH^1_\sing(K,\rH^3(\cX_{\fr_0,\fr_1}\otimes\overline\dQ,\Lambda(2))_\fm).\]
By Lemma \ref{le:singular} and Proposition \ref{pr:degenerate}, we know that the canonical map
\[\rH^1_\sing(K,\rH^3(\cX_{\fr_0,\fr_1}\otimes\overline\dQ,\Lambda(2))_\fm)/\Ker\phi^{\fr\fl}_\rho
\to\rH^1_\sing(K,\rH^3(\cX_{\fr_0,\fr_1}\otimes\overline\dQ,\Lambda(2))/\Ker\phi^{\fr\fl}_\rho)\]
is an isomorphism. Thus the corollary follows from Proposition \ref{pr:corank_cubic}.
\end{proof}

To ease notation, we put $\rM\coloneqq\rH^3(\cX_{\fr_0,\fr_1}\otimes\overline\dQ,\Lambda(2))/\Ker\phi^{\fr\fl}_\rho$ as a $\Lambda[\rG_\dQ]$-module, and $\bar\rM\coloneqq\rM\otimes\dF_p=\rH^3(\cX_{\fr_0,\fr_1}\otimes\overline\dQ,\dF_p(2))/\fm$.

\begin{proposition}\label{pr:dimension}
As an $\dF_p[\rG_{\TF}]$-module, the semisimplification $\bar\rM^{\r{ss}}$ is isomorphic to $\bar\rN_\rho^\sharp(2)^{\oplus|\fD(\fr,\fr_\rho)|}$ (Notation \ref{no:reduction}). In particular, $\dim_{\dF_p}\bar\rM=8|\fD(\fr,\fr_\rho)|$.
\end{proposition}

\begin{proof}
\emph{Step 1.} By the Eichler--Shimura relation in \cite{Wed00}, the Chebotarev density theorem, and \cite{Dim05}*{Lemma 6.5}, we have $\bar\rM^{\r{ss}}\simeq\bar\rN_\rho^\sharp(2)^{\oplus n}$ as an $\dF_p[\rG_{\TF}]$-module for some integer $n\geq 0$. Fix an embedding $\TF\hookrightarrow K(=\dQ_{\ell^6})$.

\emph{Step 2.} We claim that $n=\dim_{\dF_p}(\bar\rM(1))_{\rG_K}=\dim_{\dF_p}\dE^{-3,6}_2/\fm$. In fact, by Propositions \ref{pr:featuring} (2) and \ref{pr:degenerate}, we have a filtration $\sF^3\bar\rM\subset\sF^1\bar\rM\subset \sF^{-1}\bar\rM\subset\bar\rM$ such that
\begin{enumerate}[label=(\alph*)]
  \item $(T-1)\bar\rM\subset\sF^{-1}\bar\rM$ where we recall that $T$ is a topological generator of $\rI_K/\rP_K$;

  \item $(\bar\rM(1))_{\rG_K}=((\bar\rM/(T-1)\bar\rM)(1))_{\rG_\kappa}\simeq(\bar\rM/\sF^{-1}\bar\rM)(1)\simeq\dE^{-3,6}_2(1)/\fm$;

  \item as an $\dF_p[\rG_\kappa]$-module, $\sF^{-1}\bar\rM/(T-1)\bar\rM$ is isomorphic to the direct sum of finitely many copies of $\dF_p$, $\dF_p(1)$ and $\dF_p(2)$.
\end{enumerate}
Take a filtration $0=\bar\rM_0\subset\bar\rM_1\subset\cdots\subset\bar\rM_n=\bar\rM$ of $\dF_p[\rG_{\TF}]$-modules such that $\bar\rM_i/\bar\rM_{i-1}\simeq\rN_\rho^\sharp(2)$ for $1\leq i\leq n$. We prove by induction that
\begin{align}\label{eq:dimension1}
\dim_{\dF_p}\frac{\bar\rM_i}{\sF^{-1}\bar\rM\cap\bar\rM_i}=i.
\end{align}
For each $1\leq i\leq n$, we have the exact sequence of $\dF_p[\rG_K]$-modules
\[0\to\frac{\bar\rM_{i-1}}{\sF^{-1}\bar\rM\cap\bar\rM_{i-1}}\to\frac{\bar\rM_i}{\sF^{-1}\bar\rM\cap\bar\rM_i}
\to\frac{\bar\rM_i/\bar\rM_{i-1}}{(\sF^{-1}\bar\rM\cap\bar\rM_i)/(\sF^{-1}\bar\rM\cap\bar\rM_{i-1})}\to0.\]
Note that as an $\dF_p[\rG_K]$-module, $\bar\rM_i/\bar\rM_{i-1}$ is canonically isomorphic to $\dF_p(-1)\oplus\dF_p^{\oplus 3}\oplus\dF_p(1)^{\oplus 3}\oplus\dF_p(2)$. On one hand, since $\frac{\bar\rM_i}{\sF^{-1}\bar\rM\cap\bar\rM_i}\subset\frac{\bar\rM}{\sF^{-1}\bar\rM}$, which is isomorphic to the direct sum of finitely many copies of $\dF_p(-1)$ by (b), we know that $\dim_{\dF_p}\frac{\bar\rM_i/\bar\rM_{i-1}}{(\sF^{-1}\bar\rM\cap\bar\rM_i)/(\sF^{-1}\bar\rM\cap\bar\rM_{i-1})}\leq 1$. On the other hand, $\frac{\sF^{-1}\bar\rM\cap\bar\rM_i}{(T-1)\bar\rM\cap\bar\rM_i}$ is isomorphic to the direct sum of finitely many copies of $\dF_p$, $\dF_p(1)$ and $\dF_p(2)$ by (c), we know that $(\sF^{-1}\bar\rM\cap\bar\rM_i)/(\sF^{-1}\bar\rM\cap\bar\rM_{i-1})\cap\Lambda(-1)=0$ inside $\bar\rM_i/\bar\rM_{i-1}$, hence $\dim_{\dF_p}\frac{\bar\rM_i/\bar\rM_{i-1}}{(\sF^{-1}\bar\rM\cap\bar\rM_i)/(\sF^{-1}\bar\rM\cap\bar\rM_{i-1})}\geq 1$. Therefore, $\dim_{\dF_p}\frac{\bar\rM_i/\bar\rM_{i-1}}{(\sF^{-1}\bar\rM\cap\bar\rM_i)/(\sF^{-1}\bar\rM\cap\bar\rM_{i-1})}=1$ hence the claim follows by (b).

\emph{Step 3.} We compute $\dE^{-3,6}_2(1)$. By Lemma \ref{le:spectral_sequence_cubic}, we have
\[\dE^{-3,6}_2(1)=\Ker[\rH^0(Y_{\fr_0,\fr_1,\overline\kappa}^{(3)},\Lambda)_\fm\xrightarrow{\rd^{-3,6}_1}
\rH^1(Y_{\fr_0,\fr_1,\overline\kappa}^{(2)},\Lambda(1))_\fm].\]
From \eqref{beq:reduction}, we have $Y_{\fr_0,\fr_1}^{(3)}=\coprod_{i=0}^5 Y_{\fr_0,\fr_1}^+\cap Y_{\fr_0,\fr_1}^-\cap Y_{\fr_0,\fr_1}^{i(i+1)(i+2)}\cap Y_{\fr_0,\fr_1}^{i(i+1)(i+5)}$. By Proposition \ref{bpr:blowup} (4), the restricted morphism
\[\pi\colon Y_{\fr_0,\fr_1}^+\cap Y_{\fr_0,\fr_1}^-\cap Y_{\fr_0,\fr_1}^{i(i+1)(i+2)}\cap Y_{\fr_0,\fr_1}^{i(i+1)(i+5)}
\to X_{\fr_0,\fr_1}^+\cap X_{\fr_0,\fr_1}^-\cap X_{\fr_0,\fr_1}^{i(i+1)(i+2)}\cap X_{\fr_0,\fr_1}^{i(i+1)(i+5)}=X_{\fr_0,\fr_1}^{123456}\]
is an isomorphism. Thus we have $\rH^0(Y_{\fr_0,\fr_1,\overline\kappa}^{(3)},\Lambda)_\fm\simeq\bigoplus_{i=0}^5\Gamma(X_{\fr_0,\fr_1}^{012345}(\overline\kappa),\Lambda)_\fm$ canonically. It is straightforward to check that an element $(f_0,\dots,f_5)$ belongs to the kernel of $\rd^{-3,6}_1$ if and only if $f_0=f_1=\cdots=f_5$ and $\delta^+_*f_0=\delta^-_*f_0=0$, where we recall Notation \ref{no:tilde}. Thus if we put $\Gamma(X_{\fr_0,\fr_1}^{012345}(\overline\kappa),\Lambda)_\fm^\heartsuit=
\Ker\delta^+_*\cap\Ker\delta^-_*\subset\Gamma(X_{\fr_0,\fr_1}^{012345}(\overline\kappa),\Lambda)_\fm$, then $\dE^{-3,6}_2(1)\simeq\Gamma(X_{\fr_0,\fr_1}^{012345}(\overline\kappa),\Lambda)_\fm^\heartsuit$.

\emph{Step 4.}  We compute $\dim_{\dF_p}\Gamma(X_{\fr_0,\fr_1}^{012345}(\overline\kappa),\Lambda)_\fm^\heartsuit/\fm$. We use the (much easier) level raising for Shimura curves. Let $\tau_0,\tau_1,\tau_2\colon F\hookrightarrow\TF$ be the three distinct embeddings corresponding to $0,1,2$ in $\Phi_F$, respectively. Choose an $O_F$-Eichler order $\cR$ of discriminant $\{\fl\}\cup\nabla\setminus\{\tau_0\}$ and level $\fr$ together with a surjective $O_F$-linear homomorphism $\cR\to O_F/\fr_1$. Let $\fH_\cR$ be the symmetric hermitian domain attached to $\cR\otimes\dQ$, which is isomorphic to $\dC\setminus\dR$. The complex Shimura curve
$\cR_\dQ^\times\backslash\fH_\cR\times \widehat\cR_\dQ^\times/\Ker[\widehat\cR^\times\to(O_F/\fr_1)^\times]$ has a canonical projective model $\cC$ over $\Spec O_{\TF}[\fr^{-1}]$. It has an action by the Hecke monoid $\dT^{\fr\fl}$. By \v{C}erednik--Drinfeld uniformization \cite{Cer76}, we know that
\begin{itemize}
  \item $\cC\otimes_{O_{\TF}[\fr^{-1}]}O_K$ is a strictly semistable $O_K$-scheme of relative dimension $1$.

  \item If we put $C=\cC\otimes_{O_{\TF}[\fr^{-1}]}\kappa$, then $C^{(0)}=C^+\coprod C^-$ and $C^{(1)}=C^+\cap C^-$.

  \item We have $C_{\overline\kappa}^\pm\simeq\coprod_{\pi_0(C_{\overline\kappa}^\pm)}\dP^1_{\overline\kappa}$, and $C_{\overline\kappa}^{(1)}\simeq\coprod_{\pi_0(C_{\overline\kappa}^{(1)})}\Spec\overline\kappa$.

  \item There are (non-canonical) isomorphisms $\pi_0(C_{\overline\kappa}^\pm)\simeq Z_{\fr_0,\fr_1}^\pm(\overline\kappa)$ and $\pi_0(C_{\overline\kappa}^{(1)})\simeq X_{\fr_0,\fr_1}^{012345}(\overline\kappa)$ that are compatible under Hecke actions, which we fix.
\end{itemize}
By a much easier discussion toward Corollary \ref{co:corank_cubic} for the curve $\cC$ (and $\Lambda=\dF_p$), we obtain an isomorphism
$\rH^1_\sing(K,\rH^1(\cC\otimes_{O_{\TF}[\fr^{-1}]}\overline\dQ,\dF_p(1))/\fm)\xrightarrow{\sim}\Gamma(\cS_\fr,\dF_p)/\fm$. In particular, we have $\dim_{\dF_p}\rH^1_\sing(K,\rH^1(\cC\otimes_{O_{\TF}[\fr^{-1}]}\overline\dQ,\dF_p(1))/\fm)=|\fD(\fr,\fr_\rho)|$ by Definition \ref{de:cubic_level_raising} (C2). However, by the Eichler--Shimura relation (for $\cC$), the Chebotarev density theorem, and \cite{BLR}*{Theorem 1}, we know that as an $\dF_p[\rG_{\TF}]$-module, $\rH^1(\cC\otimes_{O_{\TF}[\fr^{-1}]}\overline\dQ,\dF_p(1))/\fm$ is isomorphic, not just up to semisimplification, to the direct sum of $n$ copies of $\bar\rN_\rho(1)$ (Notation \ref{no:reduction}). Therefore, $n=|\fD(\fr,\fr_\rho)|$ as $\dim_{\dF_p}\rH^1_\sing(K,\bar\rN_\rho(1))=1$. On the other hand, by a similar argument in Steps 2 \& 3, we also have an isomorphism
$(\rH^1(\cC\otimes_{O_{\TF}[\fr^{-1}]}\overline\dQ,\dF_p(1))/\fm)_{\rG_K}
\simeq\Gamma(X_{\fr_0,\fr_1}^{012345}(\overline\kappa),\Lambda)_\fm^\heartsuit/\fm$. In particular, we have $\dim_{\dF_p}\Gamma(X_{\fr_0,\fr_1}^{012345}(\overline\kappa),\Lambda)_\fm^\heartsuit/\fm
=|\fD(\fr,\fr_\rho)|\dim_{\dF_p}(\bar\rN_\rho(1))_{\rG_K}=|\fD(\fr,\fr_\rho)|$. In other words, $n=\dim_{\dF_p}\dE^{-3,6}_2(1)/\fm=\dim_{\dF_p}\Gamma(X_{\fr_0,\fr_1}^{012345}(\overline\kappa),\Lambda)_\fm^\heartsuit/\fm=|\fD(\fr,\fr_\rho)|$. The proposition follows.
\end{proof}

We now introduce a new set of \'{e}tale correspondences of $\cX_{\fr_0,\fr_1}$. For every $\fd\in\fD(\fr,\fr_\rho)$, we have the following composite morphism
\begin{align}\label{eq:correspondence0}
\tilde\delta^\fd\colon\cX_{\fr_0,\fr_1}\to\cX_{\fr,O_F}\xrightarrow{\delta^\fd}\cX_{\fr_\rho,O_F}
\end{align}
(see Remark \ref{bre:hecke} and \eqref{beq:degeneracy}). It is a finite \'{e}tale morphism of Deligne--Mumford stacks. As usual, we put $\tilde\delta=\tilde\delta^{O_F}$. Form the following pullback square
\begin{align}\label{eq:correspondence1}
\xymatrix{
\cX_{\fr_0,\fr_1}^\fd \ar[r]^-{\varepsilon}\ar[d]_-{\varepsilon^\fd} &  \cX_{\fr_0,\fr_1} \ar[d]^-{\tilde\delta^\fd} \\
\cX_{\fr_0,\fr_1} \ar[r]^-{\tilde\delta} & \cX_{\fr_\rho,O_F}
}
\end{align}
which produces an \'{e}tale correspondence (Definition \ref{de:correspondence})
\begin{align}\label{eq:correspondence2}
\cX_{\fr_0,\fr_1}\xleftarrow{\varepsilon^\fd}\cX_{\fr_0,\fr_1}^\fd\xrightarrow{\varepsilon}\cX_{\fr_0,\fr_1}.
\end{align}
By base change along the morphism $\cX_{\fr_0,\fr_1}^\fd\to\cX_{\fr_\rho,O_F}$, we obtain an action of the Hecke monoid $\dT^{\fr\fl}$ on $\cX_{\fr_0,\fr_1}^\fd$ which is compatible with the above \'{e}tale correspondence. For an ample set $S\subset\{0,1,2,3,4,5\}$, we put $X_{\fr_0,\fr_1}^{\fd,S}=\varepsilon^{-1}X_{\fr_0,\fr_1}^S=(\varepsilon^\fd)^{-1}X_{\fr_0,\fr_1}^S$, which is a closed subscheme of $X_{\fr_0,\fr_1}^\fd\coloneqq\cX_{\fr_0,\fr_1}^\fd\otimes\kappa$. Similar to Definition \ref{bde:blowup}, we obtain a semistable resolution $\pi\colon\cY_{\fr_0,\fr_1}^\fd\to\cX_{\fr_0,\fr_1}^\fd\otimes O_K$. It is easy to see that \eqref{eq:correspondence2} lifts uniquely to an \'{e}tale correspondence
\begin{align*}\label{eq:correspondence3}
\cY_{\fr_0,\fr_1}\xleftarrow{\varepsilon^\fd}\cY_{\fr_0,\fr_1}^\fd\xrightarrow{\varepsilon}\cY_{\fr_0,\fr_1}.
\end{align*}

\begin{lem}\label{le:degenerate}
For every $\fd\in\fD(\fr,\fr_\rho)$, the following diagram commutes
\[\xymatrix{
(\Gamma(\cS_\fr,\Lambda)/\Ker\phi^{\fr\fl}_\rho)^{\oplus 3} \ar[d]_-{|\Cl(F)_{\fr_1}|\cdot\delta^*\circ\delta^\fd_*}\ar[r]^-\simeq
& \rH^1_\sing(K,\rM) \ar[d]^-{\varepsilon^\fd_*\circ\varepsilon^*} \\
(\Gamma(\cS_\fr,\Lambda)/\Ker\phi^{\fr\fl}_\rho)^{\oplus 3} \ar[r]^-\simeq &
\rH^1_\sing(K,\rM)
}\]
where the horizontal isomorphisms are the one from Corollary \ref{co:corank_cubic}.
\end{lem}

\begin{proof}
For an ideal $\fs$ of $O_F$ that is coprime to $\nabla$ and $\fl$, we define a moduli functor $Z_{\fs,O_F}^{S^\pm}\coloneqq Z(\{\fl\}\cup\nabla\setminus\Phi_F)_{\fs,O_F}^{S^\pm}$ valued in groupoids in the same way as Definition \ref{bde:hilbert_base}. It is a smooth Deligne--Mumford stack over $\Spec\dF_{\ell^6}$. For every $\fd\in\fD(\fr,\fr_\rho)$, we have a finite \'{e}tale morphism
\[\tilde\delta^\fd\colon Z_{\fr_0,\fr_1}^\pm\to Z_{\fr,O_F}^{S^\pm}\xrightarrow{\delta^\fd}Z_{\fr_\rho,O_F}^{S^\pm}\]
similar to \eqref{eq:correspondence0}. Like \eqref{eq:correspondence1}, we form the following pullback square
\begin{align*}\label{eq:correspondence4}
\xymatrix{
Z_{\fr_0,\fr_1}^{\fd,\pm} \ar[r]^-{\varepsilon}\ar[d]_-{\varepsilon^\fd} &  Z_{\fr_0,\fr_1}^\pm \ar[d]^-{\tilde\delta^\fd} \\
Z_{\fr_0,\fr_1}^\pm \ar[r]^-{\tilde\delta} & Z_{\fr_\rho,O_F}^{S^\pm}.
}
\end{align*}
Similar to $\cY_{\fr_0,\fr_1}$, we have exceptional divisors $E^i_{g'}$ on $Y_{\fr_0,\fr_1,\overline\kappa}^{\fd,i(i+1)(i+2)}$ for $i\in S^\pm$ and $g'\in Z_{\fr_0,\fr_1}^{\fd,\mp}(\overline\kappa)$. The morphism $\varepsilon^\fd$ induces an isomorphism from $E^i_{g'}$ to its image which is $E^i_{\varepsilon^\fd(g')}$.

Note that the isomorphism in Corollary \eqref{co:corank_cubic} comes from a surjection
\[\big(\bigoplus_{i\in S^+}\Gamma(Z_{\fr_0,\fr_1}^-(\overline\kappa),\Lambda)/\Ker\phi^{\fr\fl}_\rho\big)
\bigoplus\big(\bigoplus_{i\in S^-}\Gamma(Z_{\fr_0,\fr_1}^+(\overline\kappa),\Lambda)/\Ker\phi^{\fr\fl}_\rho\big)\to\rH^1_\sing(K,\rM)\]
via the basis $\{\sfE^0_{g'},\sfE^1_{g'},\sfE^2_{g'},\sfE^3_{g'},\sfE^4_{g'},\sfE^5_{g'}\}$. Therefore, the lemma amounts to that the diagram
\begin{align*}
\xymatrix{
\Gamma(Z_{\fr_0,\fr_1}^\pm(\overline\kappa),\Lambda)/\Ker\phi^{\fr\fl}_\rho \ar[r]^-{\psi^\pm_*}\ar[d]_-{\varepsilon^\fd_*\circ\varepsilon^*}
&  \Gamma(\rS_\fr,\Lambda)/\Ker\phi^{\fr\fl}_\rho \ar[d]^-{|\Cl(F)_{\fr_1}|\cdot\delta^*\circ\delta^\fd_*} \\
\Gamma(Z_{\fr_0,\fr_1}^\pm(\overline\kappa),\Lambda)/\Ker\phi^{\fr\fl}_\rho \ar[r]^-{\psi^\pm_*}
&  \Gamma(\rS_\fr,\Lambda)/\Ker\phi^{\fr\fl}_\rho }
\end{align*}
commutes. This follows form the elementary fact that $\varepsilon^\fd_*\circ\varepsilon^*=\tilde\delta^*\circ\tilde\delta^\fd_*\colon
\Gamma(Z_{\fr_0,\fr_1}^\pm(\overline\kappa),\Lambda)\to\Gamma(Z_{\fr_0,\fr_1}^\pm(\overline\kappa),\Lambda)$, and the identity $\psi^\pm_*\circ\tilde\delta^*\circ\tilde\delta^\fd_*=|\Cl(F)_{\fr_1}|\cdot\delta^*\circ\delta^\fd_*\circ\psi^\pm_*$. The lemma follows.
\end{proof}

Now we are ready to finish the proof of Theorem \ref{th:cubic_level_raising}.

\begin{proof}[Proof of Theorem \ref{th:cubic_level_raising} (2,3)]
Put $\rM_0\coloneqq\rH^3(\cX_{\fr_\rho,O_F}\otimes\overline\dQ,\Lambda(2))/\Ker\phi^{\fr\fl}_\rho$ as a $\Lambda[\rG_\dQ]$-module, and $\bar\rM_0\coloneqq\rM_0\otimes\dF_p$.

\emph{Step 1.} For every $\fd\in\fD(\fr,\fr_\rho)$, the map $\tilde\delta^\fd_*\colon\rM\to\rM_0$ is surjective. In fact, $\tilde\delta^\fd_*\circ\tilde\delta^{\fd*}=\deg(\tilde\delta^\fd)=\mu(\fr,\fr_\rho)\cdot|\Cl(F)_{\fr_1}|$ which is invertible in $\Lambda$.

\emph{Step 2.} We claim that the $\dF_p[\rG_{\TF}]$-module $\bar\rM_0$ is isomorphic to $\bar\rN_\rho^\sharp(2)$, and the canonical map
\begin{align}\label{eq:raising1}
\sum_{\fd\in\fD(\fr,\fr_\rho)}\tilde\delta^\fd_*\colon\bar\rM\to\bigoplus_{\fd\in\fD(\fr,\fr_\rho)}\bar\rM_0
\end{align}
is an isomorphism. By Proposition \ref{pr:dimension}, we know that $\bar\rM_0^{\r{ss}}\simeq\bar\rN_\rho^\sharp(2)^{\oplus n_0}$ for some $n_0\geq 0$. Let $\bar\rM_1$ be the kernel of \eqref{eq:raising1}. Then $\bar\rM_1^{\r{ss}}\simeq(\bar\rN_\rho(2))^{\oplus n_1}$ and $(\bar\rM/\bar\rM_1)^{\r{ss}}\simeq(\bar\rN_\rho(2))^{\oplus|\fD(\fr,\fr_\rho)|-n_1}$ as $\dF_p[\rG_{\TF}]$-modules. By Lemma \ref{le:degenerate} and Definition \ref{de:cubic_level_raising} (C2), we know that the map
\begin{align}\label{eq:raising2}
\sum_{\fd\in\fD(\fr,\fr_\rho)}\tilde\delta^*\circ\tilde\delta^\fd_*=\sum_{\fd\in\fD(\fr,\fr_\rho)}\varepsilon^\fd_*\circ\varepsilon^*
\colon\rH^1_\sing(K,\bar\rM)\to\bigoplus_{\fd\in\fD(\fr,\fr_\rho)}\rH^1_\sing(K,\bar\rM)
\end{align}
has trivial kernel. By Corollary \ref{co:corank_cubic}, we know that $\dim_{\dF_p}\rH^1_\sing(K,\bar\rM)=3|\fD(\fr,\fr_\rho)|$. On the other hand, as a direct consequence of Lemma \ref{le:singular} and Proposition \ref{pr:degenerate}, we have $\dim_{\dF_p}\rH^1_\sing(K,\bar\rM/\bar\rM_1)\leq\dim_{\dF_p}\rH^1_\sing(K,(\bar\rM/\bar\rM_1)^{\r{ss}})$. Since $\dim_{\dF_p}\rH^1_\sing(K,\bar\rN_\rho(2))=3$ by Lemma \ref{le:sharp}, we have $\dim_{\dF_p}\rH^1_\sing(K,\bar\rM/\bar\rM_1)\leq 3(|\fD(\fr,\fr_\rho)|-n_1)$. This implies that $n_1=0$ hence \eqref{eq:raising1} is injective. In particular, $\rM_0$ is nontrivial; $\bar\rM$ is semisimple hence $\bar\rM_0$ is semisimple as well by Step 1. It remains to show that $n_0=1$. However, we know that for every $\fd\in\fD(\fr,\fr_\rho)$, the kernel of the map $\tilde\delta^*\circ\tilde\delta^\fd_*\colon\rH^1_\sing(K,\bar\rM)\to\rH^1_\sing(K,\bar\rM)$ has dimension $3(|\fD(\fr,\fr_\rho)|-1)$. Therefore, $n_0$ has to be $1$.

\emph{Step 3.} We claim that $\rM_0\simeq\rN_\rho^\sharp(2)$ as $\Lambda[\rG_{\TF}]$-modules, and the canonical map
\begin{align}\label{eq:raising3}
\sum_{\fd\in\fD(\fr,\fr_\rho)}\tilde\delta^\fd_*\colon\rM\to\bigoplus_{\fd\in\fD(\fr,\fr_\rho)}\rM_0
\end{align}
is an isomorphism. Note that the map \eqref{eq:raising2} also has trivial kernel if we replace $\bar\rM$ by $\rM$. In particular, the exponent of $\rM_0$ is $p^\nu$. Since $\bar\rho$ is generic (Definition \ref{de:perfect_prime}), $\bar\rM_0$ is an absolutely irreducible representation of $\rG_{\TF}$. By the Eichler--Shimura relation, the Chebotarev density theorem, and a result of Mazur and Carayol (see \cite{Kis}*{Theorem 1.4.1}), we have $\rM_0\simeq\rN_\rho^\sharp(2)$. By Step 2 and the Nakayama lemma, we know that \eqref{eq:raising3} is surjective. Since the length of $\rM$ is at most $8\nu|\fD(\fr,\fr_\rho)|$ by Proposition \ref{pr:dimension}, \eqref{eq:raising3} has to be an isomorphism.

\emph{Step 4.} Let $\varphi\colon\rM_0\to\rN_\rho^\sharp(2)$ be the isomorphism of $\Lambda[\rG_{\TF}]$-modules. We show that $\varphi$ is in fact an isomorphism of $\Lambda[\rG_{F_0}]$-modules. For every element $\gamma\in\Gal(\TF/\dQ)$, choose a lifting $\tilde\gamma\in\rG_\dQ$.
The composition $\tilde\gamma^{-1}\varphi^{-1}\tilde\gamma\varphi$ does not depend on the choice of $\tilde\gamma$, and is a $\Lambda[\rG_{\TF}]$-linear automorphism of $\rM_0$. Since $\rM_0$ has no nontrivial $\Lambda[\rG_{\TF}]$-submodule, $\tilde\gamma^{-1}\varphi^{-1}\tilde\gamma\varphi\in\Lambda^\times$. In other words, we obtain a character $\chi\colon\Gal(\TF/\dQ)\to\Lambda^\times$. It suffices to show that $\chi\res_{\Gal(\TF/F_0)}$ is trivial. The group $\Gal(\TF/F_0)$ is generated by an absolute $\ell$-Frobenius morphism $\sigma$. Since $\sigma^2$ acts trivially on $(\rN_\rho^\sharp(3))_{\rG_\kappa}$, it suffices to show that $\sigma^2$ acts trivially on $(\rM(1))_{\rG_K}$ as well. By Propositions \ref{pr:featuring} (2) and \ref{pr:degenerate}, we have a canonical isomorphism $(\rM(1))_{\rG_K}\simeq\dE^{-3,6}(1)/\Ker\phi^{\fr\fl}_\rho$. In the proof of Proposition \ref{pr:dimension}, we already showed that $\dE^{-3,6}(1)$ is contained in the diagonal submodule of $\Gamma(X_{\fr_0,\fr_1}^{012345}(\overline\kappa),\Lambda)_\fm$. By Proposition \ref{bpr:hilbert_sparse} (2) and Remark \ref{bre:frobenius}, we know that $\sigma^2$ acts on $\Gamma(X_{\fr_0,\fr_1}^{012345}(\overline\kappa),\Lambda)_\fm$ by sending $(f_0,f_1,f_2,f_3,f_4,f_5)$ to $(f_2,f_3,f_4,f_5,f_0,f_1)$. In particular, $\sigma^2$ acts trivially on $\dE^{-3,6}(1)$ hence $(\rM(1))_{\rG_K}$. At this moment, Theorem \ref{th:cubic_level_raising} (2) has been proved.

\emph{Step 5.} Again by Proposition \ref{bpr:hilbert_sparse} (1) and Remark \ref{bre:frobenius}, the $\ell$-Frobenius action on $\Hom(\cB(Y_{\fr_0,\fr_1},\Lambda),\Lambda)_\fm$ induces the translation on the coordinate $\boldsymbol{e}\in\Gamma(\cS_\fr,\Lambda)_\fm^{\oplus 6}$. In particular, by taking $\sigma$-invariants in the isomorphism of Corollary \ref{co:corank_cubic}, we obtain a canonical isomorphism
$\Gamma(\cS_\fr,\Lambda)/\Ker\phi^{\fr\fl}_\rho
\xrightarrow{\sim}\rH^1_\sing(\dQ_\ell,\rH^3(\cX_{\fr_0,\fr_1}\otimes\overline\dQ,\Lambda(2))/\Ker\phi^{\fr\fl}_\rho)$.
Part (3) then follows from the isomorphism $\Gamma(\cS_\fr,\Lambda)/\Ker\phi^{\fr\fl}_\rho\simeq\Gamma(\cS_\fr,\Lambda)/\Ker\phi^\fl_\rho$
ensured by Definition \ref{de:cubic_level_raising} (C2).

Theorem \ref{th:cubic_level_raising} has been fully proved.
\end{proof}

\section{Bounding Selmer groups}
\label{ss:4}

In this chapter, we fix an isomorphism $O_F/\fp\simeq\dF_{\Nm\fp}$ for every prime $\fp$ of $F$.

\subsection{Hirzebruch--Zagier cycle and a reciprocity law}
\label{ss:reciprocity}

We fix a finite set of even cardinality $\nabla^\flat$ of places of $\dQ$ containing $\infty$. Let $\nabla$ (resp.\ $\nabla'$, $\nabla''$) be the set of places of $F$ above $\nabla^\flat$ that is of degree either $1$ or $3$ (resp.\ unramified of degree $2$, ramified of degree $2$). Then $\nabla$ is a finite set of even cardinality of places of $F$ containing $\Phi_F$. Write $\fr_\rho=\fr'_\rho\fr''_\rho\fr'''_\rho$ where $\fr'_\rho$ (resp.\ $\fr''_\rho$, $\fr'''_\rho$) is the product of primes in $\nabla'$ (resp.\ in $\nabla''$, not in $\nabla'\cup\nabla''$).

\begin{assumption}\label{as:quadruple}
We assume that the perfect quadruple $(\rho,\fr_\rho,\fr_0,\fr_1)$ from \Sec\ref{ss:level_raising} satisfies: (1) $\fr'_\rho$ is square-free; (2) $\fr''_\rho=O_F$; (3) $\fr_0=\fr'_\rho$; and (4) $\fr_1$ is coprime to $\nabla'\cup\nabla''$.
\end{assumption}

Let $\ell$ be a cubic-level raising prime for $(\rho,\fr_\rho,\fr_0,\fr_1)$. Put
\[\cX(\ell)_{\fr_1}^\flat\coloneqq\cX(\{\ell\}\cup\nabla^\flat\setminus\{\infty\})_{\dZ,\fr_1\cap\dZ}\]
over $\Spec\dZ[\fr_1^{-1}]$ (Definition \ref{bde:hilbert_shimura}). Since $\fr_1\cap\dZ$ is an absolutely neat ideal of $\dZ$, $\cX(\ell)_{\fr_1}^\flat$ is a projective curve over $\Spec\dZ[\fr_1^{-1}]$, strictly semistable over $\Spec\dZ_{\ell^2}$ by Theorem \ref{bpr:zink}.

\begin{definition}\label{de:compatible}
Let $\cO^\flat$ be an oriented maximal order of discriminant $\{\ell\}\cup\nabla^\flat\setminus\{\infty\}$, with orientations $o_p\colon\cO^\flat\to\dF_{p^2}$. Let $\cO$ be an oriented $O_F$-maximal order of discriminant $\{\fl\}\cup\nabla\setminus\{\Phi_F\}$, with orientations $o_\fp\colon\cO\to\dF_{\Nm\fp^2}$. We say that an injective homomorphism $\cO^\flat\otimes O_F\hookrightarrow\cO$ of $O_F$-algebras is \emph{orientation compatible} if
\begin{enumerate}
  \item for every $\fp\in\{\fl\}\cup\nabla\setminus\{\Phi_F\}$, the composed map $\cO^\flat\to\cO\xrightarrow{o_\fp}\dF_{\Nm\fp^2}$ coincides with $o_p$ where $p$ is the prime underlying $\fp$;
  \item for every $\fp\in\nabla'$ above $p$, the action of $\cO^\flat$ on the one-dimensional $O_F/\fp$ (which has been identified with $\dF_{\Nm\fp}$) vector space $(\cO/O_F\otimes\cO^\flat)\otimes\dZ_p$ via left multiplication is given by the homomorphism $o_p$.
\end{enumerate}
\end{definition}

From now on, we fix an orientation compatible homomorphism $\cO^\flat\otimes O_F\hookrightarrow\cO$ of $O_F$-algebras\footnote{Such data is in fact unique up to isomorphism.}. Let $\Omega_{\fr_1}$ be the unique $O_F$-submodule of $O_F\otimes(\dZ/(\fr_1\cap\dZ))$ that is isomorphic to $O_F/\fr_1$.

Let $T$ be a scheme over $\Spec\dZ[(\fr\disc F)^{-1}]$. Let $(A,\iota_A,\lambda_A)$ be an element of $\cX(\ell)_{\fr_1}^\flat(T)$ where we recall that $A$ is an abelian surface over $T$; $\iota_A\colon\cO^\flat\hookrightarrow\End(A)$ is an injective homomorphism; and $\lambda_A\colon\underline{(\dZ/(\fr_1\cap\dZ))^{\oplus 2}}_{/T}\hookrightarrow A$ is an $\cO^\flat$-equivariant injective homomorphism. We let
\begin{itemize}
  \item $A'$ to be $\cO\otimes_{\cO^\flat}A$, which is an abelian scheme over $T$;
  \item $\iota_{A'}\colon\cO\hookrightarrow\cO\otimes_{\cO^\flat}\End(A)\subset\End(A')$ be the induced homomorphism;
  \item $\lambda_{A'}\colon\underline{\Omega_{\fr_1}^{\oplus 2}}_{/T}\subset\underline{O_F\otimes(\dZ/(\fr_1\cap\dZ))^{\oplus 2}}_{/T}
     \xrightarrow{O_F\otimes\lambda_A}O_F\otimes A\to\cO\otimes_{\cO^\flat}A$ the induced homomorphism, which is $\cO$-equivariant and injective;
  \item $C_{A'}=\prod_{\fp\mid\fr_0}e_\fp A'[\fp]$, where $e_\fp$ is a nontrivial idempotent in $(O_F\otimes\cO^\flat)/\fp\cO$, which is isomorphic to the upper triangular $O_F/\fp$-subalgebra of $\cO/\fp\cO\simeq\Mat_2(O_F/\fp)$ (since $\fp\in\nabla'$), such that $((O_F\otimes\cO^\flat)/\fp\cO)e_\fp=(O_F/\fp)e_\fp$. It is easy to see that the subgroup $e_\fp A'[\fp]$ does not depend on the choice of $e_\fp$, and is an $\cO$-stable flat subgroup of $A'[\fp]$ over $T$ such that at every geometric point, it is isomorphic to $(O_F/\fp)^{\oplus 2}$.
\end{itemize}
Therefore, once we fix an isomorphism $\Omega_{\fr_1}\simeq O_F/\fr_1$ of $O_F$-modules, the assignment $(A,\iota_A,\lambda_A)\mapsto(A',\iota_{A'},C_{A'},\lambda_{A'})$ defines a morphism
\begin{align}\label{eq:special}
\theta\colon\cX(\ell)_{\fr_1}^\flat\otimes\dZ[(\fr\disc F)^{-1}]\to\cX(\ell)_{\fr_0,\fr_1}
\end{align}
over $\Spec\dZ[(\fr\disc F)^{-1}]$. It is a finite morphism.

Let $\fs$ be either $\fr_1$ or $\fr_1\fl$. Put $\cS^\flat_\fs=\cS(\nabla^\flat)_{\fs\cap\dZ}$. We now construct a similar map
\begin{align}\label{eq:special2}
\vartheta\colon\cS^\flat_\fs\to\cS_{\fr_0\fs}
\end{align}
as follows. Let $(\cR,\cR_0,\{\epsilon_p\})$ \footnote{We use $\epsilon$ since $o$ has been used for $\cO^\flat$ and $\cO$.} be an element of $\cS^\flat_\fs$. We let
\begin{itemize}
  \item $\cR'$ be the unique minimal $O_F$-Eichler order in $F\otimes\cR_0$ containing $O_F\otimes\cR$; it is of discriminant $\nabla$ and level $\ft(\fs\cap\dZ)$, where $\ft=\prod_{\fp\in\nabla'}\fp$;

  \item $\cR'_0$ be the unique $O_F$-maximal order in $F\otimes\cR_0$ containing $O_F\otimes\cR_0$ such that for every $\fp\in\nabla'$ above $p$, the action of $\cR$ on the one-dimensional $O_F/\fp$ (which has been identified with $\dF_{\Nm\fp}$) vector space $(\cR'_0/O_F\otimes\cR_0)\otimes\dZ_p$ via \emph{right} multiplication is given by the homomorphism $\epsilon_p$;

  \item $\epsilon'_\fp\colon\cR'\to\dF_{\Nm\fp^2}$ be the unique orientation at $\fp$, for every $\fp\in\nabla$ above $p$, such that $\epsilon'_\fp\res_\cR=\epsilon_p$, and $\epsilon'_\fp\res O_F$ induces the fixed isomorphism $O_F/\fp\xrightarrow{\sim}\dF_{\Nm\fp}\subset\dF_{\Nm\fp^2}$.
\end{itemize}
The assignment $(\cR,\cR_0,\{\epsilon_p\})$ to $(\cR',\cR'_0,\{\epsilon'_\fp\})$ defines a map $\vartheta'\colon\cS^\flat_\fs\to\cS_{\ft(\fs\cap\dZ)}$. Since $\fr_0\fs$ contains $\ft(\fs\cap\dZ)$, we have the degenerate map $\delta\colon\cS_{\ft(\fs\cap\dZ)}\to\cS_{\fr_0\fs}$ \eqref{aeq:degeneracy}. Put $\vartheta=\delta\circ\vartheta'$. It is clear that the following diagram
\begin{align}\label{eq:reciprocity0}
\xymatrix{
\cS^\flat_{\fr_1\fl} \ar[r]^-{\vartheta}\ar[d]_-{\delta^\fd} &  \cS_{\fr\fl} \ar[d]^-{\delta^\fd}  \\
\cS^\flat_{\fr_1} \ar[r]^-{\vartheta} &  \cS_\fr
}\end{align}
commutes, for $\fd=O_F,\fl$.

The following proposition reveals the relation between the morphism $\theta$ \eqref{eq:special} and the map $\vartheta$ \eqref{eq:special2}. Put $X_{\fr_1}^\flat=\cX(\ell)_{\fr_1}^\flat\otimes\dF_{\ell^6}$.

\begin{proposition}\label{pr:special}
The morphism $\theta$ sends $X^{\flat(1)}_{\fr_1}$ (recall the notation from \Sec\ref{ss:semistable_schemes}) into $X_{\fr_0,\fr_1}^{012345}$. Moreover, the following diagram
\[\xymatrix{
X^{\flat(1)}_{\fr_1}(\dF_{\ell^\infty}) \ar[r]^-{\theta}\ar[d]_-{\psi^\flat} & X_{\fr_0,\fr_1}^{012345}(\dF_{\ell^\infty})  \ar[d]^-{\psi} \\
\cS^\flat_{\fr_1\fl} \ar[r]^-{\vartheta} & \cS_{\fr\fl}
}\]
commutes, where $\psi^\flat$ is the quotient map induced from the canonical isomorphism $X_{\fr_1}^{\flat(1)}(\dF_{\ell^\infty})/\Cl(\dQ)_{\fr_1\cap\dZ}\simeq\cS^\flat_{\fr_1\fl}$ by Proposition \ref{bpr:hilbert_sparse}.
\end{proposition}

\begin{proof}
To prove the proposition, we may assume that $\fr_0=\ft=\prod_{\fp\in\nabla'}\fp$. Let $(A,\iota_A,\lambda_A)$ be an element in $X^{\flat(1)}_{\fr_1}(\dF_{\ell^\infty})$ whose image under $\theta$ is $(A',\iota_{A'},C_{A'},\lambda_{A'})$. Put $C'_{A'}=C_{A'}\times\IM(\lambda_{A'})$.

We choose a decomposition $\cO^\flat\otimes\dZ_\ell=\dZ_{\ell^2}\oplus\dZ_{\ell^2}\Pi$ with $\Pi^2=\ell$. Then $\Pi$ acts trivially on $\Lie(A)$. As $\Lie(A')$ is canonically isomorphic to $O_F\otimes\Lie(A)$ and $\Pi$ remains a uniformizer in $\cO\otimes\dZ_\ell$, $\Pi$ acts trivially on $\Lie(A')$ as well. In other words, $(A',\iota_{A'},C_{A'},\lambda_{A'})\in X_{\fr_0,\fr_1}^{012345}(\dF_{\ell^\infty})$. The first part follows.

Now we show the commutativity of the diagram. Suppose that the image of $(A,\iota_A,\lambda_A)$ (resp.\ $(A',\iota_{A'},C_{A'},\lambda_{A'})$) in $\cS^\flat_{\fr_1\fl}$ (resp.\ $\cS_{\fr\fl}$) is represented by $(\cR,\cR_0,\{\epsilon_p\})$ (resp.\ $(\cR',\cR'_0,\{\epsilon'_\fp\})$). By Proposition \ref{apr:ramified}, $\cR$ (resp.\ $\cR'$) is the subalgebra of $\End(A,\iota_A)$ (resp.\ $\End(A',\iota_{A'})$) preserving the subgroup $C_A$ (resp.\ $C'_{A'}$). However, since $\cR'$ contains $O_F\otimes\cR$ and has level $\ft\fr_1\fl$, we know that $\cR'$ must be the (unique) minimal $O_F$-Eichler order in $F\otimes\cR_0$ containing $O_F\otimes\cR$. Thus $\cR$ and $\cR'$ satisfy the relation in the definition of $\vartheta$.

To check that $\cR_0$ and $\cR'_0$ satisfy the relation in the definition of $\vartheta$, it suffices to show that for every $\fp\in\nabla'$ above $p$, the action of $\cR$ on the one-dimensional $O_F/\fp$-vector space $(\cR'_0/O_F\otimes\cR_0)\otimes\dZ_p$ via right multiplication is given by the homomorphism $\epsilon_p$. We first recall how $\epsilon_p$ is obtained (see the paragraphs after \cite{Rib89}*{Theorem 4.12}). Let $p_{\cO^\flat}$ be the kernel of $o_p\colon\cO^\flat\to\dF_{p^2}$. Then $\rT_pA/p_{\cO^\flat}\rT_pA$ is a one-dimensional $\dF_{p^2}$-vector space, and $\epsilon_p\colon\cR\to\dF_{p^2}$ is the homomorphism induced from the action of $\cR$ on $\rT_pA/p_{\cO^\flat}\rT_pA$. Note that $(\cR'_0/O_F\otimes\cR_0)\otimes\dZ_p$ is canonically isomorphic $\End(\rT_pA',\iota_{A'})/O_F\otimes\End(\rT_pA,\iota_A)$. Fixing an isomorphism $\rT_p(A)\simeq\cO^\flat_p\coloneqq\cO^\flat\otimes\dZ_p$ as a left $\cO^\flat_p$-module, we have $\cR_0\otimes\dZ_p=\cR\otimes\dZ_p\simeq(\cO^\flat_p)^{\r{op}}$ under which $\epsilon_p=(o_p)^{\r{op}}$, and $\cR'_0\otimes\dZ_p\simeq(\cO\otimes\dZ_p)^{\r{op}}$. Since the action of $(\cO^\flat_p)^{\r{op}}$ on $(\cO\otimes\dZ_p)^{\r{op}}/O_F\otimes(\cO^\flat_p)^{\r{op}}$ via right multiplication is given by the homomorphism $(o_p)^{\r{op}}$, we are done.

It is straightforward to check that $\{\epsilon_p\}$ and $\{\epsilon'_\fp\}$ satisfy the relation in the definition of $\vartheta$. We will leave details to the reader. The proposition is then proved.
\end{proof}

\begin{definition}[Hirzebruch--Zagier cycle]\label{de:hirzebruch}
We define the \emph{Hirzebruch--Zagier cycle} of $\cX(\ell)_{\fr_0,\fr_1}\otimes\dQ$ to be
\[\Theta^\ell_{\fr_0,\fr_1}\coloneqq\theta_*[\cX(\ell)^\flat_{\fr_1}\otimes\dQ]\in\CH^2(\cX(\ell)_{\fr_0,\fr_1}\otimes\dQ).\]
\end{definition}

By Theorem \ref{th:cubic_level_raising} (1), we have $\CH^2(\cX(\ell)_{\fr_0,\fr_1}\otimes\dQ)^0_\fm=\CH^2(\cX(\ell)_{\fr_0,\fr_1}\otimes\dQ)_\fm$. Similar to \eqref{eq:abel_jacobi}, we have the localized Abel--Jacobi map
\[\AJ_\fm\colon\CH^2(\cX(\ell)_{\fr_0,\fr_1}\otimes\dQ)\to\rH^1(\dQ,\rH^3(\cX(\ell)_{\fr_0,\fr_1}\otimes\overline\dQ,\dZ_p(2))_\fm).\]
Note that we have the localization map
\[\loc_\ell\colon\rH^1(\dQ,\rH^3(\cX(\ell)_{\fr_0,\fr_1}\otimes\overline\dQ,\dZ_p(2))_\fm)\to
\rH^1(\dQ_\ell,\rH^3(\cX(\ell)_{\fr_0,\fr_1}\otimes\overline\dQ,\dZ_p(2))_\fm)\]
and the quotient map
\[\partial\colon\rH^1(\dQ_\ell,\rH^3(\cX(\ell)_{\fr_0,\fr_1}\otimes\overline\dQ,\dZ_p(2))_\fm)\to
\rH^1_\sing(\dQ_\ell,\rH^3(\cX(\ell)_{\fr_0,\fr_1}\otimes\overline\dQ,\dZ_p(2))_\fm)\]
from \Sec\ref{ss:notation}. We have a bilinear pairing $(\;,\;)\colon\Gamma(\cS_\fr,\dZ)\times\Gamma(\cS_\fr,\dZ)\to\dZ$ such that $(\phi_1,\phi_2)=\sum_{g\in\cS_\fr}\phi_1\phi_2(g)$. It induces a pairing $(\;,\;)\colon\Gamma(\cS_\fr,\dZ)/\Ker\phi^\fr_\rho\times\Gamma(\cS_\fr,\Lambda)[\Ker\phi^\fr_\rho]\to\Lambda$. Now we can state the reciprocity law for the Hirzebruch--Zagier cycle.

\begin{theorem}\label{th:reciprocity}
Let $(\rho,\fr_\rho,\fr_0,\fr_1)$ be a perfect quadruple satisfying Assumption \ref{as:quadruple}, and $\ell$ a cubic-level raising prime. Then under the isomorphism in Theorem \ref{th:cubic_level_raising} (3), we have
\[(\partial\loc_\ell\AJ_\fm\Theta^\ell_{\fr_0,\fr_1},\phi)=(\ell+1)\cdot|(\dZ/\fr_1\cap\dZ)^\times|\cdot\sum_{\cS^\flat_{\fr_1}}\phi(\vartheta(x))\]
for every $\phi\in\Gamma(\cS_\fr,\Lambda)[\Ker\phi^\fr_\rho]$.
\end{theorem}

\begin{proof}
Similar to $X_{\fr_0,\fr_1}$, we have subschemes $X_{\fr_1}^{\flat\pm}$ of $X_{\fr_1}^\flat$ such that $X_{\fr_1}^{\flat(0)}=X_{\fr_1}^{\flat+}\coprod X_{\fr_1}^{\flat-}$, $X_{\fr_1}^{\flat+}\cap X_{\fr_1}^{\flat-}=X_{\fr_1}^{\flat(1)}$, and the image of $X_{\fr_1}^{\flat\pm}$ under $\theta$ is contained in $X_{\fr_0,\fr_1}^\pm$. Since $\cX_{\fr_1}^\flat\otimes\dZ_{\ell^6}$ is regular, by the universal property of blow-up, we have a unique morphism $\tilde\theta\colon\cX(\ell)_{\fr_1}^\flat\otimes\dZ_{\ell^6}\to\cY_{\fr_0,\fr_1}$ such that $\pi\circ\tilde\theta=\theta$, where we recall that $\pi\colon\cY_{\fr_0,\fr_1}\to\cX(\ell)_{\fr_0,\fr_1}\otimes\dZ_{\ell^6}$ is the semistable resolution used in \Sec\ref{ss:featuring}.

Let $S$ be an (ample) subset of $\{0,1,2,3,4,5\}$ containing $S^\pm$ with $|S|=4$. We claim that the scheme-theoretical intersection of the graph of the restricted morphism $\theta\colon X_{\fr_1}^{\flat\pm}\to X_{\fr_0,\fr_1}^\pm$ and $X_{\fr_1}^{\flat\pm}\times_{\Spec\dF_{\ell^6}}X_{\fr_0,\fr_1}^S$ inside $X_{\fr_1}^{\flat\pm}\times_{\Spec\dF_{\ell^6}}X_{\fr_0,\fr_1}^\pm$ is a reduced $0$-dimensional $\dF_{\ell^6}$-scheme. In fact, from the moduli interpretation, the intersection is contained in the graph of the restricted morphism $\theta\colon X_{\fr_1}^{\flat(1)}\to X_{\fr_0,\fr_1}^\pm$, which is a reduced $0$-dimensional $\dF_{\ell^6}$-scheme.

By the above claim, we know that the image of $X_{\fr_1}^{\flat\pm}$ under $\tilde\theta$ does not intersect with $Y_{\fr_0,\fr_1}^S$ if $S$ is a type but not sparse. Moreover, the intersection number of the graph of the restricted morphism $\tilde\theta\colon X_{\fr_1,\dF_{\ell^\infty}}^{\flat+}\to Y_{\fr_0,\fr_1,\dF_{\ell^\infty}}^+$ and $X_{\fr_1,\dF_{\ell^\infty}}^{\flat+}\times_{\Spec\dF_{\ell^\infty}}H_h$ inside $X_{\fr_1,\dF_{\ell^\infty}}^{\flat+}\times_{\Spec\dF_{\ell^\infty}}Y_{\fr_0,\fr_1,\dF_{\ell^\infty}}^+$ is equal to $1$ (resp.\ $0$) if $h$ is (resp.\ is not) in the image of $\theta\colon X^{\flat(1)}_{\fr_1}(\dF_{\ell^\infty})\to X_{\fr_0,\fr_1}^{012345}(\dF_{\ell^\infty})$.

For every type $S$, denote by $z$ the geometric cycle class of $Y_{\fr_0,\fr_1}^{(0)}\times_{Y_{\fr_0,\fr_1}}\tilde\theta_*X^\flat_{\fr_1}$ in $\rH^4(Y_{\fr_0,\fr_1,\dF_{\ell^\infty}}^{(0)},\Lambda(2))$. Let $z'\in\Hom(\cB(Y_{\fr_0,\fr_1}),\Lambda)$ be the image of $z$ under the map $\beta'$ \eqref{eq:beta}. It is clear that $z'$ belongs to $\Hom(\cB(Y_{\fr_0,\fr_1}),\Lambda)^{\rG_{\dF_{\ell^6}}}$, which is canonically isomorphic to $\Hom(\cA(Y_{\fr_0,\fr_1},\Lambda),\Lambda)$ by Proposition \ref{pr:featuring} (1). By Theorem \ref{th:relation_abel}, Proposition \ref{pr:relation_abel} and Corollary \ref{co:corank_cubic}, we know that $\partial\loc_\ell\AJ_\fm\Theta^\ell_{\fr_0,\fr_1}$ coincides with $z'$, regarded as in the cokernel of $\coker\tilde\Delta_\fm/\Ker\phi^{\fr\fl}_\rho$. Moreover, it is fixed by $\Gal(\dF_{\ell^6}/\dF_\ell)$. Now we regard $z'$ as an element in $(\coker\tilde\Delta_\fm/\Ker\phi^{\fr\fl}_\rho)^{\Gal(\dF_{\ell^6}/\dF_\ell)}\simeq\Gamma(\cS_\fr,\Lambda)/\Ker\phi^\fr_\rho$.

By the previous discussion and Proposition \ref{pr:corank_cubic}, we have $z'=(\psi^+_*\delta^+_*\theta_*\b{1}+\psi^-_*\delta^-_*\theta_*\b{1})/2$, where $\b{1}$ is the constant function on $X^{\flat(1)}_{\fr_1}(\dF_{\ell^\infty})$ with value $1$. In other words, we have
\begin{align}\label{eq:reciprocity1}
(\partial\loc_\ell\AJ_\fm\Theta^\ell_{\fr_0,\fr_1},\phi)=\frac{1}{2}(\b{1},\theta^*\delta^{+*}\psi^{+*}\phi+\theta^*\delta^{-*}\psi^{-*}\phi).
\end{align}
By Proposition \ref{pr:special}, we have the following commutative diagram
\[\xymatrix{
X^{\flat(1)}_{\fr_1}(\dF_{\ell^\infty}) \ar[r]^-{\theta}\ar[d]_-{\psi^\flat} & X_{\fr_0,\fr_1}^{012345}(\dF_{\ell^\infty})
\ar[d]^-{\psi}\ar[r]^-{\delta^\pm} & Z_{\fr_0,\fr_1}^\pm(\dF_{\ell^\infty})  \ar[d] \\
\cS^\flat_{\fr_1\fl} \ar[r]^-{\vartheta} & \cS_{\fr\fl} \ar[r]^-{\delta\text{ resp. }\delta^\fl} & \cS_\fr
}\]
which, together with \eqref{eq:reciprocity0}, implies that
\begin{align*}
\eqref{eq:reciprocity1}=\frac{1}{2}(\b{1},\psi^{\flat^*}\vartheta^*\delta^*\phi+\psi^{\flat^*}\vartheta^*\delta^{\fl*}\phi)
=\frac{1}{2}(\b{1},\psi^{\flat^*}\delta^*\vartheta^*\phi+\psi^{\flat^*}\delta^{\fl*}\vartheta^*\phi).
\end{align*}
The theorem follows as $\delta_*\psi^\flat_*\b{1}=\delta^\fl_*\psi^\flat_*\b{1}$ is equal to the constant function on $\cS^\flat_{\fr_1}$ with value $(\ell+1)\cdot|(\dZ/\fr_1\cap\dZ)^\times|$.
\end{proof}

\subsection{Bloch--Kato Selmer groups}
\label{ss:bloch}

Let $E$ be an elliptic curve over $\Spec F$. For every prime $p$, we have a Galois representation $\rho_{E,p}\colon\rG_F\to\GL(\rN_{E,p})$ where $\rN_{E,p}=\rH^1(E\otimes\overline{F},\dZ_p)$. Let $\rho_{E,p}^\sharp\colon\rG_\dQ\to\GL(\rN_{E,p}^\sharp)$ be the multiplicative induction, similar to Notation \ref{no:reduction}, where $\rN_{E,p}^\sharp=\rN_{E,p}^{\otimes 3}$. Then $(\rN_{E,p}^\sharp\otimes\dQ)(2)$ is the $p$-adic realization $\sfM(E)_p$ of the motive $\sfM(E)$ \eqref{eq:cubic_motive}.

\begin{lem}\label{le:tamely_pure}
The $\dQ_p$-representation $\sfM(E)_p$ of $\rG_\dQ$ is tamely pure (\cite{Liu}*{Definition 3.3}), that is, $\rH^1(\dQ_v,\sfM(E)_p)=0$ for all primes $v\nmid p$.
\end{lem}

\begin{proof}
This follows from the purity of the monodromy filtration on $\rN_{E,p}^\sharp\otimes\dQ$. To show the purity, we are allowed to make a base change to a finite field extension. Thus, it suffices to show the purity for the $\dQ_p[\rG_F]$-module $(\rN_{E,p}\otimes\dQ)^{\otimes 3}$. By \cite{Del80}*{(1.6.9)}, it suffices to consider $\rN_{E,p}\otimes\dQ$, whose purity is well-known.
\end{proof}

\begin{definition}[\cite{BK90} or \cite{Liu}*{\Sec 3.1}]
In terms of the above lemma, the \emph{Bloch--Kato Selmer group} $\rH^1_f(\dQ,\sfM(E)_p)$ of the representation $\sfM(E)_p$ is the subspace of classes $s\in\rH^1(\dQ,\sfM(E)_p)$ such that
\[\loc_p(s)\in\rH^1_f(\dQ_p,\sfM(E)_p)\coloneqq\Ker[\rH^1(\dQ_p,\sfM(E)_p)\to\rH^1(\dQ_p,\sfM(E)_p\otimes_{\dQ_p}\rB_{\r{cris}})].\]
\end{definition}

\subsection{Proof of Theorem \ref{th:main}}
\label{ss:proof}

From now on, we assume that $E$ is a modular elliptic curve over $F$ satisfying Assumption \ref{as:elliptic_curve}. In particular, we have an irreducible cuspidal automorphic representation $\Pi_E$ of $\Res_{F/\dQ}\GL_2(\dA)$ with trivial central character, and the set $\Sigma(\Pi_E,\tau)$ contains $\infty$. We assume that $L(1/2,\Pi_E,\tau)\neq 0$.

Put $\nabla^\flat=\Sigma(\Pi_E,\tau)$. Then we have sets $\nabla,\nabla',\nabla''$ as in \Sec\ref{ss:reciprocity}. Let $\fr_\rho$ be the conductor of $E$ away from $\nabla$. Write $\fr_\rho=\fr'_\rho\fr''_\rho\fr'''_\rho$ as in \Sec\ref{ss:reciprocity}.

\begin{lem}\label{le:dichotomy}
We have that $\fr'_\rho$ is square-free and $\fr''_\rho=O_F$.
\end{lem}

\begin{proof}
For a local field $K$, denote by $\omega_K$ the unique nontrivial unramified quadratic character of $K^\times$, and $\Sp_K$ the Steinberg representation of $\GL_2(K)$.

By Assumption \ref{as:elliptic_curve} (E1), we know that $\fr'_\rho$ is square-free. For the second part, let $v$ be a rational prime underlying $\nabla''$. Then $F\otimes_\dQ\dQ_v=\dQ_v\times L$ where $L$ is a \emph{ramified} quadratic extension of $\dQ_v$. By the epsilon dichotomy \cite{Pra92}*{Theorem D}, it suffices to show that the following four spaces $\Hom_{\GL_2(\dQ_v)}(\Sp_L,\Sp_{\dQ_v})$, $\Hom_{\GL_2(\dQ_v)}(\Sp_L,\Sp_{\dQ_v}\otimes\omega_{\dQ_v})$, $\Hom_{\GL_2(\dQ_v)}(\Sp_L\otimes\omega_L,\Sp_{\dQ_v})$, $\Hom_{\GL_2(\dQ_v)}(\Sp_L\otimes\omega_L,\Sp_{\dQ_v}\otimes\omega_{\dQ_v})$ are all nonzero. As $L$ is ramified over $\dQ_v$, $\omega_L\res_{\dQ_v^\times}$ is trivial. Thus $\Hom_{\GL_2(\dQ_v)}(\Sp_L,\Sp_{\dQ_v})\simeq\Hom_{\GL_2(\dQ_v)}(\Sp_L\otimes\omega_L,\Sp_{\dQ_v})$ and $\Hom_{\GL_2(\dQ_v)}(\Sp_L,\Sp_{\dQ_v}\otimes\omega_{\dQ_v})=\Hom_{\GL_2(\dQ_v)}(\Sp_L\otimes\omega_L,\Sp_{\dQ_v}\otimes\omega_{\dQ_v})$. We claim that the trivial character of $L^\times$ appears in $\Sp_{\dQ_v}\otimes\omega_{\dQ_v}$. Then the second part follows from
\cite{Pra92}*{Remark 4.1.2}.

For the claim, if not, then by the Saito--Tunnell dichotomy \cites{Tun83, Sai93}, the trivial character of $L^\times$ appears in the Jacquet--Langlands transfer of $\Sp_{\dQ_v}\otimes\omega_{\dQ_v}$ to $D_v^\times$, where $D_v$ is the unique division $\dQ_v$-quaternion algebra. However, this is not the case since $L$ is ramified over $\dQ_v$.
\end{proof}

Let $\fr_1$ be an ideal of $O_F$ contained in $\fr'''_\rho$ and coprime to $\nabla'\cup\nabla''$. Put $\fr_0=\fr'_\rho$, and $\fr=\fr_0\fr_1$ as usual. The elliptic curve $E$ determines a homomorphism
\[\phi^\fr_E\colon\dZ[\dT^\fr]\to\dZ\]
such that $\phi^\fr_E(\rT_\fq)=a_\fq(E)=\tr\rho_{E,p}(\Frob_\fq)$ and $\phi^\fr_E(\rS_\fq)=1$.

\begin{proposition}\label{pr:period}
There is an absolutely neat ideal $\fr_1$ of $O_F$ contained in $\fr'''_\rho$ and coprime to $\nabla'\cup\nabla''$ such that the functional $\cJ\colon\Gamma(\cS_\fr,\dZ)[\Ker^\fr_E]\to\dZ$ defined by the formula
\[\cJ(\phi)=\sum_{\cS^\flat_{\fr_1}}\phi(\vartheta(x))\]
is nonzero, where $\vartheta$ is the map \eqref{eq:special2}.
\end{proposition}

\begin{proof}
Let $B^\flat$ the (definite) $\dQ$-quaternion algebra with the discriminant set $\nabla^\flat$, and put $B=B^\flat\otimes_\dQ F$. Let $\Pi$ be the Jacquet--Langlands correspondence of $\Pi_E$ to $\Res_{F/\dQ}B^\times(\dA)$. We write $\Pi=\otimes'_v\Pi_v$ where $\Pi_v$ is an irreducible admissible representation of $\Res_{F/\dQ}B^\times(\dQ_v)$ for every place $v$ of $\dQ$.

By \cite{Ich08} and Remark \ref{bre:absolutely_neat}, it suffices to show the following statement: for $v\in\Sigma(\Pi_E,\tau)$, if we denote by $l$ a generator of the $1$-dimensional space $\Hom_{B^\flat_v}(\Pi_v,\b{1})$, then $l$ takes non-zero values on new vectors. This is obvious if
$F\otimes_\dQ\dQ_v=\dQ_v\times\dQ_v\times\dQ_v$ or $F\otimes_\dQ\dQ_v$ is a field. Thus, we may assume that $F\otimes_\dQ\dQ_v=\dQ_v\times L$ for a quadratic field extension $L/\dQ_v$.

We resume the notation in the proof of Lemma \ref{le:dichotomy}. Write $\Pi_v=\pi\boxtimes\sigma$ where $\pi$ (resp.\ $\sigma$) is an irreducible admissible representation of $\GL_2(L)$ (resp.\ $D_v^\times$). There are three cases.

(1) If $\pi$ is an unramified principal series and $L/\dQ_v$ is unramified, then it follows from Case 2 of \cite{Liu}*{Lemma 3.21}.

(2) If $\pi$ is an unramified principal series and $L/\dQ_v$ is ramified, then $\sigma$ must be the trivial character by the same argument in Case 2 of \cite{Liu}*{Lemma 3.21}. The remaining argument is similar as well.

(3) If $\pi$ is either $\Sp_L$ or $\Sp_L\otimes\omega_L$ and $L/\dQ_v$ is unramified, then we may assume $\pi=\Sp_L$ without lost of generality. Then $\sigma$ must be the nontrivial unramified quadratic character of $D_v^\times$ by \cite{Pra92}*{Theorem D, Remark 4.1.2}. We identify $D_v\otimes_{\dQ_v}L$ with $\Mat_2(L)$ such that $O_{D_v}^\times$ is contained in the standard Iwahori subgroup of $\GL_2(L)$ and
$\Pi\coloneqq\big(\begin{array}{cc} 0 & 1 \\ v & 0 \\ \end{array} \big)$ is a uniformizer of $D_v$. We realize $\pi=\Sp_L$ as the unique irreducible subrepresentation of
\[\left\{f\colon\GL_2(L)\to\dC\res f\(\left(
                                              \begin{array}{cc}
                                                a & * \\
                                                0 & d \\
                                              \end{array}
                                            \right)g
\)=\left|\frac{a}{d}\right|_Lf(g),\forall g\in\GL_2(L)\right\}.\]
It is well-known the Iwahori fixed function $f_0$ satisfying
\[f_0\(\left(
             \begin{array}{cc}
               1 & 0 \\
               0 & 1 \\
             \end{array}
           \right)
\)=v^2,\quad f_0\(\left(
             \begin{array}{cc}
               0 & 1 \\
               1 & 0 \\
             \end{array}
           \right)
\)=-1\] is a new vector of $\pi$. In particular, we have $f_0(\Pi)=-v^2$. By the same argument in Case 2 of \cite{Liu}*{Lemma 3.21}, we only need to show that the integral
\[\int_{L^\times\backslash D_v^\times}\sigma(g^{-1})f_0(g)\rd g\]
does not vanish, which is true since $\sigma$ is nontrivial.
\end{proof}

From now on, we fix an ideal $\fr_1$ as in Proposition \ref{pr:period}.

\begin{lem}\label{le:prime_bad}
There exists a finite set $\cP_E$ of primes such that for every $p\not\in\cP_E$, we have (compare with Definition \ref{de:perfect_prime})
\begin{enumerate}
  \item $p\geq 11$ and $p\neq 13,19$;

  \item $p$ is coprime to $\nabla$ and $\fr\cdot\mu(\fr,\fr_\rho)\cdot|\Cl(F)_{\fr_1}|\cdot\disc F$;

  \item $\bar\rho$ is generic;

  \item $\Gamma(\cS_{\fr_\rho},\dZ)_{\fm^\fr_{E,p}}$ is a free $\dZ_p$-module of rank $1$, and the maps in the following diagram
      \[\xymatrix{
      \Gamma(\cS_\fr,\dZ)_{\fm^\fr_{E,p}}
      \ar[rrr]^-{\bigoplus_{\fd\in\fD(\fr,\fr_\rho)}\delta^\fd_*}\ar[d] &&&
      \bigoplus_{\fd\in\fD(\fr,\fr_\rho)}\Gamma(\cS_{\fr_\rho},\dZ)_{\fm^\fr_{E,p}} \ar[d] \\
      (\Gamma(\cS_\fr,\dZ)/\Ker\phi^\fr_E)\otimes\dZ_p
      \ar[rrr]^-{\bigoplus_{\fd\in\fD(\fr,\fr_\rho)}\delta^\fd_*}  &&&
      \bigoplus_{\fd\in\fD(\fr,\fr_\rho)}(\Gamma(\cS_{\fr_\rho},\dZ)/\Ker\phi^\fr_E)\otimes\dZ_p
      }\]
      are all isomorphisms.
\end{enumerate}
\end{lem}

\begin{proof}
Properties (1, 2) only exclude finitely many primes. By \cite{Dim05}*{Proposition 0.1 (iii)} and Assumption \ref{as:elliptic_curve} (E2, E3), property (3) only excludes finitely many primes as well.

For (4), by the theory of automorphic forms, there is a number field $L$ and two finite sets $R\subset R'$ of homomorphisms $\phi\colon\dZ[\dT^\fr]\to O_L$ such that
\begin{itemize}
  \item $\Gamma(\cS_{\fr_\rho},L)/\Ker\phi\neq 0$ for $\phi\in R$, and $\Gamma(\cS_{\fr_\rho},L)=\bigoplus_{\phi\in R}\Gamma(\cS_{\fr_\rho},L)/\Ker\phi$;
  \item $\Gamma(\cS_\fr,L)/\Ker\phi\neq 0$ for $\phi\in R'$, and $\Gamma(\cS_\fr,L)=\bigoplus_{\phi\in R'}\Gamma(\cS_\fr,L)/\Ker\phi$.
\end{itemize}
Note that $R$ contains $\phi_E^\fr$ and $\Gamma(\cS_{\fr_\rho},L)/\Ker\phi_E^\fr=(\Gamma(\cS_{\fr_\rho},\dQ)/\Ker\phi_E^\fr)\otimes_\dQ L$ has dimension $1$. If a prime $p$ does not divide the order of
\begin{itemize}
  \item $\coker[\Gamma(\cS_{\fr_\rho},O_L)\hookrightarrow\bigoplus_{\phi\in R}\Gamma(\cS_{\fr_\rho},O_L)/\Ker\phi]$,
  \item $\coker[\Gamma(\cS_\fr,O_L)\hookrightarrow\bigoplus_{\phi\in R'}\Gamma(\cS_\fr,O_L)/\Ker\phi]$,
  \item $\coker[\bigoplus_{\fd\in\fD(\fr,\fr_\rho)}\delta^\fd_*\colon\Gamma(\cS_\fr,\dZ)/\Ker\phi^\fr_E\to\Gamma(\cS_{\fr_\rho},\dZ)/\Ker\phi^\fr_E]$,
\end{itemize}
and satisfies that for every prime ideal $P_L$ of $O_L$ above $p$ and every $\phi\in R'$ other than $\phi_E^\fr$, the composite map $\dZ[\dT^\fr]\xrightarrow{\phi}O_L\to O_L/P_L$ is different from $\dZ[\dT^\fr]\xrightarrow{\phi^\fr_E}\dZ\to\dF_p$, then $p$ satisfies property (4). However, the above requirement only excludes finitely many primes. The lemma follows.
\end{proof}

From now on, we fix a prime $p$ not in $\cP_E$. We also fix an arbitrary integer $\nu\geq 1$. Put $\rN_\rho=\rN_{E,p}\otimes\dZ/p^\nu$ with the induced homomorphism $\rho\colon\rG_F\to\GL(\rN_\rho)$. Then we obtain a perfect quadruple $(\rho,\fr_\rho,\fr_0,\fr_1)$ (Definition \ref{de:perfect_prime}) satisfying Assumption \ref{as:quadruple}.

\begin{lem}
Let $(\rho,\fr_\rho,\fr_0,\fr_1)$ be the quadruple as above. Then there are infinitely many cubic-level raising prime (Definition \ref{de:cubic_level_raising}) for $(\rho,\fr_\rho,\fr_0,\fr_1)$.
\end{lem}

\begin{proof}
By Definition \ref{de:perfect_prime} (3a), we may choose an integer $a$ such that $p\nmid a(a^{18}-1)(a^6+1)$. Put $X_0=\big(\begin{array}{cc} a & 0 \\ 0 & 1 \\ \end{array} \big)\in\GL_2(\dZ/p^\nu)$. We fix an isomorphism $\rN_\rho\simeq(\dZ/p^\nu)^{\oplus 2}$. Since $\bar\rho$ is generic, the image of $(\Ind^\dQ_F\rho)\res_{\rG_{\TF}}$ contains an element $(X_1,X_2,X_3)\in\GL_2(\dZ/p^\nu)\times\GL_2(\dZ/p^\nu)\times\GL_2(\dZ/p^\nu)$ satisfying $X_1\equiv X_2\equiv X_3\equiv X_0 \mod p$. Raising sufficiently large $p$-power, we may assume that $X_1=X_2=X_3=X_0$. Regard the image of $(\Ind^\dQ_F\rho)\res_{\rG_{F_0}}$ as a subgroup of $(\GL_2(\dZ/p^\nu)\times\GL_2(\dZ/p^\nu)\times\GL_2(\dZ/p^\nu))\rtimes\rA_3$, where $\rA_3$ is the alternating group. Then it contains an element $(X_0,X_0,X_0,\sigma)$ where $\sigma$ is a rotation. By the Chebotarev density theorem, there are infinitely many primes $\fl$ of $F_0$ of degree $1$ that is unramified in the splitting field of $(\Ind^\dQ_F\rho)\res_{\rG_{F_0}}$, coprime to $\nabla$ and $2\fr$, and such that $(\Ind^\dQ_F\rho)(\Frob_\fl)=(X_0,X_0,X_0,\sigma)$. Now the underlying prime $\ell$ of $\fl$ will satisfy Definition \ref{de:cubic_level_raising} (C1,C3,C4).

Now we show that Definition \ref{de:cubic_level_raising} (C2) only excludes finitely many such $\ell$. Since $\End(\Gamma(\cS_\fs,\dZ/p^\nu))$ is a finite set for $\fs=\fr_\rho,\fr$, there is a finite set $\Sigma$ of primes of $F$ that are coprime to $\nabla$ and $\fr$ such that $\dZ[\coprod_{\fq\in\Sigma}\dT_\fq]\cap\Ker\phi_\rho^\fr$ and $\Ker\phi_\rho^\fr$ have the same image in $\End(\Gamma(\cS_\fs,\dZ/p^\nu))$, and $\dZ[\coprod_{\fq\in\Sigma}\dT_\fq]\cap(\dZ[\dT^\fr]\setminus\fm_\rho^\fr)$ and $\dZ[\dT^\fr]\setminus\fm_\rho^\fr$ have the same image in $\End(\Gamma(\cS_\fs,\dZ/p^\nu))$. Then any prime $\ell$ satisfying Definition \ref{de:cubic_level_raising} (C1,C3,C4) such that $\fl\not\in\Sigma$ will satisfy (C2) as well. The lemma is proved.
\end{proof}

Recall that we have the multiplicatively induced representation $\rN_\rho^\sharp$ (Notation \ref{no:reduction}) and the $\dZ/p^\nu[\rG_\dQ]$-module $\rM_0$ as in Theorem \ref{th:cubic_level_raising}. By Theorem \ref{th:cubic_level_raising} (2), we have a $\rG_\dQ$-equivariant pairing $\rN_\rho^\sharp(2)\times\rM_0\to\dZ/p^\nu(\chi)(1)$ for some character $\chi\colon\rG_\dQ\to\Gal(F_0/\dQ)\to\{\pm1\}$. It induces, for every prime power $v$, a local Tate pairing
\begin{align}\label{eq:tate_pairing}
\langle\;,\;\rangle_v\colon\rH^1(\dQ_v,\rN_\rho^\sharp(2))\times\rH^1(\dQ_v,\rM_0)\to\rH^2(\dQ_v,\dZ/p^\nu(\chi)(1))\subset\dZ/p^\nu.
\end{align}
For $s\in\rH^1(\dQ,\rN_\rho^\sharp(2))$ and $r\in\rH^1(\dQ,\rM_0)$, we will write $\langle s,r\rangle_v$ rather than $\langle\loc_v(s),\loc_v(r)\rangle_v$.

\begin{lem}\label{le:local_tate}
We have the following statements.
\begin{enumerate}
  \item The restriction of $\sum_v\langle\;,\;\rangle_v$ to $\rH^1(\dQ,\rN_\rho^\sharp(2))\times\rH^1(\dQ,\rM_0)$ is trivial, where
      the sum is taken over all primes $v$.

  \item For every prime $v\neq p$, there exists an integer $\nu_v\geq 1$, independent of $\nu$, such that the image of the pairing \eqref{eq:tate_pairing} is annihilated by $p^{\nu_v}$.

  \item For every prime $v\neq p$, the restriction of $\langle\;,\;\rangle_v$ to $\rH^1_\unr(\dQ_v,\rN_\rho^\sharp(2))\times\rH^1_\unr(\dQ_v,\rM_0)$ is trivial.

  \item If $\ell$ is a cubic-level prime for $(\rho,\fr_\rho,\fr_0,\fr_1)$, then $\langle\;,\;\rangle_\ell$ induces a perfect pairing between
      $\rH^1_\unr(\dQ_\ell,\rN_\rho^\sharp(2))$ and $\rH^1_\sing(\dQ_\ell,\rM_0)$, in which both are free $\dZ/p^\nu$-modules of rank $1$.

  \item The restriction of $\langle\;,\;\rangle_p$ to $\rH^1_f(\dQ_p,\rN_\rho^\sharp(2))\times\rH^1_f(\dQ_p,\rM_0)$ is trivial (see \cite{Liu}*{Definition 4.6}).
\end{enumerate}
\end{lem}

\begin{proof}
Part (1) follows from the global class field theory and the fact that $p$ is odd. Part (2) follows from a similar argument for \cite{Liu}*{Lemma 4.3 (1)}. Part (3) is well-known. Part (4) follows from Lemma \ref{le:sharp} and the fact that $\ell$ splits in $F_0$. Part (5) follows from a similar argument for \cite{Liu}*{Lemma 4.8}.
\end{proof}

Now we review some results in Galois theory from \cite{Liu}*{\Sec 5.1}. Denote by $\rG$ the image of $\rho^\sharp(2)\colon\rG_\dQ\to\GL(\rN_\rho^\sharp(2))$ and $\dL/\dQ$ the Galois extension determined by $\rG$. In particular, we identify $\rG$ with the Galois group $\Gal(\dL/\dQ)$ in the following discussion.

\begin{lem}
We have $\rH^i(\rG,\rN_\rho^\sharp(2))=0$ for all $i\geq 0$. The restriction of classes gives an isomorphism
$\Res^\dL_\dQ\colon\rH^1(\dQ,\rN_\rho^\sharp(2))\xrightarrow{\sim}\rH^1(\dL,\rN_\rho^\sharp(2))^\rG=\Hom_\rG(\rG_\dL^\ab,\rN_\rho^\sharp(2))$.
\end{lem}

\begin{proof}
Since $\bar\rho$ is generic, $\rG$ contains an element $X$ such that $X\mod p$ is a scalar element of order coprime to $p$. Raising sufficiently large $p$-power of $X$, we may assume that $\rG$ itself contains a scalar element of order coprime to $p$. Then the lemma follows from \cite{Liu}*{Lemma 5.3}.
\end{proof}

The above lemma yields a $\dZ/p^\nu[\rG]$-linear pairing
\[[\;,\;]\colon \rH^1(\dQ,\rN_\rho^\sharp(2))\times\rG_\dL^\ab\to\rN_\rho^\sharp(2).\]
For each finitely generated $\dZ/p^\nu$-submodule $\rS$ of $\rH^1(\dQ,\rN_\rho^\sharp(2))$, denote by $\rG_\rS$ the subgroup of $\gamma\in\rG_\dL^\ab$ such that $[s,\gamma]=0$ for all $s\in\rS$. Let $\dL_\rS\subset\overline\dQ$ be the subfield fixed by $\rG_\rS$.

\begin{lem}
The induced pairing
\[[\;,\;]\colon\rS\times\Gal(\dL_\rS/\dL)\to\rN_\rho^\sharp(2)\]
yields an isomorphism $\Gal(\dL_\rS/\dL)\simeq\Hom(\rS,\rN_\rho^\sharp(2))$ of $\dZ/p^\nu[\rG]$-modules.
\end{lem}

\begin{proof}
It is a special case of \cite{Liu}*{Lemma 5.4}.
\end{proof}

Put $\rH=\Gal(\dL_\rS/\dL)$. For each prime $w$ of $L$, we denote by $\ell_w$ the rational prime underlying $w$. If $\ell_w$ is unramified in $\dL_\rS$, then its Frobenius substitution in $\rH$ of an arbitrary place of $\dL_\rS$ above $w$ is the same one since $\rH$ is abelian. We denote it by $\Psi_w$.

Let $\gamma$ be an element of $\rG$ of order coprime to $p$. Let $\dL_\gamma\subset\dL$ be the subfield fixed by the cyclic subgroup
$\langle\gamma\rangle\subset\rG$, and put $\rH_\gamma=\Gal(\dL_\rS/\dL_\gamma)$. In particular, we have
$\rH_\gamma\simeq\rH\rtimes\langle\gamma\rangle\simeq\Hom(\rS,\rN_\rho^\sharp(2))\rtimes\langle\gamma\rangle$ as $\langle\gamma\rangle$ is of order coprime to $p$.

\begin{definition}
Let $\Box$ be a finite set of primes containing those ramified in $\dL_\rS$. We say that a prime $w$ of $\dL$ is \emph{$(\Box,\gamma)$-admissible} if its underlying prime $\ell_w$ does not belong to $\Box$ and its Frobenius substitution in $\Gal(\dL/\dQ)=\rG$, which is then well-defined, equals $\gamma$.
\end{definition}

Let $w$ be a $(\Box,\gamma)$-admissible place of $\dL$. Then for $s\in\rS\subset\rH^1(\dQ,\rN_\rho^\sharp(2))$, we have
\[\loc_{\ell_w}(s)\in\rH^1_\unr(\dQ_{\ell_w},\rN_\rho^\sharp(2))=
\rH^1(\dF_{\ell_w},\rN_\rho^\sharp(2))\subset\Hom(\rG_{\ell_w}/\rI_{\ell_w},\rN_\rho^\sharp(2)).\]
Moreover by \cite{Liu}*{Lemma 5.7} we have $[s,\Psi_w]=\loc_{\ell_w}(s)(\Psi_w)$ as an equality in $\rN_\rho^\sharp(2)$.

\begin{lem}\label{le:psi_image}
The subset consisting of elements $\Psi_w\in\rH=\Gal(\dL_\rS/\dL)=\Hom(\rS,\rN_\rho^\sharp(2))$ for all $(\Box,\gamma)$-admissible places $w$ coincides with $\Hom(\rS,\rN_\rho^\sharp(2)^{\langle\gamma\rangle})$.
\end{lem}

\begin{proof}
It is a special case of \cite{Liu}*{Lemma 5.8}.
\end{proof}

Now we are ready to prove Theorem \ref{th:main}.

\begin{proof}[Proof of Theorem \ref{th:main}]
We prove by contradiction. Suppose $\dim_{\dQ_p}\rH^1_f(\dQ,\sfM(E)_p)>0$. Then there is a free $\dZ/p^\nu$-submodule $\rS$ of $\rH^1_f(\dQ,\rN_\rho^\sharp(2))$ (\cite{Liu}*{Definition 3.1}) of rank $1$ by \cite{Liu}*{Lemma 5.9}. Let $\Box$ be the (finite) set of primes that are either ramified in $\dL_\rS$ or not coprime to $\nabla$ or $\fr\disc F$. Put $\nu_\Box=\max\{\nu_v\res v\in\Box\}$ where $\nu_v$ is in Lemma \ref{le:local_tate} (2).

Let $\ell$ be a cubic-level raising prime for the quadruple $(\rho,\fr_\rho,\fr_0,\fr_1)$ that is not in $\Box$. Put $\gamma=\rho^\nu(\Frob_\ell)\in\rG$. Then $\gamma$ has order coprime to $p$ and $\rN_\rho^\sharp(2)^{\langle\gamma\rangle}\simeq\dZ/p^\nu$. Let $w$ be a $(\Box,\gamma)$-admissible place of $\dL$ such that $p\nmid[s,\Psi_w]$ for a generator $s$ of $\rS$, which is possible by Lemma \ref{le:psi_image}. Note that $\ell_w$ is also a cubic-level raising prime for $(\rho,\fr_\rho,\fr_0,\fr_1)$ if it is sufficiently large. Therefore, we may assume $\ell=\ell_w$.

Recall that we have the Hirzebruch--Zagier cycle $\Theta^\ell_{\fr_0,\fr_1}$ from Definition \ref{de:hirzebruch}. To ease notation, we will now regard $\Theta^\ell_{\fr_0,\fr_1}$ as its induced class in $\rH^1(\dQ,\rH^3(\cX(\ell)_{\fr_0,\fr_1}\otimes\overline\dQ,\dZ_p(2))/\Ker\phi^{\fr\fl}_\rho)$ via the localized Abel--Jacobi map. We claim that
\begin{enumerate}
  \item $\loc_v\Theta^\ell_{\fr_0,\fr_1}
      \in\rH^1_\unr(\dQ_v,\rH^3(\cX(\ell)_{\fr_0,\fr_1}\otimes\overline\dQ,\dZ_p(2))/\Ker\phi^{\fr\fl}_\rho)$ for a prime $v\not\in\Box\cup\{p,\ell\}$;

  \item $\loc_p\Theta^\ell_{\fr_0,\fr_1}
      \in\rH^1_f(\dQ_p,\rH^3(\cX(\ell)_{\fr_0,\fr_1}\otimes\overline\dQ,\dZ_p(2))/\Ker\phi^{\fr\fl}_\rho)$.
\end{enumerate}
In fact, (1) follows from \cite{Liu}*{Lemma 3.4} as $\cX(\ell)_{\fr_0,\fr_1}$ is smooth over $\Spec\dZ_{(v)}$ for such $v$; and (2) follows from \cite{Nek00}*{Theorem 3.1 (ii)}.

By Proposition \ref{pr:period}, there exists an integer $\nu_\cJ\geq 0$ and an element $\phi\in\Gamma(\cS_\fr,\dZ)[\Ker^\fr_E]$ such that
$p^{\nu_\cJ}$ does \emph{not} divide $|(\dZ/\fr_1\cap\dZ)^\times|\cdot\cJ(\phi)$. Now we choose the integer $\nu$ such that $\nu\geq\nu_\cJ+\nu_\Box$. By Theorem \ref{th:reciprocity}, we know that $p^{\nu_\cJ}$ does \emph{not} divide $(\partial\loc_\ell\Theta^\ell_{\fr_0,\fr_1},\phi)$. In particular, there is an ideal $\fd\in\fD(\fr,\fr_\rho)$ such that $p^{\nu_\cJ}$ does not divide $\partial\loc_\ell\Theta^{\ell,\fd}_{\fr_0,\fr_1}$, where $\Theta^{\ell,\fd}_{\fr_0,\fr_1}\in\rM_0$ is the component of $\Theta^\ell_{\fr_0,\fr_1}$ at $\fd$ under the isomorphism in Theorem \ref{th:cubic_level_raising} (2).

By Lemma \ref{le:tamely_pure} and \cite{Liu}*{Lemma 3.4}, we have $\loc_v(s)\in\rH^1_\unr(\dQ_v,\rN_\rho^\sharp(2))$ for every prime $v\not\in\Box\cup\{p,\ell\}$. By \cite{Liu}*{Definition 4.6, Remark 4.7}, we have $\loc_p(s)\in\rH^1_f(\dQ_v,\rN_\rho^\sharp(2))$. Then by Lemma \ref{le:local_tate} (2,3,5) and (1, 2) in the above claim, we have
\[p^{\nu-\nu_\Box}\mid\sum_{v\neq \ell}\langle s,\Theta^{\ell,\fd}_{\fr_0,\fr_1}\rangle_v.\]
However, by Lemma \ref{le:local_tate} (4), we know that $p^{\nu_\cJ}$ does not divide $\langle s,\Theta^{\ell,\fd}_{\fr_0,\fr_1}\rangle_\ell$. As $\nu\geq\nu_\cJ+\nu_\Box$, these contradict with Lemma \ref{le:local_tate} (1). Therefore, Theorem \ref{th:main} is proved.
\end{proof}

\appendix

\section{Eichler orders and abelian varieties}
\label{ss:a}

In this appendix, we study oriented Eichler orders over totally real number fields, and use them to canonically parameterize certain abelian varieties in positive characteristic. This generalizes the work of Ribet \cite{Rib89}.

We fix a totally real number field $F$, with the ring of integers $O_F$ and the set $\Phi_F$ of all archimedean places.

\subsection{Oriented orders and bimodules}
\label{ass:order}

Let $\Delta$ be a finite set of places of $F$ of even cardinality, and let $\fN$ be an ideal of $O_F$ that is coprime to (primes in) $\Delta$.

\begin{definition}\label{ade:eichler}
We recall the following definitions.
\begin{enumerate}
  \item An \emph{$O_F$-Eichler order of discriminant $\Delta$ and level $\fN$} is an $O_F$-algebra $\cR$ such that
     \begin{itemize}
        \item $\cR\otimes\dQ$ is an $F$-quaternion algebra that is ramified exactly at places in $\Delta$.

        \item $\cR$ is equal to the intersection of two maximal $O_F$-orders in $\cR\otimes\dQ$;

        \item for a prime $\fq\not\in\Delta$, the $O_{F_\fq}$-algebra $\cR\otimes_{O_F}O_{F_\fq}$ is isomorphic to
         \[\left\{\left(
                    \begin{array}{cc}
                      a & b \\
                      c & d \\
                    \end{array}
                  \right)
         \in\Mat_2(O_{F_\fq})\res c\in\fq^{m_\fq}\right\}\]
         where $\fq^{m_\fq}$ exactly divides $\fN$.
     \end{itemize}

  \item When $\fN=O_F$, we simply call $\cR$ an \emph{$O_F$-maximal order of discriminant $\Delta$}.

  \item An \emph{oriented $O_F$-Eichler order of discriminant $\Delta$ and level $\fN$} is a collection of data $(\cR,\cR_0,\{o_\fp\})$ where
      \begin{itemize}
        \item $\cR$ is an $O_F$-Eichler order of discriminant $\Delta$ and level $\fN$,
        \item $\cR_0$ is an $O_F$-maximal order of discriminant $\Delta$ containing $\cR$,
        \item $o_\fp\colon\cR\to\dF_{\Nm\fp^2}$ is a (ring) homomorphism such that $O_F\cap\Ker o_\fp=\fp$ for every prime $\fp$ in $\Delta$.
      \end{itemize}

  \item Denote by $\cS(\Delta)_\fN$ the set of isomorphism classes of oriented $O_F$-Eichler orders of discriminant $\Delta$ and level $\fN$.
\end{enumerate}
\end{definition}

%\begin{notation}
%For an $O_F$-Eichler order $\cR$, we write $\Tr_{\cR/O_F}$ for the reduced trace.
%\end{notation}

Now suppose that we are given two oriented $O_F$-maximal orders $\cO$ and $\cE$. Define
\begin{itemize}
  \item $\Sigma$ to be the set of primes which belong to the discriminant of both $\cO$ and $\cE$;
  \item $\Delta$ to be the set of places (including archimedean ones) which belong to the discriminant of exactly one of $\cO$ and $\cE$.
\end{itemize}
We assume that $\Delta$ is not empty. The following definition generalizes the one in \cite{Rib89}*{\Sec 2}.

\begin{definition}\label{ade:bimodule}
An \emph{$(\cO,\cE)$-bimodule} $M$ is an $O_F$-module equipped with a left action by $\cO$ and a right action by $\cE$, both compatible with the underlying $O_F$-action.
\begin{enumerate}
  \item We say that $M$ is \emph{projective of rank $n$} if it is a projective $O_F$-module of rank $n$.

  \item We say that $M$ is \emph{admissible} if it is projective of finite rank and satisfies the condition
      \[\fp_\cO M=M\fp_\cE\]
      for every $\fp\in\Sigma$, where $\fp_\cO$ (resp.\ $\fp_\cE$) is the maximal ideal of $\cO$ (resp.\ $\cE$) such that $\fp_\cO\cap O_F$ (resp.\ $\fp_\cE\cap O_F$) equals $\fp$.

  \item We say that two $(\cO,\cE)$-bimodules $M$ and $N$ are \emph{locally isomorphic} if they are isomorphic after tensoring with $O_{F_\fp}$ for every prime $\fp$.
\end{enumerate}
\end{definition}

Let $M$ be an admissible $(\cO,\cE)$-bimodule. Put $\Lambda=\End(M)$ as the algebra of bimodule endomorphisms. For each admissible $(\cO,\cE)$-bimodule $N$ that is locally isomorphic to $M$, we put $J(N)=\Hom(M,N)$ as the set of bimodule homomorphisms, which is a right $\Lambda$-module via composition.

\begin{lem}
The assignment $N\mapsto J(N)$ establishes a bijection between the sets of isomorphism classes of the following objects:
\begin{itemize}
  \item$(\cO,\cE)$-bimodules that are locally isomorphic to $M$;
  \item locally free rank-$1$ right $\Lambda$-modules.
\end{itemize}
\end{lem}

\begin{proof}
The proof is same as for \cite{Rib89}*{Theorem 2.3}.
\end{proof}

From now on, we consider only admissible $(\cO,\cE)$-bimodules of rank $8$. For every collection of integers
\[\underline{r}=\{r_\fp\res \fp\in\Sigma\}\]
where $r_\fp\in\{0,1,2\}$, we may consider admissible rank-$8$ $(\cO,\cE)$-bimodules with invariants $\underline{r}$, defined in the same way as in \cite{Rib89}. We now briefly recall: An admissible rank-$8$ $(\cO,\cE)$-bimodule $M$ has invariants $\underline{r}$ if for every $\fp\in\Sigma$, there is an isomorphism of $(\cO,\cE)$-bimodules $M/M\fp_\cE=M/\fp_\cO M\simeq\dF_{\Nm\fp^2}^{\oplus r_\fp}\oplus\dF_{\Nm\fp^2}^{\oplus 2-r_\fp}$ such that $\cO$ and $\cE$ both act by the given orientation on the first direct summand; and exactly one of them acts by the given orientation on the second direct summand.

We denote by $\cM(\cO,\cE)^{\underline{r}}$ the set of isomorphism classes of admissible rank-$8$ $(\cO,\cE)$-bimodules with invariants $\underline{r}$. By a similar argument for \cite{Rib89}*{Proposition 2.2}, bimodules in $\cM(\cO,\cE)^{\underline{r}}$ with given $\underline{r}$ are locally isomorphic to each other. For $M\in\cM(\cO,\cE)^{\underline{r}}$, the endomorphism algebra $\End(M)$ is naturally an \emph{oriented} $O_F$-Eichler order of discriminant $\Delta$ and some level $\fs$ containing $\prod_{\fp\in\Sigma}\fp$. The class group $\Cl(F)$ of $F$ acts on $\cM(\cO,\cE)^{\underline{r}}$ via the formula $\fa. M=M\otimes_{O_F}\fa$ for every projective rank-$1$ $O_F$-module $\fa$. Obviously, we have $\End(\fa.M)\simeq\End(M)$.

More generally, for an ideal $\fN$ of $O_F$ coprime to $\Delta\cup\Sigma$. We denote by $\cM(\cO,\cE)^{\underline{r}}_\fN$ the set of pairs $(M,\cR)$ where $M\in\cM(\cO,\cE)^{\underline{r}}$ and $\cR\subset\End(M)$ is an $O_F$-Eichler order of level $\fN\fs$. The class group $\Cl(F)$ acts on $\cM(\cO,\cE)^{\underline{r}}_\fN$ via the first factor.

\begin{proposition}\label{apr:bimodule}
The assignment $(M,\cR)\mapsto(\cR,\End(M))$ induces a bijection
\begin{align}\label{aeq:bimodule}
\cM(\cO,\cE)^{\underline{r}}_\fN/\Cl(F)\xrightarrow{\sim}\cS(\Delta)_{\fN(\underline{r})},
\end{align}
where $\fN(\underline{r})$ is the product of $\fN$ and $\fp\in\Sigma$ with $r_\fp=1$. Here, we regarded $(\cR,\End(M))$ as the oriented $O_F$-Eichler order with the first variable (Definition \ref{ade:eichler} (3)) of $\End(M)$ replaced by $\cR$, and $\{o_\fp\}$ given by restriction.
\end{proposition}

\begin{proof}
It suffices to consider the case where $\fN=O_F$. Then the proof is almost same as for \cite{Rib89}*{Theorem 2.4}. We only point out the difference when one changes from $\dQ$ to $F$.

In the proof of injectivity, suppose that we have $M$ and $M'$ with $\End(M)\simeq\End(M')$. Then the $\End(M)$-module $\Hom(M,M')$ is isomorphic to $\fa\End(M)$ where $\fa$ is some ideal of $O_F$. In particular, we have $M'\simeq M\otimes_{O_F}\fa$.

In the proof of surjectivity, we fix an element $M\in\cM(\cO,\cE)^{\underline{r}}$, whose existence can be shown in the same way as in the proof of \cite{Rib89}*{Theorem 2.4}. Put $\Lambda=\End(M)$. There is a one-to-one correspondence between the set of locally free rank-$1$ right $\Lambda$-modules and the double coset $\Lambda_\dQ^\times\backslash\widehat\Lambda_\dQ^\times/\widehat\Lambda^\times$. On the other hand, the choice of the member $\Lambda\in\cS(\Delta)_{\fN(\underline{r})}$ induces a one-to-one correspondence between $\cS(\Delta)_{\fN(\underline{r})}$ and the double coset $\Lambda_\dQ^\times\backslash\widehat\Lambda_\dQ^\times/\widehat{F}^\times\widehat\Lambda^\times$. In particular, the two sides of \eqref{aeq:bimodule} have the same cardinality.
\end{proof}

\begin{notation}\label{ano:divisor}
Let $\fM,\fN$ be two ideals of $O_F$ such that $\fN\subset\fM$.
\begin{enumerate}
  \item Denote by $\fD(\fN,\fM)$ the set of all ideals of $O_F$ containing $\fN^{-1}\fM$.
  \item Put $\mu(\fN,\fM)=\Nm(\fN\fM^{-1})\cdot\prod_\fq\left(1+\frac{1}{\Nm\fq}\right)$, where the product is taken over all primes $\fq$ of $F$ that divide $\fM$ but not $\fN$.
\end{enumerate}
\end{notation}

Let $\fM,\fN$ be two ideals of $O_F$ such that $\fN\subset\fM$ and $\fN$ is coprime to $\Delta$, we have a \emph{degeneracy map}
\begin{align}\label{aeq:degeneracy}
\delta^\fd\colon \cS(\Delta)_\fN\to \cS(\Delta)_\fM
\end{align}
for every $\fd\in\fD(\fN,\fM)$. We write $\delta=\delta^{O_F}$ for simplicity.

\begin{definition}\label{ade:pushforward}
Let $\Lambda$ be an abelian group. Let $\fM,\fN$ be as in the previous definition. For every $\fd\in\fD(\fN,\fM)$, the map $\delta^\fd$ induces
\begin{enumerate}
  \item the \emph{pullback map} $\delta^{\fd*}\colon\Gamma(\cS(\Delta)_\fM,\Lambda)\to\Gamma(\cS(\Delta)_\fN,\Lambda)$ such that $(\delta^{\fd*}f)(y)=f(\delta^\fd(y))$ for every $f\in\Gamma(\cS(\Delta)_\fM,\Lambda)$ and $y\in\cS(\Delta)_\fN$;

  \item the \emph{(normalized) pushforward map} $\delta^\fd_*\colon\Gamma(\cS(\Delta)_\fN,\Lambda)\to\Gamma(\cS(\Delta)_\fM,\Lambda)$ such that \[(\delta^\fd_*g)(x)=\frac{\mu(\fN,\fM)}{|(\delta^\fd)^{-1}(x)|}\sum_{y\in(\delta^\fd)^{-1}(x)}g(y)\]
      for every $g\in\Gamma(\cS(\Delta)_\fN,\Lambda)$ and $x\in\cS(\Delta)_\fM$. Note that $|(\delta^\fd)^{-1}(x)|$ always divides $\mu(\fN,\fM)$ (Notation \ref{ano:divisor}).
\end{enumerate}
\end{definition}

In particular, the composite map $\delta^\fd_*\circ\delta^{\fd*}\colon\Gamma(\cS(\Delta)_\fM,\Lambda)\to\Gamma(\cS(\Delta)_\fM,\Lambda)$ is always the multiplication by $\mu(\fN,\fM)$.

\begin{definition}[Hecke monoid]\label{ade:hecke}
Let $\fq$ be a prime of $F$. We define the \emph{(abstract spherical) Hecke monoidal} at $\fq$, denoted by $\dT_\fq$, to be the commutative monoid generated by symbols $\rT_\fq,\rS_\fq,\rS_\fq^{-1}$ ($\rS_\fq\rS_\fq^{-1}=1$).
\end{definition}

\begin{definition}\label{ade:degeneracy}
For $\fq$ coprime to $\Delta$ and $\fN$, we define an action of the Hecke monoid $\dT_\fq$ on $\cS(\Delta)_\fN$ by correspondences such that
\begin{itemize}
  \item $\rT_\fq$ acts by the correspondence
     \[\cS(\Delta)_\fN\xleftarrow{\delta}\cS(\Delta)_{\fN\fq}\xrightarrow{\delta^\fq}\cS(\Delta)_\fN;\]

  \item $\rS_\fq$ and $\rS_{\fq}^{-1}$ act by the trivial correspondence.
\end{itemize}
\end{definition}

\subsection{Parameterizing abelian varieties: unramified case}
\label{ass:unramified}

Let $d$ be the degree of $F$. We fix an oriented $O_F$-maximal order $\cO$ of discriminant $\Delta$ containing no archimedean places, and a rational prime $\ell$ that is inert in $O_F$ and such that $\fl\not\in\Delta$, where $\fl$ is the unique prime of $F$ above $\ell$. Let $k$ be an algebraically closed field containing $\dF_{\ell^d}\dF_{\ell^2}$, and $\sigma$ the $\ell$-Frobenius map. We fix a homomorphism $o_\fl\colon O_F\to\dF_{\Nm\fl}=\dF_{\ell^d}$.

\begin{lem}\label{ale:rapoport}
Let $A$ be a superspecial abelian variety of dimension $2d$ over $\Spec k$. Let $\iota\colon\cO\to\End(A_{/k})$ be an injective homomorphism. We have the following.
\begin{enumerate}
  \item If $d$ is odd, then for every $r\in O_F$,
      \[\tr(\iota(r)\res\Lie(A_{/k}))=2\Tr_{\dF_{\ell^d}/\dF_\ell}o_\fl(r).\]

  \item If $d$ is even, then exactly one of the following holds:
      \begin{enumerate}
        \item for every $r\in O_F$, we have $\tr(\iota(r)\res\Lie(A_{/k}))=2\Tr_{\dF_{\ell^d}/\dF_\ell}o_\fl(r)$;
        \item for every $r\in O_F$, we have $\tr(\iota(r)\res\Lie(A_{/k}))=4\Tr_{\dF_{\ell^d}/\dF_{\ell^2}}o_\fl(r)$;
        \item for every $r\in O_F$, we have $\tr(\iota(r)\res\Lie(A_{/k}))=4\Tr_{\dF_{\ell^d}/\dF_{\ell^2}}\sigma(o_\fl(r))$.
      \end{enumerate}
\end{enumerate}
\end{lem}

\begin{proof}
We fix a supersingular elliptic curve $E_0$ over $\dF_{\ell^2}$ and put $E=E_0\otimes_{\dF_{\ell^2}}k$. Then we have $\End(E_{0/\dF_{\ell^2}})=\End(E_{/k})$. By a well-known theorem of Deligne, Shioda and Ogus (see for example \cite{LO98}*{\Sec 1.6}), we have $A\simeq E^{\oplus 2d}$. If we put $A_0=E_0^{\oplus 2d}$, then we have an isomorphism $\End(A_{0/\dF_{\ell^2}})\simeq\End(A_{/k})$. In particular, we obtain an injective homomorphism $\iota\colon \cO\to\End(A_{0/\dF_{\ell^2}})$. Therefore, $\cO$ acts on $\Lie(A_{0/\dF_{\ell^2}})$, which is an $\dF_{\ell^2}$-vector space of dimension $2d$. On the other hand, $\cO/\fl\cO$ is isomorphic to $\Mat_2(\dF_\fl)$. Take a nontrivial idempotent $e$ of $\cO/\fl\cO$, we have for every $r\in O_F$,
\[\tr(\iota(r)\res\Lie(A_{/k}))=2\tr(\iota(r)\res e\Lie(A_{0/\dF_{\ell^2}})).\]
The lemma follows as $e\Lie(A_{0/\dF_{\ell^2}})$ is an $\dF_{\ell^2}$-vector space of dimension $d$.
\end{proof}

\begin{definition}\label{ade:pure_unramified}
Let $(A,\iota)$ be as in Lemma \ref{ale:rapoport}. We say that $(A,\iota)$ is \emph{mixed} if it is in case (1) or (2a), \emph{$+$-pure} if it is in case (2b) and, \emph{$-$-pure} if it is in case (2c).
\end{definition}

\begin{definition}\label{ade:abelian_unramified}
Let $\fN_0$ and $\fN_1$ be two ideals of $O_F$ such that $\fN_0$, $\fN_1$ and $\Delta\cup\{\fl\}$ are mutually coprime. We define $\cA(\Delta)^0_{\fN_0,\fN_1}$ (resp.\ $\cA(\Delta)^{\pm1}_{\fN_0,\fN_1}$) to be the set of isomorphism classes of quadruples $(A,\iota,C,\lambda)$ where
\begin{itemize}
  \item $A$ is a superspecial abelian variety of dimension $2d$ over $\Spec k$;
  \item $\iota\colon\cO\to\End(A_{/k})$ is an injective homomorphism such that $(A,\iota)$ is mixed (resp.\ $\pm$-pure);
  \item $C\subset A(k)$ is an $\cO$-submodule that is isomorphic to $(O_F/\fN_0)^{\oplus 2}$ (Remark \ref{are:cyclic});
  \item $\lambda\colon(O_F/\fN_1)^{\oplus 2}\to A(k)$ is an $\cO$-equivariant injective homomorphism (Remark \ref{are:cyclic}).
\end{itemize}
We simply write $\cA(\Delta)^i_\fN$ for $\cA(\Delta)^i_{\fN,O_F}$ for $i\in\{-1,0,+1\}$.
\end{definition}

\begin{remark}\label{are:cyclic}
Note that for an ideal $\fN$ of $O_F$ coprime to $\Delta$, the abelian group $(O_F/\fN)^{\oplus 2}$ is equipped with a unique up to isomorphism cyclic $\cO$-module structure that is compatible with the underlying $O_F$-module structure, since $\cO/\fN\cO$ is isomorphic to $\Mat_2(O_F/\fN)$.
\end{remark}

Let $\Cl(F)_{\fN_1}$ be the ray class group of $F$ with respect to the ideal $\fN_1$. We realize $\Cl(F)_{\fN_1}$ as the set of isomorphism classes of rank-$1$ projective $O_F$-module $\fa$ together with an $O_F$-linear surjective homomorphism $\fa\to O_F/\fN_1$. It acts on $\cA(\Delta)^i_{\fN_0,\fN_1}$ for each $i\in\{-1,0,+1\}$ such that the data $\fa\to O_F/\fN_1$ sends $(A,\iota,C,\lambda)$ to $(A\otimes_{O_F}\fa,\iota\otimes_{O_F}\fa,C\otimes_{O_F}\fa,\lambda)$, where we have used the isomorphism $(A\otimes_{O_F}\fa)[\fN_1]\simeq A[\fN_1]$ induced from $\fa\to O_F/\fN_1$.

For an element $(A,\iota,C,\lambda)\in\cA(\Delta)^i_{\fN_0,\fN_1}$, denote by $\End(A,\iota,C,\lambda)$ the subalgebra of $\End(A_{/k})$ of endomorphisms that commute with $\iota(O_F)$ and preserve $C\times\IM(\lambda)$.

\begin{proposition}\label{apr:unramified}
The $O_F$-algebra $\End(A,\iota,C,\lambda)$ is naturally an oriented $O_F$-Eichler order. The assignment $(A,\iota,C,\lambda)\mapsto\End(A,\iota,C,\lambda)$ induces a canonical isomorphism
\[\cA(\Delta)^i_{\fN_0,\fN_1}/\Cl(F)_{\fN_1}\xrightarrow{\sim}\cS(\Delta')_{\fN'}\]
where
\begin{enumerate}
  \item $\fN'=\fN_0\fN_1$ and $\Delta'=\Phi_F\cup\Delta\cup\{\fl\}$, if $d$ is odd and $i=0$;
  \item $\fN'=\fN_0\fN_1\fl$ and $\Delta'=\Phi_F\cup\Delta$, if $d$ is even and $i=0$;
  \item $\fN'=\fN_0\fN_1$ and $\Delta'=\Phi_F\cup\Delta$, if $d$ is even and $i=\pm1$.
\end{enumerate}
Moreover, the induced action of $\Gal(k/\dF_{\ell^d}\dF_{\ell^2})$ on $\cA(\Delta)^i_{\fN_0,\fN_1}/\Cl(F)_{\fN_1}$ is trivial.
\end{proposition}

\begin{proof}
Note that $\End(A,\iota,C,\lambda)$ contains $\iota(O_F)$ hence is an $O_F$-algebra via $\iota$.

\emph{Step 1.} Put $\fN=\fN_0\fN_1$. We have a canonical map from $\cA(\Delta)^i_{\fN_0,\fN_1}$ to $\cA(\Delta)^i_\fN$ sending $(A,\iota,C,\lambda)$ to $(A,\iota,C\times\IM(\lambda))$. It induces an isomorphism
\[\cA(\Delta)^i_{\fN_0,\fN_1}/\Cl(F)_{\fN_1}\xrightarrow{\sim}\cA(\Delta)^i_\fN/\Cl(F).\]
Therefore, we only need to study the set $\cA(\Delta)^i_\fN$.

\emph{Step 2.} We first assume $\fN=O_F$, and will write $\cA(\Delta)^i=\cA(\Delta)^i_{O_F}$ for short.

Fix a supersingular elliptic curve $E$ over $\Spec k$. Then $\cE_0\coloneqq\End(E)$ is a ($\dZ$-)maximal order of discriminant $\{\infty,\ell\}$. Moreover, $\cE_0$ is equipped with a homomorphism $\epsilon\colon\cE_0\to k$ through its action on $\Lie(E)$, whose image is contained in $\dF_{\ell^2}$. Put $\cA(\Delta)=\bigcup_{\{-1,0,1\}}\cA(\Delta)^i$. It is canonically isomorphic to the set of isomorphism classes of homomorphisms $\cO\to\Mat_{2d}(\cE_0)$. Now we construct the isomorphism case by case.

For (1), put $\cE=O_F\otimes\cE_0$ which is an $O_F$-maximal order of discriminant $\Phi_F\cup\{\fl\}$. We have a homomorphism $\epsilon_\fl\colon\cE\to\dF_{\ell^{2d}}=\dF_{\ell^d}\dF_{\ell^2}\subset k$ induced by $\epsilon$. In other words, $\cE$ is an oriented $O_F$-maximal order of discriminant $\Phi_F\cup\{\fl\}$. There is a one-to-one correspondence between the set $\cA(\Delta)^0=\cA(\Delta)$ and the set of isomorphism classes of (admissible) $(\cO,\cE)$-bimodules of rank $8$. In other words, we have an isomorphism $\cA(\Delta)^0\xrightarrow{\sim}\cM(\cO,\cE)$ sending $(A,\iota)$ to $M_{(A,\iota)}$, compatible with the $\Cl(F)$-action. Moreover, we have $\End(A,\iota)\simeq\End(M_{(A,\iota)})$ and we endow $\End(A,\iota)$ with the orientation from the latter. Therefore, the proposition follows from Proposition \ref{apr:bimodule}.

Now we consider (2) and (3). In particular, $d$ is even. For an element $(A,\iota)\in\cA(\Delta)$, let $f_{(A,\iota)}\colon\cO\to\Mat_{2d}(\cE_0)$ be the corresponding homomorphism. Then the commutator of $f_{(A,\iota)}(\cO)$ in $\Mat_{2d}(\cE_0)$ is canonically isomorphic to $\End(A,\iota)$. It is an $O_F$-Eichler order of discriminant $\Phi_F\cup\Delta$ whose level is a power of $\fl$. Therefore, the question for determining $\End(A,\iota)$ is local at $\ell$.

As $\cO\otimes\dZ_\ell\simeq\Mat_2(O_{F_\fl})$, the homomorphism $f_{(A,\iota)}\otimes\dZ_\ell$ induces another homomorphism $g_{(A,\iota)}\colon O_{F_\fl}\to\Mat_d(\cE_{0,\ell})$, where $\cE_{0,\ell}\coloneqq\cE_0\otimes\dZ_\ell$ is the maximal order in the division quaternion $\dZ_\ell$-algebra. The homomorphism $g_{(A,\iota)}$ amounts to an $O_F\otimes\cE_{0,\ell}$-module $N_{(A,\iota)}$ that is free of rank $4d$ as a $\dZ_\ell$-module. Such module is semisimple, which is a direct sum of (torsion-free) simple modules. Up to isomorphism, there are two simple torsion-free $O_F\otimes\cE_{0,\ell}$-modules $N^+$ and $N^-$: the underlying $\cE_{0,\ell}$-module of both $N^\pm$ are $\cE_{0,\ell}^{\oplus d/2}$, and the induced homomorphism $O_F\to\Mat_{d/2}(\cE_{0,\ell})\xrightarrow{\epsilon}\Mat_{d/2}(\dF_{\ell^2})$ has trace $\Tr_{\dF_{\ell^d}/\dF_{\ell^2}}o_\fl(r)$ (resp.\ $\Tr_{\dF_{\ell^d}/\dF_{\ell^2}}\sigma(o_\fl(r))$) for every $r\in O_F$ in the case of $N^+$ (resp.\ $N^-$). Comparing with Definition \ref{ade:abelian_unramified}, we know that $N_{(A,\iota)}$ is isomorphic to $N^+\oplus N^-$ (resp.\ $(N^+)^{\oplus 2}$, and $(N^-)^{\oplus 2}$) if and only if $(A,\iota)$ is mixed (resp.\ $+$-pure, and $-$-pure). Note that we have the isomorphism
\begin{align}\label{aeq:local}
\End(A,\iota)\otimes\dZ_\ell\simeq\End_{O_F\otimes\cE_{0,\ell}}(N_{(A,\iota)}).
\end{align}
Let $\cE^\pm_\fl$ be the two maximal orders in $F\otimes\cE_{0,\ell}$ containing $O_F\otimes\cE_{0,\ell}$ such that the action of $\cE_0$ on the one-dimensional $O_F/\fl$ (which is isomorphic to $\dF_{\ell^d}$ via $o_\fl$) vector space $\cE^+_\fl/O_F\otimes\cE_{0,\ell}$ (resp.\ $\cE^-_\fl/O_F\otimes\cE_{0,\ell}$) via left multiplication is given by the homomorphism $\epsilon$ (resp.\ $\sigma\circ\epsilon$). Let $\cE^\pm$ be the two maximal orders in $F\otimes\cE_0$ containing $O_F\otimes\cE_0$ such that $\cE^\pm\otimes\dZ_\ell=\cE^\pm_\fl$.

For (3), take an element $(A,\iota)$ in $\cA(\Delta)^{\pm 1}$. By \eqref{aeq:local}, we have \[\End(A,\iota)\otimes\dZ_\ell\simeq\End_{O_F\otimes\cE_{0,\ell}}(N_{(A,\iota)})
\simeq\Mat_2(\End_{O_F\otimes\cE_{0,\ell}}(N^\pm))=\Mat_2(O_{F_\fl}).\]
Therefore, $\End(A,\iota)$ is an $O_F$-maximal order of discriminant $\Phi_F\cup\Delta$.

We claim that the commutator of $\End_{O_F\otimes\cE_{0,\ell}}(N_{(A,\iota)})$ in $\End_{O_{F_\fl}}(N_{(A,\iota)})$, which is a subalgebra of $F\otimes\cE_{0,\ell}$, coincides with $\cE^\pm_\fl$. Without lost of generality, we only consider the $+$-pure case. Note that this commutator is same as the commutator $R$ of $\End_{O_F\otimes\cE_{0,\ell}}(N^+)$ in $\End_{O_{F_\fl}}(N^+)$. Let $W$ be the unique \'{e}tale $\dZ_\ell$-subalgebra of $O_{F_\fl}$ of degree $2$. Write $\cE_{0,\ell}=W\oplus W\Pi$ with $\Pi^2=\ell$ and such that $\epsilon\res_W=o_\fl\res_W\colon W\to\dF_{\ell^2}$. We may assume that the $O_{F_\fl}$-action preserves $W^{\oplus d/2}$ and $(W\Pi)^{\oplus d/2}$ respectively, such that the subalgebra $W$ acts by the scalar multiplication. This is possible by the definition of $N^+$. Therefore, $R$ is identified with $\Mat_2(O_{F_\fl})$ if we use the basis $\{(1,\dots,1),(\Pi,\dots,\Pi)\}$. Under the same basis, we have \[O_F\otimes\cE_{0,\ell}=\left\{\left(
                    \begin{array}{cc}
                      a & b \\
                      c & d \\
                    \end{array}
                  \right)
         \in\Mat_2(O_{F_\fl})\res b\in\fl\right\},\]
and that the homomorphism $W\subset\cE_{0,\ell}\to\Mat_2(O_{F_\fl})$ is given by the formula
\[w\mapsto\left(
                    \begin{array}{cc}
                      w &  \\
                       & \sigma(w) \\
                    \end{array}
                  \right)
         \in\Mat_2(O_{F_\fl})\]
in which $\sigma$ is understood as the canonical lift of the $\ell$-Frobenius map. This implies $R=\cE^\pm_\fl$.

The homomorphism $f_{(A,\iota)}$ amounts to an $(\cO,O_F\otimes\cE_0)$-bimodule $M_{(A,\iota)}$ that is a projective $O_F$-module of rank $8$. The above claim implies that if $(A,\iota)$ is $\pm$-pure, then the action of $\cE^\pm$ on $M_{(A,\iota)}\otimes\dQ$ preserves $M_{(A,\iota)}$; in other words, $M_{(A,\iota)}$ is an (admissible) $(\cO,\cE^\pm)$-bimodules of rank $8$. The remaining argument is same as in (1).

For (2),  take an element $(A,\iota)$ in $\cA(\Delta)^0$. By \eqref{aeq:local}, we have \[\End(A,\iota)\otimes\dZ_\ell\simeq\End_{O_F\otimes\cE_{0,\ell}}(N_{(A,\iota)})
\simeq\End_{O_F\otimes\cE_{0,\ell}}(N^+\oplus N^-),\]
which is of level $\fl$. Therefore, $\End(A,\iota)\simeq\End(M_{(A,\iota)})$ is an $O_F$-Eichler order of discriminant $\Phi_F\cup\Delta$ and level $\fl$. Put $M^+_{(A,\iota)}=M_{(A,\iota)}\otimes_{O_F\otimes\cE_0}\cE^+$. Then $M^+_{(A,\iota)}$ belongs to $\cM(\cO,\cE^+)$. Thus the assignment $(A,\iota)\mapsto(M^+_{(A,\iota)},\End(A,\iota))$ induces a bijection from $\cA(\Delta)^0$ to $\cM(\cO,\cE^+)_\fl$. Therefore, the proposition follows from Proposition \ref{apr:bimodule}.

\emph{Step 3.} Finally we consider a general ideal $\fN$. In case
\begin{enumerate}
  \item the assignment $(A,\iota,C)\mapsto (\End(A,\iota)\cap\End(A/C,\iota),M_{(A,\iota)})$ induces a bijection from $\cA(\Delta)_\fN^0$ to $\cM(\cO,\cE)_\fN$;

  \item the assignment $(A,\iota,C)\mapsto (\End(A,\iota)\cap\End(A/C,\iota),M^+_{(A,\iota)})$ induces a bijection from $\cA(\Delta)_\fN^0$ to $\cM(\cO,\cE^+)_{\fN\fl}$;

  \item the assignment $(A,\iota,C)\mapsto (\End(A,\iota)\cap\End(A/C,\iota),M_{(A,\iota)})$ induces a bijection from $\cA(\Delta)_\fN^{\pm1}$ to $\cM(\cO,\cE^\pm)_\fN$.
\end{enumerate}
Then the isomorphism follows from Proposition \ref{apr:bimodule}.

\emph{Step 4.} Note that the action of $\Gal(k/\dF_{\fl^2})$ on $\cA(\Delta)^i_{\fN_0,\fN_1}/\Cl(F)_{\fN_1}$ does not change values in $\cS(\Delta')_{\fN'}$ under the previous construction. Therefore, such action is trivial. The proposition is proved.
\end{proof}

\begin{remark}\label{are:erratum}
In \cite{Liu}, we consider the (coarse) moduli scheme $\cX^\natural_\fM$ (see \cite{Liu}*{\Sec 2.4}). However, in \cite{Liu}*{Proposition 2.21}, we overlooked the factor $\Cl(F)$. Instead, we should have canonical isomorphisms $\pi_0(\cX^?_{\fM;\dF_{\ell^2}})/\Cl(F)\simeq\cS_\fM$ for $?=\bullet,\circ$ and $\pi_0(\cX^{\r{ssp}}_{\fM;\dF_{\ell^2}})/\Cl(F)\simeq\cS_{\fM\ell}$, by Proposition \ref{apr:unramified} (with $d=2$) and the original proof of \cite{Liu}*{Proposition 2.21}. Such change will not affect the main results in \cite{Liu} (Theorem 4.10, 4.11, and all in \Sec1) and their proofs since the automorphic representation $\pi$ of $\Res_{F/\dQ}\GL_2(\dA)$ has trivial central character, as long as we require $p\nmid|\Cl(F)|$ at the beginning in \cite{Liu}*{Assumption 3.8}.
%\footnote{In \cite{Liu}*{Definition 2.4}, the italic word \emph{even} should be \emph{odd}.}
\end{remark}

\subsection{Parameterizing abelian varieties: ramified case}

Let $d$ be the degree of $F$. Suppose that we are given an oriented $O_F$-maximal order $\cO$ of discriminant $\Delta$ containing no archimedean places, and a rational prime $\ell$ that is inert in $O_F$ and such that $\fl\in\Delta$, where $\fl$ is the unique prime of $F$ above $\ell$. In particular, we have a homomorphism $o_\fl\colon\cO\to\dF_{\ell^{2d}}$ coming from the orientation at $\fl$. Let $k$ be an algebraically closed field containing $\dF_{\ell^{2d}}$.

The following definition generalizes the one in \cite{Rib89}*{\Sec 4}.

\begin{definition}\label{ade:exceptional}
Let $(A,\iota)$ be a pair where $A$ is an abelian variety over $\Spec k$ and $\iota\colon\cO\to\End(A_{/k})$ is an injective homomorphism. We say that $(A,\iota)$ is \emph{exceptional} if the action of $\cO$ on $\Lie(A_{/k})$ factorizes through the maximal ideal of $\cO$ containing $\fl$.
\end{definition}

We choose a uniformizer $\Pi$ of $\cO\otimes\dZ_\ell$ such that $\Pi^2=\ell$, and a $O_{F_\fl}$-subalgebra $W\subset\cO\otimes\dZ_\ell$ that is \'{e}tale of degree $2$.

\begin{lem}\label{ale:rapoport_ramified}
Let $A$ be an abelian variety of dimension $2d$ over $\Spec k$. Let $\iota\colon\cO\to\End(A_{/k})$ be an injective homomorphism such that $(A,\iota)$ is exceptional. Then exactly one of the following holds:
\begin{enumerate}
  \item for every $w\in W$, we have $\tr(\iota(w)\res\Lie(A_{/k}))=\Tr_{\dF_{\ell^{2d}}/\dF_\ell}o_\fl(w)$;
  \item for every $w\in W$, we have $\tr(\iota(w)\res\Lie(A_{/k}))=2\Tr_{\dF_{\ell^{2d}}/\dF_{\ell^2}}o_\fl(w)$;
  \item for every $w\in W$, we have $\tr(\iota(w)\res\Lie(A_{/k}))=2\Tr_{\dF_{\ell^{2d}}/\dF_{\ell^2}}\sigma(o_\fl(w))$.
\end{enumerate}
Note that $\iota$ extends to an action of $\cO\otimes\dZ_\ell$ on $\Lie(A_{/k})$.
\end{lem}

\begin{proof}
The same proof of \cite{Rib89}*{Proposition 4.2} shows that $A$ must be superspecial. Then the rest of proof is same as Lemma \ref{ale:rapoport}.
\end{proof}

\begin{definition}\label{ade:pure_ramified}
Let $(A,\iota)$ be as in Lemma \ref{ale:rapoport_ramified}. We say that $(A,\iota)$ is \emph{mixed} if it is in case (1), \emph{$+$-pure} if it is in case (2) and, \emph{$-$-pure} if it is in case (3).
\end{definition}

\begin{definition}\label{ade:abelian_ramified}
Let $\fN_0$ and $\fN_1$ be two ideals of $O_F$ such that $\fN_0$, $\fN_1$ and $\Delta$ are mutually coprime. We define $\cA(\Delta)^0_{\fN_0,\fN_1}$ (resp.\ $\cA(\Delta)^{\pm1}_{\fN_0,\fN_1}$) to be the set of isomorphism classes of quadruples $(A,\iota,C,\lambda)$ where
\begin{itemize}
  \item $A$ is an abelian variety of dimension $2d$ over $\Spec k$;
  \item $\iota\colon\cO\to\End(A_{/k})$ is an injective homomorphism such that $(A,\iota)$ is exceptional and mixed (resp.\ $\pm$-pure);
  \item $C\subset A(k)$ is an $\cO$-submodule that is isomorphic to $(O_F/\fN_0)^{\oplus 2}$ (Remark \ref{are:cyclic});
  \item $\lambda\colon(O_F/\fN_1)^{\oplus 2}\to A(k)$ is an $\cO$-equivariant injective homomorphism (Remark \ref{are:cyclic}).
\end{itemize}
We simply write $\cA(\Delta)^i_\fN$ for $\cA(\Delta)^i_{\fN,O_F}$ for $i\in\{-1,0,+1\}$.
\end{definition}

Similar to the situation in \ref{ass:unramified}, the ray class group $\Cl(F)_{\fN_1}$ acts on $\cA(\Delta)^i_{\fN_0,\fN_1}$, and we have a subalgebra $\End(A,\iota,C,\lambda)$ of $\End(A_{/k})$ for every element $(A,\iota,C,\lambda)\in\cA(\Delta)^i_{\fN_0,\fN_1}$.

\begin{proposition}\label{apr:ramified}
Suppose that $d$ is odd. The $O_F$-algebra $\End(A,\iota,C,\lambda)$ is naturally an oriented $O_F$-Eichler order. The assignment $(A,\iota,C,\lambda)\mapsto\End(A,\iota,C,\lambda)$ induces a canonical isomorphism
\[\cA(\Delta)^i_{\fN_0,\fN_1}/\Cl(F)_{\fN_1}\xrightarrow{\sim}\cS(\Phi_F\cup\Delta\setminus\{\fl\})_{\fN'}\]
where $\fN'=\fN_0\fN_1$ (resp.\ $\fN'=\fN_0\fN_1\fl$) if $i=\pm 1$ (resp.\ $i=0$). Moreover, the induced action of $\Gal(k/\dF_{\fl^2})$ on $\cA(\Delta)^i_{\fN_0,\fN_1}/\Cl(F)_{\fN_1}$ is trivial.
\end{proposition}

\begin{proof}
Put $\fN=\fN_0\fN_1$. Same to the proof of Proposition \ref{apr:unramified}, we have a canonical isomorphism
$\cA(\Delta)^i_{\fN_0,\fN_1}/\Cl(F)_{\fN_1}\xrightarrow{\sim}\cA(\Delta)^i_\fN/\Cl(F)$. Fix a supersingular elliptic curve $E$ over $\Spec k$. Then $\cE_0\coloneqq\End(E)$ is a maximal order of discriminant $\{\infty,\ell\}$, equipped with a homomorphism $\epsilon\colon\cE_0\to\dF_{\ell^2}$. Put $\cE=O_F\otimes\cE_0$ which is an $O_F$-maximal order of discriminant $\Phi_F\cup\{\fl\}$ (as $d$ is odd). We have a homomorphism $\epsilon_\fl\colon\cE\to\dF_{\ell^{2d}}$ induced by $\epsilon$. For the pair $(\cO,\cE)$, the corresponding set $\Sigma$ equals $\{\fl\}$. Thus, we have sets $\cM(\cO,\cE)^r_\fN$ for $r=0,1,2$. By the same argument in \cite{Rib89}*{\Sec 4}, we have a canonical bijection between $\cA(\Delta)^i_\fN$ and $\cM(\cO,\cE)^{i+1}_\fN$. Then the proposition follows from Proposition \ref{apr:bimodule}.
\end{proof}

\section{Stratification on Hilbert modular schemes with bad reduction}
\label{ss:b}

In this appendix, we improve (see Remark \ref{bre:zink}) the work of Zink \cite{Zin82} on the global description of some pluri-nodal integral model of Hilbert modular varieties at certain bad primes. We also prove some extra results which are used in \Sec\ref{ss:3}. We use the ideal in \cite{Hel12} and \cite{TX16}, known as the isogeny trick.

We fix a totally real number field $F$, with the ring of integers $O_F$ and the set $\Phi_F$ of all archimedean places.

\subsection{Hilbert modular schemes}
\label{bss:hilbert}

\begin{definition}[Absolutely neat ideal]\label{bde:absolutely_neat}
An ideal $\fr$ of $O_F$ is said to be \emph{absolutely neat} if for \emph{every} $O_F$-Eichler order $\cR$ of discriminant $\Delta$ to which $\fr$ is coprime and level $\fr$, together with a surjective $O_F$-linear homomorphism $\cR\to O_F/\fr$, we have
\begin{itemize}
  \item $\Ker[\cR^\times\to(O_F/\fr)^\times]=1$ if $\Phi_F\subset\Delta$;

  \item $\Ker[\cR^\times\to(O_F/\fr)^\times]$ acts freely on the symmetric hermitian domain attached to $\cR\otimes\dQ$ if $\Phi_F\not\subset\Delta$.
\end{itemize}
\end{definition}

\begin{remark}\label{bre:absolutely_neat}
Given an arbitrary finite set $\Box$ of primes of $F$, there exists an absolutely neat ideal that is coprime to $\Box$. If $\fr$ is an absolutely neat idea, then every ideal contained in $\fr$ is also absolutely neat.
\end{remark}

Let $\cO$ be an \emph{oriented} $O_F$-maximal order of discriminant $\Delta$ containing no archimedean places, which is unique up to isomorphism. Let $\fr_0$ and $\fr_1$ be two ideals of $O_F$ such that $\fr_0$, $\fr_1$ and $\Delta$ are mutually coprime.

\begin{definition}\label{bde:hilbert_shimura}
We define the moduli functor $\cX(\Delta)_{\fr_0,\fr_1}$ on the category of schemes over $\Spec\dZ[(\fr_0\fr_1\disc F)^{-1}]$ as follows. For every $\dZ[(\fr_0\fr_1\disc F)^{-1}]$-scheme $T$,  we let $\cX(\Delta)_{\fr_0,\fr_1}(T)$ be the groupoid of quadruples $(A,\iota_A,C_A,\lambda_A)$ where
\begin{itemize}
  \item $A$ is an abelian scheme over $T$;

  \item $\iota_A\colon\cO\to\End(A_{/T})$ is an injective homomorphism satisfying:
      \[\tr(\iota_A(r)\res\Lie(A_{/T}))=\Tr_{O_F/\dZ}\Tr^0_\cO(r)\]
      for every $r\in\cO$, where the right-hand side is regarded as a section of $\sO_T$ via the structure homomorphism;

  \item $C_A$ is an $\cO$-stable flat subgroup of $A[\fr_0]$ over $T$ such that at every geometric point, it is isomorphic to $(O_F/\fr_0)^{\oplus 2}$ (Remark \ref{are:cyclic});

  \item $\lambda_A\colon\underline{(O_F/\fr_1)^{\oplus 2}}_{/T}\to A$ is an $\cO$-equivariant injective homomorphism of group schemes over $T$ (Remark \ref{are:cyclic}).
\end{itemize}
\end{definition}

The functor $\cX(\Delta)_{\fr_0,\fr_1}$ receives an action by the ray class group $\Cl(F)_{\fr_1}$ as follows. We realize $\Cl(F)_{\fr_1}$ as the set of isomorphism classes of rank-$1$ projective $O_F$-module $\fa$ together with an $O_F$-linear surjective homomorphism $\fa\to O_F/\fr_1$. It acts on $\cX(\Delta)_{\fr_0,\fr_1}(T)$ in the way that the data $\fa\to O_F/\fr_1$ sends $(A,\iota_A,C_A,\lambda_A)$ to $(A\otimes_{O_F}\fa,\iota_A\otimes_{O_F}\fa,C_A\otimes_{O_F}\fa,\lambda_A)$, where we have used the isomorphism $(A\otimes_{O_F}\fa)[\fr_1]\simeq A[\fr_1]$ induced from $\fa\to O_F/\fr_1$.

Suppose that we have another ideal $\fs_0\subset\fr_0$ of $O_F$ coprime to $\Delta$ and $\fr_1$. For each $\fd\in\fD(\fs_0,\fr_0)$ (Notation \ref{ano:divisor}), we have a \emph{degeneracy morphism}:
\begin{align}\label{beq:degeneracy}
\delta^\fd\colon \cX(\Delta)_{\fs_0,\fr_1}\to \cX(\Delta)_{\fr_0,\fr_1}.
\end{align}
It is finite \'{e}tale. We write $\delta=\delta^{O_F}$ for simplicity.

\begin{definition}\label{bde:hecke}
For a prime $\fq$ of $F$ coprime to $\Delta$ and $\fr_0\fr_1$, we define an action of the Hecke monoid $\dT_\fq$ on $\cX(\Delta)_{\fr_0,\fr_1}$ via \'{e}tale correspondences (Definition \ref{de:correspondence}) as follows:
\begin{itemize}
  \item $\rT_\fq$ acts by the correspondence
     \[\cX(\Delta)_{\fr_0,\fr_1}\xleftarrow{\delta}\cX(\Delta)_{\fr_0\fq,\fr_1}\xrightarrow{\delta^\fq}\cX(\Delta)_{\fr_0,\fr_1};\]

  \item $\rS_\fq$ acts via the morphism given by $\fq$ (with the canonical surjection $\fq\to O_F/\fr_1$) in $\Cl(F)_{\fr_1}$;

  \item $\rS_\fq^{-1}$ acts via the morphism given by $\fq^{-1}$ (with the canonical surjection $\fq^{-1}\to O_F/\fr_1$) in $\Cl(F)_{\fr_1}$.
\end{itemize}
\end{definition}

\begin{remark}\label{bre:hecke}
Let $\fq'$ be another prime as in Definition \ref{bde:hecke}. Then the actions of $\dT_\fq$ and $\dT_{\fq'}$ commute canonically. Note that we have a canonical morphism $\cX(\Delta)_{\fr_0,\fr_1}\to\cX(\Delta)_{\fr_0\fr_1,O_F}$ by taking the image of $\lambda_A$; it is compatible with the action of Hecke monoids.
\end{remark}

It is known that $\cX(\Delta)_{\fr_0,\fr_1}$ is a Deligne--Mumford stack of finite type and relative dimension $\deg F$ over $\Spec\dZ[(\fr_0\fr_1\disc F)^{-1}]$, and smooth over $\Spec\dZ[(\fr_0\fr_1\disc F)^{-1},\Delta^{-1}]$. We call it the \emph{Hilbert modular stack} of discriminant $\Delta$ and level $(\fr_0,\fr_1)$. If $\fr_1$ is absolutely neat, then $\cX(\Delta)_{\fr_0,\fr_1}$ is representable by a quasi-projective scheme over $\Spec\dZ[(\fr_0\fr_1\disc F)^{-1}]$ (\cite{Zin82}*{1.7, 3.4, 3.8}). It is projective if and only if $\Delta\neq\emptyset$.

\subsection{Pluri-nodal reduction}
\label{bss:reduction}

Put $d=\deg F$. We fix the following data:
\begin{itemize}
  \item a rational prime $\ell$ that is inert in $F$, which induces a prime $\fl$ of $F$,

  \item an oriented $O_F$-maximal order $\cO$ of discriminant $\Delta$ containing $\fl$ but none from $\Phi_F$,

  \item an embedding $\dZ_{\ell^\infty}\hookrightarrow\dC$.
\end{itemize}
Denote by $\sigma$ the $\ell$-Frobenius map. In particular, $\Phi_F$ is now identified with $\Hom(O_F,\dF_{\ell^\infty})$ hence equipped with an action of $\sigma$ through the target $\dF_{\ell^\infty}$. Put $\Phi=\Hom(\cO,\dF_{\ell^\infty})$ which is also equipped with an action of $\sigma$. By restriction, there is a canonical map $\pi\colon\Phi\to\Phi_F$. It is $\sigma$-equivariant and two-to-one.

\begin{definition}\label{bde:ample_subset}
We say that a subset $S$ of $\Phi$ is \emph{ample} \footnote{It is called \emph{zul\"{a}ssig} in the terminology of \cite{Zin82}, which means \emph{admissible}.} if $\pi^{-1}\tau\cap S\neq\emptyset$ for every $\tau\in\Phi_F$.
\end{definition}

We choose a uniformizer $\Pi$ of $\cO\otimes\dZ_\ell$ such that $\Pi^2=\ell$, and an $O_{F_\fl}$-subalgebra $W\subset\cO\otimes\dZ_\ell$ that is \'{e}tale of degree $2$. Let $T$ be an $\dF_{\ell^{2d}}$-scheme and $A$ an abelian scheme over $T$ with $\iota_A\colon\cO\to\End(A_{/T})$. Then the locally free $\sO_T$-sheaf $\Lie(A_{/T})$ has a canonical decomposition
\begin{align}\label{beq:split}
\Lie(A_{/T})=\bigoplus_{\tau\in\Phi}\Lie(A_{/T})_\tau
\end{align}
of $\sO_T$-sheaves such that $\cO$ acts on $\Lie(A_{/T})_\tau$ via $\tau$.

Recall that we have the Hilbert modular stack $\cX(\Delta)_{\fr_0,\fr_1}$ introduced in \Sec\ref{bss:hilbert}. Put \[X(\Delta)_{\fr_0,\fr_1}=\cX(\Delta)_{\fr_0,\fr_1}\otimes\dF_{\ell^{2d}}.\]
For every ample subset $S\subset\Phi$, we are going to introduce a subfunctor $X(\Delta)_{\fr_0,\fr_1}^S$ of $X(\Delta)_{\fr_0,\fr_1}$ as follows.

\begin{definition}\label{bde:hilbert_stratum}
We define $X(\Delta)_{\fr_0,\fr_1}^S$ to be the subfunctor of $X(\Delta)_{\fr_0,\fr_1}$ such that the induced map
$\iota_A(\Pi)_*\colon\Lie(A_{/T})_\tau\to\Lie(A_{/T})_{\sigma^d\tau}$ vanishes for $\tau\in S$.
\end{definition}

\begin{remark}
The subfunctor $X(\Delta)_{\fr_0,\fr_1}^S$ in Definition \ref{bde:hilbert_stratum} does not depend on the choices of $\Pi$ and $W$.
\end{remark}

We collect some elementary properties of the subfunctors $X(\Delta)_{\fr_0,\fr_1}^S$.

\begin{lem}\label{ble:strata}
We have
\begin{enumerate}
  \item $X(\Delta)_{\fr_0,\fr_1}^S$ is a closed subfunctor of $X(\Delta)_{\fr_0,\fr_1}$ for every ample subset $S$.

  \item $X(\Delta)_{\fr_0,\fr_1}^S\cap X(\Delta)_{\fr_0,\fr_1}^{S'}=X(\Delta)_{\fr_0,\fr_1}^{S\cup S'}$ for ample subsets $S$, $S'$.

  \item For a field $k$ containing $\dF_{\ell^{2d}}$, we have
      \[X(\Delta)_{\fr_0,\fr_1}(k)=\bigcup_{|S|=d}X(\Delta)_{\fr_0,\fr_1}^S(k).\]

  \item Let $\fs_0\subset\fr_0$ be another ideal of $O_F$ coprime to $\Delta$ and $\fr_1$. For every $\fd\in\fD(\fs_0,\fr_0)$, the following diagram
       \[\xymatrix{
       X(\Delta)_{\fs_0,\fr_1}^S \ar[r]\ar[d]_-{\delta^\fd} & X(\Delta)_{\fs_0,\fr_1} \ar[d]^-{\delta^\fd}  \\
       X(\Delta)_{\fr_0,\fr_1}^S \ar[r] & X(\Delta)_{\fr_0,\fr_1}
       }\]
       is Cartesian.
\end{enumerate}
\end{lem}

\begin{proof}
All but (3) are straightforward from the definitions. For (3), we only have to note that at least one of
$\iota_A(\Pi)_*\colon\Lie(A_{/T})_\tau\to\Lie(A_{/T})_{\sigma^d\tau}$ and $\iota_A(\Pi)_*\colon\Lie(A_{/T})_{\sigma^d\tau}\to\Lie(A_{/T})_\tau$
should vanish, as their composition vanishes.
\end{proof}

We recall the following definition.

\begin{definition}\label{bde:polynodal_scheme}
Let $X$ be a scheme locally of finite presentation over $\Spec\dZ_{\ell^{2d}}$. We say that $X$ is \emph{pluri-nodal} if for every closed point $x$ of $X$ of characteristic $\ell$, the completion of the strict henselization of the local ring of $X$ at $x$ is isomorphic to
\[\dZ_{\ell^\infty}[[X_1,\dots,X_s;Y_1,\dots,Y_s;Z_1,\dots]]/(X_1Y_1-\ell,\dots,X_sY_s-\ell)\]
for some (but unique) $s=s(x)\geq 0$. Denote by $(X\otimes\dF_{\ell^{2d}})^{[s]}$ the locus where $s(x)\geq s$, which is a closed subset of the special fiber $X\otimes\dF_{\ell^{2d}}$. We equip $(X\otimes\dF_{\ell^{2d}})^{[s]}$ with the induced reduced scheme structure.
\end{definition}

\begin{proposition}\label{bpr:zink}
Let $\fr_0$ and $\fr_1$ be two ideals of $O_F$ such that $\fr_0$, $\fr_1$ and $\Delta$ are mutually coprime. Suppose that $\fr_1$ is absolutely neat. Then
\begin{enumerate}
  \item $\cX(\Delta)_{\fr_0,\fr_1}\otimes\dZ_{\ell^{2d}}$ is a pluri-nodal scheme over $\Spec\dZ_{\ell^{2d}}$ (Definition \ref{bde:polynodal_scheme});

  \item for every ample subset $S$, $X(\Delta)_{\fr_0,\fr_1}^S$ is a smooth proper scheme over $\Spec\dF_{\ell^{2d}}$ of dimension $2d-|S|$, and is a disjoint union of irreducible components of $X(\Delta)_{\fr_0,\fr_1}^{[|S|-d]}$;

  \item $X(\Delta)_{\fr_0,\fr_1}^S\cap X(\Delta)_{\fr_0,\fr_1}^{S'}=X(\Delta)_{\fr_0,\fr_1}^{S\cup S'}$ for ample subsets $S$, $S'$;

  \item $X(\Delta)_{\fr_0,\fr_1}^{[s]}=\bigcup_{|S|=d+s}X(\Delta)_{\fr_0,\fr_1}^S$; in particular, $X(\Delta)_{\fr_0,\fr_1}=\bigcup_{|S|=d}X(\Delta)_{\fr_0,\fr_1}^S$.
\end{enumerate}
\end{proposition}

\begin{proof}
This is proved in \cite{Zin82}*{Satz 3.10}. In fact, it is not difficult to deduce it from the Grothendieck--Messing theory.
\end{proof}

\begin{example}\label{bex:cube}
We consider the special case where $d=3$. We identify $\Phi$ with $\dZ/6\dZ=\{0,1,2,3,4,5\}$ such that $\sigma i=i+1$, and $\Phi_F$ with $\dZ/3\dZ=\{0,1,2\}$ such that $\pi^{-1}\{i\}=\{i,i+3\}$. The dual reduction building of $X(\Delta)_{\fr_0,\fr_1}$ is depicted as follows.
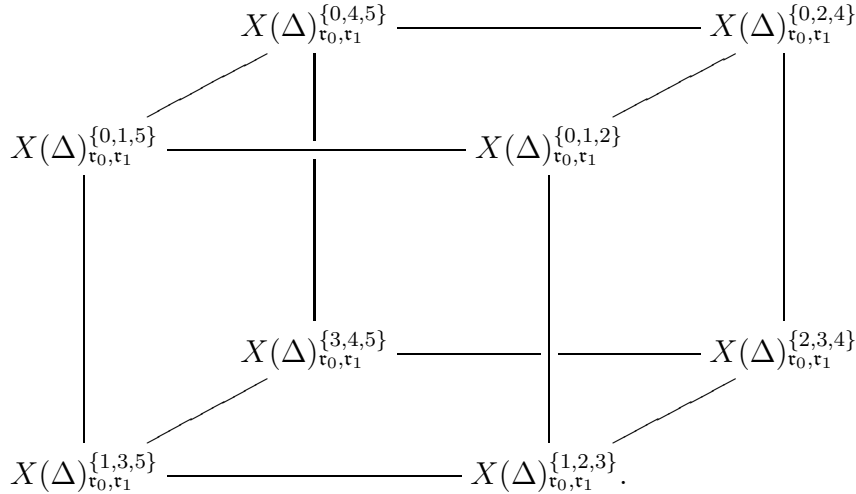
\[\xymatrix{
& X(\Delta)_{\fr_0,\fr_1}^{\{0,4,5\}} \ar@{-}[rr] && X(\Delta)_{\fr_0,\fr_1}^{\{0,2,4\}} \\
X(\Delta)_{\fr_0,\fr_1}^{\{0,1,5\}} \ar@{-}[rr]\ar@{-}[ur]  && X(\Delta)_{\fr_0,\fr_1}^{\{0,1,2\}} \ar@{-}[ur]  \\ &&&\\
& X(\Delta)_{\fr_0,\fr_1}^{\{3,4,5\}} \ar@{-}[rr]|\hole\ar@{-}[uuu]|!{[u],[uu]}\hole && X(\Delta)_{\fr_0,\fr_1}^{\{2,3,4\}} \ar@{-}[uuu] \\
X(\Delta)_{\fr_0,\fr_1}^{\{1,3,5\}} \ar@{-}[rr]\ar@{-}[ur]\ar@{-}[uuu]  && X(\Delta)_{\fr_0,\fr_1}^{\{1,2,3\}}. \ar@{-}[ur]\ar@{-}[uuu]
}\]
Here, every cell (that is, a vertex, an edge, a face, or the cube) represents $X(\Delta)_{\fr_0,\fr_1}^S$ for some ample subset $S$ such that $|S|-3$ equals to the dimension of the cell. For two cells representing $X(\Delta)_{\fr_0,\fr_1}^S$ and $X(\Delta)_{\fr_0,\fr_1}^{S'}$ respectively, the minimal cell containing both of them represents $X(\Delta)_{\fr_0,\fr_1}^{S\cup S'}=X(\Delta)_{\fr_0,\fr_1}^S\cap X(\Delta)_{\fr_0,\fr_1}^{S'}$.
\end{example}

\subsection{Description of strata: primitive case}

We keep the setup in \Sec\ref{bss:reduction}. In this section, we study the strata $X(\Delta)_{\fr_0,\fr_1}^S$ when $S$ is a type (Definition \ref{bde:type}) in more details. From now to the end of this appendix, we will assume that $\fr_1$ is absolutely neat.

\begin{definition}\label{bde:type}
Let $S$ be a subset of $\Phi$.
\begin{enumerate}
  \item We say that $S$ is a \emph{type} if it is ample (Definition \ref{bde:ample_subset}) and $|S|=d$.

  \item We say that $S$ is a \emph{sparse type} if it is a type and such that $\tau\in S$ implies $\sigma\tau\not\in S$.

  \item If $S$ is a type, then we define the function $r_S\colon\Phi\to\{-1,0,1\}$ via the formula
      \[r_S(\tau)=\b{1}_S(\sigma\tau)-\b{1}_S(\tau),\]
      where $\b{1}_S$ denotes the characteristic function of $S$. Put
      \[\Phi_S=r_S^{-1}\{-1\}\subset\Phi,\qquad \Delta_S=\pi(\Phi_S)\subset\Phi_F.\]
\end{enumerate}
\end{definition}

The following lemma is elementary.

\begin{lem}\label{ble:type}
Let $S$ be a type. We have
\begin{enumerate}
  \item $\pi_*r_S=0$; in particular, $\pi$ induces a bijection $\Phi_S\simeq\Delta_S$;

  \item $\Phi_S\subset S$ and they are equal if and only if $S$ is a sparse type;

  \item $|\Phi_S|$ is always odd.
\end{enumerate}
\end{lem}

We review some Hodge theory for abelian schemes in characteristic $\ell$. Let $T$ be a scheme of characteristic $\ell$ and $A$ an abelian scheme over $T$. Put $A^{(\ell)}=A\times_{T,\rF_\abs}T$ as an abelian scheme over $T$ via the second factor, where $\rF_\abs\colon T\to T$ is the absolute $\ell$-Frobenius morphism of $T$. We have the corresponding Frobenius morphism $\r{Fr}_A\colon A\to A^{(\ell)}$ and Verschiebung morphism $\r{Ver}_A\colon A^{(\ell)}\to A$, both over $T$. There is a short exact sequence
\begin{align}\label{beq:hodge}
0 \to \omega(A_{/T}) \to \rH^\dr_1(A_{/T}) \to  \Lie(A_{/T}) \to 0
\end{align}
of locally free $\sO_T$-sheaves on $T$ \footnote{The kernel $\omega(A_{/T})$ is canonically isomorphic to $\omega_{A^\vee/T}$. However, we choose the current notation to keep uniformity and also because the dual abelian scheme $A^\vee$ will never be used.}. We put
\[\cF_A\coloneqq(\r{Ver}_A)_*\colon\rH^\dr_1(A^{(\ell)}_{/T})\to\rH^\dr_1(A_{/T}),\qquad
\cV_A\coloneqq(\r{Fr}_A)_*\colon\rH^\dr_1(A_{/T})\to\rH^\dr_1(A^{(\ell)}_{/T}).\]
We have $\Ker\cF_A=\IM\cV_A=\omega(A^{(\ell)}_{/T})$ and $\Ker\cV_A=\IM\cF_A\simeq\Lie(A^{(\ell)}_{/T})$.

Recall that we have chosen the uniformizer $\Pi$ of $\cO\otimes\dZ_\ell$ and the $O_{F_\fl}$-subalgebra $W\subset\cO\otimes\dZ_\ell$. Suppose that $T$ is an $\dF_{\ell^{2d}}$-scheme and that $A$ is equipped with an action $\iota_A\colon\cO\to\End(A_{/T})$. Then the sequence \eqref{beq:hodge} is the direct sum of sequences
\begin{align*}
0 \to \omega(A_{/T})_\tau \to \rH^\dr_1(A_{/T})_\tau \to  \Lie(A_{/T})_\tau \to 0
\end{align*}
over $\tau\in\Phi$, compatible with \eqref{beq:split}, in which $\rH^\dr_1(A_{/T})_\tau$ has rank $2$. By the canonical isomorphisms \[\rH^\dr_1(A^{(\ell)}_{/T})_\tau\simeq(\rF_\abs^*\rH^\dr_1(A_{/T}))_\tau\simeq\rF_\abs^*\rH^\dr_1(A_{/T})_{\sigma^{-1}\tau},\]
we have the induced maps
\[\cF_A\colon\rF_\abs^*\rH^\dr_1(A_{/T})_{\sigma^{-1}\tau}\to\rH^\dr_1(A_{/T})_\tau,\qquad
\cV_A\colon\rH^\dr_1(A_{/T})_\tau\to\rF_\abs^*\rH^\dr_1(A_{/T})_{\sigma^{-1}\tau}.\]
The action of $\Pi$ induces a map $\Pi_A\colon\rH^\dr_1(A_{/T})_\tau\to \rH^\dr_1(A_{/T})_{\sigma^d\tau}$ for each $\tau$.

In what follows, we will suppress the subscript $T$ if no confusion arises.

\begin{remark}\label{bre:hodge}
When $T=\Spec k$ for a perfect field $k$ (containing $\dF_{\ell^{2d}}$), the canonical map $V\to V\otimes_{k,\sigma}k=\rF_\abs^* V$ is a bijection for every $k$-vector space $V$. Therefore, we will regard $\cF_A$ as a $\sigma$-linear map $\cF_A\colon\rH^\dr_1(A)_{\sigma^{-1}\tau}\to\rH^\dr_1(A)_\tau$ and $\cV_A$ as a $\sigma^{-1}$-linear map $\cV_A\colon\rH^\dr_1(A)_\tau\to\rH^\dr_1(A)_{\sigma^{-1}\tau}$. Then we have $\Ker\cF_A=\IM\cV_A=\omega(A)$ and $\Ker\cV_A=\IM\cF_A\simeq\Lie(A)$.
\end{remark}

For every type $S$, we will define three new moduli functors in Definitions \ref{bde:hilbert_base}, \ref{bde:hilbert_shift} and \ref{bde:hilbert_bundle}, respectively.

\begin{definition}\label{bde:hilbert_base}
Let $S$ be type. We define the moduli functor $Z(\Delta)_{\fr_0,\fr_1}^S$ on the category of schemes over $\Spec\dF_{\ell^{2d}}$ as follows. For every $\dF_{\ell^{2d}}$-scheme $T$,  we let $Z(\Delta)_{\fr_0,\fr_1}^S(T)$ be the set of isomorphism classes of quadruples $(B,\iota_B,C_B,\lambda_B)$ where
\begin{itemize}
  \item $B$ is an abelian scheme over $T$;

  \item $\iota_B\colon\cO\to\End(B)$ is an injective homomorphism satisfying:
      \[\tr(\iota_B(w)\res\Lie(B))=\sum_{\tau\in\Phi}(r_S(\tau)+1)\tau(w)\]
      for every $w\in W$;

  \item $C_B$ and $\lambda_B$ are similar data as in Definition \ref{bde:hilbert_shimura}.
\end{itemize}
Similar to $\cX(\Delta)_{\fr_0,\fr_1}$, it receives actions by $\Cl(F)_{\fr_1}$ and $\dT_\fq$ for $\fq$ coprime to $\Delta$ and $\fr_0\fr_1$.
\end{definition}

\begin{remark}
Like $X(\Delta)_{\fr_0,\fr_1}^S$, the functor $Z(\Delta)_{\fr_0,\fr_1}^S$ does not depend on the choices of $\Pi$ and $W$.
\end{remark}

By \cite{Zin82}*{4.8} (with the remark after it), it is known that $Z(\Delta)_{\fr_0,\fr_1}^S$ is represented by a projective scheme over $\Spec\dF_{\ell^{2d}}$. To describe it, we need to introduce some quaternionic Shimura varieties of non-PEL type. Choose an $O_F$-Eichler order $\cR$ of discriminant $\Delta_S\cup\Delta\setminus\{\fl\}$ and level $\fr_0\fr_1$ together with a surjective $O_F$-linear homomorphism $\cR\to O_F/\fr_1$. Let $\fH_\cR$ be the symmetric hermitian domain attached to $\cR\otimes\dQ$. The complex Shimura variety
\[\cR_\dQ^\times\backslash\fH_\cR\times \widehat\cR_\dQ^\times/\Ker[\widehat\cR^\times\to(O_F/\fr_1)^\times]\]
has a canonical smooth projective model $\cX(\Delta_S\cup\Delta\setminus\{\fl\})_{\fr_0,\fr_1}$ over $\Spec\dZ_{\ell^{2d}}$ (via the fixed embedding $\dZ_{\ell^\infty}\hookrightarrow\dC$). It is of relative dimension $d-|\Delta_S|$ and does not depend on the choice of the data $\cR\to O_F/\fr_1$ up to isomorphism. It receives actions by $\Cl(F)_{\fr_1}$ and $\dT_\fq$ for $\fq$ coprime to $\Delta$ and $\fr_0\fr_1$.
Put \[X(\Delta_S\cup\Delta\setminus\{\fl\})_{\fr_0,\fr_1}=\cX(\Delta_S\cup\Delta\setminus\{\fl\})_{\fr_0,\fr_1}\otimes\dF_{\ell^{2d}}.\]

\begin{lem}\label{ble:zink}
There is an isomorphism
\[Z(\Delta)_{\fr_0,\fr_1}^S\otimes_{\dF_{\ell^{2d}}}\dF_{\ell^\infty}\xrightarrow{\sim}
X(\Delta_S\cup\Delta\setminus\{\fl\})_{\fr_0,\fr_1}\otimes_{\dF_{\ell^{2d}}}\dF_{\ell^\infty}\]
that is compatible with the actions of $\Cl(F)_{\fr_1}$ and $\dT_\fq$ for $\fq$ coprime to $\Delta$ and $\fr_0\fr_1$, and such that the
automorphism $1\times\sigma$ on the right-hand side corresponds to the automorphism $\fl^{|\Delta_S|}\times\sigma$ on the left-hand side. Here, we regard $\fl$ as an element in $\Cl(O_F)_{\fr_1}$.

In particular, $Z(\Delta)_{\fr_0,\fr_1}^S$ is of dimension $0$ if and only if $S$ is sparse (and $d$ is odd).
\end{lem}

\begin{proof}
This is \cite{Zin82}*{4.21}.
\end{proof}

We will also use the following lemma.

\begin{lem}\label{ble:rank}
Suppose that $A$ is an abelian scheme over an $\dF_{\ell^{2d}}$-scheme $T$ equipped with an injective homomorphism $\iota_A\colon\cO\to\End(A)$. We have
\begin{enumerate}
  \item If $(A,\iota_A)$ satisfies the trace condition in Definition \ref{bde:hilbert_shimura}, then $\Pi_A\rH^\dr_1(A)_\tau$ is a subbundle of $\rH^\dr_1(A)_{\sigma^d\tau}$ of rank $1$ for every $\tau\in\Phi$.

  \item If $(A,\iota_A)$ satisfies the trace condition in Definition \ref{bde:hilbert_base}, then $\Pi_A\rH^\dr_1(A)_\tau=0$ if and only if $\tau\in S$.
\end{enumerate}
\end{lem}

\begin{proof}
These are special cases of \cite{Zin82}*{2.6}.
\end{proof}

\begin{definition}\label{bde:hilbert_shift}
Let $S$ be type. We define the moduli functor $X_1(\Delta)_{\fr_0,\fr_1}^S$ on the category of schemes over $\Spec\dF_{\ell^{2d}}$ as follows. For every $\dF_{\ell^{2d}}$-scheme $T$,  we let $X_1(\Delta)_{\fr_0,\fr_1}^S(T)$ be the set of isomorphism classes of data $(A,\iota_A,C_A,\lambda_A;B,\iota_B,C_B,\lambda_B;\phi)$ where
\begin{itemize}
  \item $(A,\iota_A,C_A,\lambda_A)\in X(\Delta)_{\fr_0,\fr_1}^S$;

  \item $(B,\iota_B,C_B,\lambda_B)\in Z(\Delta)_{\fr_0,\fr_1}^S$;

  \item $\phi\colon A\to B$ is an isogeny compatible with other data in the obvious sense and such that
     \begin{itemize}
       \item $\Ker\phi\subset A[\ell]$;
       \item $\Ker[\phi_*\colon\rH^\dr_1(A)\to\rH^\dr_1(B)]=\bigoplus_{\tau\in S}\omega(A)_{\sigma^d\tau}$.
     \end{itemize}
\end{itemize}
Similar to $\cX(\Delta)_{\fr_0,\fr_1}$, the functor receives actions by $\Cl(F)_{\fr_1}$ and $\dT_\fq$ for $\fq$ coprime to $\Delta$ and $\fr_0\fr_1$.
\end{definition}

\begin{proposition}\label{bpr:hilbert_base}
The canonical morphism $p\colon X_1(\Delta)_{\fr_0,\fr_1}^S\to X(\Delta)_{\fr_0,\fr_1}^S$ induced by remembering $(A,\iota_A,C_A,\lambda_A)$ is an isomorphism.
\end{proposition}

\begin{proof}
Note that $X_1(\Delta)_{\fr_0,\fr_1}^S$ is a closed subfunctor of $X(\Delta)_{\fr_0,\fr_1}^S\times_{\Spec\dF_{\ell^{2d}}}Z(\Delta)_{\fr_0,\fr_1}^S$. Thus by Proposition \ref{bpr:zink} and Lemma \ref{ble:zink}, $X_1(\Delta)_{\fr_0,\fr_1}^S$ is represented by a projective scheme. In particular, $p$ is proper. Now we write $X_1=X_1(\Delta)_{\fr_0,\fr_1}^S$ and $X=X(\Delta)_{\fr_0,\fr_1}^S$ for simplicity. Since $X$ is smooth by Proposition \ref{bpr:zink}, it suffices to show that
\begin{enumerate}
  \item For every algebraically closed field $k$ containing $\dF_{\ell^{2d}}$, the map $p\colon X_1(k)\to X(k)$ is a bijection. In particular, for every closed point $x$ of $X$, there is a unique closed point $x_1$ of $X_1$ such that $p(x_1)=x$ and that $p$ induces an isomorphism on residue fields.

  \item In (1), the induced map $p_*\colon\sT_{X_1,x_1}\otimes_{\sO_{X_1,x_1}}k(x_1)\to\sT_{X,x}\otimes_{\sO_{X,x}}k(x)$ on tangent spaces is injective for every closed point $x$ of $X$.
\end{enumerate}
For (1), the injectivity is clear. To show the surjectivity, it suffices to show that for every element $(A,\iota_A,C_A,\lambda_A)\in X(k)$,
\begin{enumerate}[label=(\alph*)]
  \item the subspace $K\coloneqq\bigoplus_{\tau\in S}\omega(A)_{\sigma^d\tau}$ is stable under $\cF_A$, $\cV_A$ and $\iota_A(\Pi)_*$;

  \item if $\phi\colon A\to B$ is the unique up to isomorphism isogeny such that $\Ker\phi\subset A[\ell]$ and $\Ker[\phi_*\colon\rH^\dr_1(A)\to\rH^\dr_1(B)]=K$, then $\dim_k\Lie(B)_\tau=r_S(\tau)+1$ for every $\tau\in\Phi$.
\end{enumerate}

We first consider (a). Note that for every $\tau\in\Phi$, the kernel of $\Pi_A\colon\rH^\dr_1(A)_\tau\to\rH^\dr_1(A)_{\sigma^d\tau}$ coincides with the image of $\Pi_A\colon\rH^\dr_1(A)_{\sigma^d\tau}\to\rH^\dr_1(A)_\tau$, which has dimension $1$ by Lemma \ref{ble:rank} (1).

For $\cF_A$: We have $\cF_A K=0$ as $\omega(A)=\Ker\cF_A$ (Remark \ref{bre:hodge}).

For $\Pi_A$: If $\tau\in S$, then the image of $\Pi_A\colon\rH^\dr_1(A)_\tau\to\rH^\dr_1(A)_{\sigma^d\tau}$ coincides with $\omega(A)_{\sigma^d\tau}$. Thus $\Pi_A\omega(A)_{\sigma^d\tau}=0$ and $\Pi_A K=0$.

For $\cV_A$: If $\sigma^{-1}\tau\in S$, then $\cV_A\omega(A)_{\sigma^d\tau}\subset\omega(A)_{\sigma^{d-1}\tau}\subset K$. If $\sigma^{-1}\tau\not\in S$, then we claim that $\cV_A\omega(A)_{\sigma^d\tau}=0$. This is equivalent to showing that in the following commutative diagram
\[\xymatrix{
\rH^\dr_1(A)_\tau \ar[r]^-{\cV_A}\ar[d]^-{\Pi_A} & \rH^\dr_1(A)_{\sigma^{-1}\tau}  \ar[d]^-{\Pi_A}\\
\rH^\dr_1(A)_{\sigma^d\tau} \ar[r]^-{\cV_A} & \rH^\dr_1(A)_{\sigma^{d-1}\tau}
}\]
$\cV_A\Pi_A\rH^\dr_1(A)_\tau=0$. However, we have $\Pi_A\circ\cV_A=0$ since $\cV_A\rH^\dr_1(A)_\tau=\omega(A)_{\sigma^{-1}\tau}$ is the image of $\Pi_A\colon\rH^\dr_1(A)_{\sigma^{d-1}\tau}\to\rH^\dr_1(A)_{\sigma^{-1}\tau}$ as $\sigma^{d-1}\tau\in S$.

Now we consider (b). Let $\rH^\cris_1(A)=\bigoplus_{\tau\in\Phi}\rH^\cris_1(A)_\tau$ be the ``crystal homology'' (same as the covariant Dieudonn\'{e} module) of $A$ with lifts of the two maps $\cF_A$ and $\cV_A$ such that $\cF_A\cV_A=\cV_A\cF_A=\ell$, and similarly for $\rH^\cris_1(B)$. We have the following formulae:
\begin{itemize}
  \item $\dim_k\Lie(B)_\tau=\dim_k\cF_B\rH^\dr_1(B)_\tau=2-\dim_k(\rH^\dr_1(B)_{\sigma\tau}/\cF_B\rH^\dr_1(B)_\tau)$;
  \item $\dim_k(\rH^\dr_1(B)_{\sigma\tau}/\cF_B\rH^\dr_1(B)_\tau)=\dim_k(\rH^\cris_1(B)_{\sigma\tau}/\cF_B\rH^\cris_1(B)_\tau)$;
  \item $\dim_k(\rH^\cris_1(B)_{\sigma\tau}/\phi_*\rH^\cris_1(A)_{\sigma\tau})=\b{1}_S(\sigma^{d+1}\tau)$;
  \item $\dim_k(\phi_*\rH^\cris_1(A)_{\sigma\tau}/\phi_*\cF_A\rH^\cris_1(A)_\tau)=\dim_k(\rH^\cris_1(A)_{\sigma\tau}/\cF_A\rH^\cris_1(A)_\tau)=1$;
  \item $\dim_k(\cF_B\rH^\cris_1(B)_\tau/\phi_*\cF_A\rH^\cris_1(A)_\tau)=\dim_k(\rH^\cris_1(B)_\tau/\phi_*\rH^\cris_1(A)_\tau)=\b{1}_S(\sigma^d\tau)$.
\end{itemize}
Therefore, $\dim_k\Lie(B)_\tau=2-(\b{1}_S(\sigma^{d+1}\tau)+1-\b{1}_S(\sigma^d\tau))=r_S(\tau)+1$. Part (1) is proved.

For (2), we fix a closed point $x$ of $X$ represented by $(A,\iota_A,C_A,\lambda_A)$ with the (perfect) residue field $k$. Let $(A,\iota_A,C_A,\lambda_A;B,\iota_B,C_B,\lambda_B;\phi)$ be the element of $X_1$ representing the unique point above $x$. Put $k^\sharp=k[\epsilon]/\epsilon^2$. By the Grothendieck--Messing theory, to give an element in the tangent space $\sT_{X,x}\otimes_{\sO_{X,x}}k$ at $x$ is equivalent to giving a direct factor $\omega(A)_\tau^\sharp$ of the $k^\sharp$-module $\rH^\dr_1(A)_\tau^\sharp\coloneqq\rH^\dr_1(A)_\tau\otimes_kk^\sharp$ lifting $\omega(A)_\tau$ for every $\tau\in\Phi$ such that $\omega(A)_{\sigma^d\tau}^\sharp=\Pi_A\rH^\dr_1(A)_\tau^\sharp$ for $\tau\in S$. In other words, it is equivalent to specifying a lifting $\omega(A)_\tau^\sharp$ for $\tau\in S$.

Therefore, for the injectivity, it amounts to showing that for every given lifting $\{\omega(A)_\tau^\sharp\res\tau\in S\}$, there is at most one lifting $\{\omega(B)_\tau^\sharp\res\tau\in\Phi\}$ such that
\begin{itemize}
  \item $\Pi_B\omega(B)_\tau^\sharp\subset \omega(B)_{\sigma^d\tau}^\sharp$ for $\tau\in\Phi$,
  \item $\phi_*\omega(A)_\tau^\sharp\subset\omega(B)_\tau^\sharp$ for $\tau\in\Phi$.
\end{itemize}
The only situation that requires discussion is when we have $\dim_k\omega(B)_\tau=1$, that is, when $\b{1}_S(\tau)=\b{1}_S(\sigma\tau)$. There are two cases. First, if $\tau\in S$, then $\phi_*\colon\rH^\dr_1(A)_\tau^\sharp\to\rH^\dr_1(B)_\tau^\sharp$ is an isomorphism, hence $\omega(B)_\tau^\sharp=\phi_*\omega(A)_\tau^\sharp$. Second, if $\tau\not\in S$, then $\omega(B)_{\sigma^d\tau}^\sharp$ is unique, hence $\omega(B)_\tau^\sharp$ is uniquely determined as $\Pi_B\colon\rH^\dr_1(B)_\tau^\sharp\to\rH^\dr_1(B)_{\sigma^d\tau}^\sharp$ is an isomorphism by Lemma \ref{ble:rank} (2). Part (2) hence the proposition are proved.
\end{proof}

\begin{definition}\label{bde:hilbert_bundle}
Let $S$ be type. We define the moduli functor $Z_1(\Delta)_{\fr_0,\fr_1}^S$ on the category of schemes over $\Spec\dF_{\ell^{2d}}$ as follows. For every $\dF_{\ell^{2d}}$-scheme $T$,  we let $Z_1(\Delta)_{\fr_0,\fr_1}^S(T)$ be the set of isomorphism classes of data $(B,\iota_B,C_B,\lambda_B;(L_\tau)_{\tau\in\Phi_S})$ where
\begin{itemize}
  \item $(B,\iota_B,C_B,\lambda_B)\in Z(\Delta)_{\fr_0,\fr_1}^S$;

  \item $L_\tau$ is a subbundle of $\rH^\dr_1(B)_\tau$ of rank $1$ for each $\tau\in\Phi_S$.
\end{itemize}
Similar to $\cX(\Delta)_{\fr_0,\fr_1}$, the functor receives actions by $\Cl(F)_{\fr_1}$ and $\dT_\fq$ for $\fq$ coprime to $\Delta$ and $\fr_0\fr_1$.
\end{definition}

\begin{lem}\label{ble:hilbert_bundle}
Let $(\cB,\iota_\cB,C_\cB,\lambda_\cB)$ be the universal object over $Z(\Delta)_{\fr_0,\fr_1}^S$. Then the morphism $Z_1(\Delta)_{\fr_0,\fr_1}^S\to Z(\Delta)_{\fr_0,\fr_1}^S$ induced by remembering $(B,\iota_B,C_B,\lambda_B)$ induces an isomorphism
\[Z_1(\Delta)_{\fr_0,\fr_1}^S\simeq\prod_{\tau\in\Phi_S}\dP^1(\rH^\dr_1(\cB)_\tau)\]
where the product denotes the fiber product over $Z(\Delta)_{\fr_0,\fr_1}^S$.
\end{lem}

\begin{proof}
This is straightforward from Definition \ref{bde:hilbert_bundle}.
\end{proof}

\begin{proposition}\label{bpr:hilbert_bundle}
The morphism $q\colon X_1(\Delta)_{\fr_0,\fr_1}^S\to Z_1(\Delta)_{\fr_0,\fr_1}^S$ induced by the assignment \[(A,\iota_A,C_A,\lambda_A;B,\iota_B,C_B,\lambda_B;\phi)\mapsto (B,\iota_B,C_B,\lambda_B;(\phi_*\omega(A)_\tau)_{\tau\in\Phi_S})\]
is an isomorphism.
\end{proposition}

\begin{proof}
First note that for $\tau\in\Phi_S$, the map $\phi_*\colon\rH^\dr_1(A)_\tau\to\rH^\dr_1(B)_\tau$ is an isomorphism. Thus $\phi_*\omega(A)_\tau$ is a subbundle of rank $1$. The map $q$ is well-defined.

Note that the target of $q$ is a smooth projective scheme by Lemmas \ref{ble:zink} and \ref{ble:hilbert_bundle}, and $q$ is a proper morphism of schemes. So we can use the same strategy as for the proof of Proposition \ref{bpr:hilbert_base}. Write $X_1=X_1(\Delta)_{\fr_0,\fr_1}^S$ and $Z_1=Z_1(\Delta)_{\fr_0,\fr_1}^S$ for simplicity. The proposition will follow from
\begin{enumerate}
  \item For every algebraically closed field $k$ containing $\dF_{\ell^{2d}}$, the map $q\colon X_1(k)\to Y_1(k)$ is a bijection.

  \item The induced map $q_*\colon\sT_{X_1,x}\otimes_{\sO_{X_1,x}}k(x)\to\sT_{Z_1,z}\otimes_{\sO_{Z_1,z}}k(z)$ on tangent spaces is injective for every closed point $x$ of $X_1$ and $z=q(x)$.
\end{enumerate}

For (1), we will construct a map $q'\colon Z_1(k)\to X_1(k)$ inverse to the map $q\colon X_1(k)\to Z_1(k)$. Let $(B,\iota_B,C_B,\lambda_B;(L_\tau)_{\tau\in\Phi_S})$ be an element in $Z_1(k)$. We define $M=\bigoplus_{\tau\in S}M_\tau$ such that
\begin{itemize}
  \item $M_\tau=\rH^\dr_1(B)_\tau$ if $\tau\in S$;
  \item $M_\tau=\cF_B\rH^\dr_1(B)_{\sigma^{-1}\tau}$ if $\tau\not\in S$ and $\sigma^{-1}\tau\not\in S$ (Remark \ref{bre:hodge});
  \item $\cV_B M_\tau=L_{\sigma^{-1}\tau}$ if $\tau\not\in S$ and $\sigma^{-1}\tau\in S$ (that is, $\sigma^{-1}\tau\in\Phi_S$). Note that $\cV_B\colon\rH^1(B)_\tau\to\rH^1(B)_{\sigma^{-1}\tau}$ has the trivial kernel if $\sigma^{-1}\tau\in\Phi_S$.
\end{itemize}
We show that $M$ is stable under $\cF_B$, $\cV_B$ and $\Pi_B$. The stability under $\Pi_B$ is obvious by Lemma \ref{ble:rank} (2).

For $\cF_B$: If $\tau\in S$ and $\sigma\tau\not\in S$, then $\cF_B M_\tau\subset\cF_A\rH^\dr_1(B)_\tau=0$. If $\tau\not\in S$ and $\sigma\tau\not\in S$, then $\cF_B \rH^\dr_1(B)_\tau=M_{\sigma\tau}$. If $\sigma\tau\in S$, then $\cF_B\rH^\dr_1(B)_\tau\subset\rH^\dr_1(B)_{\sigma\tau}=M_{\sigma\tau}$.

For $\cV_B$: If $\tau\in S$ and $\sigma^{-1}\tau\not\in S$, then $\cV_B M_\tau\subset\cV_B\rH^\dr_1(B)_\tau=0$. If $\tau\not\in S$ and $\sigma^{-1}\tau\not\in S$, then $\cV_B M_\tau=\cV_B\cF_B\rH^\dr_1(B)_{\sigma^{-1}\tau}=0$. If $\sigma^{-1}\tau\in S$, then $\cV_B\rH^\dr_1(B)_\tau\subset\rH^\dr_1(B)_{\sigma^{-1}\tau}=M_{\sigma^{-1}\tau}$.

Let $\phi'\colon B\to A$ be the unique isogeny such that $\Ker\phi'\subset B[\ell]$ and $\Ker[\phi'_*\colon\rH^\dr_1(B)\to\rH^\dr_1(A)]=M$. Then $A$ inherits an action $\iota_A\colon\cO\to\End(A)$. Since $\Ker\phi'\subset B[\ell]$, there is a unique ($\cO$-equivariant) isogeny $\phi\colon A\to B$ such that $\phi\circ\phi'=[\ell]_B$. As the restricted map $\phi\colon A[\fr_0\fr_1]\to B[\fr_0\fr_1]$ is an isomorphism, $(C_B,\lambda_B)$ canonically induce $(C_A,\lambda_A)$ for $A$. Thus we obtain the data $(A,\iota_A,C_A,\lambda_A;B,\iota_B,C_B,\lambda_B;\phi)$. We claim that it is an element of $X_1(k)$, which follows from the three claims below.
\begin{enumerate}[label=(\alph*)]
  \item $\dim_k\Lie(A)_\tau=1$ for every $\tau\in\Phi$.

  \item $\Ker[\phi_*\colon\rH^\dr_1(A)\to\rH^\dr_1(B)]=\bigoplus_{\tau\in S}\omega(A)_{\sigma^d\tau}$.

  \item $\Pi_A\colon\Lie(A)_\tau\to\Lie(A)_{\sigma^d\tau}$ vanishes for $\tau\in S$.
\end{enumerate}

For (a), we use the same computation as in the proof of Proposition \ref{bpr:hilbert_base} and leave the details to the reader.

For (b), by construction, $\phi_*\colon\rH^\dr_1(A)_\tau\to\rH^\dr_1(B)_\tau$ is an isomorphism if $\tau\in S$. If $\tau\not\in S$ and $\sigma\tau\in S$, we have the following commutative diagram
\[\xymatrix{
\rH^\dr(A)_\tau \ar[r]^-{\cF_A}\ar[d]^-{\phi_*} & \rH^\dr_1(A)_{\sigma\tau} \ar[d]^-{\phi_*}_-{\simeq} \\
\rH^\dr(B)_\tau \ar[r]^-{\cF_B} & \rH^\dr_1(B)_{\sigma\tau}
}\]
in which $\cF_B$ has the trivial kernel. We have that $\Ker[\phi_*\colon\rH^\dr_1(A)_\tau\to\rH^\dr_1(B)_\tau]=(\Ker\cF_A)_\tau=\omega(A)_\tau$. If $\tau\not\in S$ and $\sigma\tau\not\in S$, then from the following commutative diagram
\[\xymatrix{
\rH^\dr_1(A)_{\sigma\tau} \ar[r]^-{\cV_A}\ar[d]^-{\phi_*} & \rH^\dr_1(A)_\tau \ar[d]^-{\phi_*} \\
\rH^\dr_1(B)_{\sigma\tau} \ar[r]^-{\cV_B} & \rH^\dr_1(B)_\tau
}\]
and $\phi_*\rH^\dr_1(A)_{\sigma\tau}=M_{\sigma\tau}=\cF_B\rH^\dr_1(B)_\tau$, we know that $\phi_*\cV_A\rH^\dr_1(A)_{\sigma\tau}=0$. In other words, $\Ker[\phi_*\colon\rH^\dr_1(A)\to\rH^\dr_1(B)]=\omega(A)_\tau$. We obtain (b).

For (c), if $\tau\in S$, we have the commutative diagram
\[\xymatrix{
\rH^\dr_1(A)_\tau \ar[r]^-{\Pi_A}\ar[d]^-{\phi_*}_\simeq  & \rH^\dr_1(A)_{\sigma^d\tau} \ar[d]^-{\phi_*}  \\
\rH^\dr_1(B)_\tau \ar[r]^-{\Pi_B}_-0  & \rH^\dr_1(B)_{\sigma^d\tau}
}\]
which implies that $\Pi_A\rH^\dr_1(A)_\tau=\Ker[\phi_*\colon\rH^\dr_1(A)_{\sigma^d\tau}\to\rH^\dr_1(B)_{\sigma^d\tau}]=\omega(A)_{\sigma^d\tau}$ by (b). Thus (c) follows.

In all, we have constructed a map $q'\colon Z_1(k)\to X_1(k)$ sending $(B,\iota_B,C_B,\lambda_B;(L_\tau)_{\tau\in\Phi_S})$ to $(A,\iota_A,C_A,\lambda_A;B,\iota_B,C_B,\lambda_B;\phi)$ as above. Now we show that $q'$ is an inverse of $q$.

To show that $q\circ q'=\r{id}_{Z_1(k)}$, it suffices to show that for $(A,\iota_A,C_A,\lambda_A;B,\iota_B,C_B,\lambda_B;\phi)=q'(B,\iota_B,C_B,\lambda_B;(L_\tau)_{\tau\in\Phi_S})$, we have $\phi_*\omega(A)_\tau=L_\tau$ for $\tau\in\Phi_S$. However, we have $\phi_*\omega(A)_\tau=\phi_*\cV_A\rH^1_\dr(A)_{\sigma\tau}=\cV_B\phi_*\rH^1_\dr(A)_{\sigma\tau}=\cV_B M_\tau=L_\tau$.

To show that $q'\circ q=\r{id}_{X_1(k)}$, we consider an element $(A,\iota_A,C_A,\lambda_A;B,\iota_B,C_B,\lambda_B;\phi)\in X_1(k)$. It suffices to show that $\phi_*\rH^\dr_1(A)_\tau=M_\tau$ for $\tau\not\in S$, where $M_\tau=\cF_B\rH^\dr_1(B)_{\sigma^{-1}\tau}$ if $\sigma^{-1}\tau\not\in S$ and $\cV_B M_\tau=\phi_*\omega(A)_{\sigma^{-1}\tau}$ if $\sigma^{-1}\tau\in S$.

For the first case, it suffices to show that $\cV_B\phi_*\rH^\dr_1(A)_\tau=0$. By the property of $\phi$, we have $\phi_*\omega(A)_{\sigma^{-1}\tau}=0$, that is, $\phi_*\cV_A\rH^\dr_1(A)_\tau=0$ and we are done.

The second case follows from the commutative diagram
\[\xymatrix{
\rH^\dr_1(A)_\tau  \ar[r]^-{\cV_A}\ar[d]^-{\phi_*}  & \rH^\dr_1(A)_{\sigma^{-1}\tau} \ar[d]^-{\phi_*}_-{\simeq} \\
\rH^\dr_1(B)_\tau  \ar[r]^-{\cV_B}  & \rH^\dr_1(B)_{\sigma^{-1}\tau}
}\]
in which $\cV_B$ has the trivial kernel, and the identity $\omega(A)_{\sigma^{-1}\tau}=\cV_A\rH^\dr_1(A)_\tau$. Part (1) has been proved.

For (2), we fix a closed point $z$ of $Z_1$ represented by the data $(B,\iota_B,C_B,\lambda_B;(L_\tau)_{\tau\in\Phi_S})$ with the (perfect) residue field $k$. Let $(A,\iota_A,C_A,\lambda_A;B,\iota_B,C_B,\lambda_B;\phi)$ be the element of $X_1$ representing the unique point above $z$. Put $k^\sharp=k[\epsilon]/\epsilon^2$. To give an element in the tangent space $\sT_{Z_1,z}\otimes_{\sO_{Z_1,z}}k$ at $z$ is equivalent to giving a direct factor $\omega(B)_\tau^\sharp$ of the $k^\sharp$-module $\rH^\dr_1(B)_\tau^\sharp\coloneqq\rH^\dr_1(B)_\tau\otimes_kk^\sharp$ lifting $\omega(B)_\tau$ for every $\tau\in\Phi$ such that $\Pi_B\omega(B)_\tau^\sharp\subset\omega(B)_{\sigma^d\tau}^\sharp$, and a direct factor $L_\tau^\sharp$ of $\rH^\dr_1(B)_\tau^\sharp$ lifting $L_\tau$ for every $\tau\in\Phi_S$.

For the injectivity, it suffices to show that for every given lifting $\{\omega(B)_\tau\res\tau\in\Phi\}\cup\{L_\tau^\sharp\res\tau\in\Phi_S\}$, there is at most one lifting $\{\omega(A)_\tau^\sharp\res\tau\in\Phi\}$ such that
\begin{itemize}
  \item $\Pi_A\rH^1_\dr(A)_\tau^\sharp=\omega(A)_{\sigma^d\tau}^\sharp$ for $\tau\in S$,
  \item $\phi_*\omega(A)_\tau^\sharp=L_\tau^\sharp$ for $\tau\in\Phi_S$,
  \item $\phi_*\omega(A)_\tau^\sharp\subset\omega(B)_\tau^\sharp$ for $\tau\in\Phi$.
\end{itemize}
In fact, if $\tau\not\in S$, then $\omega(A)_\tau^\sharp$ must be $\Pi_A\rH^1_\dr(A)_{\sigma^d\tau}^\sharp$. If $\tau\in\Phi_S$, then $\omega(A)_\tau^\sharp$ is the unique direct summand such that $\phi_*\omega(A)_\tau^\sharp=L_\tau^\sharp$. If $\tau\in S\setminus\Phi_S$, then $\phi_*\colon\rH^1_\dr(A)_\tau^\sharp\to\rH^1_\dr(B)_\tau^\sharp$ is an isomorphism and $\omega(B)_\tau^\sharp$ is of rank $1$. Thus, $\omega(A)_\tau^\sharp$ is uniquely determined. Part (2) hence the proposition are proved.
\end{proof}

The following corollary is a consequence of Propositions \ref{bpr:hilbert_base}, \ref{bpr:hilbert_bundle}, and Lemma \ref{ble:hilbert_bundle}.

\begin{corollary}\label{bco:hilbert_bundle}
Let $S$ be a type. The isomorphisms $p$ and $q$ induce a canonical isomorphism
\[X(\Delta)_{\fr_0,\fr_1}^S\simeq\prod_{\tau\in\Phi_S}\dP^1(\rH^\dr_1(\cB)_\tau).\]
\end{corollary}

\subsection{Description of strata: general case}

Now we consider general ample subsets.

\begin{notation}\label{bno:dag_subset}
Let $S$ be a proper ample subset of $\Phi$.
\begin{enumerate}
  \item Let $S^\dag$ be the unique type contained in $S$ such that for $\tau\in S\setminus S^\dag$, we have $\sigma\tau\in S^\dag$.
  \item Let $\Phi_S$ be the subset of $\Phi_{S^\dag}$ consisting of $\tau$ such that $\sigma^d\tau\not\in S$.
\end{enumerate}
\end{notation}

\begin{notation}\label{bno:hilbert_subbundle}
We denote by $\wp_S$ the following composite morphism
\[X(\Delta)_{\fr_0,\fr_1}^S\to X(\Delta)_{\fr_0,\fr_1}^{S^\dag} \xrightarrow{q\circ p^{-1}}
Z_1(\Delta)_{\fr_0,\fr_1}^{S^\dag}\to Z(\Delta)_{\fr_0,\fr_1}^{S^\dag}\]
where the last morphism is the canonical projection in Lemma \ref{ble:hilbert_bundle}.
\end{notation}

\begin{definition}\label{bde:hilbert_subbundle}
Let $S$ be a proper ample subset of $\Phi$. We define a subfunctor $Z_1(\Delta)_{\fr_0,\fr_1}^S$ of $Z_1(\Delta)_{\fr_0,\fr_1}^{S^\dag}$ such that the data $(B,\iota_B,C_B,\lambda_B;(L_\tau)_{\tau\in\Phi_{S^\dag}})$ satisfy
\begin{enumerate}[label=(\roman*)]
  \item $L_\tau=\cF_B\rF_\abs^*\rH^1(B)_{\sigma^{-1}\tau}$ if $\tau\in\Phi_{S^\dag}\setminus\Phi_S$ and $\sigma^{-1}\tau\in S^\dag$;
  \item $L_\tau=\Pi_B\cV_{B,\sigma^d\tau}^{-1}\rF_\abs^*L_{\sigma^{d-1}\tau}$ if $\tau\in\Phi_{S^\dag}\setminus\Phi_S$ and $\sigma^{d-1}\tau\in S^\dag$.
\end{enumerate}
Note that in (ii), we have $\sigma^{d-1}\tau\in\Phi_{S^\dag}$ and $\cV_{B,\sigma^d\tau}\coloneqq\cV_B\colon\rH^\dr_1(B)_{\sigma^d\tau}\to\rF_\abs^*\rH^\dr_1(B)_{\sigma^{d-1}\tau}$ is an isomorphism.
\end{definition}

\begin{theorem}\label{bth:hilbert_subbundle}
Let $S$ be a proper ample subset of $\Phi$. Then
\begin{enumerate}
  \item The composite morphism
     \[Z_1(\Delta)_{\fr_0,\fr_1}^S\to Z_1(\Delta)_{\fr_0,\fr_1}^{S^\dag}\xrightarrow{\sim}\prod_{\tau\in\Phi_{S^\dag}}\dP^1(\rH^\dr_1(\cB)_\tau)
     \to\prod_{\tau\in\Phi_S}\dP^1(\rH^\dr_1(\cB)_\tau)\]
     is an isomorphism, where the middle isomorphism is due to Lemma \ref{ble:hilbert_bundle} and the last morphism is the canonical projection.

  \item Under the isomorphism $q\circ p^{-1}\colon X(\Delta)_{\fr_0,\fr_1}^{S^\dag}\xrightarrow{\sim}Z_1(\Delta)_{\fr_0,\fr_1}^{S^\dag}$, the image of $X(\Delta)_{\fr_0,\fr_1}^S$ is contained in $Z_1(\Delta)_{\fr_0,\fr_1}^S$, and the following diagram
     \[\xymatrix{
     X(\Delta)_{\fr_0,\fr_1}^S  \ar[r]\ar[d] & Z_1(\Delta)_{\fr_0,\fr_1}^S \ar[d]\\
     X(\Delta)_{\fr_0,\fr_1}^{S^\dag} \ar[r]^-{q\circ p^{-1}} & Z_1(\Delta)_{\fr_0,\fr_1}^{S^\dag}
     }\]
     is Cartesian.

  \item The morphism $\wp_S\colon X(\Delta)_{\fr_0,\fr_1}^S\to Z(\Delta)_{\fr_0,\fr_1}^{S^\dag}$ is a $(\dP^1)^{|\Phi_S|}$-bundle. It is compatible with the actions by $\Cl(F)_{\fr_1}$ and $\dT_\fq$ for $\fq$ coprime to $\Delta$ and $\fr_0\fr_1$.
\end{enumerate}
\end{theorem}

\begin{proof}
Note that in Definition \ref{bde:hilbert_subbundle} (ii), the restricted map $\Pi_B\colon\rH^\dr_1(B)_{\sigma^d\tau}\to\rH^\dr_1(B)_\tau$ is an isomorphism. Part (1) is an elementary consequence of the fact that $S\neq\Phi$.

For (2), it suffices to show that for an element $(A,\iota_A,C_A,\lambda_A;B,\iota_B,C_B,\lambda_B;\phi)\in X_1(\Delta)_{\fr_0,\fr_1}^{S^\dag}(T)$ with some $\dF_{\ell^{2d}}$-scheme $T$, the collection $\{\phi_*\omega(A)_\tau\res\tau\in S^\dag\}$ satisfies the two conditions in Definition \ref{bde:hilbert_subbundle} if and only if $(A,\iota_A,C_A,\lambda_A)$ belongs to $X(\Delta)_{\fr_0,\fr_1}^S(T)$. First we note that $\sigma^d$ induces a bijection from $\Phi_{S^\dag}\setminus\Phi_S$ to $S\setminus S^\dag$. We show that the condition for $\tau\in\Phi_{S^\dag}\setminus\Phi_S$ in Definition \ref{bde:hilbert_subbundle} is satisfied if and only if $\Pi_A\colon\Lie(A)_{\sigma^d\tau}\to\Lie(A)_\tau$ vanishes.

In case (i), we look at the following commutative diagram
\[\xymatrix{
\rF_\abs^*\rH^\dr_1(A)_{\sigma^{-1}\tau}  \ar[r]^-{\cF_A}\ar[d]^-{\rF_\abs^*\phi_*}_-\simeq & \rH^\dr_1(A)_\tau \ar[d]^-{\phi_*}_-\simeq \\
\rF_\abs^*\rH^\dr_1(B)_{\sigma^{-1}\tau}  \ar[r]^-{\cF_B} & \rH^\dr_1(B)_\tau
}\]
where both vertical maps are isomorphisms. As $\sigma^{-1}\tau\in S$, we have $\Pi_A\cF_A\rF_\abs^*\rH^\dr_1(A)_{\sigma^{-1}\tau}=0$. Thus $\phi_*\omega(A)_\tau=\cF_B\rF_\abs^*\rH^1(B)_{\sigma^{-1}\tau}$ if and only if $\omega(A)_\tau=\cF_A\rF_\abs^*\rH^\dr_1(A)_{\sigma^{-1}\tau}$ that is, $\Pi_A\omega(A)_\tau=0$. However, this is equivalent to the vanishing of $\Pi_A\colon\Lie(A)_{\sigma^d\tau}\to\Lie(A)_\tau$.

In case (ii), we look at the following commutative diagram
\[\xymatrix{
\rF_\abs^*\rH^\dr_1(A)_{\sigma^{d-1}\tau}  \ar[d]^-{\rF_\abs^*\phi_*}_-\simeq & \rH^\dr_1(A)_{\sigma^d\tau}
\ar[d]^-{\phi_*}\ar[l]_-{\cV_A}\ar[r]^-{\Pi_A} & \rH^\dr_1(A)_\tau \ar[d]^-{\phi_*}_-\simeq \\
\rF_\abs^*\rH^\dr_1(B)_{\sigma^{d-1}\tau}  & \rH^\dr_1(B)_{\sigma^d\tau} \ar[l]_-{\cV_B}\ar[r]^-{\Pi_B}_-\simeq & \rH^\dr_1(B)_\tau
}\]
where the two vertical maps on sides are isomorphisms. The map $\Pi_A\colon\Lie(A)_{\sigma^d\tau}\to\Lie(A)_\tau$ vanishes if and only if $\Pi_A\rH^\dr_1(A)_{\sigma^d\tau}=\omega(A)_\tau$, that is, $\phi_*\omega(A)_\tau=\Pi_B\phi_*\Lie(A)_{\sigma^d\tau}$. Now it suffices to show that $\cV_B\phi_*\Lie(A)_{\sigma^d\tau}=\rF_\abs^*\phi_*\omega(A)_{\sigma^{d-1}\tau}$. However, this is obvious since $\rF_\abs^*\omega(A)_{\sigma^{d-1}\tau}=\cV_A\rH^\dr_1(A)_{\sigma^d\tau}$. Part (2) is proved.

Part (3) follows from (1) and (2), in which the compatibility is obvious.
\end{proof}

\begin{remark}\label{bre:zink}
If we choose a type $S'\subset S$ other than $S^\dag$, then the composite morphism $X(\Delta)_{\fr_0,\fr_1}^S\to X(\Delta)_{\fr_0,\fr_1}^{S'}\to Z(\Delta)_{\fr_0,\fr_1}^{S'}$ is only a $(\dP^1)^{|\Phi_S|}$-bundle up to a Frobenius factor (\cite{Hel12}*{Definition 4.7} or Definition \ref{bde:frobenius} below) in general (see Proposition \ref{bpr:translation} for example).

In \cite{Zin82}, the author only proved that $X(\Delta)_{\fr_0,\fr_1}^S$ is a $(\dP^1)^{|\Phi_S|}$-bundle over  $Z(\Delta)_{\fr_0,\fr_1}^{S^\dag}$ \emph{up to a Frobenius factor} for certain $S^\dag$ he chose.
\end{remark}

\begin{corollary}\label{bco:translation}
Let $S$ be a type that is not sparse. For a proper ample subset $\tilde{S}$ with $\tilde{S}^\dag=S$, the morphism $\wp_{\tilde{S}}$ is an isomorphism if and only if $\tilde{S}=S\cup\sigma^{-1}S$. In particular, the morphism $\wp_S$ has a canonical section which is $\wp_{S\cup\sigma^{-1}S}^{-1}$.
\end{corollary}

\begin{proof}
Note that $S\cup\sigma^{-1}S$ is the unique proper ample set such that $(S\cup\sigma^{-1}S)^\dag=S$ and $\Phi_{S\cup\sigma^{-1}S}=\emptyset$. Then the corollary follows from Theorem \ref{bth:hilbert_subbundle}.
\end{proof}

\begin{proposition}\label{bpr:hilbert_sparse}
Suppose that $d=\deg F$ is odd. Let $S^+$ (resp.\ $S^-$) be the unique sparse type that contains (resp.\ does not contain) the distinguished element in $\Phi$ coming the orientation of $\cO$ at $\fl$. Let $k$ be an algebraically closed field containing $\dF_{\ell^{2d}}$.
\begin{enumerate}
  \item There are canonical $\Gal(k/\dF_{\ell^{2d}})$-invariant isomorphisms
      \[Z(\Delta)_{\fr_0,\fr_1}^{S^\pm}(k)/\Cl(F)_{\fr_1}\xrightarrow{\sim}\cS(\Phi_F\cup\Delta\setminus\{\fl\})_{\fr_0\fr_1}.\]
      Moreover, the maps $Z(\Delta)_{\fr_0,\fr_1}^{S^\pm}(k)/\Cl(F)_{\fr_1}\to Z(\Delta)_{\fr_0,\fr_1}^{S^\mp}(k)/\Cl(F)_{\fr_1}$ induced by $\sigma_{S^\pm}$ (Remark \ref{bre:frobenius}) coincide with the identity map on the right-hand side under these isomorphisms.

  \item There is a canonical $\Gal(k/\dF_{\ell^{2d}})$-invariant isomorphism
      \[X(\Delta)_{\fr_0,\fr_1}^\Phi(k)/\Cl(F)_{\fr_1}\xrightarrow{\sim}\cS(\Phi_F\cup\Delta\setminus\{\fl\})_{\fr_0\fr_1\fl}.\]
      Moreover, the map induced by $\sigma^\Phi$ (Remark \ref{bre:frobenius}) on the left-hand side corresponds to the map of changing orientation at $\fl$ on the right-hand side.

  \item the isomorphisms in (1) and (2) are compatible with the respective actions of the Hecke monoid $\dT_\fq$ for every prime $\fq$ of $F$ coprime to $\Delta$ and $\fr_0\fr_1$.

  \item Under the isomorphisms in (1) and (2), the map
      \[X(\Delta)_{\fr_0,\fr_1}^\Phi(k)/\Cl(F)_{\fr_1}\to Z(\Delta)_{\fr_0,\fr_1}^{S^?}(k)/\Cl(F)_{\fr_1}\]
      induced from the composite morphism
      \[X(\Delta)_{\fr_0,\fr_1}^\Phi\to X(\Delta)_{\fr_0,\fr_1}^{S^?}\xrightarrow{\wp_{S^?}}Z(\Delta)_{\fr_0,\fr_1}^{S^?}\]
      coincides with the degeneracy map $\delta$ (resp.\ $\delta^\fl$) for $?=+$ (resp.\ $?=-$) (\Sec\ref{ass:order}).
\end{enumerate}
\end{proposition}

\begin{proof}
Let $(A,\iota_A,C_A,\lambda_A)$ be an element in either $Z(\Delta)_{\fr_0,\fr_1}^{S^\pm}(k)$ or $X(\Delta)_{\fr_0,\fr_1}^\Phi(k)$. We claim that the pair $(A,\iota_A)$ is exceptional (Definition \ref{ade:exceptional}). In the first case, this follows from the fact that one of $\Lie(B)_\tau$ and $\Lie(B)_{\sigma^d\tau}$ vanishes for every $\tau\in\Phi$. In the second case, it is already in the definition.

Then we can identify $Z(\Delta)_{\fr_0,\fr_1}^{S^\pm}(k)$ (resp.\ $X(\Delta)_{\fr_0,\fr_1}^\Phi(k)$) with $\cA(\Delta)^{\pm1}_{\fr_0,\fr_1}$ (resp.\ $\cA(\Delta)^0_{\fr_0,\fr_1}$). The proposition follows from Proposition \ref{apr:ramified}, where part (3) can be checked directly from Definitions \ref{ade:degeneracy}, \ref{bde:hecke}, and the construction.
\end{proof}

Before the end of this section, we discuss some relation between our construction and the Goren--Oort stratification in \cite{TX16}.

We fix a type $S$. It is elementary to see that $\pi$ induces a bijection $\pi\colon\Phi\setminus(S\cup\sigma^{-1}S)\xrightarrow{\sim}\Phi_F\setminus\Delta_S$. For every $\tau\in\Phi_F\setminus\Delta_S$, we denote by $\tilde\tau$ the unique element in $\Phi\setminus(S\cup\sigma^{-1}S)$ such that $\pi\tilde\tau=\tau$. Then $X(\Delta)_{\fr_0,\fr_1}^{S\cup\sigma^{-1}S\cup\{\tilde\tau\}}$ is a divisor of $X(\Delta)_{\fr_0,\fr_1}^{S\cup\sigma^{-1}S}$. We denote by $Z'(\Delta)_{\fr_0,\fr_1}^{S,\tau}$ its image under the isomorphism $\wp_{S\cup\sigma^{-1}S}\colon X(\Delta)_{\fr_0,\fr_1}^{S\cup\sigma^{-1}S}\to Z(\Delta)_{\fr_0,\fr_1}^S$. On the other hand, we have the Goren--Oort divisor $Z(\Delta)_{\fr_0,\fr_1}^{S,\tau}$ of $Z(\Delta)_{\fr_0,\fr_1}^S$ attached to the element $\tau$, defined in \cite{TX16}, which we recall below.

Let $(B,\iota_B,C_B,\lambda_B)$ be an object of $Z(\Delta)_{\fr_0,\fr_1}^S$. For every $\tau'\in\Phi$, we define the essential Verschiebung map (at $\tau')$ $\cV_{B,\r{es},\tau'}\colon \rH^\dr_1(B)_{\tau'}\to\rF_\abs^*\rH^\dr_1(B)_{\sigma^{-1}\tau'}$ to be
\begin{itemize}
  \item the original one $\cV_B\colon\rH^\dr_1(B)_{\tau'}\to\rF_\abs^*\rH^\dr_1(B)_{\sigma^{-1}\tau'}$ if it is not the zero map;
  \item the inverse of $\cF_B\colon \rF_\abs^*\rH^\dr_1(B)_{\sigma^{-1}\tau'}\to\rH^\dr_1(B)_{\tau'}$ if it is an isomorphism.
\end{itemize}
For $n\geq 1$, put
$\cV_{B,\r{es},\tau'}^n\coloneqq \rF_\abs^{(n-1)*}\cV_{B,\r{es},\sigma^{-(n-1)}\tau'}\circ\cdots\circ\cV_{B,\r{es},\tau'}
\colon\rH^\dr_1(B)_{\tau'}\to\rF_\abs^{n*}\rH^\dr_1(B)_{\sigma^{-n}\tau'}$.
Now for $\tau\in\Phi_F\setminus\Delta_S$, let $n_\tau$ be the smallest \emph{positive} integer such that $\sigma^{-n_\tau}\tau\in\Phi_F\setminus\Delta_S$. Then the map $\cV_{B,\r{es},\tilde\tau}^{n_\tau}$ is of rank $1$. According to \cite{TX16}*{Definition 4.6}, the vanishing locus of the restricted map $\cV_{\cB,\r{es},\tilde\tau}^{n_\tau}\res_{\omega(\cB)_{\tilde\tau}}$ for the universal object $(\cB,\iota_\cB,C_\cB,\lambda_\cB)$ is the Goren--Oort stratum, which we denote by $Z(\Delta)_{\fr_0,\fr_1}^{S,\tau}$, attached to the subset $\{\tau\}$. By \cite{TX16}*{Proposition 4.7}, it is smooth of codimension $1$. We have the following comparison result.

\begin{proposition}\label{bpr:divisor}
For every $\tau\in\Phi_F\setminus\Delta_S$, the subschemes $Z'(\Delta)_{\fr_0,\fr_1}^{S,\tau}$ and $Z(\Delta)_{\fr_0,\fr_1}^{S,\tau}$ coincide.
\end{proposition}

\begin{proof}
We may assume that $S$ is not sparse; otherwise, $\Phi_F\setminus\Delta_S=\emptyset$. By \cite{TX16}*{Lemma 4.5}, $Z(\Delta)_{\fr_0,\fr_1}^{S,\tau}$ is also the vanishing locus of the restricted map $\cV_{B,\r{es},\sigma^d\tilde\tau}^{n_\tau}\res_{\omega(B)_{\sigma^d\tilde\tau}}$.

It suffices to check that the two subschemes have the same set of geometric points. We fix an algebraically closed field $k$ containing $\dF_{\ell^{2d}}$. Let $(B,\iota_B,C_B,\lambda_B)$ be an element of $Z(\Delta)_{\fr_0,\fr_1}^S(k)$. Let $(A,\iota_A,C_A,\lambda_A;B,\iota_B,C_B,\lambda_B;\phi)$ be its inverse image in $X_1(\Delta)_{\fr_0,\fr_1}^{S\cup\sigma^{-1}S}(k)$. There are two cases.

Suppose that $n_\tau$ is even. We have the following commutative diagram
\[\xymatrix{
\rH^\dr_1(A)_{\sigma^{d-n_\tau}\tilde\tau} \ar[d]^-{\phi_*}  & \rH^\dr_1(A)_{\sigma^{d-n_\tau+1}\tilde\tau} \ar[d]^-{\phi_*}\ar[l]_-{\cV_A}\ar[r]^-{\cF_A} & \cdots \ar[r]^-{\cF_A} & \rH^\dr_1(A)_{\sigma^d\tilde\tau} \ar[d]^-{\phi_*}_-\simeq  \\
\rH^\dr_1(B)_{\sigma^{d-n_\tau}\tilde\tau}   & \rH^\dr_1(B)_{\sigma^{d-n_\tau+1}\tilde\tau} \ar[l]_-{\cV_B}\ar[r]^-{\cF_B}_-\simeq
& \cdots \ar[r]^-{\cF_B}_-\simeq & \rH^\dr_1(B)_{\sigma^d\tilde\tau}
}\]
where we recall Remark \ref{bre:frobenius}. As $(A,\iota_A,C_A,\lambda_A)$ belongs to $X(\Delta)_{\fr_0,\fr_1}^{S\cup\sigma^{-1}S}(k)$, we know that $\cF_A\rH^\dr_1(A)_{\sigma^{-1}\tau'}=\cV_A\rH^\dr_1(A)_{\sigma\tau'}$ for $\tau'=\sigma^{d-2}\tilde\tau,\dots,\sigma^{d-n_\tau+2}\tilde\tau$. It implies that $\Ker\cV_{B,\r{es},\sigma^d\tilde\tau}^{n_\tau}=\phi_*\cF_A\rH^\dr_1(A)_{\sigma^{d-1}\tilde\tau}$ as $\phi_*\cV_A\rH^\dr_1(A)_{\sigma^{d-n_\tau+1}\tilde\tau}=0$. On the other hand, we know that $(B,\iota_B,C_B,\lambda_B)$ belongs to $Z'(\Delta)_{\fr_0,\fr_1}^{S,\tau}(k)$ if and only if $\cF_A\rH^\dr_1(A)_{\sigma^{d-1}\tilde\tau}=\omega(A)_{\sigma^d\tilde\tau}$, which is equivalent to $\phi_*\cF_A\rH^\dr_1(A)_{\sigma^{d-1}\tilde\tau}=\omega(B)_{\sigma^d\tilde\tau}$, which is then equivalent to $\cV_{B,\r{es},\sigma^d\tilde\tau}^{n_\tau}\omega(B)_{\sigma^d\tilde\tau}=0$. The last condition is equivalent to $(B,\iota_B,C_B,\lambda_B)\in Z(\Delta)_{\fr_0,\fr_1}^{S,\tau}(k)$.

Suppose that $n_\tau$ is odd. We have the following commutative diagram
\[\xymatrix{
\rH^\dr_1(A)_{\sigma^{d-n_\tau}\tilde\tau} \ar[d]^-{\phi_*}_-\simeq  & \rH^\dr_1(A)_{\sigma^{d-n_\tau+1}\tilde\tau} \ar[d]^-{\phi_*}_-\simeq\ar[l]_-{\cV_A} & \cdots \ar[l]_-{\cV_A}\ar[r]^-{\cF_A} & \rH^\dr_1(A)_{\sigma^d\tilde\tau} \ar[d]^-{\phi_*}_-\simeq  \\
\rH^\dr_1(B)_{\sigma^{d-n_\tau}\tilde\tau}   & \rH^\dr_1(B)_{\sigma^{d-n_\tau+1}\tilde\tau} \ar[l]_-{\cV_B}
& \cdots \ar[l]_-{\cV_A}^-\simeq\ar[r]^-{\cF_B}_-\simeq & \rH^\dr_1(B)_{\sigma^d\tilde\tau}
}\]
in which $\phi_*$ on the left is an isomorphism now. By a similar argument, we also have $\Ker\cV_{B,\r{es},\sigma^d\tilde\tau}^{n_\tau}=\phi_*\cF_A\rH^\dr_1(A)_{\sigma^{d-1}\tilde\tau}$. The rest is the same.

Thus the proposition follows.
\end{proof}

\subsection{Frobenius factor}

In this section, we discuss some purely inseparable morphisms between different strata. We start from the following remark on the arithmetic Frobenius action.

\begin{remark}[Arithmetic Frobenius morphism]\label{bre:frobenius}
As $X(\Delta)_{\fr_0,\fr_1}=(\cX(\Delta)_{\fr_0,\fr_1}\otimes\dF_\ell)\otimes_{\dF_\ell}\dF_{\ell^{2d}}$ by definition, the $\ell$-Frobenius $\sigma\in\Gal(\dF_{\ell^{2d}}/\dF_\ell)$ acts on $X(\Delta)_{\fr_0,\fr_1}$ via an isomorphism, known as the arithmetic Frobenius morphism. It induces an isomorphism
\[\sigma^S\colon X(\Delta)_{\fr_0,\fr_1}^S\to X(\Delta)_{\fr_0,\fr_1}^{\sigma^{-1}S}\]
of schemes over the isomorphism $\sigma\colon\Spec\dF_{\ell^{2d}}\to\Spec\dF_{\ell^{2d}}$, for every ample subset $S$ of $\Phi$.

On the other hand, for every type $S$, we have a canonical isomorphism $Z(\Delta)_{\fr_0,\fr_1}^S\simeq (Z(\Delta)_{\fr_0,\fr_1}^{\sigma^{-1}S})^{(\ell)}$ of $\dF_{\ell^{2d}}$-schemes. It induces an isomorphism
\[\sigma_S\colon Z(\Delta)_{\fr_0,\fr_1}^S\to Z(\Delta)_{\fr_0,\fr_1}^{\sigma^{-1}S}\]
of schemes over the isomorphism $\sigma\colon\Spec\dF_{\ell^{2d}}\to\Spec\dF_{\ell^{2d}}$.
It is easy to see that for every proper ample subset set $S$, we have $\wp_{\sigma^{-1}S}\circ\sigma^S=\sigma_{S^\dag}\circ\wp_S$.
\end{remark}

Let $k$ be a field containing $\dF_\ell$. Let $X$ be a scheme over $\Spec k$. For every integer $n\geq 0$, we denote by $X^{\sigma^n}$ the $k$-scheme $X\to\Spec k\xrightarrow{\sigma^n}\Spec k$. We have a chain of morphisms
\[X^{\sigma^n}\to X^{\sigma^{n-1}}\to\cdots\to X\]
of $k$-schemes, induced from the absolute $\ell$-Frobenius morphisms. We denote by $\rF_\sigma^n$ the composition. If $k/\dF_\ell$ is a finite extension of degree $n$, then $\rF_\sigma^n\colon X\to X$ is simply the relative Frobenius morphism of the $k$-scheme $X$. We recall the following definition from \cite{Hel12}*{Definition 4.7}.

\begin{definition}\label{bde:frobenius}
A morphism $f\colon Y\to X$ of $k$-schemes is a \emph{Frobenius factor} if there exist an integer $n\geq 0$ and a morphism $g\colon X^{\sigma^n}\to Y$ of $k$-schemes such that $f\circ g=\rF_\sigma^n$.
\end{definition}

A Frobenius factor is a finite morphism if $X$ is locally of finite type. The follow lemma gives a criterion for when a Frobenius factor is a Frobenius morphism.

\begin{lem}\label{ble:frobenius}
Suppose that $k$ is perfect. Let $X,Y$ be two smooth schemes over $\Spec k$ of pure dimension $n$. Let $f\colon Y\to X$ be a Frobenius factor of degree $\ell^n$ (as a finite morphism) such that $f_*\colon\sT_{Y,y}\to\sT_{X,f(y)}$ vanishes at some closed point $y$ of $Y$. Then
there is a unique isomorphism $g\colon X^\sigma\xrightarrow{\sim}Y$ of $k$-schemes such that $f\circ g=\rF_\sigma^1$.
\end{lem}

\begin{proof}
We may assume that $X$ is irreducible. Then $Y$ must be irreducible. Let $k(X)$ and $k(Y)$ be the function fields of $X$ and $Y$, respectively. Then $k(Y)$ is a purely inseparable extension of $k(X)$ of degree $\ell^n$. By the proof of \cite{Hel12}*{Proposition 4.8}, it suffices to show that for every rational function $F\in k(Y)$, we have $F^\ell\in k(X)$.

If not, then we can find an element $F_0\in k(Y)$ such that $F_0^\ell\not\in k(X)$. Since $k$ is perfect, $F_0$ must be transcendental over $k$. Let $F_1,\dots,F_n$ be $n$ elements in $k(Y)$ that vanishes at $y$ and such that $\{\rd F_1,\dots,\rd F_n\}$ span the cotangent space at $y$. Without lost of generality, we may also assume that $F_0,F_1,\dots F_{n-1}$ are algebraically independent over $k$. As $f_*\colon\sT_{Y,y}\to\sT_{X,f(y)}$ vanishes, we have $F_i\not\in k(X)$ for every $i\leq i\leq n$. In other words, we have a consecutive extension
\[k(X)\subset k(X)(F_0^\ell) \subset k(X)(F_0)\subset k(X)(F_0,F_1)\subset \cdots\subset k(X)(F_0,F_1,\dots,F_{n-1})\]
in which every extension is purely inseparable of degree $\ell$. As $k(X)(F_0,F_1,\dots,F_{n-1})$ is contained in $k(Y)$, the degree of $k(Y)$ over $k(X)$ is at least $\ell^{n+1}$, which is a contradiction. The lemma follows.
\end{proof}

\begin{proposition}\label{bpr:translation}
Let $S$ be a type that is not sparse. The composite morphism
\begin{align}\label{beq:translation}
X(\Delta)_{\fr_0,\fr_1}^{S\cup\sigma^{-1}S}\to X(\Delta)_{\fr_0,\fr_1}^{\sigma^{-1}S}\xrightarrow{\wp_{\sigma^{-1}S}} Z(\Delta)_{\fr_0,\fr_1}^{\sigma^{-1}S}
\end{align}
is a Frobenius factor. It is finite flat of degree $\ell^{|\Phi_S|}$. Here, the first morphism is the canonical embedding as $S\cup\sigma^{-1}S$ contains the type $\sigma^{-1}S$.
\end{proposition}

\begin{proof}
We compute, under the isomorphism $X(\Delta)_{\fr_0,\fr_1}^{\sigma^{-1}S}\xrightarrow{q\circ p^{-1}}Z_1(\Delta)_{\fr_0,\fr_1}^{\sigma^{-1}S}$, the subscheme $X(\Delta)_{\fr_0,\fr_1}^{S\cup\sigma^{-1}S}$ in terms of the moduli interpretation of $Z_1(\Delta)_{\fr_0,\fr_1}^{\sigma^{-1}S}$.

We claim that under Definition \ref{bde:hilbert_bundle}, an element $(B,\iota_B,C_B,\lambda_B;(L_{\sigma^{-1}\tau})_{\tau\in\Phi_S})\in Z_1(\Delta)_{\fr_0,\fr_1}^{\sigma^{-1}S}(T)$ belongs to $qp^{-1}X(\Delta)_{\fr_0,\fr_1}^{S\cup\sigma^{-1}S}(T)$ if and only if for $\tau\in\Phi_S$,
\begin{enumerate}
  \item $\rF_\abs^*L_{\sigma^{-1}\tau}=\cV_B\omega(B)_\tau$ if $\sigma^{d+2}\tau\in S$. Note that in this case $\omega(B)_\tau$ has rank $1$, and $\cV_B\rH^\dr_1(B)_\tau=\rF_\abs^*\rH^\dr_1(B)_{\sigma^{-1}\tau}$;

  \item $\rF_\abs^*L_{\sigma^{-1}\tau}=\cV_B\Pi_{B,\tau}^{-1}L_{\sigma^d\tau}$ if $\sigma^{d+2}\tau\not\in S$. Note that in this case $\sigma^{d+1}\tau\in\Phi_S$ (hence $L_{\sigma^d\tau}$ makes sense), and $\cV_B\rH^\dr_1(B)_\tau=\rF_\abs^*\rH^\dr_1(B)_{\sigma^{-1}\tau}$.
\end{enumerate}
Note that in both cases, $\Pi_{B,\tau}\coloneqq\Pi_B\colon\rH^\dr_1(B)_\tau\to\rH^\dr_1(B)_{\sigma^d\tau}$ is an isomorphism as $\tau\not\in\sigma^{-1}S$.

We assume the claim and prove the proposition. For two elements $\tau,\tau'\in\Phi_S$, we write $\tau\prec\tau'$ if $\tau'=\sigma^{d+1}\tau$ or $\tau'=\tau$. It is an elementary exercise to see that if $S$ is not sparse, then $\prec$ defines a partial order on $\Phi_S$. We define the depth of an element $\tau\in\Phi_S$ to be the largest possible integer $i$ such that there exists a strictly increasing sequence $\tau_1\prec\cdots\prec\tau_i=\tau$. Denote by $\Phi_S^i$ (resp.\ $\Phi_S^{>i}$) the subset $\Phi_S$ consisting of elements of depth $i$ (resp.\ $>i$). We identify $Z_1(\Delta)_{\fr_0,\fr_1}^{\sigma^{-1}S}$ with $\prod_{\tau\in\Phi_S}\dP^1(\rH^\dr_1(\cB)_{\sigma^{-1}\tau})$ by Lemma \ref{ble:hilbert_bundle}. Consider the following commutative diagram
\[\xymatrix{
X(\Delta)_{\fr_0,\fr_1}^{S\cup\sigma^{-1}S}  \ar[r]\ar[d] & \prod_{\tau\in\Phi_S}\dP^1(\rH^\dr_1(\cB)_{\sigma^{-1}\tau}) \ar[d] \\
X(\Delta)_{\fr_0,\fr_1}^{S\cup\sigma^{-1}S,>1}  \ar[r]\ar[d] & \prod_{\tau\in\Phi_S^{>1}}\dP^1(\rH^\dr_1(\cB)_{\sigma^{-1}\tau}) \ar[d] \\
X(\Delta)_{\fr_0,\fr_1}^{S\cup\sigma^{-1}S,>2}  \ar[r]\ar[d] & \prod_{\tau\in\Phi_S^{>2}}\dP^1(\rH^\dr_1(\cB)_{\sigma^{-1}\tau}) \ar[d] \\
\vdots \ar[d]   & \vdots \ar[d] \\
Z(\Delta)_{\fr_0,\fr_1}^{\sigma^{-1}S} \ar[r]^-{=} & Z(\Delta)_{\fr_0,\fr_1}^{\sigma^{-1}S}
}\]
in which all right vertical morphisms are natural projections, and $X(\Delta)_{\fr_0,\fr_1}^{S\cup\sigma^{-1}S,>i}$ is the image of $X(\Delta)_{\fr_0,\fr_1}^{S\cup\sigma^{-1}S}$ under the corresponding projection. Note that for sufficiently large $i$, we have $X(\Delta)_{\fr_0,\fr_1}^{S\cup\sigma^{-1}S,>i}=Z(\Delta)_{\fr_0,\fr_1}^{\sigma^{-1}S}$ by the claim.

Now we fix an integer $i\geq 1$. Define a morphism
\[\prod_{\tau\in\Phi_S^{>i}}\dP^1(\rH^\dr_1(\cB)_{\sigma^{-1}\tau})\to\prod_{\tau\in\Phi_S^i}\dP^1(\rF_\abs^*\rH^\dr_1(\cB)_{\sigma^{-1}\tau})\]
by sending $(B,\iota_B,C_B,\lambda_B;(L_{\sigma^{-1}\tau})_{\tau\in\Phi_S^{>i}})$ to $(B,\iota_B,C_B,\lambda_B;(L'_{\sigma^{-1}\tau})_{\tau\in\Phi_S^i})$ such that
\begin{itemize}
  \item $L'_{\sigma^{-1}\tau}=\cV_B\omega(B)_\tau$ if $\tau$ is maximal with respect to the partial order;

  \item $L'_{\sigma^{-1}\tau}=\cV_B\Pi_{B,\tau}^{-1}L_{\sigma^d\tau}$ if $\sigma^{d+1}\tau\in\Phi_S$.
\end{itemize}
Then the claim implies that we have a Cartesian diagram
\[\xymatrix{
X(\Delta)_{\fr_0,\fr_1}^{S\cup\sigma^{-1}S,>i-1} \ar[r]\ar[d] & \prod_{\tau\in\Phi_S^i}\dP^1(\rH^\dr_1(\cB)_{\sigma^{-1}\tau}) \ar[d] \\
X(\Delta)_{\fr_0,\fr_1}^{S\cup\sigma^{-1}S,>i}  \ar[r] &  \prod_{\tau\in\Phi_S^i}\dP^1(\rF_\abs^*\rH^\dr_1(\cB)_{\sigma^{-1}\tau})
}\]
in which the right vertical morphism is the fiber product of the relative Frobenius morphisms
\[\dP^1(\rH^\dr_1(\cB)_{\sigma^{-1}\tau})\to\rF_\abs^*\dP^1(\rH^\dr_1(\cB)_{\sigma^{-1}\tau})
\simeq\dP^1(\rF_\abs^*\rH^\dr_1(\cB)_{\sigma^{-1}\tau}).\]
In particular, the morphism $X(\Delta)_{\fr_0,\fr_1}^{S\cup\sigma^{-1}S,>i-1}\to X(\Delta)_{\fr_0,\fr_1}^{S\cup\sigma^{-1}S,>i}$ is a Frobenius factor, which is finite flat of degree $\ell^{|\Phi_S^i|}$. The proposition follows as $\Phi_S=\coprod_{i>0}\Phi_S^i$.

To prove the claim, we note the following elementary but crucial fact that $S\cup\sigma^{-1}S\setminus \sigma^{-1}S=\Phi_S$. Thus the claim follows if we can show that the map $\Pi_A\colon\Lie(A)_\tau\to\Lie(A)_{\sigma^d\tau}$ vanishes if and only if $L_{\sigma^{-1}\tau}$ satisfies the corresponding condition in (1) or (2).

Suppose that we are in (1). Consider the following commutative diagram
\[\xymatrix{
\rF_\abs^*\rH^\dr_1(A)_{\sigma^{-1}\tau}  \ar[d]^-{\rF_\abs^*\phi_*}_-\simeq  & \rH^\dr_1(A)_\tau \ar[l]_-{\cV_A}\ar[r]^-{\Pi_A}\ar[d]^-{\phi_*}
& \rH^\dr_1(A)_{\sigma^d\tau} \ar[d]^-{\phi_*}_-\simeq \\
\rF_\abs^*\rH^\dr_1(B)_{\sigma^{-1}\tau}  & \rH^\dr_1(B)_\tau \ar[l]_-{\cV_B}\ar[r]^-{\Pi_{B,\tau}}_-\simeq   & \rH^\dr_1(B)_{\sigma^d\tau}
}\]
in which $\cV_B$ has the trivial kernel. By the definitions of $p$ and $q$, we have $L_{\sigma^{-1}\tau}=\phi_*\omega(A)_{\sigma^{-1}\tau}$. Then $\rF_\abs^*L_{\sigma^{-1}\tau}=\rF_\abs^*\phi_*\omega(A)_{\sigma^{-1}\tau}=\phi_*\cV_A\rH^\dr_1(A)_\tau=\cV_B\phi_*\rH^\dr_1(A)_\tau$. Thus, $\rF_\abs^*L_{\sigma^{-1}\tau}=\cV_B\omega(B)_\tau$ if and only if $\Pi_{B,\tau}\phi_*\rH^\dr_1(A)_\tau=\omega(B)_{\sigma^d\tau}$, which is equivalent to $\Pi_A\rH^\dr_1(A)_\tau=\omega(A)_{\sigma^d\tau}$. However the last equality is equivalent to the vanishing of $\Pi_A\colon\Lie(A)_\tau\to\Lie(A)_{\sigma^d\tau}$.

Suppose that we are in (2). Consider the following commutative diagram
\[\xymatrix{
\rF_\abs^*\rH^\dr_1(A)_{\sigma^{-1}\tau}  \ar[d]^-{\phi_*}_-\simeq  & \rH^\dr_1(A)_\tau \ar[l]_-{\cV_A}\ar[r]^-{\Pi_A}\ar[d]^-{\phi_*}
& \rH^\dr_1(A)_{\sigma^d\tau} \ar[d]^-{\phi_*}_-\simeq \\
\rF_\abs^*\rH^\dr_1(B)_{\sigma^{-1}\tau}  & \rH^\dr_1(B)_\tau \ar[l]_-{\cV_B}\ar[r]^-{\Pi_{B,\tau}}_-\simeq   & \rH^\dr_1(B)_{\sigma^d\tau}
}\]
in which $\cV_B$ has the trivial kernel. By a similar argument as in (1), we know that $\rF_\abs^*L_{\sigma^{-1}\tau}=\cV_B\Pi_{B,\tau}^{-1}L_{\sigma^d\tau}$ if and only if $\Pi_B\phi_*\rH^\dr_1(A)_\tau=L_{\sigma^d\tau}$. The last equality is equivalent to $\Pi_A\rH^\dr_1(A)_\tau=\omega(A)_{\sigma^d\tau}$, which is same as the vanishing of $\Pi_A\colon\Lie(A)_\tau\to\Lie(A)_{\sigma^d\tau}$.
\end{proof}

The above proposition implies that for a type $S$ that is not sparse, the following composite morphism
\begin{align}\label{beq:translation2}
\b{f}\colon Z(\Delta)_{\fr_0,\fr_1}^{\sigma S}\xrightarrow{\wp_{\sigma S\cup S}^{-1}}
X(\Delta)_{\fr_0,\fr_1}^{\sigma S\cup S}\xrightarrow{\eqref{beq:translation}} Z(\Delta)_{\fr_0,\fr_1}^S
\end{align}
is a Frobenius factor of degree $\ell^{|\Phi_S|}$. For $n\geq 1$, we have the iterated morphism
\[\b{f}^n\colon Z(\Delta)_{\fr_0,\fr_1}^{\sigma^n S}\to Z(\Delta)_{\fr_0,\fr_1}^S.\]

\begin{proposition}\label{bpr:translation2}
Suppose $d>1$ and $S=\{\tau,\sigma\tau,\dots,\sigma^{d-1}\tau\}$ for some $\tau\in\Phi$. Then there is a unique isomorphism $g\colon (Z(\Delta)_{\fr_0,\fr_1}^S)^\sigma\xrightarrow{\sim}Z(\Delta)_{\fr_0,\fr_1}^{\sigma^{d-1}S}$ of $\dF_{\ell^{2d}}$-schemes such that
\[\b{f}^{d-1}\circ g=\rF_\sigma^1\colon (Z(\Delta)_{\fr_0,\fr_1}^S)^\sigma\to Z(\Delta)_{\fr_0,\fr_1}^S.\]
In particular, the morphism
\[\b{f}^{2d(d-1)}\colon Z(\Delta)_{\fr_0,\fr_1}^S=Z(\Delta)_{\fr_0,\fr_1}^{\sigma^{2d(d-1)}S}\to Z(\Delta)_{\fr_0,\fr_1}^S\]
coincides with the relative Frobenius morphism over $\Spec\dF_{\ell^{2d}}$.
\end{proposition}

\begin{proof}
By Lemma \ref{ble:zink}, $Z(\Delta)_{\fr_0,\fr_1}^S$ is smooth of dimension $d-1$ over $\Spec\dF_{\ell^{2d}}$. By Proposition \ref{bpr:translation}, we know that $\b{f}^{d-1}$ is a Frobenius factor of degree $\ell^{d-1}$. Thus by Lemma \ref{ble:frobenius}, it suffices to show that $\b{f}^{d-1}_*$ induces the zero map on tangent spaces at, actually, every closed point. Put $Z_i=Z(\Delta)_{\fr_0,\fr_1}^{\sigma^iS}$ for $i=0,\dots,d-1$ for simplicity.

We fix a closed point $z_0$ of $Z_0$ represented by $(B_0,\iota_{B_0},C_{B_0},\lambda_{B_0})$ with the (perfect) residue field $k$. Suppose that the unique closed point $z_i$ in $Z_i$ above $z_0$ under $\b{f}^i$ is represented by $(B_i,\iota_{B_i},C_{B_i},\lambda_{B_i})$. Put $k^\sharp=k[\epsilon]/\epsilon^2$. Recall from the proof of Proposition \ref{bpr:hilbert_base} that we have a canonical isomorphism
\[T_i\coloneqq \sT_{Z_i,z_i}\otimes_{\sO_{Z_i,z_i}}k\simeq\bigoplus_{j=1}^{d-1}V_{i,j}\]
where $V_{i,j}$ is the $1$-dimensional $k$-vector space parameterizing direct factors $\omega(B_i)_{\sigma^{i+j-1}\tau}^\sharp$ of the $k^\sharp$-module $\rH^\dr_1(B_i)_{\sigma^{i+j-1}\tau}^\sharp\coloneqq\rH^\dr_1(B_i)_{\sigma^{i+j-1}\tau}\otimes_kk^\sharp$ lifting $\omega(B_i)_{\sigma^{i+j-1}\tau}$.

We prove by induction on $i=0,\dots,d-1$ that $\b{f}^i_*$ induces isomorphisms $V_{i,j}\xrightarrow{\sim}V_{0,i+j}$ for $1\leq j\leq d-1-i$, and annihilates $V_{i,j}$ for $j>d-1-i$.

For the case where $i=0$, it is trivial. Suppose that we know for the case where $i<d-1$, and consider the case of $i+1$. However the case of $i+1$ for $Z_0$ follows from the case where $i=1$ for $Z_i$. Therefore, it suffices to consider the case where $i=1$, that is, the morphism $\b{f}\colon Z_1\to Z_0$. Note that we have isomorphisms
\[\rF_\abs^*\rH^\dr_1(B)_{\sigma^{d-1}\tau}\xleftarrow{\cV_B}\rH^\dr_1(B)_{\sigma^d\tau}\xrightarrow{\Pi_B}\rH^\dr_1(B)_\tau\]
for every object $(B,\iota_B,C_B,\lambda_B)$ of $Z_0$. Therefore, $V_{0,1}$ is canonically isomorphic to the $1$-dimensional $k$-vector space parameterizing direct factors $M_{\sigma^{d-1}\tau}^\sharp$ of $\rF_\abs^*\rH^\dr_1(B_0)_{\sigma^{d-1}\tau}^\sharp$ lifting $M_{\sigma^{d-1}\tau}\coloneqq\cV_{B_0}\Pi_{B_0,\sigma^d\tau}^{-1}\omega(B_0)_\tau$.

By the proof of Proposition \ref{bpr:translation}, we already have canonical isomorphisms $\b{f}_*\colon V_{1,j}\xrightarrow{\sim}V_{0,1+j}$ for $1\leq j\leq d-2$; and the subspace $V_{1,d-1}$ is canonically isomorphic to the $k$-vector space parameterizing direct factors $L_{\sigma^{d-1}\tau}^\sharp$ of $\rH^\dr_1(B_0)_{\sigma^{d-1}\tau}^\sharp$ lifting $M_{\sigma^{d-1}\tau}$ (Remark \ref{bre:hodge}). Moreover, $\b{f}_*V_{1,d-1}\subset V_{0,1}$ and it sends $L_{\sigma^{d-1}\tau}^\sharp$ to $\rF_\abs^*L_{\sigma^{d-1}\tau}^\sharp$. Therefore, $\b{f}_*V_{1,d-1}=0$. The proposition follows.
\end{proof}

\subsection{A semistable model for $d=3$}
\label{bss:semistable}

Now we assume that $d=3$. As in Example \ref{bex:cube}, we identify $\Phi$ with the set $\dZ/6\dZ=\{0,1,2,3,4,5\}$ such that $\sigma i=i+1$, and $\Phi_F$ with the set $\dZ/3\dZ=\{0,1,2\}$ such that $\Phi\to\Phi_F$ is the natural map of modulo $3$. For an ample subset $S=\{i,j,k,\cdots\}$, we directly write $X(\Delta)_{\fr_0,\fr_1}^{ijk\cdots}$ for $X(\Delta)_{\fr_0,\fr_1}^S$. If $S^\dag=\{i,j,k\}$, then we write the morphism $\wp_S$ (Notation \ref{bno:hilbert_subbundle}) as $\wp_{ijk\cdots}\colon X(\Delta)_{\fr_0,\fr_1}^{ijk\cdots}\to Z(\Delta)_{\fr_0,\fr_1}^{ijk}$.

\begin{definition}\label{bde:blowup}
Let
\begin{align}\label{beq:blowup}
\pi\colon\cY(\Delta)_{\fr_0,\fr_1}\to\cX(\Delta)_{\fr_0,\fr_1}\otimes\dZ_{\ell^6}
\end{align}
\footnote{The projection $\pi\colon\Phi\to\Phi_F$ will not be used anymore, including the main text. Thus, using $\pi$ for this morphism will not cause confusion.} be the composition of the blow-up of $\cX(\Delta)_{\fr_0,\fr_1}\otimes\dZ_{\ell^6}$ along the closed subscheme $X(\Delta)_{\fr_0,\fr_1}^{024}$ and the further blow-up along the strict transform of $X(\Delta)_{\fr_0,\fr_1}^{135}$. We call \eqref{beq:blowup} the \emph{canonical semistable resolution} of $\cX(\Delta)_{\fr_0,\fr_1}\otimes\dZ_{\ell^6}$ (see Remark \ref{bre:blowup}).
\end{definition}

It follows from \cite{GS95}*{Example 6.15} that $\cY(\Delta)_{\fr_0,\fr_1}$ is a proper strictly semistable scheme over $\Spec\dZ_{\ell^6}$.

\begin{remark}\label{bre:blowup}
Let $\pi'\colon\cY'(\Delta)_{\fr_0,\fr_1}\to\cX(\Delta)_{\fr_0,\fr_1}\otimes\dZ_{\ell^6}$ be the composition of the blow-up of $\cX(\Delta)_{\fr_0,\fr_1}\otimes\dZ_{\ell^6}$ along the closed subscheme $X(\Delta)_{\fr_0,\fr_1}^{135}$ and the further blow-up along the strict transform of $X(\Delta)_{\fr_0,\fr_1}^{024}$. Then $\cY(\Delta)_{\fr_0,\fr_1}$ and $\cY'(\Delta)_{\fr_0,\fr_1}$ are canonically isomorphic as schemes over $\cX(\Delta)_{\fr_0,\fr_1}$. In particular, $\sigma$ acts on $\cY(\Delta)_{\fr_0,\fr_1}$ as an isomorphism.
\end{remark}

We put $Y(\Delta)_{\fr_0,\fr_1}=\cY(\Delta)_{\fr_0,\fr_1}\otimes_{\dZ_{\ell^6}}\dF_{\ell^6}$. We define some strata of $Y(\Delta)_{\fr_0,\fr_1}$ indexed by ample subsets $S$ of $\{0,1,2,3,4,5\}$ as follows.
\begin{itemize}
  \item If $|S|=3$ (that is, $S$ is a type), then let $Y(\Delta)_{\fr_0,\fr_1}^S$ be the strict transform of $X(\Delta)_{\fr_0,\fr_1}^S$ under $\pi$.

  \item If $|S|=4$, then we can write $S=S_1\cup S_2$ uniquely (up to permutation) where both $S_1$ and $S_2$ are types. Put $Y(\Delta)_{\fr_0,\fr_1}^S=Y(\Delta)_{\fr_0,\fr_1}^{S_1}\cap Y(\Delta)_{\fr_0,\fr_1}^{S_2}$.

  \item If $|S|=5$, then we can write $S=S_1\cup S_2$ uniquely where $S_1$ is a sparse type and $S_2$ is a type. Put $Y(\Delta)_{\fr_0,\fr_1}^S=Y(\Delta)_{\fr_0,\fr_1}^{S_1}\cap Y(\Delta)_{\fr_0,\fr_1}^{S_2}$.

  \item Put $Y(\Delta)_{\fr_0,\fr_1}^{012345}=Y(\Delta)_{\fr_0,\fr_1}^{024}\cap Y(\Delta)_{\fr_0,\fr_1}^{135}$.
\end{itemize}

The dual reduction building of the strictly semistable scheme $\cY(\Delta)_{\fr_0,\fr_1}$ can be described as the following diagram (compare with the one in Example \ref{bex:cube})
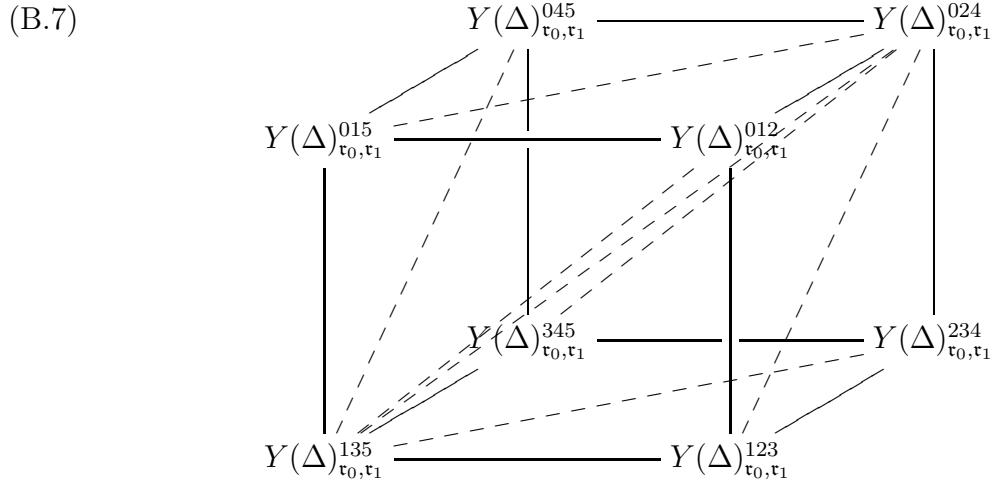
\begin{align}\label{beq:reduction}
\xymatrix{
& Y(\Delta)_{\fr_0,\fr_1}^{045} \ar@{-}[rr] && Y(\Delta)_{\fr_0,\fr_1}^{024} \ar@{--}[llldddd]
\ar@{--}[ldddd]\ar@{--}[llddd]\ar@{--}[llld] \\
Y(\Delta)_{\fr_0,\fr_1}^{015} \ar@{-}[rr]\ar@{-}[ur]  && Y(\Delta)_{\fr_0,\fr_1}^{012} \ar@{-}[ur]  \\ &&&\\
& Y(\Delta)_{\fr_0,\fr_1}^{345} \ar@{-}[rr]|\hole\ar@{-}[uuu]|!{[u],[uu]}\hole && Y(\Delta)_{\fr_0,\fr_1}^{234} \ar@{-}[uuu] \\
Y(\Delta)_{\fr_0,\fr_1}^{135} \ar@{--}[rrru]\ar@{--}[rruuu]\ar@{--}[ruuuu]
\ar@{-}[rr]\ar@{-}[ur]\ar@{-}[uuu]  && Y(\Delta)_{\fr_0,\fr_1}^{123}
\ar@{-}[ur]\ar@{-}[uuu] }
\end{align}
in which the dashed lines indicates new (two-dimensional) intersection of irreducible components on the special fiber $Y(\Delta)_{\fr_0,\fr_1}$ caused by blow-ups. The line, either continuous or dashed, connecting two vertices representing $Y(\Delta)_{\fr_0,\fr_1}^S$ and $Y(\Delta)_{\fr_0,\fr_1}^{S'}$ (with $S$ and $S'$ two different types) represents the subscheme $Y(\Delta)_{\fr_0,\fr_1}^{S\cup S'}$.

\begin{proposition}\label{bpr:blowup}
We have
\begin{enumerate}
  \item the restricted morphism $\pi\colon Y(\Delta)_{\fr_0,\fr_1}^{i(i+1)(i+2)}\to X(\Delta)_{\fr_0,\fr_1}^{i(i+1)(i+2)}$
      is the blow-up along $X(\Delta)_{\fr_0,\fr_1}^{i(i+1)(i+2)(i+3)(i+5)}$, for every $i\in\{0,1,2,3,4,5\}$;

  \item the restricted morphism
      $\pi\colon Y(\Delta)_{\fr_0,\fr_1}^{024}\to X(\Delta)_{\fr_0,\fr_1}^{024}$
      is the blow-up along $X(\Delta)_{\fr_0,\fr_1}^{012345}$ followed by the blow-up along the strict transform of $X(\Delta)_{\fr_0,\fr_1}^{01234}\cup X(\Delta)_{\fr_0,\fr_1}^{01245}\cup X(\Delta)_{\fr_0,\fr_1}^{02345}$;

  \item the restricted morphism
      $\pi\colon Y(\Delta)_{\fr_0,\fr_1}^{135}\to X(\Delta)_{\fr_0,\fr_1}^{135}$
      is the blow-up along $X(\Delta)_{\fr_0,\fr_1}^{012345}$ followed by the blow-up along the strict transform of $X(\Delta)_{\fr_0,\fr_1}^{01235}\cup X(\Delta)_{\fr_0,\fr_1}^{01345}\cup X(\Delta)_{\fr_0,\fr_1}^{12345}$;

  \item the restricted morphism
      $\pi\colon Y(\Delta)_{\fr_0,\fr_1}^{i(i+1)(i+2)(i+3)}\to X(\Delta)_{\fr_0,\fr_1}^{i(i+1)(i+2)(i+3)}$
      is an isomorphism, for every $i\in\{0,1,2,3,4,5\}$;

  \item the restricted morphism
      $\pi\colon Y(\Delta)_{\fr_0,\fr_1}^{i(i+1)(i+2)(i+4)}\to X(\Delta)_{\fr_0,\fr_1}^{i(i+1)(i+2)(i+4)}$
      is the blow-up along $X(\Delta)_{\fr_0,\fr_1}^{012345}$, for every $i\in\{0,1,2,3,4,5\}$;

  \item the restricted morphism
      $\pi\colon Y(\Delta)_{\fr_0,\fr_1}^{i(i+1)(i+2)(i+3)(i+4)}\to X(\Delta)_{\fr_0,\fr_1}^{i(i+1)(i+2)(i+3)(i+4)}$
      is a $\dP^1$-bundle, for every $i\in\{0,1,2,3,4,5\}$;

  \item we have $Y(\Delta)_{\fr_0,\fr_1}^{(0)}=\coprod_{|S|=3}Y(\Delta)_{\fr_0,\fr_1}^S$;

  \item we have $Y(\Delta)_{\fr_0,\fr_1}^{(1)}=\coprod_{|S|\geq 4}Y(\Delta)_{\fr_0,\fr_1}^S$.
\end{enumerate}
Here in (7,8), we adopt the notation in \Sec\ref{ss:semistable_schemes} for the strictly semistable scheme $\cY(\Delta)_{\fr_0,\fr_1}$.
\end{proposition}

\begin{proof}
By Remark \ref{bre:blowup}, we can change the order of blow-up in the decomposition of $\pi$. Thus we may assume $i=1$ in (1,4,5,6) without lost of generality.

For (1), we decompose $\pi$ as $\pi_1\circ\pi_2$ where $\pi_1$ is the blow-up along $X(\Delta)_{\fr_0,\fr_1}^{024}$ and $\pi_2$ is the blow-up along the strict transform of $X(\Delta)_{\fr_0,\fr_1}^{135}$. For the first blow-up, the induced morphism $\pi_1\colon X'(\Delta)_{\fr_0,\fr_1}^{012}\to X(\Delta)_{\fr_0,\fr_1}^{012}$ is an isomorphism where $X'(\Delta)_{\fr_0,\fr_1}^{012}$ is the strict transform of $X(\Delta)_{\fr_0,\fr_1}^{012}$. For the second blow-up, the induced morphism $\pi_2\colon Y(\Delta)_{\fr_0,\fr_1}^{012}\to X'(\Delta)_{\fr_0,\fr_1}^{012}$ is the blow-up along $X(\Delta)_{\fr_0,\fr_1}^{012}\cap X(\Delta)_{\fr_0,\fr_1}^{135}=X(\Delta)_{\fr_0,\fr_1}^{01235}$, regarded as a subscheme of $X'(\Delta)_{\fr_0,\fr_1}^{012}$. Part (1) is proved.

For (2) (and similarly for (3)), we adopt the same strategy as in (1). We decompose $\pi$ as $\pi_1\circ\pi_2$ where $\pi_1$ is the blow-up along $X(\Delta)_{\fr_0,\fr_1}^{135}$ and $\pi_2$ is the blow-up along the strict transform of $X(\Delta)_{\fr_0,\fr_1}^{024}$. Then after restriction to $Y(\Delta)_{\fr_0,\fr_1}^{024}$, $\pi_1$ is the blow-up along $X(\Delta)_{\fr_0,\fr_1}^{012345}$, and $\pi_2$ is the blow-up along the strict transform of $X(\Delta)_{\fr_0,\fr_1}^{01234}\cup X(\Delta)_{\fr_0,\fr_1}^{01245}\cup X(\Delta)_{\fr_0,\fr_1}^{02345}$.

For (4), by definition, $Y(\Delta)_{\fr_0,\fr_1}^{0123}=Y(\Delta)_{\fr_0,\fr_1}^{012}\cap Y(\Delta)_{\fr_0,\fr_1}^{123}$. By (8), whose proof only uses (1--3), we know that $Y(\Delta)_{\fr_0,\fr_1}^{0123}$ is smooth of pure dimension $2$. Thus (4) follows from (1).

For (5), by definition, $Y(\Delta)_{\fr_0,\fr_1}^{0124}=Y(\Delta)_{\fr_0,\fr_1}^{012}\cap Y(\Delta)_{\fr_0,\fr_1}^{024}$. Then it follows from (2) and (8).

For (6), by definition, $Y(\Delta)_{\fr_0,\fr_1}^{01234}=Y(\Delta)_{\fr_0,\fr_1}^{123}\cap Y(\Delta)_{\fr_0,\fr_1}^{024}$. From (8) and (1), we know that $Y(\Delta)_{\fr_0,\fr_1}^{01234}$ has to be the exceptional divisor over $X(\Delta)_{\fr_0,\fr_1}^{01234}$ in the blow-up $Y(\Delta)_{\fr_0,\fr_1}^{123}\to X(\Delta)_{\fr_0,\fr_1}^{123}$. Thus (6) follows.

For (7), note that $Y(\Delta)_{\fr_0,\fr_1}^S$ are smooth proper schemes over $\Spec\dF_{\ell^6}$ of pure dimension $3$ by (1--3) and Proposition \ref{bpr:zink}. Therefore, it suffices to show that $Y(\Delta)_{\fr_0,\fr_1}=\bigcup_{|S|=3}Y(\Delta)_{\fr_0,\fr_1}^S$. However, by Proposition \ref{bpr:zink} (4), we know that $\bigcup_{|S|=3}\pi Y(\Delta)_{\fr_0,\fr_1}^S=X(\Delta)_{\fr_0,\fr_1}$, which suffices.

For (8), it follows from (7) and the definition of $Y(\Delta)_{\fr_0,\fr_1}^S$ for $|S|\geq 4$.
\end{proof}

\begin{remark}\label{bre:hecke_cube}
Suppose that we have another ideal $\fs_0\subset\fr_0$ of $O_F$ coprime to $\Delta$ and $\fr_1$. For each $\fd\in\fD(\fs_0,\fr_1)$, the composite morphism $\cY(\Delta)_{\fs_0,\fr_1}\xrightarrow{\pi}\cX(\Delta)_{\fs_0,\fr_1}\xrightarrow{\delta^\fd}\cX(\Delta)_{\fr_0,\fr_1}$ lifts uniquely to a morphism (with the same notation) $\delta^\fd\colon \cY(\Delta)_{\fs_0,\fr_1}\to \cY(\Delta)_{\fr_0,\fr_1}$. It is finite \'{e}tale. We write $\delta=\delta^{O_F}$ for simplicity as always. The similar property in Lemma \ref{ble:strata} (4) holds for $Y(\Delta)_{\fr_0,\fr_1}$ as well. Similarly, the action of the ray class group $\Cl(F)_{\fr_1}$ on $\cX(\Delta)_{\fr_0,\fr_1}$ lifts uniquely to an action on $\cY(\Delta)_{\fr_0,\fr_1}$.

For a prime $\fq$ of $F$ coprime to $\Delta$ and $\fr_0\fr_1$, the Hecke monoid $\dT_\fq$ acts on $\cY(\Delta)_{\fr_0,\fr_1}$ via \'{e}tale correspondences (Definition \ref{de:correspondence}) in a fashion similar to $\cX(\Delta)_{\fr_0,\fr_1}$ as follows:
\begin{itemize}
  \item $\rT_\fq$ acts by the correspondence
     \[\cY(\Delta)_{\fr_0,\fr_1}\xleftarrow{\delta}\cY(\Delta)_{\fr_0\fq,\fr_1}\xrightarrow{\delta^\fq}\cY(\Delta)_{\fr_0,\fr_1};\]

  \item $\rS_\fq$ (resp.\ $\rS_\fq^{-1}$) acts via the morphism given by $\fq$ (resp.\ $\fq^{-1}$) in $\Cl(F)_{\fr_1}$.
\end{itemize}
\end{remark}

\begin{remark}
All discussions in this appendix except Proposition \ref{bpr:hilbert_sparse} are valid for general neat level structures, not necessarily the one in Definition \ref{bde:hilbert_shimura} and those of other moduli functors.
\end{remark}

\end{document}